%% file: TesiPoggesiarxiv.tex
\documentclass[11 pt, a4paper,titlepage]{book}
\usepackage{amsmath,amsthm,amssymb,latexsym,enumerate,color,hyperref}
\usepackage{graphicx}

\usepackage[english]{babel}
\usepackage{verbatim,mathrsfs,paralist}
\usepackage{epstopdf}
\usepackage{indentfirst}
\usepackage[a4paper]{geometry}

\usepackage{frontespizio}
\usepackage[titletoc]{appendix}

\usepackage{emptypage}

\usepackage{url}

\usepackage[multiple]{footmisc}

\usepackage{caption}

\usepackage{bbm}

\begin{document}

\newtheorem{thm}{Theorem}[chapter]
\newtheorem{prop}[thm]{Proposition}
\newtheorem{lem}[thm]{Lemma}
\newtheorem{cor}[thm]{Corollary}
\newtheorem{rem}[thm]{Remark}
\newtheorem*{defn}{Definition}

\newtheorem{SBT}{Theorem}
\renewcommand*{\theSBT}{\Alph{SBT}}


\newcommand{\DD}{\mathbb{D}}
\newcommand{\NN}{\mathbb{N}}
\newcommand{\ZZ}{\mathbb{Z}}
\newcommand{\QQ}{\mathbb{Q}}
\newcommand{\RR}{\mathbb{R}}
\newcommand{\CC}{\mathbb{C}}
\renewcommand{\SS}{\mathbb{S}}

\newcommand{\re}{\mathop{\mathrm{Re}}}   
\newcommand{\im}{\mathop{\mathrm{Im}}}   
\newcommand{\dist}{\mathop{\mathrm{dist}}}  
\newcommand{\link}{\mathop{\circ\kern-.35em -}}
\newcommand{\spn}{\mathop{\mathrm{span}}}   
\newcommand{\ind}{\mathop{\mathrm{ind}}}   
\newcommand{\rank}{\mathop{\mathrm{rank}}}   
\newcommand{\Fix}{\mathop{\mathrm{Fix}}}   
\newcommand{\codim}{\mathop{\mathrm{codim}}}   
\newcommand{\conv}{\mathop{\mathrm{conv}}}   
\newcommand{\epsi}{\mbox{$\varepsilon$}}
\newcommand{\cl}{\overline}
\newcommand{\pa}{\partial}
\newcommand{\ve}{\varepsilon}
\newcommand{\zi}{\zeta}
\newcommand{\Si}{\Sigma}
\newcommand{\cA}{{\mathcal A}}
\newcommand{\cG}{{\mathcal G}}
\newcommand{\cI}{{\mathcal I}}
\newcommand{\cJ}{{\mathcal J}}
\newcommand{\cK}{{\mathcal K}}
\newcommand{\cL}{{\mathcal L}}
\newcommand{\cN}{{\mathcal N}}
\newcommand{\cR}{{\mathcal R}}
\newcommand{\cS}{{\mathcal S}}
\newcommand{\cT}{{\mathcal T}}
\newcommand{\cU}{{\mathcal U}}
\newcommand{\OM}{\Omega}
\newcommand{\B}{\bullet}
\newcommand{\ol}{\overline}
\newcommand{\ul}{\underline}
\newcommand{\vp}{\varphi}
\newcommand{\AC}{\mathop{\mathrm{AC}}}   
\newcommand{\Lip}{\mathop{\mathrm{Lip}}}   
\newcommand{\es}{\mathop{\mathrm{esssup}}}   
\newcommand{\les}{\mathop{\mathrm{les}}}   
\newcommand{\nid}{\noindent}
\newcommand{\pzr}{\phi^0_R}
\newcommand{\pir}{\phi^\infty_R}
\newcommand{\psr}{\phi^*_R}
\newcommand{\pow}{\frac{N}{N-1}}
\newcommand{\ncl}{\mathop{\mathrm{nc-lim}}}   
\newcommand{\nvl}{\mathop{\mathrm{nv-lim}}}  
\newcommand{\la}{\lambda}
\newcommand{\La}{\Lambda}    
\newcommand{\de}{\delta}    
\newcommand{\fhi}{\varphi} 
\newcommand{\ga}{\gamma}    
\newcommand{\ka}{\kappa}   

\newcommand{\core}{\heartsuit}
\newcommand{\diam}{\mathrm{diam}}

\newcommand{\lan}{\langle}
\newcommand{\ran}{\rangle}
\newcommand{\tr}{\mathop{\mathrm{tr}}}
\newcommand{\diag}{\mathop{\mathrm{diag}}}
\newcommand{\dv}{\mathop{\mathrm{div}}}

\newcommand{\al}{\alpha}
\newcommand{\be}{\beta}
\newcommand{\Om}{\Omega}
\newcommand{\na}{\nabla}

\newcommand{\cC}{\mathcal{C}}
\newcommand{\cM}{\mathcal{M}}
\newcommand{\nr}{\Vert}
\newcommand{\De}{\Delta}
\newcommand{\cX}{\mathcal{X}}
\newcommand{\cP}{\mathcal{P}}
\newcommand{\om}{\omega}
\newcommand{\si}{\sigma}
\newcommand{\te}{\theta}
\newcommand{\Ga}{\Gamma}

\newcommand{\pCap}{\operatorname{Cap}}




\newcommand{\eps}{\varepsilon}
\newcommand{\ovr}{\overline}

\theoremstyle{plain}

\newtheorem*{principle1}{First Principle}
\newtheorem*{principle2}{Second Principle}
\newtheorem*{principle3}{Third Principle}
\newtheorem*{principle4}{Fourth Principle}




\theoremstyle{remark}
\newtheorem*{sketch alternative}{Outline of an alternative proof}
\newtheorem*{alternative}{Alternative proof}
\newtheorem*{sketch}{Outline of the proof}
\newtheorem*{hint}{Hint}

\theoremstyle{definition}
\newtheorem*{notations}{Notation}{\bf}{}


\theoremstyle{plain} 
\newtheorem{theorem}{Theorem}[chapter]
\newtheorem{proposition}[theorem]{Proposition}
\newtheorem{lemma}[theorem]{Lemma}
\newtheorem{corollary}[theorem]{Corollary}

\theoremstyle{definition}
\newtheorem{remark}[theorem]{Remark}
\newtheorem{definition}[theorem]{Definition}
\newtheorem{exercise}{Exercise}[chapter]
\newtheorem{open problem}[theorem]{Open problem}



\newcommand{\cSC}{{\mathcal C}_S}
\newcommand{\cH}{{\mathcal H}}

\newcommand{\DLB}{\De_{LB} \, }


\makeatletter
\renewenvironment{quotation}
                 {\list{}{\listparindent 1.5em
                          \rightmargin \leftmargin 
                                     \parsep \z@ \@plus\p@}
                                     \item\relax}
                                     {\endlist}
\makeatother                                     

%

\frontmatter

\begin{titlepage}

\linespread{1.1}
\pagestyle{myheadings}

\thispagestyle{empty}
\newgeometry{top=3cm,bottom=2cm}
\begin{center}
	\vspace{-4cm}
	\includegraphics[scale=0.41]{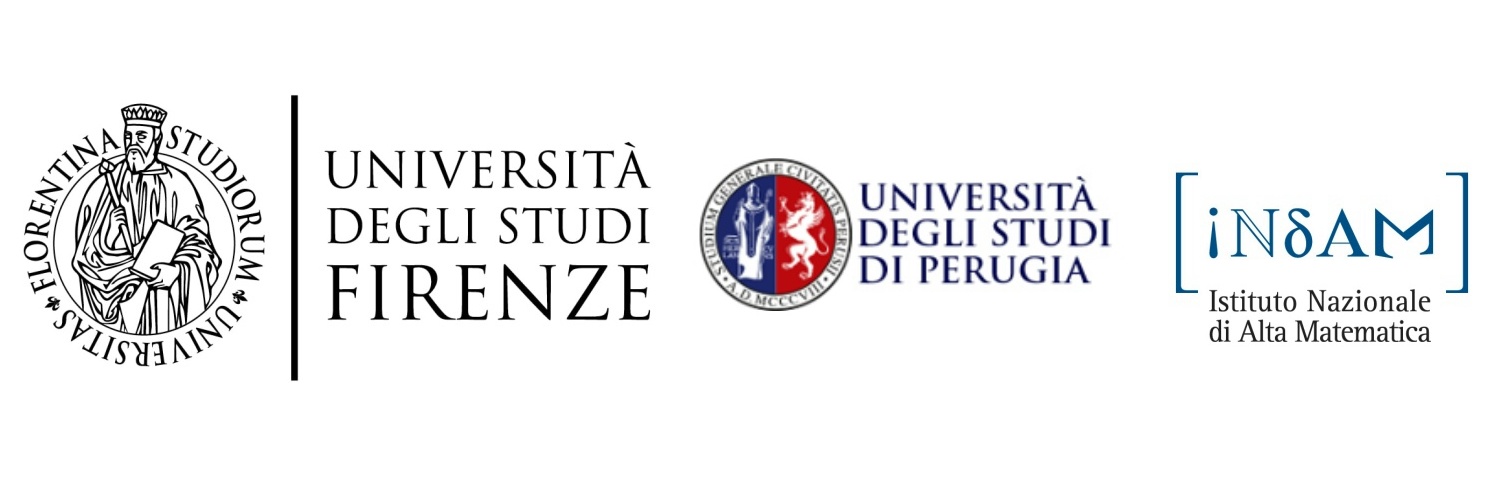}
	\par
	 { 
	  Universit\`a di Firenze, Universit\`a di Perugia, INdAM  consorziate nel {\small CIAFM}\\}
	 \vspace{1mm}
\vskip .8cm
	\par \vspace{2mm}
	\large
	\textbf{DOTTORATO DI RICERCA\\
	IN MATEMATICA, INFORMATICA, STATISTICA}\\
		 \vskip.2cm
		 CURRICULUM IN MATEMATICA\\
	CICLO XXXI
	\par \vspace{5mm}
	\large
	
		 
	 {\bf Sede amministrativa Universit\`a degli Studi di Firenze}\\
	Coordinatore Prof.~Graziano Gentili
	\par \vspace{8mm}
	\huge
	\textbf{ The Soap Bubble Theorem\\ and Serrin's problem:\\ quantitative symmetry}
	\par \vspace{5mm}
	
	\large
	Settore Scientifico Disciplinare MAT/05
\end{center}
\par \vspace{10mm}

\normalsize
\hspace{1cm}\begin{minipage}{0.42\linewidth}
	\textbf{Dottorando}:
	\\
	{Giorgio Poggesi}
\end{minipage}
\hspace{2cm}
\begin{minipage}{0.42\linewidth}
	\textbf{Tutore}
	\\
	{Prof. Rolando Magnanini}
	
\end{minipage}
\par \vspace{10mm}

\begin{center}
	\begin{minipage}{0.30\linewidth}
		\textbf{Coordinatore}
		\\
		{Prof.~Graziano Gentili}
	\end{minipage}
\end{center}
\par \vspace{9mm}
\begin{center}
	\hrule
	\par \vspace{5mm}
	Anni 2015/2018
	
\end{center}
\restoregeometry

%
%
%
%
%
%

\end{titlepage}


\chapter*{Acknowledgements}

First and foremost I thank my advisor Rolando Magnanini for introducing me to mathematical research. I really appreciated his constant guidance and encouraging criticism, as well as the many remarkable opportunities that he gave me during the PhD period.

I am also indebted with all the professors who, in these three years, have invited me at their institutions in order to collaborate in mathematical research:
Xavier Cabr\'e at the Universitat Polit\`{e}cnica de Catalunya (UPC) in Barcelona, Daniel Peralta-Salas at the Instituto de Ciencias Matem\'aticas (ICMAT) in Madrid, Shigeru Sakaguchi and Kazuhiro Ishige at the Tohoku University in Sendai, Lorenzo Brasco at the Universit\`{a} di Ferrara.
It has been a honour for me to interact and work with them.

I would like to thank Paolo Salani, Andrea Colesanti, Chiara Bianchini, and Elisa Francini for their kind and constructive availability.
Also, I thank my friend and colleague Diego Berti, with whom I shared Magnanini as advisor, for the long and good time spent together.

Special thanks to my partner Matilde, my parents, and the rest of my family. Their unfailing and unconditional support has been extremely important to me.

\tableofcontents

\mainmatter
\chapter*{Introduction}
\markboth{Introduction}{Introduction}
\addcontentsline{toc}{chapter}{Introduction}

\begin{quotation}
The thesis is divided in two parts. The first
%
%
is based on the results obtained in \cite{MP, MP2, Pog} and further improvements that give the title to the thesis. The following introduction is dedicated {\it only} to this part.

The second part, entitled ``Other works'', contains a couple of papers, as they were published.
The first are the lecture notes \cite{CP} to which the author of this thesis collaborated. They
%
%
have several points of contact with the first part. 
The second (\cite{MPLittlewood}) is an article on a different topic, published during my studies as a PhD student.
\end{quotation}

The pioneering symmetry results obtained by A. D. Alexandrov \cite{Al1}, \cite{Al2} and  J. Serrin \cite{Se} are now classical but still influential. The former -- the well-known Soap Bubble Theorem -- states that a compact hypersurface, embedded in $\RR^N$, that has constant mean curvature must be a sphere. The latter -- Serrin's symmetry result -- has to do with certain overdetermined problems for partial differential equations.
In its simplest formulation, it states that  the overdetermined boundary value problem 
\begin{eqnarray}
\label{serrin1}
&\De u=N \ \mbox{ in } \ \Om, \quad u=0 \ \mbox{ on } \ \Ga, \\
\label{serrin2}
&u_\nu=R \ \mbox{ on } \ \Ga, 
\end{eqnarray}
admits a solution 
for some positive constant $R$ if and only if $\Om$ is a ball of radius $R$ and, up to translations, $u(x)=(|x|^2-R^2)/2$. Here, $\Om$ denotes a bounded domain in $\RR^N$, $N\ge 2$, with sufficiently smooth boundary $\Ga$, say $C^2$, and $u_\nu$ is the outward normal derivative of $u$ on $\Ga$. 
This result inaugurated a new and fruitful field in mathematical research at the confluence of Analysis and Geometry.

Both problems have many applications to other areas of mathematics and natural sciences. In fact (see the next section), Alexandrov's theorem is related with soap bubbles and the classical isoperimetric problem, while Serrin's result -- as Serrin himself explains in \cite{Se} -- was actually motivated by two concrete problems in mathematical physics regarding the torsion of a straight solid bar and the tangential stress of a fluid on the walls of a rectilinear pipe.

\bigskip
%
%
\noindent{\bf Some motivations.}
Let us consider a free soap bubble in the space, that is a closed surface $\Ga$ in $\RR^3$ made of a soap film enclosing a domain $\Om$ (made of air) of fixed volume.
It is known that (see \cite{Roslibro}) if the surface is in equilibrium then it must be a minimizing minimal surface (in the terminology of Chapter \ref{chap:LNCIME}).
%
%
%
This means that its area within any compact region increases when the surface is perturbed within that region.
Thus, soap bubbles solve the classical isoperimetric problem, that asks which sets $\Om$ in $\RR^N$ (if they exist) minimize the surface area for a given volume.
%
%
Put in other terms, soap bubbles attain the sign of equality in the classical isoperimetric inequality. Hence, they must be spheres (see also Theorem \ref{thm:servepertesiLNisoperimetric}).
%
%

A standard proof of the isoperimetric inequality that hinges on rearrangement techniques can be found in \cite[Section 1.12]{PZ}.
Here, we just want to underline the relation between the isoperimetric problem and Alexandrov's theorem.

%
%
In this introduction, we adopt the terminology and techniques pertaining to \textit{shape derivatives}, which will later be useful also to give a motivation to Serrin's problem. 
See Chapter \ref{chap:LNCIME}, for a more general treatment pertaining the theory of minimal surfaces.
%
%

Thus, we assume the existence of a set $\Om$ whose boundary $\Ga$ (of class $C^2$) minimizes its surface area among sets of given volume. 
%
%
%
%
Then, we consider the family of domains $\Om_t$ such that $\Om_0= \Om $, where
\begin{equation}\label{intro:eq:evolutiondomains1}
\Om_t = \cM_t (\Om) 
\end{equation}
and $\cM_t: \RR^N \to \RR^N$ is a mapping such that
\begin{equation}\label{intro:eq:evolutiondomains2}
\cM_0 (x) =x, \; \cM'_0 (x) = \phi(x) \nu(x) .
\end{equation}
Here, the symbol $'$ means differentiation with respect to $t$, $\phi$ is any compactly supported continuous function, and $\nu$ is a proper extension of the unit normal vector field to a tubular neighborhood of $\Ga$.
%
%

Thus, we consider the area and volume functionals (in the variable $t$)
\begin{equation*}
A(t) = | \Ga_t | = \int_{\Ga_t} \, dS_x \quad \text{ and } \quad V(t) = | \Om_t| = \int_{\Om_t} \, dx ,
\end{equation*}
where $\Ga_t$ is the boundary of $\Om_t$, and $|\Ga_t|$, $|\Om_t|$ denote indifferently the $N-1$-dimensional measure of $\Ga_t$ and the $N$-dimensional measure of $\Om_t$.

Since $\Om_0 = \Om$ is the domain that minimizes $A(t)$ among all the domains in the one-parameter family $\left\lbrace \Om_t \right\rbrace_{t \in \RR}$ that have prescribed volume $|\Om_t|=V$, the method of {\it Lagrange multipliers} informs us that there exists a number $\la$ such that
$$
A'(0) - \la V'(0) = 0.
$$

Now by applying Hadamard's variational formula (see \cite[Chapter 5]{HP}) we can directly compute that the derivatives of $V$ and $A$ are given by:
\begin{equation*}
V'(0)= \int_{\Ga} \phi (x) \, dS_x 
\end{equation*}
and
\begin{equation*}
A'(0) = \int_{\Ga} H(x) \, \phi(x) \, dS_x , 
\end{equation*}
where $H$ is the mean curvature of $\Ga$. 

Therefore, we obtain that
$$
\int_{\Ga} \left( H(x) - \la \right) \, \phi(x) \, dS_x = 0.
$$
Since $\phi$ is arbitrary, we conclude that $H \equiv \la$ on $\Ga$.
The value $\la$ can be computed by using {\it Minkowski's identity},
\begin{equation*}
\int_\Ga H \, <x, \nu> \,dS_x = |\Ga| ,
\end{equation*}
and equals the number $H_0$ given by
$$
H_0= \frac{| \Ga|}{N | \Om|} .
$$

This argument proves that soap bubbles must have constant mean curvature. In turn, Alexandrov's theorem informs us that they must be spheres.


Similar arguments can be applied to optimization of other physical quantities, which also satisfy inequalities of isoperimetric type.
%
%
A celebrated example, which is connected with the overdetermined problem \eqref{serrin1}, \eqref{serrin2}, is given by {\it Saint Venant's Principle}, explained below.

To this aim we introduce the {\it torsional rigidity} $\tau(\Om)$ of a bar of cross-section $\Om$, that after some normalizations, can be defined as
$$
\max \left\lbrace Q(v) : 0 \not\equiv v \in W_0^{1,2} (\Om)  \right\rbrace , \, \text{ where } \, Q(v)= \frac{ \left( N \, \int_\Om v \, dx \right)^2}{\int_\Om |\na v|^2 \, dx  } .
$$
As mentioned in \cite{Se}, this quantity is related to the solution $u$ of \eqref{serrin1}. In fact, it turns out that $u$ realizes the maximum above, and we have that
\begin{equation*}
\tau(\Om) = Q(u) = \int_\Om | \na u|^2 \, dx = -N \int_\Om u \, dx.
\end{equation*}

Saint Venant's Principle (or the isoperimetric inequality for the torsional rigidity) states that the ball maximizes $\tau(\Om)$ among sets of given volume.
The standard proof of this result also uses rearrangement techniques and can be found in \cite[Section 1.12]{PZ}.
Here, in analogy to what done just before,
%
%
we simply show the relation between Saint Venant's Principle and problem \eqref{serrin1},\eqref{serrin2}.


Thus, we suppose that a set $\Om$ (of class $C^{1, \al}$) maximizing $\tau(\Om)$ among sets of given volume exists and consider the one-parameter family $\Om_t$ such that $\Om_0= \Om$ already defined.
%
%

This time, as the parameter $t$ changes, we must describe how the solution $u(t,x)$ of \eqref{serrin1} in $\Om_t$ changes, since the torsional rigidity of $\Om_t$ reads as
\begin{equation*}
T(t)= \tau (\Om_t)= - N \int_{\Om_t} u (t,x) \, dx .
\end{equation*}


By using the method of Lagrange multipliers as before, we deduce that there exists a number $\la$ such that
\begin{equation*}
T'(0) + \la V'(0) = 0.
\end{equation*}
By Hadamard's variational formula (see \cite[Chapter 5]{HP}), we can compute
\begin{equation*}
T'(0)= -N \int_{\Om} u'(x) \, dx - N \int_{\Ga} u(x) \phi(x) \, dS_x,
\end{equation*}
where $u(x)= u(0,x)$, and $u'(x)$ is the derivative of $u(t,x)$ with respect to $t$, evaluated at $t=0$.
%
%
Moreover, it turns out that $u'$ is the solution of the problem
\begin{equation*}
\De u' = 0 \, \mbox{ in } \Om , \, u'= u_\nu \phi \, \mbox{ on } \Ga.
\end{equation*}
Thus, 
%
%
we immediately find that
$$
T'(0) = -N \int_{\Om} u' \, dx = - \int_{\Om} u' \De u \, dx = - \int_{\Ga} u' u_\nu \, dS_x ,
$$
after an integration by parts in the last equality.
By using this last formula together with the fact that $ u'= u_\nu \phi$ on $\Ga$,
and recalling the already computed expression for $V'(0)$, we conclude that
$$
0= T'(0) + \la V'(0) =  \int_{\Ga} \left( - u_\nu^2 + \la \right) \phi \, dS_x .
$$
Since $\phi$ is arbitrary, we deduce that $u_\nu^2 \equiv \la$ on $\Ga$. By the identity
\begin{equation*}
\int_\Ga u_\nu\,dS_x=N\,|\Om|,
\end{equation*}
we then compute $\la=R^2$,
where 
$$
R=\frac{N |\Om|}{| \Ga|}.
$$

The core of this thesis concerns symmetry and stability results for the Soap Bubble Theorem, Serrin's problem, and some other related problems. Most of the material and techniques presented in this part are contained in the papers \cite{MP, MP2, Pog}, that are already published or in print. However, in the process of writing the thesis, we realized that some of the stability results described in those papers could be sensibly improved.
%
%
Thus, we present our results in this new (technically equivalent) form.
%
%

\bigskip
\noindent{\bf Integral identities and symmetry.}
The Soap Bubble Theorem and Serrin's result share several common features. To prove his result, Alexandrov introduced his reflection principle, an elegant geometric technique that also works for other symmetry results concerning curvatures. 
Serrin's proof hinges on his method of moving planes, an adaptation and refinement of the reflection principle. That method proves to be a very flexible tool, since it allows to prove
radial symmetry for positive solutions of a far more general class of non-linear equations, that includes the semi-linear equation
\begin{equation}
\label{semilinear}
\De u=f(u),
\end{equation}
where $f$ is a locally Lipschitz continuous non-linearity.
\par
Also, alternative proofs of both symmetry results can be given, based on certain integral identities and inequalities. 
In fact, in the same issue of the journal in which \cite{Se} is published, H.~F.~Weinberger \cite{We} gave a different proof of Serrin's symmetry result for problem \eqref{serrin1}-\eqref{serrin2} based on integration by parts, a maximum principle for an appropriate $P$-function, and the Cauchy-Schwarz inequality.
These ideas can be extended to the Soap Bubble Theorem,
using R. C. Reilly's argument (\cite{Re}), in which the relevant hypersurface is
regarded as the zero level surface of the solution of \eqref{serrin1}.
The connection between \eqref{serrin1}-\eqref{serrin2} and the Soap Bubble problem is then hinted by the simple differential identity 
$$
\De u=|\na u|\,\dv\frac{\na u}{| \na u|}+\frac{\lan \na^2u\,\na u, \na u\ran}{|\na u|^2};
$$
here, $\na u$ and $\na^2u$ are the gradient and the hessian matrix of $u$, as standard. If we agree
to still denote by $\nu$ the vector field $\na u/|\na u|$ (that on $\Ga$ coincides with the outward unit normal), the above identity and \eqref{serrin1} inform us that
%
%
\begin{equation*}
u_{\nu\nu}+(N-1)\,H\,u_\nu=N,
\end{equation*}
on every {\it non-critical} level surface of $u$, and hence on $\Ga$. In fact, a well known formula states that
the mean curvature $H$ (with respect to the inner normal) of a regular level surface of $u$ equals
$$
\frac1{N-1}\,\dv\frac{\na u}{|\na u|}.
$$
%
%
\par 
In both problems, the radial symmetry of $\Ga$ will follow from that of the solution $u$ of \eqref{serrin1}. In fact, we will show that in Newton's inequality,
%
%
\begin{equation*}
(\De u)^2\le N\,|\na^2 u|^2,
\end{equation*}
that holds pointwise in $\Om$ by the Cauchy-Schwarz inequality, the equality sign is identically attained in $\Om$. The point is that such equality holds if and only if 
$u$ is a quadratic polynomial $q$ of the form
%
%
\begin{equation*}
q(x)=\frac12\, (|x-z|^2-a),
\end{equation*}
for some choice of $z\in\RR^N$ and $a\in\RR$ (see Lemma \ref{lem:sphericaldetector}). The boundary condition in \eqref{serrin1} will then
tell us that $\Ga$ must be a sphere centered at $z$.
\par
The starting points of our analysis are the following two integral identities (see Theorems \ref{thm:serrinidentity} and \ref{thm:identitySBT}):
\begin{equation}
\label{intro:idwps}
\int_{\Om} (-u) \left\{ |\na ^2 u|^2- \frac{ (\De u)^2}{N} \right\} dx=
\frac{1}{2}\,\int_\Ga \left( u_\nu^2- R^2\right) (u_\nu-q_\nu)\,dS_x
\end{equation}
and
\begin{multline}
\label{intro:H-fundamental}
\frac1{N-1}\int_{\Om} \left\{ |\na ^2 u|^2-\frac{(\De u)^2}{N}\right\}dx+
\frac1{R}\,\int_\Ga (u_\nu-R)^2 dS_x = \\
\int_{\Ga}(H_0-H)\, (u_\nu)^2 dS_x.
\end{multline}
Here, $R$ and $H_0$ are the already mentioned reference constants:
\begin{equation*}
R=\frac{N\,|\Om|}{|\Ga|}, \quad H_0=\frac1{R}=\frac{|\Ga|}{N\,|\Om|}.
\end{equation*}

Identities \eqref{intro:idwps} and \eqref{intro:H-fundamental} hold regardless of how the point $z$ or the constant $a$ are chosen for $q$. 
Identity \eqref{intro:idwps}, claimed in \cite[Remark 2.5]{MP} and proved in \cite{MP2}, puts together and refine Weinberger's identities and some remarks of L.~E.~Payne and P.~W.~Schaefer \cite{PS}. Identity \eqref{intro:H-fundamental} was proved in \cite[Theorem 2.2]{MP} by polishing the arguments contained in \cite{Re}.

It is thus evident that each of the two identities gives spherical symmetry if respectively $u_\nu=R$ or $H=H_0$ on $\Ga$, since Newton's inequality holds with the equality sign (notice that in \eqref{intro:idwps} $-u>0$ by the strong maximum principle). The same conclusion is also achieved if we only assume that $u_\nu$ or $H$ are constant on $\Ga$, since those constants must equal $R$ and $H_0$, as already observed.

%
%
%
%

Thus, \eqref{intro:idwps} and \eqref{intro:H-fundamental} give new elegant proofs of Alexandrov's and Serrin's results.
Moreover,
%
%
they have several advantages that we shall describe in this and the next section.
%
%

One advantage is that we can obtain symmetry under weaker assumptions.
In fact, observe that to get spherical symmetry it is enough to prove that the right-hand sides in the two identities are non-positive. For instance, in the case of the Soap Bubble Theorem (see Theorem \ref{th:SBT}) one gets the spherical symmetry if 
$$\int_{\Ga}(H_0-H)\, (u_\nu)^2 dS_x \le 0 ,$$
which is certainly true if $H\ge H_0$ on $\Ga$.

Such slight generalization has been exploited in \cite{Pog} to improve a symmetry result of Garofalo and Sartori (\cite{GS}) for the p-capacitary potential with constant normal derivative. In fact, in \cite{Pog} the star-shaped assumption used in \cite{GS} has been removed. 
%
%

%
%

Another advantage is that \eqref{intro:idwps} and \eqref{intro:H-fundamental} tell us something more about Saint Venant's principle and the classical isoperimetric problem described in the previous section. 

In fact, it has been noticed in \cite[Theorem 7]{Mag} that, to infer that a domain $\Om$ is a ball, it is sufficient to require that, under the
flow \eqref{intro:eq:evolutiondomains1}-\eqref{intro:eq:evolutiondomains2} with the choice $\phi= (u_\nu - q_\nu)/2$, the function 
$$
t \to T(t) + R^2 (V(t) - V)
$$
has non-positive derivative at $t=0$. In particular, the same conclusion holds if $t=0$ is a critical point.
In fact, with that choice of $\phi$, \eqref{intro:eq:evolutiondomains2} gives that
$$
T'(0)+ R^2 V'(0)= \frac{1}{2} \int_\Ga  (u_\nu - R^2) \, (u_\nu - q_\nu) \, dS_x ,
$$
and the conclusion follows from \eqref{intro:idwps}.

For the case of the classical isoperimetric problem we can deduce a similar statement.
In fact we can infer that a domain $\Om$ is a ball if, under the
flow \eqref{intro:eq:evolutiondomains1}-\eqref{intro:eq:evolutiondomains2} with the choice $\phi=- u_\nu^2$, the function 
$$
t \to A(t) - H_0 (V(t) - V)
$$
has non-positive derivative at $t=0$. In particular, the same conclusion holds if $t=0$ is a critical point.
In fact, with that choice of $\phi$, \eqref{intro:eq:evolutiondomains2} gives that
\begin{equation*}
A'(0) - H_0 V'(0)= \int_\Ga (H_0 - H) \, ( u_\nu^2) \, dS_x ,
\end{equation*}
and the conclusion follows from \eqref{intro:H-fundamental}.

Thus, \eqref{intro:idwps} and \eqref{intro:H-fundamental} inform us that the flows generated respectively by $\phi= (u_\nu - q_\nu)/2$ and $\phi= - u_\nu^2$ are quite privileged.

\bigskip
\noindent{\bf Stability results.}
The greatest benefit
%
%
produced by our identities are undoubtedly the stability results for the Soap Bubble Theorem, Serrin's problem, and other related overdetermined problems.

Technically speaking, there are several ways to describe the closeness of a domain to a ball. In this thesis, we mainly privilege
the following: find two concentric balls $B_{\rho_i} (z)$ and $B_{\rho_e} (z)$, centered at $z \in \Om$ with radii $\rho_i$ and $\rho_e$, such that
\begin{equation}
\label{balls}
B_{\rho_i} (z) \subseteq \Om \subseteq B_{\rho_e} (z)
\end{equation} 
 and
\begin{equation}
\label{stability}
\rho_e-\rho_i\le \psi(\eta),
\end{equation}
where
$\psi:[0,\infty)\to[0,\infty)$ is a continuous function vanishing at $0$ and $\eta$ is a suitable measure of the deviation of $u_\nu$ or $H$ from being a constant.

%
%

The landmark results of this thesis are the following stability estimates:
\begin{equation}\label{intro:eq:Improved-Serrin-stability}
\rho_e-\rho_i\le C \, \nr u_\nu - R \nr_{2,\Ga}^{\tau_N} 
\end{equation}
and
\begin{equation}
\label{intro:eq:SBT-improved-stability}
\rho_e-\rho_i\le C\,\nr H_0-H\nr_{2,\Ga}^{\tau_N} .
\end{equation}

In \eqref{intro:eq:Improved-Serrin-stability} (see Theorem \ref{thm:Improved-Serrin-stability} for details), $\tau_2 = 1$, $\tau_3$ is arbitrarily close to one,  
%
%
and $\tau_N = 2/(N-1)$ for $N \ge 4$.
In \eqref{intro:eq:SBT-improved-stability} (see Theorem \ref{thm:SBT-improved-stability} for details), $\tau_N=1$ for $N=2, 3$, $\tau_4$ is arbitrarily close to one, and
%
%
$\tau_N = 2/(N-2) $ for $N\ge 5$.

The constants $C$ depend on the dimension $N$, the diameter $d_\Om$, the radii $r_i$, $r_e$ of the uniform interior and exterior sphere conditions, and the distance $\de_\Ga (z)$ of $z$ to $\Ga$.
%
%
As pointed out in Remarks \ref{rem:removing-distance-convex} and \ref{remarkprova:nuova scelte z}, the dependence on $\de_\Ga (z)$ of the constants in \eqref{intro:eq:Improved-Serrin-stability} and \eqref{intro:eq:SBT-improved-stability} can (so far) be removed when $\Om$ is convex or by assuming some additional requirements.

The stability results for Serrin's problem so far obtained in the literature can be divided into two groups, depending on the method employed. One is based on a quantitative study of the method of moving planes and hinges on the use of Harnack's inequality. In \cite{ABR}, where the stability issue was considered for the first time, \eqref{balls} and \eqref{stability} are obtained with 
$\psi(\eta)=C\,|\log \eta|^{-1/N}$ and $\eta=\nr u_\nu-c\nr_{C^1(\Ga)}$.
The estimate was later improved in \cite{CMV} to get
\begin{equation*}
\psi(\eta)=C\,\eta^{\tau_N} \quad \mbox{and} \quad   \eta=\sup_{\substack{x,y \in \Ga\\ \ x \neq y}} \frac{|u_\nu(x) - u_\nu(y)|}{|x-y|}.
\end{equation*}
The exponent $\tau_N \in (0,1)$ can be computed for a general setting and, if $\Om$ is convex,
is arbitrarily close to $1/(N+1)$.
It should be also stressed that the method works in the more general case of the semilinear equation \eqref{semilinear}.

The approach of using integral identities and inequalities, in the wake of Weinberger's proof of symmetry, was inaugurated in \cite{BNST}, and later improved in \cite{Fe} and \cite{MP2}. This method has worked so far only for problem \eqref{serrin1}, though.

The main result in \cite{BNST} states that (if a uniform bound on $u_\nu$ is assumed) $\Om$ can be approximated in measure by a finite number of mutually disjoint balls $B_i$. The error in the approximation is 
$\psi(\eta)=C\,\eta^{1/(4N+9)}$ where $\eta=\nr u_\nu-c\nr_{1,\Ga}$. There, it is also proved that
%
%
\eqref{balls} and \eqref{stability} hold with
$\psi(\eta)=C\,\eta^{1/2(4N+9)(N-1)}$ and $\eta=\nr u_\nu-c\nr_{\infty, \Ga}$.

In \cite[Theorem 1.1]{MP2} we have obtained
$\psi(\eta)= C\,\eta^{\frac{2}{N+2} }$ where $\eta = \nr u_\nu - R \nr_{2,\Ga}$, which already improves on \cite{ABR, CMV, BNST}.
While writing this thesis, we obtained \eqref{intro:eq:Improved-Serrin-stability}, that improves on \cite{MP2} (for every $N \ge 2$) to the extent that it gains the (optimal) Lipschitz stability in the case $N=2$.
%
%
%

In this thesis we will also consider a weaker $L^1$-deviation (as in \cite{BNST}) and prove (see Theorem \ref{thm:Serrin-stability}) the inequality
\begin{equation}
\label{stability-Serrin}
\rho_e-\rho_i\le C\,\nr u_\nu-R\nr_{1,\Ga}^{\tau_N/2},
\end{equation}
where $\tau_N$ is the same appearing in \eqref{intro:eq:Improved-Serrin-stability}.
Also this inequality is new and refines one stated in \cite[Theorem 3.6]{MP2}
%
%
in which $\tau_N /2$ was replaced by $1/(N+2)$.
Of course, that inequality was already better than that obtained in \cite{BNST}.

In \cite{Fe} instead, a different measure of closeness to spherical symmetry is adopted, by considering a slight modification of the so-called Fraenkel asymmetry:
\begin{equation}
\label{intro:asymmetry}
\cA(\Om)=\inf\left\{\frac{|\Om\De B^x|}{|B^x|}: x \mbox{ center of a ball $B^x$ with radius $R$} \right\}.
\end{equation}
Here, $\Om\De B^x$ denotes the symmetric difference of $\Om$ and $B^x$.
It is then obtained a Lipschitz-type estimate: $\cA(\Om)\le C\,\nr u_\nu-R\nr_{2,\Ga}$.
The control given by $\rho_e-\rho_i$ is stronger than that given by $\cA(\Om)$ (see Section \ref{sec:stabbyasymmetrySBT}).

\medskip

Let us now turn our attention to the Soap Bubble Theorem and comment on our estimate \eqref{intro:eq:SBT-improved-stability}. 
%
%

The only stability result based on Alexandrov's reflection principle has been obtained in \cite{CV} and states that there exist two positive constants $C$, $\ve$ such that \eqref{balls} and \eqref{stability} hold for $\psi(\eta)=C\eta$ if $\eta= \nr H_0-H\nr_{\infty,\Ga} < \ve .$
In \cite{CM} and \cite{KM}, similar results (again with the uniform deviation) are obtained, based on the proof of the Soap Bubble Theorem via Heintze-Karcher's inequality (that holds if $\Ga$ is mean convex) given in \cite{Ro}. Moreover, in \cite{KM} \eqref{balls} and \eqref{stability} are obtained also with $\psi(\eta)= C \, \eta$ and $\eta= \nr H_0-H\nr_{2,\Ga}$, for surfaces that are small normal deformations of spheres.

%

Further advances have been obtained in \cite[Theorem 1.2]{MP2} and \cite[Theorem 4.1]{MP}. In fact, based on Reilly's proof (\cite{Re}) of the Soap Bubble Theorem, \eqref{balls} and \eqref{stability} are shown to hold for general surfaces respectively with 
$$
\psi(\eta)= C\, \eta^{\tau_N} \quad \mbox{and} \quad \eta=\nr H_0-H\nr_{2,\Ga}.
$$
$$
\psi(\eta) = C\, \eta^{\tau_N /2} \quad \mbox{and} \quad \eta = \int_{\Ga}(H_0-H)^+\,dS_x;
$$
Here, $\tau_N=1$ for $N=2,3$, and $\tau_N=2/(N+2)$ for $N \ge 4$.
%
%
%
%
%
Both these estimates are improved
%
%
by \eqref{intro:eq:SBT-improved-stability}
and
\begin{equation}
\label{the-estimate}
\rho_e-\rho_i\le C\,\left\{\int_{\Ga}(H_0-H)^+\,dS_x\right\}^{\tau_N /2} 
\end{equation}
(see Theorem \ref{thm:SBT-stability}), which are original material of this thesis. In \eqref{the-estimate}, $\tau_N$ is the same appearing in \eqref{intro:eq:SBT-improved-stability}.
If we compare the exponents in \eqref{intro:eq:SBT-improved-stability} and \eqref{the-estimate} to those obtained in \cite[Theorem 1.2]{MP2} and \cite[Theorem 4.1]{MP}, we notice that the dependence of $\tau_N$ on $N$ has become virtually continuous, in the sense that $\tau_N\to 1$, if $N$ ``approaches'' $4$ from below or from above.

\par
As a final important achievement, by arguments similar to those of \cite{Fe}, we report in Theorem \ref{th:asymmetry} the (optimal) inequality for the asymmetry \eqref{intro:asymmetry} obtained in \cite[Theorem 4.6]{MP2}:
\begin{equation}\label{eq:introasymmetrystab}
\cA(\Om)\le C\,\nr H_0-H\nr_{2,\Ga}.
\end{equation}
\medskip

In the remaining part of this introduction, we shall pinpoint the key remarks in the proof of \eqref{intro:eq:Improved-Serrin-stability} and \eqref{intro:eq:SBT-improved-stability}.
\par
First, notice that the function $h=q-u $ is harmonic and we have that
\begin{equation*}
|\na ^2 u|^2-\frac{(\De u)^2}{N} = |\na^2 h|^2 .
\end{equation*}
Thus, \eqref{intro:idwps} reads as:
\begin{equation}
\label{intro:idwps-h}
\int_{\Om} (-u)\, |\na ^2 h|^2\,dx=
\frac{1}{2}\,\int_\Ga ( R^2-u_\nu^2)\, h_\nu\,dS_x.
\end{equation}
Also, since $h=q$ on $\Ga$, if we choose $z$ in $\Om$, it holds that
\begin{equation*}
\max_{\Ga} h-\min_{\Ga} h=\frac12\,(\rho_e^2-\rho_i^2)\ge \frac{r_i}{2} \, (\rho_e-\rho_i).
\end{equation*}

Now, observe that \eqref{intro:idwps-h} holds regardless of the choice of the parameters $z$ and $a$ defining $q$. We will thus complete the first step of our proof by choosing $z\in\Om$ in a way that the oscillation of $h$ on $\Ga$ can be bounded in terms of the volume integral in \eqref{intro:idwps-h}.
\par

To carry out this plan, we use three ingredients. First, as done in Lemmas \ref{lem:Lp-estimate-oscillation-generic-v} and \ref{lem:L2-estimate-oscillation}, we show that the oscillation of $h$ on $\Ga$, and hence $\rho_e - \rho_i$ can be bounded from above in the following way: 
\begin{equation}\label{intro:eq:lemmanuovooscillationLpnorm}
\rho_e - \rho_i \le C (N, p, d_\Om, r_i, r_e) \, \nr h - h_\Om \nr_{p, \Om}^{p/(N+p)} ,
\end{equation}
where $h_\Om$ is the mean value of $h$ on $\Om$ and $p \in \left[1, \infty \right)$.
We emphasize that this inequality is new and generalizes to any $p$ the estimate obtained in \cite[Lemma 3.3]{MP} for $p=2$.

Secondly (see Lemma \ref{lem:relationdist}), we easily obtain the bound
\begin{equation*}
\frac{r_i}{2} \,\de_\Ga(x) \le -u \ \mbox { on } \ \ol{\Om}.
\end{equation*}

Thirdly, we choose $z \in \Om$ as a minimum (or any critical) point of $u$ and we apply two integral inequalities to $h$ and its first (harmonic) derivatives. One is the Hardy-Poincar\'e-type inequality 
\begin{equation}
\label{boas-straube}
\nr v \nr_{r, \Om} \le C(N, r, p, \al, d_\Om, r_i, \de_\Ga (z)) \, \nr \de_\Ga(x)^\al \, \na v(x) \nr_{p, \Om},
\end{equation}
that is applied to the first (harmonic) derivatives of $h$. It holds for any harmonic function $v$
in $\Om$ that is zero at some given point in $\Om$ (in our case that point will be $z$, since $\na h(z)=0$), when $r, p, \al$ are three numbers such that either $1 \le p \le r \le \frac{Np}{N-p(1 - \al )}$, $p(1 - \al)<N$, $0 \le \al \le 1$ (see Lemma \ref{lem:John-two-inequalities}),
or $1 \le r=p <\infty$ and $0 \le \al \le 1$ (see Lemma \ref{lem:two-inequalities}).
The other one is applied to
$h- h_\Om$ and is the Poincar\'e-type inequality
\begin{equation}
\label{classicalpoincare}
\nr v \nr_{r, \Om} \le C(N, r, p, d_\Om, r_i) \, \nr \na v(x) \nr_{p, \Om}
\end{equation}
that holds for any function $v\in W^{1,p}(\Om)$ with zero mean value on $\Om$, where $r$ and $p$ must satisfy the same inequalities mentioned above for \eqref{boas-straube} when $\al=0$ (see
Lemmas \ref{lem:John-two-inequalities} and \ref{lem:two-inequalities}).
This third step is accomplished in Theorem \ref{thm:serrin-W22-stability}, where all the details can be found.

\par
Thus, putting together the above arguments gives that (see Theorem \ref{thm:serrin-W22-stability}) there exists a constant $C$ depending on $N$, $d_\Om$, $r_i$, $r_e$, $\de_\Ga (z)$
%
%
such that
\begin{equation}
\label{the-estimate-for-h}
\rho_e - \rho_i \le C \left(\int_{\Om} (-u)\, |\na ^2 h|^2\,dx\right)^{\tau_N /2},
\end{equation}
where $\tau_N$ is that appearing in \eqref{intro:eq:Improved-Serrin-stability}.
We mention that in dimension $N=2$ there is no need to use \eqref{intro:eq:lemmanuovooscillationLpnorm}, thanks to Sobolev imbedding theorem (see item (i) of Theorem \ref{thm:serrin-W22-stability}).

Next, we work on the right-hand side of \eqref{intro:idwps-h}. The important observation is that, if $u_\nu-R$ tends to $0$, also $h_\nu$ does. Quantitatively, this fact can be expressed by the inequality
\begin{equation}\label{intro:eq:tracehnu}
\nr h_\nu\nr_{2,\Ga}\le C(N, d_\Om, r_i, r_e, \de_\Ga (z) ) \,\nr u_\nu-R\nr_{2,\Ga},
\end{equation}
that can be derived (see the proof of Theorem \ref{thm:Improved-Serrin-stability}) by exploiting arguments contained in \cite{Fe}. 
Thus, after an application of H\"older's inequality to the right-hand side of \eqref{intro:idwps-h}, by \eqref{intro:eq:tracehnu} we deduce that
$$
\int_\Om (-u) |\na^2 h|^2 dx \le C(N, d_\Om, r_i, r_e, \de_\Ga (z) ) \,\nr u_\nu-R\nr_{2,\Ga}^2 ,
$$
and \eqref{intro:eq:Improved-Serrin-stability} will follow by \eqref{the-estimate-for-h}.

%
%

\smallskip

Let us now sketch the proof of \eqref{intro:eq:SBT-improved-stability}.
We fix again $z$ at a local minimum point of $u$ in $\Om$.
In this way, $z \in \Om$ and $h=q-u$ is such that $\na h(z)=0.$

Thus, as before we can apply to $h-h_\Om$ and to its first (harmonic) derivatives the Poincar\'e-type inequalities \eqref{classicalpoincare} and \eqref{boas-straube} (with $\al=0$).
In this way, by exploiting again \eqref{intro:eq:lemmanuovooscillationLpnorm}, we get that (see Theorem \ref{thm:SBT-W22-stability}) there exists a constant $C$ depending on $N$, $d_\Om$, $r_i$, $r_e$, $\de_\Ga (z)$
%
%
such that
\begin{equation}
\label{eq:step2}
\rho_e - \rho_i \le C \| \na^2 h \|^{\tau_N}_{2, \Om},
\end{equation}
where $\tau_N$ is that appearing in \eqref{intro:eq:SBT-improved-stability}.
We mention that when $N=2,3$, there is no need to use \eqref{intro:eq:lemmanuovooscillationLpnorm}, thanks to Sobolev imbedding theorem (see item (i) of Theorem \ref{thm:SBT-W22-stability}).

Next, we work on the right-hand side of \eqref{intro:H-fundamental} and we notice that
%
%
it can be rewritten as (see Theorem \ref{thm:identitySBT})
%
%
$$
\int_{\Ga}(H_0-H)\,(u_\nu-q_\nu)\,u_\nu\,dS_x+
\int_{\Ga}(H_0-H)\, (u_\nu-R)\,q_\nu\, dS_x.
$$
Hence, \eqref{intro:H-fundamental} can be written in terms of $h$ as
\begin{multline}
\label{intro:identity-SBT-h}
\frac1{N-1}\int_{\Om} |\na ^2 h|^2 dx+
\frac1{R}\,\int_\Ga (u_\nu-R)^2 dS_x = \\
-\int_{\Ga}(H_0-H)\,h_\nu\,u_\nu\,dS_x+
\int_{\Ga}(H_0-H)\, (u_\nu-R)\,q_\nu\, dS_x.
\end{multline}
Discarding the first summand at the left-hand side of \eqref{intro:identity-SBT-h} and applying H\"older's inequality and \eqref{intro:eq:tracehnu} to its right-hand side yield that
$$
\nr u_\nu-R\nr_{2,\Ga}\le C (N, d_\Om, r_i, r_e, \de_\Ga (z) ) \,\nr H_0-H\nr_{2,\Ga}.
$$

Thus, by discarding the second summand at the left-hand side of \eqref{intro:identity-SBT-h} we get
\begin{equation*}
\| \na ^2 h \|_{2,\Om} \le C (N, d_\Om, r_i, r_e, \de_\Ga (z) ) \| H_0 - H \|_{2, \Ga}.
\end{equation*}
Hence, \eqref{intro:eq:SBT-improved-stability} easily follows by putting together this last inequality and \eqref{eq:step2}.
%
%
%
%
%

\bigskip
\noindent{\bf Plan of the work.}
%
The core of the thesis comprises four chapters.

Chapter \ref{chapter:integral identities and symmetry results} contains integral identities and related symmetry results.
In Section \ref{sec:integral identities} we prove \eqref{intro:idwps} and \eqref{intro:H-fundamental}, as well as other related identities. 
Section \ref{sec:symmetry results via integral identities} contains the corresponding relevant symmetry results.
%
%

Chapter \ref{chap:PogAA} describes symmetry theorems that concern overdetermined problems for $p$-harmonic functions in exterior and punctured domains. They are essentially the results proved in \cite{Pog}.

In Chapter \ref{chapter:various estimates} we collect all the estimates for the torsional rigidity density $u$ and the harmonic function $h$ that will be useful to derive our stability results.
Section \ref{sec:gradient estimate for torsion} is dedicated to the function $u$ and contains pointwise estimates for $u$ and its gradient.
Section \ref{sec:harmonic functions in weighted spaces} collects Hardy-Poincar\'e inequalities as well as a proof of a trace inequality for harmonic functions which is the key ingredient to get \eqref{intro:eq:tracehnu}.
In Section \ref{sec:estimates for the oscillation of harmonic functions} we present the crucial lemma which allows to prove \eqref{intro:eq:lemmanuovooscillationLpnorm}.
%
%
All the inequalities for the particular harmonic function $h= q- u$ that are necessary to get our stability estimates are presented in Section \ref{sec:estimates for h}, as a consequence of Sections \ref{sec:gradient estimate for torsion}, \ref{sec:harmonic functions in weighted spaces}, \ref{sec:estimates for the oscillation of harmonic functions}. 
%
%

Chapter \ref{chapter:Stability results} contains all the stability results for the spherical configuration.
Section \ref{sec:stability Serrin} is devoted to the case of Serrin's problem \eqref{serrin1}, \eqref{serrin2} and contains the proof of inequalities \eqref{intro:eq:Improved-Serrin-stability} and \eqref{stability-Serrin}.
Section \ref{sec:Alexandrov's SBT} is devoted to the case of Alexandrov's Soap Bubble Theorem and contains the proof of \eqref{intro:eq:SBT-improved-stability} and \eqref{the-estimate}.
In Section \ref{sec:stabbyasymmetrySBT} we consider the asymmetry $\cA(\Om)$
defined in \eqref{intro:asymmetry}, we prove \eqref{eq:introasymmetrystab}, and we compare $\cA(\Om)$ with $\rho_e -\rho_i$. 
Finally, Section \ref{sec:Other stability results mean curvature} collects other related stability results.

As already mentioned, the second part of the thesis, entitled ``Other works'', contains a couple of papers as they were published. For the sake of coherence with the topic of this dissertation, we decided not to present them in the main part.

In chapter \ref{chap:LNCIME} are reported the lecture notes \cite{CP} of a CIME summer course taught by Prof. Xavier Cabr\'e in Cetraro during the week of June 19-23, 2017. The author of the present thesis attended that course and then collaborated in writing those notes.

The second paper, reported in Chapter \ref{chap:Littlewoodfourthprinciple}, is \cite{MPLittlewood}. That is a short article stimulated by some remarks made while the author of the present thesis was attending (as undergraduate student) the class ``Analisi Matematica III'' taught by Prof. Rolando Magnanini at the Universit\`a di Firenze.

\chapter{Integral identities and symmetry}\label{chapter:integral identities and symmetry results}
In this chapter, we shall present a number of integral identities which our symmetry and stability results are based on. In them, the torsional rigidity density $u$, i.e. the solution of \eqref{serrin1}, plays the main role. Their proofs are combinations of the divergence theorem, differential identities, and some ad hoc manipulations.

Before we start, we set some relevant notations. 
By $\Om\subset\RR^N$, $N\ge 2$, we shall denote a bounded domain, that is a connected bounded open set, and call $\Ga$ its boundary. 
By $|\Om|$ and $|\Ga|$, we will denote indifferently the $N$-dimensional Lebesgue measure of $\Om$
and the surface measure of $\Ga$. When $\Ga$ is of class $C^1$, $\nu$ will denote the (exterior) unit normal vector field to $\Ga$ and, when $\Ga$ is
%
%
of class $C^2$, $H(x)$ will denote its mean curvature (with respect to $-\nu(x)$) at $x\in\Ga$. 
\par
As already done in the introduction, we set $R$ and $H_0$ to be the two reference constants given by
\begin{equation}
\label{def-R-H0}
R=\frac{N\,|\Om|}{|\Ga|}, \quad H_0=\frac1{R}=\frac{|\Ga|}{N\,|\Om|},
\end{equation}
and we use the letter $q$ to denote the quadratic polynomial defined by
\begin{equation}
\label{quadratic}
q(x)=\frac12\, (|x-z|^2-a),
\end{equation}
where $z$ is any point in $\RR^N$ and $a$ is any real number.

\section{Integral identities for the torsional rigidity density}\label{sec:integral identities}

We start by recalling the following classical Rellich-Pohozaev identity (\cite{Poh}).

\begin{lem}[Rellich-Pohozaev identity]
Let $\Om$ be a domain with boundary $\Ga$ of class $C^1$.
For every function $u \in C^2(\Om) \cap C^1(\ol{\Om})$, it holds that
\begin{multline}
\label{eq:idPohozaevgeneral}
\frac{N-2}{2} \int_{\Om} |\na u|^2 \, dx - \int_{\Om} <x , \na u > \De u \, dx =
\\
\frac{1}{2} \int_\Ga |\na u|^2 <x, \nu> \, dS_x - \int_\Ga <x,\na u> u_\nu \, dS_x.
\end{multline}
\end{lem}
\begin{proof}
By direct computation it is easy to verify the following differential identity:
$$
\dv \left\lbrace \frac{| \na u|^2}{2} \, x - <x, \na u> \na u \right\rbrace = \frac{N-2}{2} \, |\na u|^2 - <x,\na u> \De u.
$$
Thus, \eqref{eq:idPohozaevgeneral} easily follows by integrating over $\Om$ and applying the divergence theorem.
\end{proof}

We now focus our attention on the solution $u$ of \eqref{serrin1}.
In this case, it is easy to check that \eqref{eq:idPohozaevgeneral} becomes the identity stated in the following corollary.
\begin{cor}[Rellich-Pohozaev identity for the torsional rigidity density]
Let $\Om \subset \mathbb R^N$ be a bounded domain with boundary $\Ga$ of class $C^{1, \al}$, $0< \al \le 1$.
Then the solution $u$ of  \eqref{serrin1} satisfies the identity
\begin{equation}
\label{Pohozaev}
(N+2) \int_{\Om} |\na u|^2 \, dx =\int_\Ga (u_\nu)^2 \, q_\nu\,dS_x ,
\end{equation}
where $q$ denotes the quadratic polynomial defined in \eqref{quadratic}.
\end{cor}
\begin{proof}
If $\Ga$ is of class $C^{1,\al}$, then $u\in C^{1,\al}(\ol{\Om})\cap C^2(\Om)$, and hence \eqref{eq:idPohozaevgeneral} holds.
Since $u$ satisfies \eqref{serrin1}, the following differential identity holds
$$
\dv \left( u \, x - u \, \na u \right) = < \na u, x> - | \na u |^2.
$$
%
%
By integrating over $\Om$, applying the divergence theorem, and recalling the boundary
condition in \eqref{serrin1}, we obtain that
\begin{equation*}
\int_\Om | \na u|^2 \, dx = \int_{\Om} <\na u , x> \, dx.
\end{equation*}
Thus, the left-hand side of \eqref{eq:idPohozaevgeneral} becomes
$$
- \left( \frac{N+2}{2} \right) \int_\Om | \na u|^2.
$$
On the other hand, by using the fact that $\na u = u_\nu \, \nu$ on $\Ga$, the right-hand side of \eqref{eq:idPohozaevgeneral} becomes
$$
- \frac{1}{2} \int_\Ga (u_\nu)^2 <x, \nu> , 
$$
and the conclusion follows.
\end{proof}

Following the tracks of Weinberger \cite{We} we introduce the P-function,
\begin{equation}
\label{P-function}
P = \frac{1}{2}\,|\nabla u|^2 - u,
\end{equation}
and we easily compute that
\begin{equation}
\label{differential-identity}
\De P = |\na^2 u|^2-\frac{(\De u)^2}{N}.
\end{equation}

We are now ready to provide the proof of the fundamental identity for Serrin's problem.

%
%
%
%

\begin{thm}[Fundamental identity for Serrin's problem, \cite{MP,MP2}] 
\label{thm:serrinidentity}
Let $\Om \subset \mathbb R^N$ be a bounded domain with boundary $\Ga$ of class $C^{1,\al}$, $0<\al\le 1$, and $R$ be the positive constant defined in \eqref{def-R-H0}.
Then the solution $u$ of  \eqref{serrin1} satisfies 
\begin{equation}
\label{idwps}
\int_{\Om} (-u) \left\{ |\na ^2 u|^2- \frac{ (\De u)^2}{N} \right\} dx=
\frac{1}{2}\,\int_\Ga \left( u_\nu^2- R^2\right) (u_\nu-q_\nu)\,dS_x .
\end{equation}
\end{thm}

\begin{proof}
First, suppose that $\Ga$ is of class $C^{2,\al}$, so that $u\in C^{2,\al}(\ol{\Om})$. Integration by parts then gives:
$$
\int_\Om (u\,\De P-P\,\De u)\,dx=\int_\Ga(u\,P_\nu-u_\nu\,P)\,dS_x.
$$
Thus, since $u$ satisfies \eqref{serrin1}, we have that
\begin{equation}
\label{parts}
\int_\Om (-u)\,\De P\,dx=-N\,\int_\Om P\,dx+\frac12\,\int_\Ga u_\nu^3\,dS_x,
\end{equation}
being $P=|\na u|^2/2=u_\nu^2/2$ on $\Ga$.
\par
Next, notice that we can write
$
u_\nu^3 = (u_\nu^2 - R^2)(u_\nu - q_\nu) + R^2 (u_\nu - q_\nu) + u_\nu^2 q_\nu.
$
By the divergence theorem and \eqref{Pohozaev} we then compute:
\begin{multline*}
N \int_{\Om} P\,dx=\frac{N}{2} \int_\Om |\na u|^2 dx-\int_\Om u\,\De u\,dx =
\left( \frac{N}{2} + 1 \right) \int_{\Om} |\na u|^2 \, dx
= 
\\
\frac{1}{2} \int_\Ga u_\nu^2 q_\nu \, dS_x .
\end{multline*}
Thus, this identity, \eqref{parts}, and \eqref{differential-identity} give \eqref{idwps}, since
$$
\int_\Ga (u_\nu-q_\nu)\,dS_x=0,
$$
being $u-q$ harmonic in $\Om$.
\par
If $\Ga$ is of class $C^{1,\al}$, then $u\in C^{1,\al}(\ol{\Om})\cap C^2(\Om)$. Thus, by a standard approximation argument, we conclude that \eqref{idwps} holds also in this case.
\end{proof}

\begin{rem}
{\rm
The assumptions on the regularity of $\Ga$ can further be weakened. For instance, if $\Om$ is a 
(bounded) convex domain, then inequality \eqref{bound-M-convex} below gives that $u_\nu$ is essentially bounded on $\Ga$ with respect to the $(N-1)$-dimensional Hausdorff measure on $\Ga$. Thus, an approximation argument again gives that \eqref{idwps} holds true.
}
\end{rem}

We now present a couple of integral identities, involving the torsional rigidity density $u$ and a harmonic function $v$, which are necessary to establish a useful trace inequality (see Lemma \ref{lem:genericv-trace inequality} below).

\begin{lem}\label{lem:identityfortraceineq}
Let $\Om \subset \mathbb R^N$ be a bounded domain with boundary $\Ga$ of class $C^{1,\al}$, $0<\al\le 1$, and let $u$ be the solution of \eqref{serrin1}. Then, for every harmonic function $v$ in $\Om$ the following two identities hold:
\begin{equation}\label{eq:nuovaidentityforidentityfortraceinequality}
\int_{\Ga} v^2 u_{\nu} \, dS_x = N \int_{\Om} v^2 dx + 2 \int_{\Om} (-u) |\na v|^2 dx ;
\end{equation}
\begin{equation}\label{eq:identityfortraceinequality}
\int_{\Ga} |\na v|^2 u_{\nu} dS_x = N \int_{\Om} |\na v|^2 dx + 2 \int_{\Om} (-u) |\na^2 v|^2 dx.
\end{equation}
\end{lem}
\begin{proof}
We begin with the following differential identity:
\begin{equation*}
\dv\,\{v^2 \na u - u \, \na(v^2)\}= v^2 \De u - u \, \De (v^2)= N \, v^2 - 2 u \, |\na v|^2,
\end{equation*}
that holds for any $v$ harmonic function in $\Om$, if $u$ satisfies \eqref{serrin1}.
Next, we integrate on $\Om$ and, by the divergence theorem,
we get \eqref{eq:nuovaidentityforidentityfortraceinequality}.
If we use \eqref{eq:nuovaidentityforidentityfortraceinequality} with $v=v_{i}$, and hence we sum up over $i=1,\dots, N$, we obtain \eqref{eq:identityfortraceinequality}.
\end{proof}

We now prove several integral identities which involve the mean curvature function $H$.
We start with the following one, that can be obtained by polishing the arguments contained in \cite{Re}.
%
%

\begin{thm}
\label{th:fundamental-identity}
Let $\Omega \subset \mathbb R^N$ be a bounded domain with boundary $\Ga$ of class $C^2$ and
denote by $H$ the mean curvature of $\Ga$.
\par
If $u$ is the solution of  \eqref{serrin1},
then the following identity holds:
\begin{equation}\label{fundamental}
\frac1{N-1}\int_{\Om} \left\{ |\na ^2 u|^2-\frac{(\De u)^2}{N}\right\}\,dx = N |\Om| -\int_{\Ga}H\, (u_\nu)^2\,dS_x.
\end{equation}
\end{thm}

\begin{proof}
Let $P$ be given by \eqref{P-function}. By the divergence theorem we can write:
\begin{equation}
\label{div-p-function}
\int_{\Om} \De P \,dx = \int_{\Ga} P_\nu \, dS_x.
\end{equation}
To compute $P_\nu$, we observe that $\na u$ is
parallel to $\nu$ on $\Ga$, that is $\na u= u_\nu \,\nu$ on $\Ga$. Thus,
$$
P_\nu=\lan D^2 u\, \na u, \nu\ran-u_\nu=u_\nu \lan (D^2 u)\,\nu,\nu\ran-u_\nu=u_{\nu\nu}\, u_\nu-u_\nu.
$$
We recall Reilly's identity from the introduction,
\begin{equation*}
u_{\nu\nu}+(N-1)\,H\,u_\nu=N ,
\end{equation*}
from which we obtain that
$$
u_{\nu\nu}\, u_\nu+(N-1)\,H\, (u_\nu)^2=N\,u_\nu,
$$
and hence
$$
P_\nu=(N-1)\,u_\nu-(N-1)\,H\, (u_\nu)^2
$$
on $\Ga$.
\par
Therefore, \eqref{fundamental} follows from this identity, \eqref{differential-identity}, \eqref{div-p-function} and the formula
\begin{equation}
\label{volume}
\int_\Ga u_\nu\,dS_x=N\,|\Om|,
\end{equation}
that is an easy consequence of the divergence theorem.
\end{proof}

\begin{thm}[Identities for the Soap Bubble Theorem, \cite{MP,MP2}]
\label{thm:identitySBT}
Let $\Om \subset \mathbb R^N$ be a bounded domain with boundary $\Ga$ of class $C^2$ and
denote by $H$ the mean curvature of $\Ga$.
%
%
Let $R$ and $H_0$ be the two positive constants defined in \eqref{def-R-H0}.

If $u$ is the solution of \eqref{serrin1}, then the following two identities hold true:
\begin{multline}
\label{H-fundamental}
\frac1{N-1}\int_{\Om} \left\{ |\na ^2 u|^2-\frac{(\De u)^2}{N}\right\}dx+
\frac1{R}\,\int_\Ga (u_\nu-R)^2 dS_x = \\
\int_{\Ga}(H_0-H)\, (u_\nu)^2 dS_x ,
\end{multline}
\begin{multline}
\label{identity-SBT2}
\frac1{N-1}\int_{\Om} \left\{ |\na ^2 u|^2-\frac{(\De u)^2}{N}\right\}dx+
\frac1{R}\,\int_\Ga (u_\nu-R)^2 dS_x = \\
\int_{\Ga}(H_0-H)\,(u_\nu-q_\nu)\,u_\nu\,dS_x+
\int_{\Ga}(H_0-H)\, (u_\nu-R)\,q_\nu\, dS_x.
\end{multline}
\end{thm}

\begin{proof}
Since, by \eqref{volume}, we have that
$$
\frac1{R}\,\int_\Ga u_\nu^2\,dS_x=
\frac1{R}\,\int_\Ga (u_\nu-R)^2\,dS_x+N |\Om|,
$$
then
\begin{multline*}
\int_\Ga H\,(u_\nu)^2\,dS_x=H_0\,\int_\Ga (u_\nu)^2\,dS_x+
\int_\Ga \left( H-H_0\right)\,(u_\nu)^2\,dS_x= \\
\frac1{R}\,\int_\Ga (u_\nu-R)^2\,dS_x+N |\Om|+
\int_\Ga ( H-H_0)\,(u_\nu)^2\,dS_x.
\end{multline*}
Thus, \eqref{H-fundamental} follows from this identity and \eqref{fundamental} at once.

Identity \eqref{identity-SBT2} can be proved just by rearranging the right-hand side of \eqref{H-fundamental}.
In fact, straightforward calculations that use \eqref{def-R-H0} and Minkowski's identity
\begin{equation}
\label{minkowski}
\int_\Ga H\,q_\nu\,dS_x=|\Ga|
\end{equation}
tell us that 
\begin{multline*}
\int_{\Ga}(H_0-H)\,(u_\nu-q_\nu)\,u_\nu\,dS_x+\\
\int_{\Ga}(H_0-H)\, (u_\nu-R)\,q_\nu\, dS_x=
\int_\Ga \left( H_0-H \right) u_\nu^2 \, dS_x ,
\end{multline*}
and hence \eqref{identity-SBT2} follows.
\end{proof}

%
%
%
%
%
%
We finally show that, if $\Ga$ is mean-convex, that is $H \ge 0$, \eqref{fundamental-identity2} can also be rearranged into an identity leading to Heintze-Karcher's inequality \eqref{heintze-karcher} below (see \cite{HK}).
%
%

\begin{thm}[\cite{MP}]
\label{thm:identityheintze-karcher}
Let $\Om \subset \mathbb R^N$ be a bounded domain with boundary $\Ga$ of class $C^2$ and
denote by $H$ the mean curvature of $\Ga$.
%
%
Let $u$ be the solution of \eqref{serrin1}. 
\par
If $\Ga$ is mean-convex, then we have the following identity:
\begin{multline}
\label{heintze-karcher-identity}
\frac1{N-1}\,\int_{\Om} \left\{ |\na ^2 u|^2-\frac{(\De u)^2}{N}\right\}\,dx +\int_\Ga\frac{(1-H\,u_\nu)^2}{H}\,dS_x =
\\
\int_\Ga\frac{dS_x}{H}-N |\Om|.
\end{multline}
\end{thm}

\begin{proof}
Since \eqref{volume} holds, from \eqref{fundamental} we obtain that
\begin{equation}
\label{fundamental-identity2}
\frac1{N-1}\int_{\Om} \left\{ |\na ^2 u|^2-\frac{(\De u)^2}{N}\right\}\,dx = \int_{\Ga}(1-H u_\nu)\, u_\nu\,dS_x.
\end{equation}
Notice that \eqref{fundamental-identity2} still holds without the assumption $H\ge0$.

Since $u_\nu$ and $H$ are continuous, the identity
\begin{equation}\label{eq:HKsenzastrictly}
\frac{(1-H\,u_\nu)^2}{H}=-(1-H\,u_\nu) u_\nu+\frac1{H}-u_\nu,
\end{equation}
holds pointwise for $H\ge0$.

If the integral $\int_\Ga \frac{dS_x}{H}$ is infinite, then \eqref{heintze-karcher-identity} trivially holds. Otherwise $1/H$ is finite (and hence $H > 0$) almost everywhere
on $\Ga$ and hence, \eqref{eq:HKsenzastrictly} holds almost everywhere.
Thus, by integrating \eqref{eq:HKsenzastrictly} on $\Ga$, summing the result up to \eqref{fundamental-identity2} and taking into account \eqref{volume},
we get \eqref{heintze-karcher-identity}.
\end{proof}

\section{Symmetry results}\label{sec:symmetry results via integral identities}
%
%
The quantity at the right-hand side of \eqref{differential-identity}, that we call Cauchy-Schwarz deficit for the hessian matrix $\na^2 u$, will play the role of detector of spherical symmetry, as is clear from the following lemma. In what follows, $I$ denotes the identity matrix.

\begin{lem}[Spherical detector]
\label{lem:sphericaldetector}
Let $\Om \subseteq \RR^N$ be a domain and $u \in C^2 (\Om)$.
Then it holds that
\begin{equation}
\label{newton}
|\na^2 u|^2 - \frac{(\De u)^2}{ N} \ge 0 \quad \mbox{in } \Om,
\end{equation}
and the equality sign holds if and only if $u$ is a quadratic polynomial. 

Moreover, if $u$ is the solution of \eqref{serrin1}, the equality sign holds if and only if $\Ga$ is a sphere (and hence $\Om$ is a ball) of radius $R$ given by \eqref{def-R-H0} and
$$
u(x)= \frac{1}{2} \, (|x|^2 - R^2),
$$
up to a translation.
\end{lem}
\begin{proof}
We regard the matrices $\na^2 u$ and $I$ as vectors in $\RR^{N^2}$.
Inequality \eqref{newton} is then the classical Cauchy-Schwarz inequality.
%
%

For the characterization of the equality case, we first consider $u$ solution of \eqref{serrin1} and we set $w(x)= u(x) - |x|^2 /2$. Since $w$ is harmonic, direct computations show that
$$
0=|\na^2 u|^2 - \frac{(\De u)^2}{ N} = | \na w|^2.
$$
Thus, $w$ is affine and $u$ is quadratic. Therefore, $u$ can be written in the form
%
\eqref{quadratic},
for some $z\in\RR^N$ and $a\in\RR$.
\par
Since $u=0$ on $\Ga$, then $|x-z|^2=a$ for $x\in\Ga$, that is $\Ga$ must be a sphere centered at $z$.
Moreover, $a$ must be positive and
$$
\sqrt{a}\,|\Ga|=\int_\Ga |x-z|\,dS_x=\int_\Ga (x-z)\cdot\nu(x)\,dS_x=N\,|\Om|.
$$
In conclusion, $\Ga$ is a sphere centered at $z$ with radius $R$.

If $u$ is any $C^2$ function for which the equality holds in \eqref{newton}, we can still
conclude that $u$ is a quadratic polynomial.
In fact, being \eqref{newton} the classical Cauchy-Schwarz inequality in $\RR^{N^2}$, if the equality sign holds we have that $\na^2 u (x) = \la(x) \, I$.
Since 
$$u_{iij}= u_{iji} = 0 \quad \mbox{and} \quad u_{jjj}= \la_j = u_{iij} , $$
we immediately realize that the function
$\la (x)$ must be a constant $\la$.
Notice that, this conclusion surely holds if $u$ is of class $C^3$. However, even if $u$ is of class $C^2$, the last two differential identities still hold in the sense of distributions and the same conclusion follows.
Thus, by integrating directly the system $\na^2 u = \la \, I$, we obtain that $u$ must be a quadratic polynomial. 
%
%
%
%
\end{proof}

As an immediate corollary of Theorem \ref{thm:serrinidentity} we have the following more general version of Serrin's symmetry result.

\begin{thm}[Symmetry for the torsional rigidity density, \cite{MP2}] 
Let $\Om \subset \mathbb R^N$ be a bounded domain with boundary $\Ga$ of class $C^{1,\al}$, $0<\al\le 1$. Let $R$ be the positive constant defined in \eqref{def-R-H0}, and $u$ be the solution of  \eqref{serrin1}.

If the right-hand side of \eqref{idwps} is non-positive,
$\Ga$ must be a sphere (and hence $\Om$ a ball) of radius $R$. 
The same conclusion clearly holds if either $u_\nu$ is constant on $\Ga$ or $u_\nu - q_\nu = 0$ on $\Ga$.
\end{thm}

\begin{proof}
If the right-hand side of \eqref{idwps} is non-positive, then the integrand at the left-hand side must be zero,
 being non-negative by \eqref{newton} and the maximum principle for $u$. Then \eqref{newton} must hold with the equality sign, since $u<0$ on $\Om$, by the strong maximum principle. The conclusion follows from Lemma \ref{lem:sphericaldetector}.
\par
Finally, if $u_\nu\equiv c$ on $\Ga$ for some constant $c$, then 
$$
c\,|\Ga|=\int_\Ga u_\nu\,dS_x=N\,|\Om|,
$$
that is $c=R$, and hence we can apply the previous argument.

The same conclusion clearly holds if $u_\nu-q_\nu=0$ on $\Ga$. Notice that in this case a simpler alternative proof is available. In fact, since $u-q$ is harmonic in $\Om$ and $(u-q)_\nu=0$ on $\Ga$, then $u-q$ must be constant  on $\ol{\Om}$.
\end{proof}

\begin{rem}
{\rm
Based on the work of Vogel \cite{Vo}, the regularity assumption on $\Om$ can be dropped, if \eqref{serrin1} and the conditions $u=0$, $u_\nu=c$ on $\Ga$ (for some constant $c$) are stated in a weak sense.
In fact, if $u \in W_{loc}^{1,p} \left( \Om \right)$ is a function such that
$$
\int_{\Om} \na u \cdot \na \phi \, dx + N \int_\Om \phi \, dx =0
$$
for every $\phi \in C_0^\infty \left( \Om \right)$, 
and for all $\ve >0$ there exists a neighborhood $U_\ve \supset \Ga$ such that $|u(x)| < \ve$ and $\left| |\na u| - c \right| < \ve$ for a.e. $x \in U_\ve \cap  \Om $, \cite[Theorem 1]{Vo} ensures that $\Ga$ must be of class $C^2$.
%
%
%
}
\end{rem}

As an immediate corollary of
%
%
Theorem \ref{thm:identitySBT} we have the following
more general version of the Soap Bubble Theorem.

\begin{thm}[Soap Bubble-type Theorem, \cite{MP}] 
\label{th:SBT}
Let $\Ga\subset\RR^N$ be a surface of class  $C^2$, which is the boundary of a bounded domain $\Om\subset\RR^N$, and let $u$ be the solution of \eqref{serrin1}.  
Let $R$ and $H_0$ be the two positive constants defined in \eqref{def-R-H0}.

If the right-hand side of \eqref{H-fundamental} is non-positive, then $\Ga$ must be a sphere
%
%
of radius $R$.
In particular, the same conclusion holds if either the mean curvature $H$ of $\Ga$ satisfies the inequality $H\ge H_0$ or $H$ is constant on $\Ga$.
\end{thm}

\begin{proof}
If the right-hand side in \eqref{H-fundamental} is non-positive (and this certainly happens if $H\ge H_0$ on $\Ga$), then both summands at the left-hand side must 
be zero, being non-negative.
%
%
\par
The fact that also the first summand is zero gives that the Cauchy-Schwarz deficit for the hessian matrix $\na^2 u$ must be identically zero and the conclusion follows from Lemma \ref{lem:sphericaldetector}.
\par
If $H$ equals some constant, instead, then \eqref{minkowski} tells us that the constant must equal $H_0$, and hence we can apply the previous argument.
\end{proof}

\begin{rem}
{\rm
(i) As pointed out in the previous proof, the assumption of the theorem also gives that the second summand at the left-hand side of \eqref{H-fundamental} must be zero, and hence $u_\nu = R$ on $\Ga$. Thus, Serrin's overdetermining condition \eqref{serrin2} holds {\it before showing that $\Om$ is a ball}.
%
%
\par
(ii) We observe that the assumption that $H\ge H_0$ on $\Ga$ gives that $H\equiv H_0$, anyway, if $\Om$ is strictly star-shaped with respect to some origin $p$. In fact, by Minkowski's identity \eqref{minkowski}, we obtain that
$$
0\le\int_\Ga [H(x)-H_0] \lan(x-p), \nu(x)\ran\,dS_x=|\Ga|-H_0 \int_\Ga \lan(x-p), \nu(x)\ran\,dS_x=0,
$$
and we know that $\lan(x-p), \nu(x)\ran>0$ for $x\in\Ga$.
}
\end{rem}


%
We now show that \eqref{heintze-karcher-identity} implies Heintze-Karcher's inequality (see \cite{HK}). 
We mention that our proof is slightly different from that of A. Ros in \cite{Ro} and relates the equality case for Heintze-Karcher's inequality to a new overdetermined problem, which is described in item (ii) of the following theorem. As it will be clear, \eqref{heintze-karcher-identity} also gives an alternative proof of Alexandrov's theorem via the characterization of the equality sign in
Heintze-Karcher's inequality \eqref{heintze-karcher}.

\begin{thm}[Heintze-Karcher's inequality and a related overdetermined problem, \cite{MP}]
\label{thm:heintze-karcher}
Let $\Ga\subset\RR^N$ be a mean-convex surface of class  $C^2$, which is the boundary of a bounded domain $\Om\subset\RR^N$, and denote by $H$ its mean curvature. Then,

\begin{enumerate}[(i)]
\item Heintze-Karcher's inequality
\begin{equation}
\label{heintze-karcher}
\int_\Ga \frac{dS_x}{H}\ge N |\Om|
\end{equation}
holds and the equality sign is attained if and only if $\Om$ is a ball;
\item $\Om$ is a ball if and only if the solution $u$ of \eqref{serrin1} satisfies 
\begin{equation}\label{equa:overdeteitemi}
u_\nu(x)=1/H(x) \quad \mbox{for every} \quad x\in\Ga .
\end{equation}
\end{enumerate}
 
\end{thm}

\begin{proof}
(i) Both summands at the left-hand side of \eqref{heintze-karcher-identity} are non-negative and hence \eqref{heintze-karcher} follows.
If the right-hand side is zero, those summands must be zero. The vanishing of the first summand implies that $\Om$ is a ball, as already noticed. Note in passing that the vanishing of the second summand gives that $u_\nu=1/H$ on $\Ga$, which also implies radial symmetry, by item (ii).

(ii) It is clear that, if $\Om$ is a ball, then \eqref{equa:overdeteitemi} holds. Conversely, it is easy to check that the right-hand side and the second summand of the left-hand side of \eqref{heintze-karcher-identity} are zero when \eqref{equa:overdeteitemi} occurs.
Thus, the conclusion follows from Lemma \ref{lem:sphericaldetector}.
\par
Notice that assumption \eqref{equa:overdeteitemi} implies that $H$ must be positive, since $u_{\nu}$ is positive and finite.     
\end{proof}

\begin{rem}
{\rm
(i) As already noticed in \cite{Ro}, the characterization of the equality sign of Heintze-Karcher's inequality given in item (i) of Theorem \ref{thm:heintze-karcher}, leads to a different proof of the Soap Bubble Theorem. In fact, if $H$ equals some constant on $\Ga$, by using Minkowski's identity \eqref{minkowski} we know that such a constant must have the value $H_0$ in \eqref{def-R-H0}.
Thus, \eqref{heintze-karcher} holds with the equality sign, and hence $\Ga$ is a sphere by item (i) of Theorem \ref{thm:heintze-karcher}.

(ii) Item (ii) of Theorem \ref{thm:heintze-karcher} could be proved also by using \eqref{fundamental-identity2} instead of \eqref{heintze-karcher-identity}. In fact, it is immediate to check that the right-hand side of \eqref{fundamental-identity2} is zero when \eqref{equa:overdeteitemi} occurs.
}
\end{rem}

We conclude this section by giving an alternative proof of Serrin's theorem that can have its own interest. Unfortunately, it only works for star-shaped domains.

\begin{thm}[Theorem 2.4, \cite{MP}]
Let $\Om\subset\RR^N$, $N\ge2$, be a bounded domain with boundary $\Ga$ of class $C^2$.
%
%
Assume that $\lan(x-p), \nu(x)\ran>0$ for every $x\in\Ga$ and some $p\in\Om$, and let $u$ be the solution of \eqref{serrin1}.
\par
Then, $\Om$ is a ball if and only if $u$ satisfies \eqref{serrin2}.
\end{thm}

\begin{proof}
It is clear that, if $\Om$ is a ball, then \eqref{serrin2} holds. Conversely, we shall check that the right-hand side of \eqref{fundamental} is zero when \eqref{serrin2} occurs.
\par
Let $u_\nu$ be constant on $\Ga$; by \eqref{volume} we know that that constant equals the value $R$ given in \eqref{def-R-H0}. Also, notice that $1-H u_\nu\ge 0$ on $\Ga$. In fact, the function $P$ in \eqref{P-function}
is subharmonic in $\Om$, since $\De P\ge 0$ by \eqref{differential-identity}. Thus, it attains its maximum on $\Ga$, where it is constant.
We thus have that
$$
0\le P_\nu=u_\nu u_{\nu \nu}-u_\nu=(N-1)\,(1-H u_\nu)\,u_\nu \ \mbox{ on } \ \Ga.
$$
Now,
\begin{multline*}
0\le\int_\Ga [1-H(x) u_\nu(x)]\,\lan (x-p),\nu(x)\ran\,dS_x=\\
\int_\Ga [1-H(x) R]\,\lan (x-p),\nu(x)\ran\,dS_x=0,
\end{multline*}
by \eqref{minkowski}. Thus, $1-H u_\nu\equiv 0$ on $\Ga$ and hence item (ii) of Theorem \ref{thm:heintze-karcher} applies. 
\end{proof}

\chapter{Radial symmetry for \texorpdfstring{$p$}{p}-harmonic functions in exterior and punctured domains}\label{chap:PogAA}
\chaptermark{Radial symmetry for $p$-harmonic functions}

In this chapter, we collect symmetry results obtained in \cite{Pog}, involving $p$-harmonic functions in exterior and punctured domains.

\section{Motivations and statement of results}

The \textit{electrostatic $p$-capacity} of a bounded open set $\Om \subset \RR^N$, $1<p<N$, is defined by
\begin{equation}\label{def:Capacity}
\pCap_p(\Om)= \inf \left\lbrace \int_{\RR^N} | \na w|^p : \; w \in C^{\infty}_{0} (\RR^N), \; w \ge 1 \; \text{in} \; \Om \right\rbrace
\end{equation}

Under appropriate sufficient conditions, there exists a unique minimizing function $u$ of \eqref{def:Capacity}; such function is called the {\it $p$-capacitary potential} of $\Om$, and satisfies
\begin{equation}
\label{Problemcapacity}
\De_p u = 0 \, \mbox{ in } \RR^N \setminus \ol{\Om} , \; u=1 \, \mbox{ on } \Ga, \; \lim_{|x|\rightarrow \infty} u(x)=0 ,
\end{equation}
where $\Ga$ is the boundary of $\Om$ and $\De_p$ denotes the $p$-Laplace operator defined by
\begin{equation*}
\De_p u = \dv \left( |\na u|^{p-2} \na u \right).
\end{equation*}


It is well known that the $p$-capacity could be equivalently defined by means of the $p$-capacitary potential $u$ as
\begin{equation}\label{eq:caponboundary}
\pCap_p(\Om) = \int_{\RR^N \setminus \ol{\Om} } |\na u|^p \, dx = \int_\Ga | \na u|^{p-1} \, dS_x,
\end{equation}
where the second equality follows by integration by parts (see the proof of Lemma~\ref{lem:asymptoticcap}).

In Section \ref{sec:exterior} we consider Problem \eqref{Problemcapacity} under Serrin's overdetermined condition given by
\begin{equation}\label{eq:overdetermination}
| \na u| = c  \mbox{ on } \Ga,
\end{equation}
and we prove the following result.

\begin{thm}[\cite{Pog}]\label{thm:Serrinesterno}
Let $\Om \subset \RR^N$ be a bounded domain (i.e., a connected bounded open set).
For $1<p<N$, the problem \eqref{Problemcapacity}, \eqref{eq:overdetermination} admits a weak solution if and only if $\Om$ is a ball.
\end{thm}


By a weak solution in the statement of Theorem \ref{thm:Serrinesterno} we mean a function $u \in W_{loc}^{1,p} \left( \RR^N \setminus \ol{\Om} \right)$ such that
$$
\int_{\RR^N \setminus \ol{\Om}} |\na u|^{p-2} \, \na u \cdot \na \phi \, dx=0
$$
for every $\phi \in C_0^\infty \left( \RR^N \setminus \ol{\Om} \right)$
and satisfying the boundary conditions in the weak sense, i.e.: for all $\ve >0$ there exists a neighborhood $U_\ve \supset \Ga$ such that $|u(x)- 1| < \ve$ and $\left| |\na u| - c \right| < \ve$ for a.e. $x \in U_\ve \cap \left( \RR^N \setminus \ol{\Om} \right)$.

Notice the complete absence in Theorem \ref{thm:Serrinesterno} of any smoothness assumption as well as of any other assumption on the domain $\Om$.

Under the additional assumption that $\Om$ is star-shaped, Theorem \ref{thm:Serrinesterno} has been proved by Garofalo and Sartori in \cite{GS}, by extending to the case $1<p<N$ the tools developed in the case $p=2$ by Payne and Philippin (\cite{PP2}, \cite{Ph}).
The proof in \cite{GS}, which combines integral identities and a maximum principle for an appropriate $P$-function, bears a resemblance to Weinberger's proof (\cite{We}) of symmetry for the archetype torsion problem \eqref{serrin1}
under Serrin's overdetermined condition
\eqref{serrin2}.
%

%
%
%
%

To prove Theorem \ref{thm:Serrinesterno}, we improve on the arguments used in \cite{GS} and we exploit, as a new crucial ingredient, Theorem \ref{th:SBT}, that is the Soap Bubble-type Theorem proved via integral identities in Section \ref{sec:symmetry results via integral identities}.

%


Symmetry for Problem \eqref{Problemcapacity}, \eqref{eq:overdetermination} was first obtained by Reichel (\cite{Re1}, \cite{Re2}) by adapting the method of moving planes introduced by Serrin (\cite{Se}) to prove symmetry for the overdetermined torsion problem \eqref{serrin1}, \eqref{serrin2}.
In Reichel's works (\cite{Re1}, \cite{Re2}) the star-shapedness assumption is not requested, but the domain $\Om$ is a priori assumed to be $C^{2,\al}$ and the solution $u$ is assumed to be of class $C^{1, \al}(\RR^N \setminus \Om)$.

For completeness let us mention that many alternative proofs and improvements in various directions of symmetry results for Serrin's problems relative to the equations \eqref{serrin1} and \eqref{Problemcapacity} have been obtained in the years and can be found in the literature: for the overdetermined torsion problem \eqref{serrin1}, \eqref{serrin2} see for example \cite{PS}, \cite{GL}, \cite{DP}, \cite{BH}, \cite{BNST}, \cite{FK}, \cite{CiS}, \cite{WX}, \cite{MP2}, \cite{BC}, and the surveys \cite{Mag}, \cite{NT}, \cite{Ka}; for the exterior overdetermined problem \eqref{Problemcapacity}, \eqref{eq:overdetermination} where the domain $\Om$ is assumed to be convex see \cite{MR}, \cite{BCS}, \cite{BC}, and \cite{FMP}.
%
%


In Section \ref{sec:casoconforme}, we establish the result corresponding to Theorem \ref{thm:Serrinesterno} in the special case $1<p=N$.
In this case, the problem corresponding to \eqref{Problemcapacity} is (see e.g. \cite{CC}):
\begin{equation}
\label{ConformalProblemcapacity}
\De_N u = 0 \, \mbox{ in } \RR^N \setminus \ol{\Om} , \; u=1 \, \mbox{ on } \Ga, \; u(x) \sim - \ln{|x|}  \text{ as } |x|\to \infty ,
\end{equation}
where $\sim$ means that
\begin{equation}\label{eq:confasymppre}
c_1 \le \frac{u(x)}{ ( - \ln{|x|} ) } \le c_2, \text{ as } |x| \to \infty
\end{equation}
for some positive constants $c_1$,$c_2$.
What we prove is the following.
\begin{thm}[\cite{Pog}]\label{thm:conforme}
Let $\Om \subset \RR^N$, $1<N$, be a bounded domain.
The problem \eqref{ConformalProblemcapacity} with the overdetermined condition \eqref{eq:overdetermination}
admits a weak solution if and only if $\Om$ is a ball.
\end{thm}

A proof of Theorem \ref{thm:conforme} that uses the method of moving planes is contained in \cite{Re2}, under the additional a priori smoothness assumptions $\Ga \in C^{2, \al}$, $u \in C^{1,\al}(\RR^N \setminus \Om)$.
Our proof via integral identities seems to be new and cannot be found in the literature unless for the classical case $p=N=2$, which has been treated with similar arguments in \cite{Mar} (for piecewise smooth domains) and \cite{MR} (for Lipschitz domains). Moreover, in Theorem \ref{thm:conforme} no assumptions on the domain $\Om$ are made.

We mention that a related symmetry result for the $N$-capacitary potential in a bounded (smooth) star-shaped ring domain has been established in \cite{PP3}.


In Section \ref{sec:Interior}, we show how the same ideas used in our proof of Theorem \ref{thm:Serrinesterno} can be adapted to give a symmetry result for a similar problem in a bounded punctured domain. More precisely, we prove the following theorem concerning the problem
\begin{equation}\label{Pdelta}
- \De_p u = K \, \de_0 \, \mbox{ in } \Om , \, u=c \, \mbox{ on } \Ga,
\end{equation} 
under Serrin's overdetermined condition 
\begin{equation}\label{eq:interno-overdetermination}
| \na u| = 1  \mbox{ on } \Ga ,
\end{equation}
where with $\de_0$ we denote the Dirac delta centered at the origin $0 \in \Om$ and $K$ is some positive normalization constant.


\begin{thm}[\cite{Pog}]
\label{thm:Serrininterno}
Let $\Om \subset \RR^N$ be a bounded star-shaped domain.
For $1<p<N$, the problem \eqref{Pdelta}, \eqref{eq:interno-overdetermination}
admits a weak solution if and only if $\Om$ is a ball centered at the origin.
\end{thm}

It should be noticed, however, that in Theorem \ref{thm:Serrininterno} we need to assume $\Om$ to be star-shaped, restriction that is not present in the proofs of Payne and Schaefer (\cite{PS})(for the case $p=2$), Alessandrini and Rosset (\cite{AR}), and Enciso and Peralta-Salas (\cite{EP}).
We mention that the proof in \cite{AR} uses an adaptation of the method of moving planes, the proof in \cite{EP} is in the wake of Weinberger, and both of them also cover the special case $p=N$.



For $p=2$, Problem \eqref{Problemcapacity} arises naturally in electrostatics. In this context $u$ is the (normalized) potential of the electric field $\na u$ generated by a conductor $\Om$.
We recall that when $\Om$ is in the electric equilibrium, the electric field in the interior of $\Om$ is null, and hence the electric potential $u$ is constant in $\Om$ (i.e. $u \equiv 1$ in $\Om$); moreover, the electric charges present in the conductor are distributed on the boundary $\Ga$ of $\Om$.
It is also known that the electric field $\na u$ on $\Ga$ is orthogonal to $\Ga$ -- i.e. $\na u = u_\nu \, \nu$, where
%
%
$\nu$ denotes the outer unit normal with respect to $\Om$ and $u_\nu$ denotes the derivative of $u$ in the direction $\nu$
-- and its intensity is given\footnote{More precisely, in free space the intensity of the electric field on $\Ga$ is given by $|\na u|_{| \Ga} = - u_\nu = \frac{\rho}{\eps_0}$, where $\rho$ is the surface charge density over $\Ga$ and $\eps_0$ is the vacuum permittivity. We ignored the constant $\eps_0$ to be coherent with the mathematical definition of capacity given in \eqref{def:Capacity}.} by the surface charge density over $\Ga$. In this context the capacity is defined as the total electric charge needed to induce the potential $u$, that is 
$$\pCap(\Om)= \int_\Ga | \na u | \, dS_x, $$
in accordance with \eqref{eq:caponboundary} for $p=2$.
Following this physical interpretation Theorem \ref{thm:Serrinesterno} simply states that the electric field on the boundary $\Ga$ of the conductor $\Om$ is constant -- or equivalently that the charges present in the conductor are uniformly distributed on $\Ga$ (i.e. the surface charge density is constant over $\Ga$) -- if and only if $\Om$ is a round ball.


Another result of interest in the same context that is related to Problem \eqref{Problemcapacity}, \eqref{eq:overdetermination} is a Poincar\'e's theorem known as the isoperimetric inequality for the capacity, stating that, among sets having given volume, the ball minimizes $\pCap(\Om)$.
We mention that a proof of this inequality that hinges on rearrangement techniques can be found in \cite[Section 1.12]{PZ} (see also \cite{Ja} for a useful review of that proof).
Here, -- as done in the introduction for the isoperimetric problem and Saint Venant's principle -- we just want to underline the relation present between this result and Problem \eqref{Problemcapacity}, \eqref{eq:overdetermination} (for $p=2$).
In fact, once that the existence of a minimizing set $\Om$ is established, we can show through the technique of shape derivatives that the solution of \eqref{Problemcapacity} in $\Om$ also satisfies the overdetermined condition \eqref{eq:overdetermination} on $\Ga$; the reasoning is the following.


We consider the evolution of the domains $\Om_t$ given by
\begin{equation*}
\Om_t = \cM_t (\Om), 
\end{equation*}
where $\Om= \Om_0$ is fixed, and $\cM_t: \RR^N \to \RR^N$ is a mapping such that
$$
\cM_0 (x) =x, \; \cM'_0 (x) = \phi(x) \nu(x),
$$
where the symbol $'$ means differentiation with respect to $t$, $\phi$ is any compactly supported continuous function, and $\nu$ is a proper extension of the unit normal vector field to a tubular neighborhood of $\Ga_0$.
%
%
Thus, we consider $u(t,x)$, solution of Problem \eqref{Problemcapacity} in $\Om= \Om_t$, and the two functions (in the variable $t$) $\pCap (\Om_t)$ and $| \Om_t|$.
%
%
%
Since $\Om=\Om_0$ is the domain that minimizes $\pCap(\Om_t)$ among all the domains in the one-parameter family $\left\lbrace \Om_t \right\rbrace_{t \in \RR}$ that have prescribed volume $|\Om_t|=V$, by using the method of Lagrange multipliers and Hadamard's variational formula (see \cite[Chapter 5]{HP}), standard computations lead to prove that there exists a number $\la$ such that
\begin{equation*}
\int_{\Ga} \phi(x) \left[ u_\nu^2 (x) - \la \right] \, dS_x = 0,
\end{equation*}
where we have set $u(x)=u(0,x)$.
Since $\phi$ is arbitrary, we deduce that $u_\nu^2 \equiv \la$ on $\Ga$, that is, $u$ satisfies the overdetermined condition \eqref{eq:overdetermination} on $\Ga$.



Other applications of Problem \eqref{Problemcapacity} are related to quantum theory, acoustic, theory of musical instruments, and the study of heat, electrical and fluid flow (see for example \cite{CFG}, \cite{DZH}, \cite{BC} and references therein).

Finally, we just mention that also Problem \eqref{Pdelta} arises in electrostatics for $p=2$: in this case, Theorem \ref{thm:Serrininterno} states that the electric field on a conducting hypersurface enclosing a charge is constant if and only if the conductor is a sphere centered at the charge.

\section{The exterior problem: proof of Theorem \ref{thm:Serrinesterno}}\label{sec:exterior}
\sectionmark{The exterior problem: proof of Theorem \ref{thm:Serrinesterno}}

In order to prove Theorem \ref{thm:Serrinesterno}, we start by collecting all the necessary ingredients.
In this section $u$ denotes a weak solution to \eqref{Problemcapacity},\eqref{eq:overdetermination}, in the sense explained in the Introduction, and it holds that $1<p<N$.



\begin{rem}[On the regularity]\label{rem:regularity}
{\rm
Due to the degeneracy or singularity of the $p$-Laplacian (when $p \neq 2$) at the critical points of $u$, $u$ is in general only $C^{1, \al}_{loc}(\RR^N \setminus \ol{\Om})$ (see \cite{Di}, \cite{Le}, \cite{To}), whereas it is $C^{\infty}$ in a neighborhood of any of its regular points thanks to standard elliptic regularity theory (see \cite{GT}). However, as already noticed in \cite{GS}, the additional assumption given by the weak boundary condition \eqref{eq:overdetermination} ensures that $u$ can be extended to a $C^2$-function in a neighborhood of $\Ga$, so that by using the work of Vogel \cite{Vo} we get that $\Ga$ is of class $C^2$. Thus, by \cite[Theorem 1]{Li} it turns out that $u$ is $C^{1, \al}_{loc} (\RR^N \setminus \Om)$.
As a consequence we can now interpret the boundary condition in \eqref{Problemcapacity} and the one in \eqref{eq:overdetermination} in the classical strong sense.
}
\end{rem}

\begin{rem}
{\rm
More precisely, by using Vogel's work \cite{Vo}, which is based on the deep results on free boundaries contained in \cite{AC} and \cite{ACF}, one can prove that $\Ga$ is of class $C^{2,\al}$ from each side. Even if 
%
%
here we do not need this refinement, it should be noticed that in light of this remark the arguments contained in \cite{Re2} give an alternative and complete proof of Theorem \ref{thm:Serrinesterno} (and also of Theorem \ref{thm:conforme}). In fact, the smoothness assumptions of \cite[Theorem 1]{Re2} are satisfied.
}
\end{rem}


By using the ideas contained in \cite{GS} together with a result of Kichenassamy and V\'eron (\cite{KV}), we can recover the following useful asymptotic expansion for $u$ as $|x|$ tends to infinity.
\begin{lem}[Asymptotic expansion, \cite{Pog}]\label{lem:asymptoticcap}
As $|x|$ tends to infinity it holds that
\begin{equation}\label{eq:cap-asymptotic u}
u(x) = \frac{p-1}{N-p} \left( \frac{\pCap_p (\Om)}{\om_N} \right)^{\frac{1}{p-1}} |x|^{- \frac{N-p}{p-1}} + o (|x|^{- \frac{N-p}{p-1}}).
\end{equation}
\end{lem}


The computations leading to determine the constant of proportionality in \eqref{eq:cap-asymptotic u} originate from \cite{GS}, where they have been used to give a complete proof -- which works without invoking \cite{KV} -- of a different but related result
(\cite[Theorem~3.1]{GS}), which is described in more details in Remark \ref{rem:Garofaoloasymptotic} below.
We mention that, later, in \cite{CoS} the same computations have been treated with tools of convex analysis and used together with the result contained in \cite[Remark 1.5]{KV} to prove Lemma \ref{lem:asymptoticcap} for convex sets.

\begin{proof}[Proof of Lemma \ref{lem:asymptoticcap}]
As noticed in \cite{GS}, if $u$ is a solution of \eqref{Problemcapacity}, then the weak comparison principle for the $p$-Laplacian (see \cite{HKM}) implies the existence of positive constants $c_1$,$c_2$, $R_0$ such that
$$
c_1 \mu(x) \le u(x) \le c_2 \mu(x), \quad \text{ if } |x| \ge R_0,
$$
where $\mu(x)$ denotes the radial fundamental solution of the $p$-Laplace operator given by
\begin{equation*}
\mu(x) = \frac{(p-1)}{(N-p)} \frac{1}{\om_N^{\frac{1}{p-1}}} |x|^{\frac{p-N}{p-1}}.
\end{equation*}
Thus we can apply the result of Kichenassamy and V\'eron (\cite[Remark 1.5]{KV}) and state that there exists a constant $\ga$ such that
\begin{equation}\label{eq:passKV}
\lim_{|x| \to \infty} \frac{u(x)}{\mu(x)} = \ga, \quad \lim_{|x| \to \infty} |x|^{\frac{N-p}{p-1} + | \al|}  D^\al \left( u - \ga \mu \right)=0,
\end{equation}
for all multi-indices $\al=(\al_1, \dots, \al_N)$ with $|\al|= \al_1+ \dots + \al_N \ge 1$.
To establish \eqref{eq:cap-asymptotic u} now it is enough to prove that
\begin{equation}\label{eq:0gamma}
\ga = \pCap_p (\Om)^{\frac{1}{p-1}};
\end{equation}
this can be easily done, as already noticed in \cite{GS}, through the following integration by parts that holds true by the $p$-harmonicity of $u$:
\begin{equation}\label{eq:capintegrperasymptotic}
- \int_{\Ga} |\na u|^{p-2} u_{\nu} \, dS_x = - \lim_{R \to \infty} \int_{\pa B_R} |\na u|^{p-2} u_{\nu_{B_R}} \, dS_x .
\end{equation}
Here as in the rest of the chapter $\nu_{B_R}$ denotes the outer unit normal with respect to the ball $B_R$ of radius $R$, $\nu$ is the outer unit normal with respect to $\Om$, and $u_\nu$ (resp.
$u_{\nu_{B_R}}$) is the derivative of $u$ in the direction $\nu$ (resp. $\nu_{B_R}$). 
The left-hand side of \eqref{eq:capintegrperasymptotic} is exactly $\pCap_p (\Om)$ as it is clear by \eqref{eq:caponboundary} and the fact that 
\begin{equation}\label{eq:capnormagradmenodernormale}
| \na u|= - u_\nu \text{ on } \Ga ;
\end{equation}
moreover, the limit in the right-hand side of \eqref{eq:capintegrperasymptotic} can be explicitly computed by using the second equation in \eqref{eq:passKV} (with $|\al|=1$) and it turns out to be $\ga^{p-1}$. Thus, \eqref{eq:0gamma} is proved and \eqref{eq:cap-asymptotic u} follows.

For completeness, we explain here how to prove the second identity in \eqref{eq:caponboundary}: we take the limit for $R \to \infty$ of the following integration by parts made on $B_R \setminus \ol{\Om} $ and we note that the integral on $\pa B_R$ converge to zero due to
\eqref{eq:passKV}:
\begin{equation*}
\int_{B_R \setminus \ol{\Om} } |\na u|^p \, dx = \int_{\pa B_R} u | \na u|^{p-2} u_{\nu_{B_R}} \, dS_x - \int_{\Ga} |\na u|^{p-2} u_\nu \, dS_x .
\end{equation*}
The desired identity is thus proved just by recalling \eqref{eq:capnormagradmenodernormale}.
\end{proof}


It is well known that the value $c$ of $| \na u|$ on $ \Ga$ appearing in the overdetermined condition \eqref{eq:overdetermination} can be explicitly computed.

\begin{lem}[Explicit value of $c$ in \eqref{eq:overdetermination}]
The constant $c$ appearing in \eqref{eq:overdetermination} equals
\begin{equation}\label{eq:valoreunu serrinesterno}
c = \frac{N-p}{p-1} \, H_0,
\end{equation}
where $H_0$ is the constant defined in \eqref{def-R-H0}.
Moreover, the following explicit expression of the $p$-capacity of $\Om$ holds:
\begin{equation}\label{eq:cap in terms isop}
\pCap_p (\Om) = \left( \frac{N-p}{p-1} \right)^{p-1} \frac{| \Ga|^p}{(N | \Om|)^{p-1}}.
\end{equation}
\end{lem}
\proof
It is enough to use \eqref{eq:caponboundary} together with the following Rellich-Pohozaev-type identity:
\begin{equation}\label{eq:Poho}
(N-p)\int_{\RR^N \setminus \ol{\Om} } |\na u|^p \, dx = (p-1)\int_\Ga |\na u|^p <x, \nu> \, dS_x .
\end{equation}
In fact, by using \eqref{eq:overdetermination} in \eqref{eq:caponboundary} and \eqref{eq:Poho} we deduce respectively that
\begin{equation*}
\pCap_p (\Om) = c^{p-1} | \Ga| \quad \text{and} \quad \pCap_p (\Om)=\frac{p-1}{N-p} c^p N | \Om| ,
\end{equation*}
from which we get \eqref{eq:valoreunu serrinesterno} and \eqref{eq:cap in terms isop}.

Equation \eqref{eq:Poho} comes directly by taking the limit for $R \to \infty$ of the following integration by parts made on $B_R \setminus \ol{\Om} $ and noting that the integrals on $\pa B_R$ converge to zero due to the asymptotic going of $u$ at infinity given by \eqref{eq:cap-asymptotic u}:
\begin{multline}\label{eq:provaconforme}
(N-p)\int_{B_R \setminus \ol{\Om} } |\na u|^p \, dx = 
\\
p \int_\Ga |\na u|^{p-2} <x, \na u> u_\nu \, dS_x - \int_\Ga |\na u|^p <x, \nu> \, dS_x -
\\
 p \int_{\pa B_R } |\na u|^{p-2} <x, \na u> u_{\nu_{B_R}} \, dS_x + \int_{\pa B_R } |\na u|^p <x, \nu_{B_R}> \, dS_x .
\end{multline}
Equation \eqref{eq:Poho} is in fact proved just by recalling that 
\begin{equation}\label{eq:capgraddernormnormale}
\na u = u_\nu \, \nu \text{ on }\Ga.
\end{equation}
\endproof


\vspace{1pt}
\noindent
\textbf{The P-function.}
As last ingredient, we introduce the $P$-function
\begin{equation}
\label{def:P}
P= \frac{|\na u|^p}{u^{\frac{p(N-1)}{(N-p)}}}.
\end{equation}
Notice that in the radial case, i.e. if $\Om= B_R (x_0)$ is a ball of radius $R$ centered at the point $x_0$, we have that
\begin{equation}
\label{eq:esplicitaradiale1}
u(x) = \left( \frac{R}{|x  - x_0  |} \right)^{\frac{N-p}{p-1}}
\end{equation}
and thus
$$
P \equiv \left( \frac{N-p}{p-1} \right)^{p} R^{-p}.
$$

In \cite{GS} the authors have studied extensively the properties of the function $P$.
In particular, in \cite[Theorem 2.2]{GS} it is proved that $P$ satisfies the strong maximum principle, i.e. the function $P$ cannot attain a local maximum at an interior point of $\RR^N \setminus \ol{ \Om}$, unless $P$ is constant. We mention that this property for the case $p=2$ was first established in \cite{PP2}.

\vspace{1pt}

Now that we collected all the ingredients, we are in position to give the proof of Theorem \ref{thm:Serrinesterno}.


\begin{proof}[Proof of Theorem \ref{thm:Serrinesterno}]

By \eqref{eq:cap-asymptotic u} it is easy to check that
\begin{equation*}
\lim_{|x| \rightarrow \infty} P(x)= \left( \frac{N-p}{p-1} \right)^{\frac{p(N-1)}{N-p}} \left( \frac{ \om_N }{ \pCap_p(\Om)} \right)^{\frac{p}{N-p}} ,
\end{equation*}
from which, by using \eqref{eq:cap in terms isop}, we get
\begin{equation}\label{eq:prova2}
\lim_{|x| \rightarrow \infty} P(x)= \left( \frac{N-p}{p-1} \right)^{p} \left( \frac{ \om_N \left( N | \Om| \right)^{p-1} }{ | \Ga|^p } \right)^{\frac{p}{N-p}}.
\end{equation}

Moreover, by recalling the boundary condition in 
\eqref{Problemcapacity}, \eqref{eq:overdetermination}, and \eqref{eq:valoreunu serrinesterno}
we can compute that
\begin{equation}\label{eq:prova}
P_{| \Ga} =  \left( \frac{N-p}{p-1} \right)^p  \left( \frac{| \Ga| }{ N | \Om| } \right)^p.
\end{equation}

By using the classical isoperimetric inequality (see, e.g., \cite{BZ})
\begin{equation}\label{eq:ISOPerimetrica}
\frac{| \Ga|^{\frac{N}{N-1}}}{N \om_N^{\frac{1}{N-1}}} \ge |\Om|,
\end{equation}
by \eqref{eq:prova2} and \eqref{eq:prova} it is easy to check that
$$
\lim_{|x| \rightarrow \infty} P(x) \le P_{| \Ga}.
$$
Hence, by the strong maximum principle proved in \cite[Theorem 2.2]{GS}, $P$ attains its maximum on $\Ga$ and thus we can affirm that
\begin{equation}\label{eq:00}
P_{\nu} \le 0 ,
\end{equation}
where, $\nu$ is still the outer unit normal with respect to $\Om$.
If we directly compute $P_{\nu}$ and we use \eqref{eq:capgraddernormnormale}, we find
\begin{equation}\label{eq:01}
P_\nu = p \, u^{- \frac{p(N-1)}{N-p}} |\na u|^{p-2}  \left\lbrace u_{\nu \nu} u_\nu - \frac{N-1}{N-p} \, |\na u|^2 \, \frac{u_\nu}{u} \right\rbrace.
\end{equation}
By the well known differential identity
\begin{equation*}
\De_p u = | \na u|^{p-2} \left\lbrace (p-1) u_{\nu \nu} + (N-1) H u_\nu \right\rbrace \quad \mbox{ on $\Ga$    }
\end{equation*}
and the $p$-harmonicity of $u$ we deduce that
\begin{equation}\label{eq:02}
u_{\nu \nu}= - \frac{N-1}{p-1} H u_\nu.
\end{equation} 
By combining \eqref{eq:00}, \eqref{eq:01}, and \eqref{eq:02} we get
\begin{equation*}
p (N-1)  u^{- \frac{p(N-1)}{N-p}} |\na u|^p  \left\lbrace \frac{- u_{\nu }}{(N-p) u} - \frac{H}{p-1} \right\rbrace \le 0,
\end{equation*}
from which, by using the fact that $u=1$ on $\Ga$, \eqref{eq:overdetermination}, \eqref{eq:capnormagradmenodernormale}, and \eqref{eq:valoreunu serrinesterno} we get
$$H_0 -H \le 0.$$
We can now conclude by using Theorem \ref{th:SBT}.
\end{proof}

\begin{rem}
{\rm
Since the solution of \eqref{Problemcapacity} in a ball is explicitly known, as a corollary of Theorem \ref{thm:Serrinesterno} we get that $u$ is spherically symmetric about the center $x_0$ of (the ball) $\Om$ and it is given by \eqref{eq:esplicitaradiale1} with $R= \frac{N-p}{(p-1)c}$.
}
\end{rem}

\begin{rem}\label{rem:Garofaoloasymptotic}
{\rm
As already mentioned before, in \cite{GS} a result slightly different from Lemma~\ref{lem:asymptoticcap} is used. In fact, in \cite[Theorem 3.1]{GS} it is proved, independently from the work \cite{KV}, that if $P$ takes its supremum at infinity, then the asymptotic expansion \eqref{eq:cap-asymptotic u} holds true and hence
$\lim_{|x| \to \infty} P(x)$ exists and it is given by \eqref{eq:prova2}; clearly, that result would be sufficient to complete the proof of Theorem \ref{thm:Serrinesterno} without invoking Lemma \ref{lem:asymptoticcap}.
}
\end{rem}

\begin{rem}
{\rm
In \cite{GS}, instead of the classical isoperimetric inequality, the authors use the nonlinear version of the isoperimetric inequality for the capacity mentioned in the Introduction stating that, for any bounded open set $A$, if $B_\rho$ denotes a ball such that $|B_\rho|= \frac{ \om_N }{N} \rho^N=|A|$,  it holds that
\begin{equation}\label{eq:Poincare}
\pCap_p (A) \ge \pCap_p (B_\rho) ,
\end{equation}
with equality if and only if $A$ is a ball (for a proof see, e.g., \cite{Ge}).

Since the $p$-capacity of a ball is known explicitly (see, e.g., \cite[Equation (4.7)]{GS}):
\begin{equation*}
\pCap_p(B_\rho)= \om_N \left( \frac{N-p}{p-1} \right)^{p-1} \rho^{N-p} ,
\end{equation*}
if we take $B_\rho$ such that $|B_\rho|= |\Om|$, for $\Om$ \eqref{eq:Poincare} becomes
\begin{equation}\label{eq:isop cap}
\pCap_p (\Om) \ge \om_N \left( \frac{N-p}{p-1} \right)^{p-1} \left( \frac{N |\Om|}{\om_N} \right)^{\frac{N-p}{N}}.
\end{equation}

Now we notice that, since \eqref{eq:cap in terms isop} holds, \eqref{eq:isop cap} is equivalent to the classical isoperimetric inequality \eqref{eq:ISOPerimetrica}.
In fact, if we put \eqref{eq:cap in terms isop} in \eqref{eq:isop cap}, it is easy to check that \eqref{eq:isop cap} becomes exactly \eqref{eq:ISOPerimetrica}.
}
\end{rem}




\section{The case \texorpdfstring{$1<p=N$}{1<p=N}: proof of Theorem \ref{thm:conforme}}\label{sec:casoconforme}

\sectionmark{The case \texorpdfstring{$1<p=N$}{1<p=N}: proof of Theorem \ref{thm:conforme}}

In the present section, we consider $u$ solution to the exterior problem \eqref{ConformalProblemcapacity}, \eqref{eq:overdetermination}, and $1<N$. 
In order to give the proof of Theorem \ref{thm:conforme}, we collect all the necessary ingredients in the following remark.

%
%
%

\begin{rem}
{\rm
(i)(On the regularity). We notice that the regularity results invoked in Remark \ref{rem:regularity} hold when $p=N$, too.

(ii)(Asymptotic expansion). By \eqref{eq:confasymppre}, the result of Kichenassamy and V\'eron (\cite[Remark 1.5]{KV}) applies also in this case and hence we have that \eqref{eq:passKV} holds with 
$$
\mu(x)= - \frac{ \ln{| x |} }{\om_N^{\frac{1}{N-1}}}.
$$
It is easy to check that also Identity \eqref{eq:capintegrperasymptotic} still holds (with $p$ replaced by $N$); by computing the limit in the right-hand side, and by using \eqref{eq:overdetermination} and the fact that $|\na u| = - u_\nu$ in the left-hand side, \eqref{eq:capintegrperasymptotic} leads to 
$$
c^{N-1} | \Ga|= \ga^{N-1},
$$
from which we deduce the following asymptotic expansion for $u$ at infinity:
\begin{equation}\label{eq:conformasymptoticexp}
u(x)= - c \left( \frac{|\Ga|}{\om_N} \right)^{\frac{1}{N-1}} \, \ln{|x|} + O(1).
\end{equation}
}
\end{rem}

We are ready now to prove Theorem \ref{thm:conforme}.

\begin{proof}[Proof of Theorem \ref{thm:conforme}]
We just find the analogous of the Rellich-Pohozaev-type identity \eqref{eq:provaconforme} when $p=N$; since $u$ is $N$-harmonic, now the vector field
$$X= N <x, \na u> | \na u|^{N-2} \na u - | \na u|^N x$$
is divergence-free and thus integration by parts leads to
\begin{equation}\label{eq:conformeintegrbyparts}
\int_{\Ga} <X, \nu> \, dS_x= \lim_{R \to \infty} \int_{\pa B_R} <X, \nu_{B_R}> \, dS_x.
\end{equation}

For the left-hand side in \eqref{eq:conformeintegrbyparts}, by using that $\na u = u_\nu \, \nu$ on $\Ga$ and \eqref{eq:overdetermination} we immediately find that
$$
\int_{\Ga} <X, \nu> \, dS_x= (N-1) c^N N|\Om|.
$$
For the right-hand side in \eqref{eq:conformeintegrbyparts}, by the asymptotic expansion \eqref{eq:conformasymptoticexp} we easily compute that
$$
\lim_{R \to \infty} \int_{\pa B_R} <X, \nu_{B_R}> \, dS_x = (N-1) c^N \frac{|\Ga|^{\frac{N}{N-1}}}{\om_N^{\frac{1}{N-1}}}.
$$
Thus, \eqref{eq:conformeintegrbyparts} becomes
$$
N |\Om| = \frac{|\Ga|^\frac{N}{N-1}}{\om_N^{\frac{1}{N-1}}},
$$
that is the equality case of the classical isoperimetric inequality. Hence $\Om$ must be a ball.
\end{proof}

\begin{rem}
{\rm
As a corollary of Theorem \ref{thm:conforme} we get that $u$ is spherically symmetric about the center $x_0$ of (the ball) $\Om$ and it is given by
$$
u(x)= - c \left( \frac{|\Ga|}{\om_N} \right)^{\frac{1}{N-1}} \, \ln{|x - x_0|} ,
$$
up to an additive constant.
}
\end{rem}

\section{The interior problem: proof of Theorem \ref{thm:Serrininterno}}\label{sec:Interior}

\sectionmark{The interior problem: proof of Theorem \ref{thm:Serrininterno}}

In the present section $u$ is a weak solution to \eqref{Pdelta}, \eqref{eq:interno-overdetermination}, and $1<p<N$.
By a weak solution of \eqref{Pdelta}, \eqref{eq:interno-overdetermination} we mean a function $u \in W^{1,p}_{loc} \left( \Om \setminus \left\lbrace 0 \right\rbrace \right)$ such that
$$
\int_{\Om} |\na u|^{p-2} \, \na u \cdot \na \phi \, dx=0
$$
for every $\phi \in C_0^\infty \left( \Om \setminus \left\lbrace 0 \right\rbrace \right)$ and satisfying the boundary condition in \eqref{Pdelta} and the one in \eqref{eq:interno-overdetermination} in the weak sense explained after Theorem \ref{thm:Serrinesterno}.

In order to give our proof of Theorem \ref{thm:Serrininterno}, we collect all the necessary ingredients in the following remark.

\begin{rem}
{\rm
(i)(On the regularity). The regularity results presented for the exterior problem in
Remark \ref{rem:regularity} hold in the same way also for the interior problem, so that, reasoning as explained there, we can affirm that $u$ can be extended to a $C^2$-function in a neighborhood of $\Ga$, $\Ga$ is of class $C^2$, and $u \in C^{1, \al}_{loc} (\ol{\Om})$.

(ii)(Explicit value of $K$ in \eqref{Pdelta}). It is easy to show that the normalization constant $K$ that appears in \eqref{Pdelta} must take the value
$$
K= |\Ga |,
$$
to be compatible with the overdetermined condition \eqref{eq:interno-overdetermination} (see for example \cite{EP}).

(iii)(Asymptotic expansion). As a direct application of \cite[Theorem 1.1]{KV}, we deduce that the asymptotic behaviour of the solution $u$ of \eqref{Pdelta} near the origin is given by
\begin{equation}
\label{asymptoticu}
u(x) = \frac{p-1}{N-p} \left( \frac{|\Ga|}{ \om_N} \right)^{\frac{1}{p-1}} \, |x|^{- \frac{N-p}{p-1}} + o(|x|^{- \frac{N-p}{p-1}}).
\end{equation}
We mention that this expansion has been used also in \cite{EP}.

}
\end{rem}

\vspace{1pt}
\noindent
\textbf{The P-function.}
We consider again the P-function defined in \eqref{def:P}; by virtue of the weak $p$-harmonicity of $u$ in $\Om \setminus \lbrace 0 \rbrace$, \cite[Theorem 2.2]{GS} ensures that the strong maximum principle holds for $P$ in $\Om \setminus \lbrace 0 \rbrace$.


\vspace{1pt}


We are ready now to prove Theorem \ref{thm:Serrininterno}; as announced in the Introduction, the proof uses arguments similar to those used in the proof of Theorem \ref{thm:Serrinesterno}.


\begin{proof}[Proof of Theorem \ref{thm:Serrininterno}]
As already noticed in \cite{EP}, concerning the existence of a weak solution to the overdetermined problem \eqref{Pdelta}, \eqref{eq:interno-overdetermination}, the actual value of the function $u$ on $\Ga$ is irrelevant, since any function differing from $u$ by a constant is $p$-harmonic whenever $u$ is.
Thus, let us now fix the constant $c$ that appears in \eqref{Pdelta} as
\begin{equation}\label{def:cSBT}
c= \frac{p-1}{N-p} \left( \frac{N |\Om|}{|\Ga|} \right);
\end{equation}
with this choice and by recalling \eqref{eq:interno-overdetermination} we have that
$$
P_{| \Ga} = \left( \frac{p-1}{N-p} \right)^{- \frac{p(N-1)}{N-p}} \left( \frac{N| \Om|}{| \Ga|} \right)^{- \frac{p(N-1)}{N-p}}.
$$
Moreover, by using \eqref{asymptoticu} we find also that
$$
\lim_{|x| \rightarrow 0} P(x)= \left( \frac{p-1}{N-p} \right)^{- \frac{p(N-1)}{N-p}}\left( \frac{|\Ga|}{\om_N} \right)^{- \frac{p}{N-p}}.
$$
By the isoperimetric inequality \eqref{eq:ISOPerimetrica} it is easy to check that
$$
\lim_{|x| \rightarrow 0} P(x) \le P_{| \Ga} ,
$$
and hence, by the maximum principle proved in \cite[Theorem 2.2]{GS}, we realize that $P$ attains its maximum on $\Ga$. We thus have that
\begin{equation}\label{eq:MPtoP}
0\le P_{\nu}.
\end{equation}
By a direct computation, with exactly the same manipulations used for the exterior problem in the proof of Theorem \ref{thm:Serrinesterno}, we find that
\begin{equation}\label{eq:PnuSBT}
P_{\nu} = p (N-1)  u^{- \frac{p(N-1)}{N-p}} |\na u|^p  \left\lbrace \frac{- u_{\nu}}{(N-p) u} - \frac{H}{p-1} \right\rbrace.
\end{equation}

By coupling \eqref{eq:MPtoP} with \eqref{eq:PnuSBT}, we can deduce that
\begin{equation*}
H \le \frac{p-1}{N-p} \left(\frac{- u_\nu}{u}\right),
\end{equation*}
that by using \eqref{def:cSBT}, \eqref{Pdelta}, \eqref{eq:interno-overdetermination}, and the fact that $|\na u|= - u_\nu$ on $\Ga$, leads to
\begin{equation}\label{eq:penultima SBT}
H \le H_0,
\end{equation}
where $H_0$ is the constant defined in \eqref{def-R-H0}.

Since $\Om$ is star-shaped with respect to a point $z \in \Om$ (possibly distinct from $0$), we have that $<(x - z) ,\nu >$ is non-negative on $\Ga$.
Thus, multiplying \eqref{eq:penultima SBT} by $<(x - z) ,\nu>$,
%
%
and integrating over $\Ga$, we get
\begin{equation*}
\int_\Ga H <(x - z ) ,\nu> \, dS_x \le |\Ga|.
\end{equation*}
By recalling the regularity of $\Ga$ and {\it Minkowski's identity}
\begin{equation*}
\int_\Ga H <(x - z ) ,\nu> \, dS_x = |\Ga| ,
\end{equation*}
we deduce -- as already noticed in \cite[Proof of Theorem 1.1]{GS} -- that the equality sign must hold in \eqref{eq:penultima SBT}, that is
$$
H \equiv H_0.
$$
Thus, the conclusion follows by the classical Alexandrov's Soap Bubble Theorem (\cite{Al2}) or, if we want, again by Theorem \ref{th:SBT}.
\end{proof}

\begin{rem}
{\rm
As a corollary of Theorem \ref{thm:Serrininterno} we get that $u$ is spherically symmetric about the center $0$ of (the ball) $\Om$ and it is given by
$$
u(x) = \frac{p-1}{N-p} \left( \frac{|\Ga|}{ \om_N} \right)^{\frac{1}{p-1}} \, |x|^{- \frac{N-p}{p-1}} ,
$$
up to an additive constant.
}
\end{rem}

\chapter{Some estimates for harmonic functions}\label{chapter:various estimates}

As already sketched in the introduction, the desired stability estimates for the spherical symmetry of $\Om$ will be obtained by linking the oscillation on $\Ga$ of the harmonic function $h= q - u$, where $u$ is the solution of \eqref{serrin1} and $q$ is the quadratic polynomial defined in \eqref{quadratic},
to the integrals
$$
\int_\Om (-u)|\na^2 h|^2 dx \quad \mbox{and} \quad \int_\Om |\na^2 h|^2 dx.
$$
%
%
\par
In order to fulfill this agenda, in Sections \ref{sec:harmonic functions in weighted spaces} and \ref{sec:estimates for the oscillation of harmonic functions} we collect some useful estimates for harmonic functions that have their own interest. 

%
%
Then, in Section \ref{sec:estimates for h} we deduce some inequalities for the particular harmonic function $h= q- u$ that will be useful for the study of the stability issue addressed in Chapter \ref{chapter:Stability results}.
%
%

In Sections \ref{sec:estimates for the oscillation of harmonic functions} and \ref{sec:estimates for h}, some pointwise estimates for the torsional rigidity density will be useful. We collect them in the following section.

Before we start, we set some relevant notations.
For a point $z\in\Om$, $\rho_i$ and $\rho_e$ denote the radius of the largest ball centered at $z$ and contained in $\Om$ and that of the smallest ball that contains $\Om$ with the same center, that is
\begin{equation}
\label{def-rhos}
\rho_i=\min_{x\in\Ga}|x-z| \ \mbox{ and } \ \rho_e=\max_{x\in\Ga}|x-z|.
\end{equation}
%
%
%
%
%
%
The diameter of $\Om$ is indicated by $d_\Om$, while $\de_\Ga (x)$ denotes the distance of a point $x$ to the boundary $\Ga$.
We recall that if $\Ga$ is of class $C^2$, $\Om$ has the properties of the uniform interior and exterior sphere condition, whose respective radii we have designated by $r_i$ and $r_e$. In other words,
there exists $r_e > 0$ (resp. $r_i>0$) such that for each $p \in \Ga$ there exists a ball contained in $\RR^N \setminus \ol{\Om}$ (resp. contained in $\Om$) of radius $r_e$ (resp. $r_i$) such that
its closure intersects $\Ga$ only at $p$.

Finally, if $\Ga$ is of class $C^{1,\al}$, the unique solution of \eqref{serrin1} is of class at least
$C^{1,\al}(\ol{\Om})$.
%
%
Thus, we can define
\begin{equation}
\label{bound-gradient}
M = \max_{\ol{\Om}} |\na u| = \max_\Ga u_\nu .
\end{equation}

\section{Pointwise estimates for the torsional rigidity density}\label{sec:gradient estimate for torsion}
We start with the following lemma in which we relate $u(x)$ with the distance function $\de_\Ga (x)$.
\begin{lem}
\label{lem:relationdist}
Let $\Om\subset\RR^N$, $N\ge 2$, be a bounded domain such that $\Ga$ is made of regular points for the Dirichlet problem, and let $u$ be the solution of \eqref{serrin1}. Then
$$
-u(x)\ge\frac12\,\de_\Ga(x)^2 \ \mbox{ for every } \ x\in\ol{\Om}.
$$
Moreover, if $\Ga$ is of class $C^{2}$, then it holds that
\begin{equation}
\label{eq:relationdist}
-u(x) \ge \frac{r_i}{2}\,\de_\Ga(x)\ \mbox{ for every } \ x\in\ol{\Om}.
\end{equation}
\end{lem}

\begin{proof}
If every point of $\Ga$ is regular, then a unique solution $u\in C^0(\ol{\Om})\cap C^2(\Om)$ exists for \eqref{serrin1}. Now, for $x\in\Om$, let $r=\de_\Ga(x)$ and consider the ball $B=B_r(x)$. 
Let $w$ be the solution of \eqref{serrin1} in $B$, that is $w(y)=(|y-x|^2-r^2)/2$. By comparison we have that $w\ge u$ on $\ol{B}$ and hence, in particular, $w(x) \ge u(x)$. Thus, we infer the first inequality in the lemma.
\par
If $\Ga$ is of class $C^{2}$, \eqref{eq:relationdist} certainly holds if $\de_\Ga(x)\ge r_i$. If $\de_\Ga (x) < r_i$, instead, let $z$ be the closest point in $\Ga$ to $x$ and call $B$ the ball of radius $r_i$ touching $\Ga$ at $z$ and containing $x$. Up to a translation, we can always suppose that the center of the ball $B$ is the origin $0$. If $w$ is the solution of \eqref{serrin1} in $B$, that is $w(y)=\left(|y|^2- r_i^2 \right)/2$, by comparison we have that $w \ge u$ in $B$, and hence
$$
-u(x) \ge \frac12\,(|x|^2-r_i^2)=
\frac12\,( r_i + |x| )(r_i-|x|)\ge\frac12\,r_i\,(r_i-|x|).
$$
This implies \eqref{eq:relationdist}, since $r_i-|x|=\de_\Ga(x)$.
\end{proof}

In the remaining part of this section we present a simple method to estimate the number $M$ (defined in \eqref{bound-gradient})
%
%
in a quite general domain. The following lemma results from a simple inspection and by the uniqueness for the Dirichlet problem.

\begin{lem}[Torsional rigidity density in an annulus]\label{soluzioneincoronacircolare}
\label{lem:torsion-annulus}

Let $A=A_{r,R}\subset \RR^N$ be the annulus centered at the  origin and radii $0<r<R$,
and set $\ka=r/R$.
\par
Then, the solution $w$ of the Dirichlet problem 
\begin{equation*}
\label{torsion-annulus}
\Delta w = N \ \textrm{ in } \ A, \quad w = 0 \ \textrm{ on } \ \pa A,
\end{equation*}
is defined for $r\le |x|\le R$ by
\begin{equation*}
w(x) =
\begin{cases}
\displaystyle\frac12\, |x|^2 +\frac{R^2}{2}\,(1-\ka^2)\,\frac{\log(|x|/r)}{\log\ka} -\frac{r^2}{2} 
\ &\mbox{ for } \ N=2, 
\vspace{5pt} \\
\displaystyle\frac12\,|x|^2 +\frac12\,\frac{R^2}{1-\ka^{N-2}}\,\left\{ (1-\ka^2)\,(|x|/r)^{2-N}+\ka^N-1\right\} \ &\mbox{ for } \ N \ge 3.
\end{cases}
\end{equation*}
\end{lem}

\medskip

\begin{thm}[A bound for the gradient on $\Ga$, \cite{MP}]
\label{thm:boundary-gradient}
Let $\Om\subset\RR^N$ be a bounded domain that satisfies the uniform interior and exterior conditions with radii $r_i$ and $r_e$ and let $u\in C^1(\ol{\Om})\cap C^2(\Om)$ be a solution
of \eqref{serrin1} in $\Om$.
\par
Then, we have that
\begin{equation}
\label{gradient-estimate}
r_i\le |\na u| \le 
c_N\,\frac{d_\Om(d_\Om+r_e)}{r_e} \ \mbox{ on } \ \Ga,
\end{equation}
where $d_\Om$ is the diameter of $\Om$ and 
$c_N=3/2$ for $N=2$ and $c_N=N/2$ for $N\ge 3$.
\end{thm}

\begin{proof}
We first prove the first inequality in \eqref{gradient-estimate}. Fix any  $p\in\Ga$.
Let $B=B_{r_i}$ be the interior ball touching $\Ga$ at $p$ and place the origin of cartesian axes at the center of $B$.
\par
If $w$ is the solution of \eqref{serrin1} in $B$, that is $w(x)= (|x|^2-r_i^2)/2$, by comparison we have that $w\ge u$ on $\ol{\Om}$ and hence, since $u(p)=w(p)=0$, we obtain:
$$
u_\nu(p)\ge w_\nu(p)=r_i.
$$
\par
To prove the second inequality, we place the origin of axes at the center of  the exterior ball $B=B_{r_e}$ touching $\Ga$ at $p$. Denote by $A$ the smallest annulus containing $\Om$, concentric with $B$ and having $\pa B$ as internal boundary and let $R$ be the radius of its external boundary.
\par
If $w$ is the solution of \eqref{serrin1} in $A$, by comparison we have that $w\le u$ on $\ol{\Om}$.
Moreover, since $u(p)=w(p)=0$, we have that
$$
u_\nu(p)\le w_\nu(p).
$$
By Lemma \ref{lem:torsion-annulus} we then compute that
$$
w_\nu(p) =\frac{R (R-r_e)}{r_e}\,f(\ka)
$$
where, for $0<\ka<1$,
%
%
\begin{equation*}
f(\ka)=
\begin{cases}
\displaystyle \frac{2\ka^2\log(1/\ka)+\ka^2-1}{2(1-\ka) \log(1/\ka)}
\ &\mbox{ for } \ N=2, 
\vspace{5pt} \\
\displaystyle\frac{2 \ka^N-N \ka^2+N-2}{2(1-\ka)(1-\ka^{N-2})} \ &\mbox{ for } \ N \ge 3.
\end{cases}
\end{equation*}
Notice that $f$ is bounded since it can be extended to a continuous function on $[0,1]$.
Tedious calculations yield that 
$$
\sup_{0<\ka<1} f(\ka)=
\begin{cases}
\frac32 \ &\mbox{ for } \ N=2, \\
\frac{N}2 \ &\mbox{ for } \ N\ge 3.
\end{cases}
$$
Finally, observe that $R\le d_\Om+r_e$.
\end{proof}

\begin{rem}\label{rem:boundgradienttorsionconvexr_eremoving}
{\rm
%
(i) Notice that, when $\Om$ is convex, we can choose $r_e=+\infty$ in \eqref{gradient-estimate} and obtain
\begin{equation}
\label{bound-M-convex}
M\le c_N\,d_\Om.
\end{equation}

%
%

(ii) To the best of our knowledge, inequality \eqref{gradient-estimate}, established in \cite{MP}, was not 
present in the literature for general smooth domains.
%
%
Other estimates are given in \cite{PP} for planar strictly convex domains (but the same argument can be generalized to general dimension for strictly mean convex domains) and in \cite{CM} for strictly mean convex domains in general dimension.
In particular, in \cite[Lemma 2.2]{CM} the authors prove that there exists a universal constant $c_0$ such that
\begin{equation*}
| \na u | \leq c_0 | \Om |^{1/N} \mbox{ in } \ol{\Om}.
\end{equation*}
%
%
}
\end{rem}

\section{Harmonic functions in weighted spaces}\label{sec:harmonic functions in weighted spaces}
To start, we briefly report on some Hardy-Poincar\'e-type inequalities for harmonic functions that are present in the literature. 
%
%
The following lemma can be deduced from the work of Boas and Straube \cite{BS}, which improves a result of Ziemer \cite{Zi}.
%
%
In what follows, for a set $A$ and a function $v: A \to \RR$, $v_A$ denotes the {\it mean value of $v$ in $A$} that is
$$
v_A= \frac{1}{|A|} \, \int_A v \, dx.
$$
Also, for a function $v:\Om \to \RR$ we define
$$
\nr \de_\Ga^\al \, \na v \nr_{p,\Om} = \left( \sum_{i=1}^N \nr \de_\Ga^\al \,  v_i \nr_{p,\Om}^p \right)^\frac{1}{p} \quad \mbox{and} \quad
\nr \de_\Ga^\al \, \na^2 v \nr_{p,\Om} = \left( \sum_{i,j=1}^N \nr \de_\Ga^\al \, v_{ij} \nr_{p,\Om}^p \right)^\frac{1}{p},
$$
for $0 \le \al \le 1$ and $p \in [1, \infty)$.

\begin{lem}[\cite{BS}]
\label{lem:two-inequalities}
Let $\Om\subset\RR^N$, $N\ge 2$, be a bounded domain with boundary $\Ga$ of class $C^{0,\al}$,
$0 \le \al \le 1$, let $z$ be a point in $\Om$, and consider $p \in \left[ 1, \infty \right)$. Then,
\par
(i) there exists a positive constant $ \mu_{p, \al} ( \Om, z ) $, such that
\begin{equation}
\label{harmonic-quasi-poincare}
\nr v \nr_{p,\Om} \le \mu_{p, \al} ( \Om, z)^{-1} \nr \de_\Ga^{\al} \, \na v  \nr_{p, \Om},
\end{equation}
for every harmonic function $v$ in $\Om$ such that $v(z)=0$;
\par
(ii)
there exists a positive constant, $\ol{\mu}_{p, \al} (\Om)$ such that
\begin{equation}
\label{harmonic-poincare}
\nr v - v_\Om \nr_{p,\Om} \le \ol{\mu}_{p, \al} (\Om)^{-1} \nr \de_\Ga^{\al} \, \na v  \nr_{p, \Om},
\end{equation}
for every function $v$ which is harmonic in $\Om$.
\par
In particular, if $\Ga$ has a Lipschitz boundary, the number $\al$ can be replaced by any exponent in $[0,1]$. 
\end{lem}

\begin{proof}
The assertions (i) and (ii) are easy consequences of a general result of Boas and Straube (see \cite{BS}).
%
%
In case (i), we apply \cite[Example 2.5]{BS}). In case (ii), \cite[Example 2.1]{BS} is appropriate. In fact, \cite[Example 2.1]{BS} proves \eqref{harmonic-poincare} for every function $v \in L^p (\Om) \cap W^{1,p}_{loc} (\Om)$. The addition of the harmonicity of $v$ in \eqref{harmonic-poincare} clearly gives a better constant. 
\par
The (solvable) variational problems 
%
%
\begin{equation*}
\mu_{p, \al} (\Om, z) = 
\min \left\{ \nr \de_\Ga^{\al} \, \na v  \nr_{p, \Om} : \nr v \nr_{p, \Om} = 1, \,\De v =0 \text{ in } \Om, \, v(z) = 0 \right\}
\end{equation*}
and
%
%
\begin{equation*}
\ol{\mu}_{p, \al} (\Om) = 
\min \left\{ \nr \de_\Ga^{\al} \, \na v  \nr_{p, \Om} : \nr v \nr_{p, \Om} = 1, \,\De v =0 \text{ in } \Om, \, v_\Om = 0 \right\}
\end{equation*}
then characterize the two constants.
\end{proof}

\begin{rem}
{\rm
When $\al=0$ we understand the boundary of $\Om$ to be locally the graph of a continuous function. In this case, \eqref{harmonic-quasi-poincare} is exactly the result contained in \cite{Zi}.
Also, in the case $p=2$ and $\al=0$ from \eqref{harmonic-quasi-poincare} and \eqref{harmonic-poincare} we recover the Poincar\'e-type inequalities proved and used in \cite{MP}.
%
}
\end{rem}

\medskip

The following corollary describes a couple of applications of the Hardy-Poincar\'e inequalities that we have just presented.

\begin{cor}\label{cor:Poincareaigradienti}
Let $\Om\subset\RR^N$, $N\ge 2$, be a bounded domain with boundary $\Ga$ of class $C^{0,\al}$, $0 \le \al \le 1$, and let $v$ be a harmonic function in $\Om$.

(i) If $z$ is a critical point of $v$ in $\Om$, then it holds that
\begin{equation*}
\nr \na v \nr_{p, \Om} \le \mu_{p, \al} ( \Om, z)^{-1} \nr \de_\Ga^{\al} \, \na^2 v  \nr_{p, \Om}.
\end{equation*}

(ii) If 
$$\int_\Om \na v \, dx = 0,$$ 
%
%
then it holds that
\begin{equation*}
\nr \na v \nr_{p, \Om} \le \ol{\mu}_{p, \al} (\Om)^{-1} \nr \de_\Ga^{\al} \, \na^2 v  \nr_{p, \Om}.
\end{equation*}
\end{cor}
\begin{proof}
Since $\na v (z)=0$ (respectively $\int_\Om \na v \, dx =0$), we can apply \eqref{harmonic-quasi-poincare} (respectively \eqref{harmonic-poincare}) to each first partial derivative $v_i$ of $v$, $i=1, \dots, N$. If we raise to the power of $p$ those inequalities and sum over $i=1, \dots, N$, the conclusion easily follows.
\end{proof}

We now present a simple lemma which will be useful in the sequel to manipulate the left-hand side of Hardy-Poincar\'e-type inequalities.

\begin{lem}\label{lem:mediaor-mediauguale}
Let $\Om$ be a domain with finite measure and let $A \subseteq \Om$ be a set of positive measure. If $v \in L^p(\Om)$, then for every $\la \in \RR$
\begin{equation}\label{eq:mediaor-mediauguale}
\nr v - v_A \nr_{p,\Om} \le \left[ 1+ \left( \frac{|\Om|}{|A|} \right)^{\frac{1}{p}} \right] \, \nr v - \la \nr_{p,\Om}.
\end{equation}
%
%
\end{lem}
\begin{proof}
By H\"older's inequality, we have that 
\begin{equation*}
|v_A - \la| \le \frac1{|A|} \int_A|v-\la| \, dx \le 
|A|^{-1/p} \nr v- \la \nr_{p,A} \le
|A|^{-1/p} \nr v - \la \nr_{p,\Om}.
\end{equation*}
Since $|v_A - \la|$ is constant, we then infer that
$$
\nr v_A-\la\nr_{p,\Om}=|\Om|^{1/p} |v_A - \la| \le \left( \frac{|\Om|}{|A|} \right)^{\frac{1}{p}} \, \nr v- \la \nr_{p,\Om}.
$$
Thus, \eqref{eq:mediaor-mediauguale} follows by an application of the triangular inequality.
\end{proof}

Other versions of Hardy-Poincar\'e inequalities different from those presented in Lemma \ref{lem:two-inequalities} can be deduced from the work of Hurri-Syrj\"anen \cite{HS}, which was stimulated by \cite{BS}.

In order to state these results, we introduce the notions of $b_0$-John domain and $L_0$-John domain with base point $z \in \Om$. Roughly speaking, a domain is a $b_0$-John domain (resp. a $L_0$-John domain with base point $z$) if it is possible to travel from one point of the domain to another (resp. from $z$ to another point of the domain) without going too close to the boundary.

A domain $\Om$ in $\RR^N$ is a {\it $b_0$-John domain}, $b_0 \ge 1$, if each pair of distinct points $a$ and $b$ in $\Om$ can be joined by a curve $\ga: \left[0,1 \right] \rightarrow \Om$ such that
%
%
\begin{equation*}
\de_\Ga (\ga(t)) \ge b_0^{-1} \min{ \left\lbrace |\ga(t) - a|, |\ga(t) - b| \right\rbrace  }.
\end{equation*}

A domain $\Om$ in $\RR^N$ is a {\it $L_0$-John domain with base point $z \in \Om$}, $L_0 \ge 1$, if each point $x \in \Om$ can be joined to $z$ by a curve $\ga: \left[0,1 \right] \rightarrow \Om$ such that
\begin{equation*}
\de_\Ga (\ga(t)) \ge L_0^{-1} |\ga(t) - x|.
\end{equation*}

It is known that, for bounded domains, the two definitions are quantitatively equivalent (see \cite[Theorem 3.6]{Va}).
%
%
The two notions could be also defined respectively through the so-called {\it $b_0$-cigar} and {\it $L_0$-carrot} properties (see \cite{Va}).

\begin{lem}\label{lem:John-two-inequalities}
Let $\Om \subset \RR^N$ be a bounded $b_0$-John domain, and consider three numbers $r, p, \al$ such that $1 \le p \le r \le \frac{Np}{N-p(1 - \al )}$, $p(1 - \al)<N$, $0 \le \al \le 1$.
Then,

(i) there exists a positive constant $ \mu_{r,p, \al} ( \Om, z ) $, such that
\begin{equation}
\label{John-harmonic-quasi-poincare}
\nr v \nr_{r,\Om} \le \mu_{r, p, \al} ( \Om, z)^{-1} \nr \de_\Ga^{\al} \, \na v  \nr_{p, \Om},
\end{equation}
for every function $v$ which is harmonic in $\Om$ and such that $v(z)=0$;
\par
(ii) there exists a positive constant, $\ol{\mu}_{r, p, \al} (\Om)$ such that
\begin{equation}
\label{John-harmonic-poincare}
\nr v - v_\Om \nr_{r,\Om} \le \ol{\mu}_{r, p, \al} (\Om)^{-1} \nr \de_\Ga^{\al} \, \na v  \nr_{p, \Om},
\end{equation}
for every function $v$ which is harmonic in $\Om$.
\end{lem}
\begin{proof}
In \cite[Theorem 1.3]{HS} it is proved that there exists a constant $c=c(N,\, r, \, p,\, \al, \, \Om)$ such that
\begin{equation}\label{eq:risultatodiHSconr-media}
\nr v - v_{r, \Om} \nr_{r,\Om} \le c \, \nr \de_\Ga^{\al} \, \na v  \nr_{p, \Om} ,
\end{equation}
for every $v \in L^1_{loc}(\Om)$ such that $\de_\Ga^\al \, \na v  \in L^p(\Om)$.
Here, $v_{r,\Om}$ denotes the {\it $r$-mean of $v$ in $\Om$} which is defined -- following \cite{IMW} -- as the unique minimizer of the problem
$$\inf_{\la \in \RR} \nr v - \la \nr_{r,\Om}.$$
Notice that, in the case $r=2$, $v_{2, \Om}$ is the classical mean value of $v$ in $\Om$, i.e. $v_{2,\Om} = v_\Om$,
as can be easily verified.

By using Lemma \ref{lem:mediaor-mediauguale} with $A= \Om$ and $\la = v_{r, \Om}$ we thus prove \eqref{John-harmonic-poincare} for every $v \in L^1_{loc}(\Om)$ such that $\de_\Ga^\al \, \na v  \in L^p(\Om)$.
The assumption of the harmonicity of $v$ in \eqref{John-harmonic-poincare} clearly gives a better constant.

Inequality \eqref{John-harmonic-quasi-poincare} can be deduced from \eqref{John-harmonic-poincare} by applying Lemma \ref{lem:mediaor-mediauguale} with $A= B_{\de_\Ga (z)} (z)$ and $\la = v_\Om$ and recalling that, since $v$ is harmonic, by the mean value property it holds that
$v(z)= v_{B_{\de_\Ga (z)} (z)}.$
%

The (solvable) variational problems 
%
%
\begin{equation*}
\mu_{r, p, \al} (\Om, z) = 
\min \left\{ \nr \de_\Ga^{\al} \, \na v  \nr_{p, \Om} : \nr v \nr_{r, \Om} = 1, \,\De v =0 \text{ in } \Om, \, v(z) = 0 \right\}
\end{equation*}
and
%
%
\begin{equation*}
\ol{\mu}_{r, p, \al} (\Om) = 
\min \left\{ \nr \de_\Ga^{\al} \, \na v  \nr_{p, \Om} : \nr v \nr_{r, \Om} = 1, \,\De v =0 \text{ in } \Om,  v_{\Om} = 0 \right\}
\end{equation*}
then characterize the two constants. 
\end{proof}

\begin{rem}
{\rm
When $r=p$, \eqref{John-harmonic-quasi-poincare} and \eqref{John-harmonic-poincare} become exactly \eqref{harmonic-quasi-poincare} and \eqref{harmonic-poincare}.
However, it should be noticed that \eqref{harmonic-quasi-poincare} and \eqref{harmonic-poincare} proved in \cite{BS} also hold true without the restriction $p (1-\al) < N$ assumed in \cite{HS}.
%
%
}
\end{rem}

%
%

From Lemma \ref{lem:John-two-inequalities} we can derive estimates for the derivatives of harmonic functions, as already done in Corollary \ref{cor:Poincareaigradienti} for Lemma \ref{lem:two-inequalities}.

\begin{cor}\label{cor:JohnPoincareaigradienti}
Let $\Om\subset\RR^N$, $N\ge 2$, be a bounded $b_0$-John domain and let $v$ be a harmonic function in $\Om$. Consider three numbers $r, p, \al$ such that $1 \le p \le r \le \frac{Np}{N-p(1 - \al )}$, $p(1 - \al)<N$, $0 \le \al \le 1$.

(i) If $z$ is a critical point of $v$ in $\Om$, then it holds that
\begin{equation*}
\nr \na v \nr_{r, \Om} \le \mu_{r, p, \al} ( \Om, z)^{-1} \nr \de_\Ga^{\al} \, \na^2 v  \nr_{p, \Om}.
\end{equation*}

(ii) If
$$\int_\Om \na v \, dx = 0,$$
%
%
then it holds that
\begin{equation*}
\nr \na v \nr_{r, \Om} \le \ol{\mu}_{r, p, \al} (\Om)^{-1} \nr \de_\Ga^{\al} \, \na^2 v  \nr_{p, \Om}.
\end{equation*}
\end{cor}
\begin{proof}
Since $\na v (z)=0$ (respectively $\int_\Om \na v \, dx = 0,$), we can apply \eqref{harmonic-quasi-poincare} (respectively \eqref{harmonic-poincare}) to each first partial derivative $v_i$ of $v$, $i=1, \dots, N$. If we raise to the power of $r$ those inequalities and sum over $i=1, \dots, N$, the conclusion easily follows in view of the inequality
$$
\sum_{i=1}^N x_i^{\frac{r}{p}} \le \left( \sum_{i=1}^N x_i \right)^{\frac{r}{p}} 
$$
that holds for every $(x_1, \dots, x_N) \in \RR^N$ with $x_i \ge 0$ for $i=1, \dots, N$, since $r/p \ge 1$.
\end{proof}

\begin{rem}[Tracing the geometric dependence of the constants]\label{rem:stime mu HS in item ii}
{\rm
In this remark, we explain how to trace the dependence on a few geometrical parameters of the constants in the relevant inequalities.

(i) The proof of \cite{HS} has the benefit of giving an explicit upper bound for the constant $c$ appearing in \eqref{eq:risultatodiHSconr-media}, from which, by following the steps of our proof, we can deduce explicit estimates for $\mu_{r,p,\al}(\Om,z)^{-1}$ and $\ol{\mu}_{r, p, \al} (\Om)^{-1}$.
In fact, we easily show that
\begin{equation*}
\ol{\mu}_{r, p, \al} (\Om)^{-1} \le k_{N,\, r, \, p,\, \al} \, b_0^N |\Om|^{\frac{1-\al}{N} +\frac{1}{r} +\frac{1}{p} } ,
\end{equation*}
\begin{equation*}
\mu_{r,p,\al}(\Om,z)^{-1} \le k_{N,\, r, \, p, \,\al} \, \left(\frac{b_0}{\de_\Ga (z)^{\frac{1}{r} }}\right)^N |\Om|^{\frac{1-\al}{N} +\frac{2}{r} +\frac{1}{p} } .
\end{equation*}

A better estimate for $\mu_{r,p,\al}(\Om,z)$ can be obtained for $L_0$-John domains with base point $z$. Since the computations are tedious and technical we present them in Appendix \ref{appendix:estimates-mu-for-L0John} and here we just report the final estimate, that is,
\begin{equation}\label{eq:estimatemu-r-p-al-generalizingFeappendix}
\mu_{r,p,\al}(\Om,z)^{-1} \le k_{N,r,p,\al} \, L_0^N |\Om|^{\frac{1-\al}{N} +\frac{1}{r} +\frac{1}{p} } .
\end{equation}

(ii) In the sequel we will also need to trace the dependence on relevant geometrical quantities of the constants $\ol{\mu}_{p, 0} (\Om)$ and $\mu_{p,0}(\Om,z)$ appearing in \eqref{harmonic-poincare} and \eqref{harmonic-quasi-poincare} in the case $\al=0$.

In \cite[Theorem 8.5]{HS1},
where the author proves \eqref{harmonic-poincare} in the case $\al=0$ for $b_0$-John domains, 
an explicit upper bound for $\ol{\mu}_{p,0} (\Om)^{-1}$, in terms of $b_0$ and $d_{\Om}$ only, can be found. We warn the reader that the definition of John domain used there is different from the definitions that we gave in this thesis, but it is equivalent in view of \cite[Theorem 8.5]{MarS}. Explicitly, by putting together \cite[Theorem 8.5]{HS1} and \cite[Theorem 8.5]{MarS} one finds that
\begin{equation*}
\ol{\mu}_{p,0} (\Om)^{-1} \le k_{N, \, p} \, b_0^{3N(1 + \frac{N}{p})} \, d_\Om.
\end{equation*}

Reasoning as in the proof of \eqref{John-harmonic-quasi-poincare}, from this estimate one can also deduce a bound for $\mu_{p,0}(\Om,z)$. In fact, by applying Lemma \ref{lem:mediaor-mediauguale} with $A= B_{\de_\Ga (z)} (z)$ and $\la = v_\Om$ and recalling the mean value property of $v$, from \eqref{harmonic-poincare} and the bound for $\ol{\mu}_{p,0} (\Om)$, we easily compute that
\begin{equation*}
\mu_{p,0} (\Om, z)^{-1} \le k_{N, \, p} \, \left( \frac{|\Om|}{\de_\Ga (z)^N} \right)^\frac{1}{r} \,b_0^{3N(1 + \frac{N}{p})} \, d_\Om.
\end{equation*}

A better estimate for $\mu_{p,0}(\Om,z)$ can be obtained for $L_0$-John domains with base point $z$, that is,
\begin{equation}\label{eq:estimatemu-p-0-generalizingFeappendix}
\mu_{p, 0}(\Om,z)^{-1} \le k_{N, \, p} \, L_0^{3N(1 + \frac{N}{p})} \, d_\Om .
\end{equation}
Complete computations to obtain \eqref{eq:estimatemu-p-0-generalizingFeappendix} can be found in Appendix \ref{appendix:estimates-mu-for-L0John}.

(iii) A domain of class $C^{2}$ is obviously a $b_0$-John domain and a $L_0$-John domain with base point $z$ for every $z \in \Om$. In fact, by the definitions, it is not difficult to prove the following bounds
%
%
\begin{equation*}
b_0 \le  \frac{d_\Om}{r_i} ,
\end{equation*}
%
%
\begin{equation*}
L_0 \le \frac{d_\Om}{\min[r_i, \de_\Ga (z)] } .
\end{equation*}

Thus, for $C^2$-domains items (i) and (ii) inform us that the following estimates hold
\begin{equation*}
\ol{\mu}_{r, p, \al} (\Om)^{-1} \le k_{N,\, r, \, p,\, \al} \, \left( \frac{d_\Om}{r_i} \right)^N |\Om|^{\frac{1-\al}{N} +\frac{1}{r} +\frac{1}{p} } ,
\end{equation*}
\begin{equation*}
\mu_{r,p,\al}(\Om,z)^{-1} \le k_{N,r,p,\al} \, \left( \frac{d_\Om}{\min[r_i, \de_\Ga (z)] } \right)^N |\Om|^{\frac{1-\al}{N} +\frac{1}{r} +\frac{1}{p} } ,
\end{equation*}
\begin{equation*}
\ol{\mu}_{p,0} (\Om)^{-1} \le k_{N, \, p} \,  \frac{d_\Om^{3N(1 + \frac{N}{p}) + 1 }  }{r_i^{3N(1 + \frac{N}{p})}  }  ,
\end{equation*}
\begin{equation*}
\mu_{p, 0}(\Om,z)^{-1} \le k_{N, \, p} \, \frac{d_\Om^{3N(1 + \frac{N}{p}) + 1 }  }{\min[r_i, \de_\Ga (z)]^{3N(1 + \frac{N}{p})} } .
\end{equation*}
}
\end{rem}

\medskip

To conclude this section, as a consequence of Lemma \ref{lem:identityfortraceineq}, we present a weighted trace inequality.
%
%
We mention that the following proof modifies an idea of W. Feldman \cite{Fe} for our purposes.
%
 
\begin{lem}[A trace inequality for harmonic functions]
\label{lem:genericv-trace inequality}
Let $\Om\subset\RR^N$, $N\ge 2$, be a bounded domain with boundary $\Ga$ of class $C^2$ and let $v$ be a harmonic function in $\Om$.

(i) If $z$ is a critical point of $v$ in $\Om$, then it holds that
\begin{equation*}
\int_{\Ga} |\na v|^2 dS_x \le \frac{2}{r_i} \left(1+\frac{N}{r_i\, \mu_{2,2, \frac{1}{2} }(\Om,z)^2 } \right)  \int_{\Om} (-u) |\na^2 v|^2 dx.
\end{equation*}

(ii) If 
$$\int_\Om \na v \, dx = 0,$$ 
then it holds that
\begin{equation*}
\int_{\Ga} |\na v|^2 dS_x \le \frac{2}{r_i} \left(1+\frac{N}{r_i\, \ol{\mu}_{2,2, \frac{1}{2} }(\Om)^2 } \right)  \int_{\Om} (-u) |\na^2 v|^2 dx.
\end{equation*}
\end{lem}
\begin{proof}
Since the term $u_\nu$ at the left-hand side of \eqref{eq:identityfortraceinequality} can be bounded from below by $r_i$, by an adaptation of Hopf's lemma (see Theorem \ref{thm:boundary-gradient}), it holds that
\begin{equation*}
r_i \, \int_{\Ga} |\na v|^2 dS_x \le N \int_{\Om} |\na v|^2 dx + 2 \int_{\Om} (-u) |\na^2 v|^2 dx.
\end{equation*}
Thus, the conclusion follows from this last formula, Corollary \ref{cor:JohnPoincareaigradienti} with $r=p=2$ and $\al=1/2$, and \eqref{eq:relationdist}.
\end{proof}

\section{An estimate for the oscillation of harmonic functions}\label{sec:estimates for the oscillation of harmonic functions}

In this section, we single out the key lemma that will produce most of the stability estimates of Chapter \ref{chapter:Stability results} below. It contains an inequality  for the oscillation of a harmonic function $v$ in terms of its $L^p$-norm and of a bound for its gradient.
We point out that the following lemma is new and generalizes the estimates proved and used in \cite{MP, MP2} for $p=2$.

To this aim, we define the {\it parallel set} as
$$
\Om_\si=\{ y\in\Om: \de_\Ga (y) >\si\} \quad \mbox{ for } \quad 0<\si \le r_i.
$$

\begin{lem}
\label{lem:Lp-estimate-oscillation-generic-v}
Let $\Om\subset\RR^N$, $N\ge 2$, be a bounded domain with boundary $\Ga$ of class $C^2$ and let $v$ be a harmonic function in $\Om$ of class $C^1 (\ol{\Om})$. Let $G$ be an upper bound for the gradient of $v$ on $\Ga$.
\par
Then, there exist two constants $a_{N,p}$ and $\al_{N,p}$ depending only on $N$ and $p$ such that if
\begin{equation}
\label{smallness-generic-v}
\nr v - v_{\Om} \nr_{p, \Om} \le \al_{N,p} \, r_{i}^{\frac{N+p}{p}}  G  
\end{equation}
holds, we have that
\begin{equation}
\label{Lp-stability-generic-v}  
\max_{\Ga} v - \min_{\Ga} v \le a_{N,p} \,  G^{ \frac{N}{N+p} } \, \nr v - v_{\Om} \nr_{p, \Om}^{ p/(N+p) }.
\end{equation} 
\end{lem}

\begin{proof}
Since $v$ is harmonic it attains its extrema on the boundary $\Ga$.
Let $x_i$ and $x_e$ be points in $\Ga$ that respectively minimize and maximize $v$ on $\Ga$ and, for 
$$
0<\si \le r_i,
$$
define the two points in $y_i, y_e\in\pa\Om_\si$ by
$y_j=x_j-\si\nu(x_j)$, $j=i, e$. 
\par
By the fundamental theorem of calculus we have that
\begin{equation}\label{eq:prova-TFCI-generic v}
v(x_j)= v(y_j) + \int_0^\si \lan\na v(x_j-t\nu(x_j)),\nu(x_j)\ran\,dt.
\end{equation}
%
%
\par
Since $v$ is harmonic and $y_j\in \overline{ \Om }_\si $, $j=i, e$, we can use the mean value property for the balls with radius $\si$ centered at  $y_j$ and obtain: 
\begin{multline*}
|v(y_j) - v_{\Om}| \le \frac1{|B|\, \si^N}\,\int_{B_\si(y_j)}|v - v_{\Om} |\,dy\le \\
\frac{1}{ \left[ |B|\, \si^N \right]^{1/p} } \, \left[\int_{B_\si(y_j)}|v - v_{\Om} |^p\,dy\right]^{1/p}\le 
\frac1{ \left[ |B|\, \si^N \right]^{1/p} } \, \left[\int_{\Om}|v- v_{\Om}|^p\,dy\right]^{1/p} 
\end{multline*}
after an application of H\"older's inequality and by the fact that $B_\si(y_j) \subseteq \Om$. This and \eqref{eq:prova-TFCI-generic v} then yield that
\begin{equation*}
\max_{\Ga} v - \min_{\Ga} v \le 2 \, \left[  \frac{\nr v - v_{\Om} \nr_{p, \Om} }{ |B|^{1/p} \, \si^{N/p}}+ \si \, G \right] ,
\end{equation*}
for every $0<\si \le r_i$.
Here we used that 
%
%
%
$|\na v|$ attains its maximum on $\Ga$, being $v$ harmonic.
\par
Therefore, by minimizing the right-hand side of the last inequality, we can conveniently choose 
$$
\si=\left(\frac{N\,\nr v - v_{\Om} \nr_{p, \Om} }{ p \, |B|^{1/p}\, G }\right)^{p/(N+p)} 
$$
and obtain \eqref{Lp-stability-generic-v}, if $\si \le r_i$;  \eqref{smallness-generic-v} will then follow. The explicit computation immediately shows that
\begin{equation}
\label{eq:costantia_Nal_Nlemmagenericv}
a_{N,p}= \frac{ 2 (N+p) }{N^{\frac{N}{N+2}} p^{\frac{p}{N+p}} } \,|B|^{\frac{1}{N+p}}
\quad \mbox{and} \quad \al_{N,p}= \frac{ p }{N} \, |B|^{\frac{1}{p} } .
\end{equation}

Notice that, the fact that \eqref{Lp-stability-generic-v} holds if \eqref{smallness-generic-v} is verified, remains true even if we replace in \eqref{smallness-generic-v} and \eqref{Lp-stability-generic-v} $v_\Om$ by any $\la \in \RR$.
\end{proof}

%
%
%
%

In the following corollary, we present a geometric sufficient condition on the domain that makes \eqref{smallness-generic-v} verified.
To this aim, for $x,$ $y \in \Om$ we denote by $d_\Om (x, y)$ the {\it intrinsic distance of $x$ to $y$ in $\Om$} induced by the euclidean metric,
%
%
that is
\begin{multline*}
d_\Om (x, y) = 
\\
\inf \left\lbrace \int_0^1 |\ga' (t)| \, dt : \, \ga: \left[ 0, 1 \right] \to \Om \text{ piecewise $C^1$, $\ga(0)=x$, $\ga(1)=y$} \right\rbrace .
\end{multline*}

\begin{cor}
Under the assumptions of Lemma \ref{lem:Lp-estimate-oscillation-generic-v}, \eqref{smallness-generic-v} holds true -- and hence also \eqref{Lp-stability-generic-v} does -- if the following inequality is verified 
%
%
\begin{equation}\label{eq:nuovacondizionegeometricainlemmaosc}
\int_\Om d_\Om (x_M,x)^p \, dx \le \al_{N,p}^p \, r_{i}^{N+p},
\end{equation}
for a point $x_M$ such that $v(x_M)= v_{\Om}$.
\end{cor}
\begin{proof}
The fact that \eqref{smallness-generic-v} holds true if \eqref{eq:nuovacondizionegeometricainlemmaosc} is verified follows from the inequality
\begin{equation}\label{eq:disuguaglianzanuovacondistanzaintrinseca}
\nr v - v_{\Om} \nr_{p, \Om} \le G \, \nr d_\Om (x_M,x) \nr_{p,\Om},
\end{equation}  
that can be proved by using the fundamental theorem of calculus.
In fact, if $\ga: \left[ 0, 1 \right] \to \Om$  is any piecewise $C^1$ curve from $x_M$ to $x$, the fundamental theorem of calculus informs us that
$$
v(x)- v(x_M)= \int_0^1 < \na v( \ga (t)), \ga'(t)> \, dt .
$$
Thus, we deduce that
$$
|v(x)- v(x_M)| \le G \, \int_0^1 | \ga'(t) | \, dt ,
$$
and \eqref{eq:disuguaglianzanuovacondistanzaintrinseca} easily follows from the definition
of $d_\Om(x,y)$, since $v(x_M) = v_\Om$.
\end{proof}

\section{Estimates for \texorpdfstring{$h=q-u$}{h=q-u}}\label{sec:estimates for h}
We now turn back our attention to the harmonic function $h = q-u$.

The following lemma is immediate.
\begin{lem}
Let $u$ be the solution of \eqref{serrin1} and set $h=q-u$, where $q$ is defined in \eqref{quadratic}.
Then, $h$ is harmonic in $\Om$ and it holds that:
\begin{equation}\label{L2-norm-hessian}
|\na^2 h|^2 = | \na ^2 u|^2 - \frac{(\De u)^2}{N}.
\end{equation}

Moreover, if the center $z$ of the polynomial $q$
%
%
is chosen in $\Om$, then the oscillation of $h$ on $\Ga$ can be bounded from below as follows:
\begin{equation}
\label{oscillationvolume}
\max_{\Ga} h-\min_{\Ga} h \ge \frac12\,(|\Om|/|B|)^{1/N}(\rho_e-\rho_i) ,
\end{equation}
or also:
\begin{equation}
\label{oscillation}
\max_{\Ga} h-\min_{\Ga} h \ge \frac{r_i}{2}(\rho_e-\rho_i) ,
\end{equation}
if $\Ga$ is of class $C^2$.
\end{lem}
\begin{proof}
Since by a direct computation it is immediate to check that $\De q = N$, the harmonicity of $h$ follows.
Simple and direct computations also give \eqref{L2-norm-hessian}.
Notice that $h=q$ on $\Ga$. Thus, by choosing $z$ in $\Om$, from \eqref{def-rhos} we get 
$$\max_{\Ga} h-\min_{\Ga} h=\frac12\,(\rho_e^2-\rho_i^2) .$$
Hence, \eqref{oscillationvolume} follows from the inequality $\rho_e+\rho_i\ge\rho_e\ge (|\Om|/|B|)^{1/N}$, that holds true since $B_{\rho_e}\supseteq\Om$.

If $\Ga$ is of class $C^2$, \eqref{oscillation} follows by noting that
%
%
$\rho_e+\rho_i \ge \rho_e \ge r_i$, or, if we want, also from \eqref{oscillationvolume} and the trivial inequality $|\Om| \ge |B| \, r_i^N$.
\end{proof}

By exploiting the additional information that we have about $h$, we now modify Lemma \ref{lem:Lp-estimate-oscillation-generic-v} to
directly link $\rho_e - \rho_i$ to the $L^p$-norm of $h$.
%
%
We do it in the following lemma which generalizes to the case of any $L^p$-norm \cite[Lemma 3.3]{MP}, that holds for $p=2$.

%
%

\begin{lem}
\label{lem:L2-estimate-oscillation}
Let $\Om\subset\RR^N$, $N\ge 2$, be a bounded domain with boundary of class $C^2$.
Set $h=q-u$, where $u$ is the solution of \eqref{serrin1} and $q$ is any quadratic polynomial as in \eqref{quadratic} with $z\in\Om$. 
\par
Then, there exists a positive constant $C$ such that
\begin{equation}
\label{L2-stability}
\rho_e-\rho_i\le C \, \nr h - h_{\Om} \nr_{p, \Om}^{ p/(N+p) }.
\end{equation}
The constant $C$ depends on $N$, $p$, $d_\Om$, $r_i$, $r_e$. If $\Om$ is convex the dependence on $r_e$ can be removed.
%
%
\end{lem}

\begin{proof}
%
%
%
By direct computations it is easy to check that
\begin{equation*}
| \na h | \le M + d_\Om \quad \mbox{on } \ol{\Om},
%
%
\end{equation*}
where $M$ is the maximum of $|\na u|$ on $\ol{\Om}$, as defined in \eqref{bound-gradient}.
Thus, we can apply Lemma \ref{lem:Lp-estimate-oscillation-generic-v} with $v=h$ and $G= M + d_\Om$.
By means of \eqref{oscillation} we deduce that
\eqref{L2-stability} holds with 
\begin{equation}\label{eq:constantCmaxnuovolemmaoscillation}
C= 2 \, a_{N,p} \, \frac{ (M + d_\Om )^{ \frac{N}{N+p} } }{ r_i } ,
\end{equation} 
if
\begin{equation*}
\nr h - h_{\Om} \nr_{p, \Om} \le \al_{N,p} \, (M + d_\Om) \, r_{i}^{\frac{N+p}{p}} .
\end{equation*}
Here,
$a_{N,p}$ and $\al_{N,p}$ are the constants defined in \eqref{eq:costantia_Nal_Nlemmagenericv}. 
On the other hand, if
\begin{equation*}
\nr h - h_{\Om} \nr_{p, \Om} > \al_{N,p} \, (M + d_\Om) \, r_{i}^{\frac{N+p}{p}} ,
\end{equation*}
it is trivial to check that \eqref{L2-stability} is verified with
$$C= \frac{d_\Om}{ \left[ \al_{N,p} \, (M + d_\Om) \right]^{\frac{p}{N+p}} \, r_{i} }.$$
Thus, \eqref{L2-stability} always holds true if we choose
the maximum between this constant and that in \eqref{eq:constantCmaxnuovolemmaoscillation}. We then can easily see that the following constant will do:
%
%
\begin{equation*}
C= \max{ \left\lbrace 2 \, a_{N,p} , \al_{N,p}^{- \frac{p}{N+p}} \right\rbrace } \, \frac{d_\Om^\frac{N}{N+p} }{r_i} \,  \left( 1 + \frac{M}{d_\Om} \right)^{ \frac{N}{N+p} } .
\end{equation*}
%
%
Now, by means of \eqref{gradient-estimate}, we obtain the constant
\begin{equation*}
C= \max{ \left\lbrace 2 \, a_{N,p} , \al_{N,p}^{- \frac{p}{N+p}} \right\rbrace } \, \frac{d_\Om^\frac{N}{N+p} }{r_i} \,  \left( 1 + c_N \, \frac{d_\Om + r_e}{r_e} \right)^{ \frac{N}{N+p} }.
\end{equation*}
If $\Om$ is convex, the dependence on $r_e$ can be avoided and we can choose
\begin{equation*}
C= \left( 1 + c_N \right)^{ \frac{N}{N+p} } \, \max{ \left\lbrace 2 \, a_{N,p} , \al_{N,p}^{- \frac{p}{N+p}} \right\rbrace } \, \frac{d_\Om^\frac{N}{N+p} }{r_i},
\end{equation*}
in light of \eqref{bound-M-convex}.
\end{proof}

%
%

%
%

For Serrin's overdetermined problem, Theorem \ref{thm:serrin-W22-stability} below will be crucial.
There, we associate the oscillation of $h$, and hence $\rho_e - \rho_i$, with the weighted $L^2$-norm of its Hessian matrix.
%
%
%

To this aim, we now choose the center $z$ of the quadratic polynomial $q$ in \eqref{quadratic} to be any critical point of $u$ in $\Om$. Notice that the (global) minimum point of $u$ is always attained in $\Om$. 
%
%
With this choice we have that $\na h (z) =0$.
%
%
We emphasize that the result that we present here improves (for every $N \ge 2$) the exponents of estimates obtained in \cite{MP2}.

\begin{thm}
\label{thm:serrin-W22-stability} 
Let $\Om\subset\RR^N$, $N\ge 2$, be a bounded domain with boundary $\Ga$ of class $C^{2}$
and $z$ be any critical point in $\Om$ of the solution $u$ of \eqref{serrin1}.
Consider the function  $h=q-u$, with $q$ given by \eqref{quadratic}.

There exists a positive constant $C$ such that
\begin{equation}\label{eq:C-provastab-serrin-W22}
\rho_e-\rho_i\le C\, \nr \de_\Ga^{\frac{1}{2} } \, \na^2 h  \nr_{2,\Om}^{\tau_N} ,
\end{equation}
with the following specifications:
%
%
\begin{enumerate}[(i)]
\item $\tau_2 = 1$;
\item $\tau_3$ is arbitrarily close to one, in the sense that for any $\theta>0$, there exists a positive constant $C$ such that  \eqref{eq:C-provastab-serrin-W22} holds with $\tau_3 = 1- \theta$;
\item $\tau_N = 2/(N-1)$ for $N \ge 4$.
\end{enumerate}

The constant $C$
%
%
depends on $N$, $r_i$, $r_e$, $d_\Om$, $\de_\Ga (z)$,
and $\theta$ (only in the case $N=3$).
\end{thm}

\begin{proof}
For the sake of clarity, we will always use the letter $c$ to denote the constants in all the inequalities appearing in the proof.
Their explicit computation will be clear by following the steps of the proof.

(i) 
Let $N=2$.
By the Sobolev immersion theorem (see for instance \cite[Theorem 3.12]{Gi} or \cite[Chapter 5]{Ad}), we have that there is a constant $c$ such that, for any $v\in W^{1,4 }(\Om)$, we have that
\begin{equation}
\label{eq:immersionSerrinN2}
\frac{|v(x)-v(y)|}{|x-y|^{\frac{1}{2} } }\le c\,\nr v\nr_{W^{1,4}(\Om)} \ \mbox{ for any  $x, y\in\ol{\Om}$ with $x\not=y$}.
\end{equation}

Applying \eqref{harmonic-poincare} with $v=h$, $p=4$, and $\al=0$ leads to
$$
\nr h - h_{\Om} \nr_{W^{1,4}(\Om)}\le c \, \nr \na h\nr_{4,\Om}.
$$
Since $\na h(z)=0$, we can apply item (i) of Corollary \ref{cor:JohnPoincareaigradienti} with $r=4$, $p=2$, and $\al=1/2$ to $h$ and obtain that
$$
\nr \na h \nr_{4,\Om} \le c \,  \nr \de_\Ga^{\frac{1}{2} } \, \na^2 h  \nr_{2,\Om} .
$$
Thus, we have that
$$
\nr h - h_{\Om} \nr_{W^{1,4}(\Om)}\le c \, \nr \de_\Ga^{\frac{1}{2} } \, \na^2 h \nr_{2,\Om} .
$$

By using the last inequality together with \eqref{eq:immersionSerrinN2}, by choosing $v= h- h_{\Om}$ and noting that $|x-y| \le d_\Om$ for any  $x, y\in\ol{\Om}$, we have that
\begin{equation*}
\max_\Ga h-\min_\Ga h \le 
c \, \nr \de_\Ga^{\frac{1}{2} } \, \na^2 h  \nr_{2,\Om}.
\end{equation*}
Thus, by recalling \eqref{oscillation} we get that \eqref{eq:C-provastab-serrin-W22} holds with $\tau_2=1$.

(ii) Let $N=3$.
By applying to the function $h$ \eqref{harmonic-poincare} with $r= \frac{3(1- \theta)}{\theta}$,
$p=3(1 -\theta )$, $\al=0$, and item (i) of Corollary \ref{cor:JohnPoincareaigradienti} with
$r=3 ( 1 - \theta )$, $p=2$, $\al=1/2$, we get
$$
\nr h - h_{\Om} \nr_{\frac{3( 1 - \theta)}{\theta}, \Om} \le c \, \nr \de_\Ga^{\frac{1}{2} } \, \na^2 h \nr_{2,\Om}.
$$
Thus, by recalling Lemma \ref{lem:L2-estimate-oscillation} we have that \eqref{eq:C-provastab-serrin-W22} holds true with $\tau_3 = 1- \theta $.

(iii) Let $N \ge 4$. Since $\na h(z)=0$, we can apply to $h$ item (i) of Corollary \ref{cor:JohnPoincareaigradienti} with $r=\frac{2N}{N-1}$, $p=2$, $\al=1/2$, and obtain that
\begin{equation*}
\nr \na h \nr_{\frac{2N}{N-1},\Om} \le c \, \nr \de_\Ga^{\frac{1}{2} } \, \na^2 h  \nr_{2,\Om} .
\end{equation*}
Being $N \ge 4$, we can apply \eqref{John-harmonic-poincare} with $v=h$, $r=\frac{2N}{N-3}$, $p=\frac{2N}{N-1}$, $\al=0$, and get
\begin{equation*}
\nr h- h_{ \Om} \nr_{\frac{2N}{N-3} } \le c \, \nr \na h \nr_{\frac{2N}{N-1},\Om}.
\end{equation*}
Thus,
$$
\nr h- h_{ \Om} \nr_{\frac{2N}{N-3} } \le c \, \nr \de_\Ga^{\frac{1}{2} } \, \na^2 h  \nr_{2,\Om},
$$
and by Lemma \ref{lem:L2-estimate-oscillation} we get that \eqref{eq:C-provastab-serrin-W22} holds with $\tau_N = 2/(N-1)$.
\end{proof}

\begin{rem}[On the constant $C$]\label{rem:dipendenzecostanti}
{\rm
The constant $C$ can be shown to depend only on the parameters mentioned in the statement of Theorem \ref{thm:serrin-W22-stability}. 
%
%
In fact, the parameters $\mu_{p, 0} (\Om,z)$, $\ol{\mu}_{p, 0} (\Om)$, $\mu_{r,p, \al} (\Om,z)$, $\ol{\mu}_{r,p,\al} (\Om)$,  can be estimated by using item (iii) of Remark \ref{rem:stime mu HS in item ii}.
To remove the dependence on the volume, then one can use the trivial bound
$$
|\Om|^{1/N} \le |B|^{1/N} d_\Om /2 .
$$
%
%

We recall that if $\Om$ has the strong local Lipschitz property (for the definition see \cite[Section 4.5]{Ad}), the immersion constant (that we used in the proof of item (i) of Theorem \ref{thm:serrin-W22-stability}) depends only on $N$ and the two Lipschitz parameters of the definition (see \cite[Chapter 5]{Ad}).
In our case $\Om$ is of class $C^2$, hence obviously it has the strong local Lipschitz property and the two Lipschitz parameters can be easily estimated in terms of $\min \lbrace r_i,r_e \rbrace $.
}
\end{rem}

In the case of Alexandrov's Soap Bubble Theorem, we have to deal with \eqref{H-fundamental} or \eqref{identity-SBT2}, where just the unweighted $L^2$-norm of the Hessian of $h$ appears.
Thus, the appropriate result in this case is Theorem \ref{thm:SBT-W22-stability} below, in which we
%
%
improve (for every $N\ge 4$) the exponents of estimates obtained in \cite{MP}.

\begin{thm}
\label{thm:SBT-W22-stability}
Let $\Om\subset\RR^N$, $N\ge 2$, be a bounded domain with boundary $\Ga$ of class $C^{2}$
and $z$ be any critical point in $\Om$ of the solution $u$ of \eqref{serrin1}.
%
Consider the function  $h=q-u$, with $q$ given by \eqref{quadratic}.
%
%

There exists a positive constant $C$ such that
\begin{equation}
\label{W22-stability}
\rho_e-\rho_i\le C\, \nr \na^2 h\nr_{2,\Om}^{\tau_N} ,
\end{equation}
with the following specifications:
%
%
\begin{enumerate}[(i)]
\item $\tau_N=1$ for $N=2$ or $3$;
\item $\tau_4$ is arbitrarily close to one, in the sense that for any $\theta>0$, there exists a positive constant $C$ such that \eqref{W22-stability} holds with $\tau_4= 1- \theta $;
\item $\tau_N = 2/(N-2) $ for $N\ge 5$.
\end{enumerate}

The constant $C$ depends on $N$, $r_i$, $r_e$, $d_\Om$, $\de_\Ga (z)$, and $\theta$ (only in the case $N=4$).
\end{thm}
\begin{proof}
As done in the proof of Theorem \ref{thm:serrin-W22-stability}, for the sake of clarity, we will always use the letter $c$ to denote the constants in all the inequalities appearing in the proof.
Their explicit computation will be clear by following the steps of the proof. By reasoning as described in Remark \ref{rem:dipendenzecostanti}, one can easily check that those constants depend only on the geometric parameters of $\Om$ mentioned in the statement of the theorem.

(i) Let $N=2$ or $3$.
By the Sobolev immersion theorem (see for instance \cite[Theorem 3.12]{Gi} or \cite[Chapter 5]{Ad}), we have that there is a constant $c$ such that, for any $v\in W^{2,2}(\Om)$, we have that
\begin{equation}
\label{eq:immersionSBTN23}
\frac{|v(x)-v(y)|}{|x-y|^\ga}\le c\, \nr v\nr_{W^{2,2}(\Om)} \ \mbox{ for any  $x, y\in\ol{\Om}$ with $x\not=y$},
\end{equation}
where $\ga$ is any number in $(0,1)$ for $N=2$ and $\ga=1/2$ for $N=3$. 

Since $\na h(z)=0$, we can apply item (i) of Corollary \ref{cor:Poincareaigradienti} with $p=2$ and $\al= 0$ to $h$ and obtain that
\begin{equation*}
\int_{\Om} |\na h|^2 \, dx \le c \, \int_{\Om} | \na^2 h|^2 \, dx.
\end{equation*}
Using this last inequality together with \eqref{harmonic-poincare} with $v= h$, $p=2$, $\al=0$, leads to
$$
\nr h - h_{\Om} \nr_{W^{2,2}(\Om)}\le c \,\nr \na^2 h\nr_{2,\Om}.
$$

Hence, by using \eqref{eq:immersionSBTN23} with $v= h-h_{ \Om}$ and noting that $|x-y| \le d_\Om$ for any  $x, y\in\ol{\Om}$, we get that
\begin{equation*}
\max_\Ga h-\min_\Ga h \le c \, \nr \na^2 h\nr_{2,\Om}.
\end{equation*}
Thus, by recalling \eqref{oscillation} we get that \eqref{W22-stability} holds with $\tau_N=1$.

(ii) Let $N=4$.
By applying \eqref{harmonic-poincare} with $r= \frac{4( 1 - \theta)}{\theta}$, $p= 4 ( 1 -\theta )$, $\al=0$, and item (i) of Corollary \ref{cor:JohnPoincareaigradienti} with $r=4 ( 1 - \theta )$, $p=2$, $\al= 0$, for $v=h$ we get
$$
\nr h - h_{ \Om} \nr_{\frac{4( 1 - \theta)}{\theta}, \Om} \le c \, \nr \na^2 h \nr_{2,\Om}.
$$
Thus, by Lemma \ref{lem:L2-estimate-oscillation} we conclude that \eqref{W22-stability} holds with $\tau_4= 1- \theta$.

(iii) Let $N\ge 5$. Applying \eqref{John-harmonic-poincare} with $r= \frac{2N}{N-4}$, $p= \frac{2N}{N-2}$, $\al=0$, and item (i) of Corollary \ref{cor:JohnPoincareaigradienti} with $r={\frac{2N}{N-2}}$, $p=2$, $\al=0$, by choosing $v=h$ we get that
$$
\nr h - h_{ \Om} \nr_{\frac{2N}{N-4},\Om} \le c \, \nr\na^2 h\nr_{2,\Om}.
$$
By Lemma \ref{lem:L2-estimate-oscillation}, we have that \eqref{W22-stability} holds
true with $\tau_N = 2/(N-2)$.
\end{proof}

%
%

\begin{rem}\label{rem:removing-distance-convex}
{\rm
The parameter $\de_\Ga (z)$ is used only to give an estimate of the parameters $\mu_{p, 0} (\Om,z)$ and $\mu_{r,p,\al} (\Om,z)$ in terms of an explicit geometrical quantity.

Moreover, in case $\Om$ is convex, by using $\de_\Ga (z)$ we 
are able to completely remove the dependence of the constants on $z$.
%
%
Indeed, in this case $u$ has a unique minimum point $z$ in $\Om$, since $u$ is analytic and the level sets of $u$ are convex by a result in \cite{Ko} (see \cite{MS}, for a similar argument), and hence we have only one choice for the point $z$. 
Thus, an estimate of $\de_\Ga(z)$ from below can be obtained, first, by putting together arguments in \cite[Lemma 2.6]{BMS} and \cite[Remark 2.5]{BMS}, to obtain that
$$
\de_\Ga(z)\ge \frac{k_N}{|\Om|\, d_\Om^{N-1}}\,\left[ \max_{\ol{\Om}}(-u) \right]^N.
$$
Secondly, if we set $x$ to be a point in $\Om$ such that $\de_\Ga (x)= r_i$, by recalling Lemma \ref{lem:relationdist} we easily find that
\begin{equation*}
\max_{\ol{\Om}}(-u) \ge - u(x) \ge \frac{r_i^2}{2}.
\end{equation*}
%
%
%
%
All in all, we have that
$$
\de_\Ga(z)\ge k_N\,\frac{r_i^{2N}}{|\Om|\, d_\Om^{N-1}}.
$$

Thus, if $\Om$ is convex, we can affirm that the constants $C$, appearing in Theorems \ref{thm:serrin-W22-stability} and \ref{thm:SBT-W22-stability}, depend on $N$, $r_i$, and $d_\Om$, only. In fact, as already noticed in item (i) of Remark \ref{rem:boundgradienttorsionconvexr_eremoving}, in this case also the parameter $r_e$ can be removed, being $r_e=+\infty$.
%
%
}
\end{rem}

Even if $\Om$ is not convex, the dependence on $\de_\Ga (z)$ can still be removed in some other cases. In fact, we can do it by choosing the point $z$ appearing in \eqref{quadratic} differently, as explained in the following remark. 
\par

\begin{rem}
\label{remarkprova:nuova scelte z}
{\rm
(i) We can choose $z$ as the center of mass of $\Om$. 
In fact, if $z$ is the center of mass of $\Om$, we have that
\begin{multline*}
\int_\Om \na h(x)\,dx=\int_\Om [x-z-\na u(x)]\,dx=\\
\int_\Om x\,dx-|\Om|\,z-\int_\Ga u(x)\,\nu(x)\,dS_x=0.
\end{multline*}
Thus, we can use item (ii) of Corollaries \ref{cor:Poincareaigradienti}, \ref{cor:JohnPoincareaigradienti} instead of item (i). In this way, in the estimates of Theorems \ref{thm:serrin-W22-stability} and \ref{thm:SBT-W22-stability} we simply obtain the same constants with $\mu_{p, 0}(\Om,z)$ and $\mu_{r,p, \al}(\Om,z)$ replaced by $\ol{\mu}_{p, 0}(\Om)$ and $\ol{\mu}_{r, p, \al}(\Om)$. Thus, we removed the presence of $\de_\Ga (z)$ and hence the dependence on $z$.

It should be noticed that, in this case the extra assumption that $z \in \Om$ is needed, since we want that the ball $B_{\rho_i}(z)$ be contained in $\Om$. 
\par
(ii) As done in \cite{Fe}, another possible way to choose $z$ is $z= x_0 - \na u (x_0)$,
where $x_0 \in \Om$ is any point such that $\de_\Ga (x_0) \ge r_i$. In fact, we obtain that $\na h (x_0)=0$ and we can thus use \eqref{harmonic-quasi-poincare} and \eqref{John-harmonic-quasi-poincare}, with $\mu_{p, 0}(\Om,z)$ and $\mu_{r,p,\al}(\Om,z)$ replaced by $\mu_{p, 0}(\Om,x_0)$ and $\mu_{r,p,\al}(\Om,x_0)$.
%
%
Thus, by recalling item (iii) of Remark \ref{rem:stime mu HS in item ii} it is clear that we removed the dependence on $\de_\Ga (x_0)$ in our constants that, as already noticed, comes from the estimation of the parameters $\mu_{p, 0} (\Om, x_0)$ and $\mu_{r,p,\al} (\Om, x_0)$.  
\par
As in item (i), we should additionally require that $z \in \Om$, to be sure that the ball $B_{\rho_i}(z)$ be contained in $\Om$.
}
\end{rem}

\chapter{Stability results}\label{chapter:Stability results}

In this chapter, we collect our results on the stability of the spherical configuration by putting together the identities derived in Chapter \ref{chapter:integral identities and symmetry results} and the estimates obtained in Chapter \ref{chapter:various estimates}.

\section{Stability for Serrin's overdetermined problem}\label{sec:stability Serrin}

In light of \eqref{L2-norm-hessian}, \eqref{idwps} can be rewritten in terms of the harmonic function $h$, as stated in the following.
\begin{lem}
Under the same assumptions of Theorem \ref{thm:serrinidentity}, if we set $h=q-u$, then
%
%
it holds that
\begin{equation}\label{idwps-h}
\int_{\Om} (-u)\, |\na ^2 h|^2\,dx=
\frac{1}{2}\,\int_\Ga ( R^2-u_\nu^2)\, h_\nu\,dS_x. 
\end{equation}
\end{lem}
\begin{proof}
Simple computations give that $h_\nu=q_\nu-u_\nu$.
By using this identity and \eqref{L2-norm-hessian}, \eqref{idwps-h} easily follows from \eqref{idwps}.
\end{proof}

%

\par
In light of \eqref{eq:relationdist}, Theorem \ref{thm:serrin-W22-stability} gives an estimate from below of the left-hand side of \eqref{idwps-h}. Now, we will take care of its right-hand side and prove our main result for Serrin's problem. The result that we present here, improves (for every $N \ge 2$) the exponents in the estimate obtained in \cite[Theorem 1.1]{MP2}.

\begin{thm}[Stability for Serrin's problem]
\label{thm:Improved-Serrin-stability}
Let $\Om\subset\RR^N$, $N\ge2$, be a bounded domain with boundary $\Ga$ of class $C^2$ and $R$ be the constant defined in \eqref{def-R-H0}.
Let $u$ be the solution of problem \eqref{serrin1} and 
$z\in\Om$ be any of its critical points.
\par
There exists a positive constant $C$ such that
\begin{equation}
\label{general improved stability serrin C}
\rho_e-\rho_i\le C\,\nr u_\nu - R \nr_{2,\Ga}^{\tau_N} ,
\end{equation}
with the following specifications:
%
%
%
\begin{enumerate}[(i)]
\item $\tau_2 = 1$;
\item $\tau_3$ is arbitrarily close to one, in the sense that for any $\theta>0$, there exists a positive constant $C$ such that  \eqref{general improved stability serrin C} holds with $\tau_3 = 1- \theta$;
\item $\tau_N = 2/(N-1)$ for $N \ge 4$.
\end{enumerate}

The constant $C$ depends on $N$, $r_i$, $r_e$, $d_\Om$, $\de_\Ga (z)$,
and $\theta$ (only in the case $N=3$).
\end{thm}
\begin{proof}
We have that 
$$
\int_\Ga ( R^2-u_\nu^2)\, h_\nu\,dS_x\le (M+R)\,\nr u_\nu-R\nr_{2,\Ga} \nr h_\nu\nr_{2,\Ga},
$$
after an an application of H\"older's inequality.
Thus, by item (i) of Lemma \ref{lem:genericv-trace inequality} with $v=h$, \eqref{idwps-h}, and this inequality, we infer that
\begin{multline*}
\nr h_\nu\nr_{2,\Ga}^2\le \frac{2}{r_i} \left(1+\frac{N}{r_i\, \mu_{2,2,\frac{1}{2} }(\Om,z)^2 } \right)  \int_{\Om} (-u) |\na^2h|^2 dx \le \\ 
\frac{M+R}{r_i} \left(1+\frac{N}{r_i\, \mu_{2,2,\frac{1}{2} }(\Om,z)^2} \right)\nr u_\nu-R\nr_{2,\Ga} \nr h_\nu\nr_{2,\Ga},
\end{multline*}
and hence
\begin{equation}
\label{ineq-feldman}
\nr h_\nu\nr_{2,\Ga}\le \frac{M+R}{r_i} \left(1+\frac{N}{r_i\, \mu_{2,2,\frac{1}{2} }(\Om,z)^2} \right)\nr u_\nu-R\nr_{2,\Ga}.
\end{equation}
Therefore,
\begin{multline*}
\int_\Om |\na ^2 h|^2 \de_\Ga (x)\, dx \le \frac{2}{r_i}\int_\Om (-u) |\na^2 h|^2 dx \le \\
 \left(\frac{M+R}{r_i}\right)^2 \left(1+\frac{N}{r_i \, \mu_{2,2,\frac{1}{2} }(\Om,z)^2} \right)\nr u_\nu-R\nr_{2,\Ga}^2,
\end{multline*}
by Lemma \ref{lem:relationdist}.
These inequalities and Theorem \ref{thm:serrin-W22-stability} then give the desired conclusion.

We recall that $\mu_{2,2,\frac{1}{2} }(\Om,z)$ appearing in the constant in the last inequality can be estimated in terms of $d_\Om$ and $\min \left[ r_i, \de_\Ga (z) \right]$, by proceeding as described in Remark \ref{rem:dipendenzecostanti}.
The ratio $R$ can be estimated (from above) in terms of $|\Om|^{1/N}$ -- just by using the isoperimetric inequality; in turn, $|\Om|^{1/N}$ can be bounded in terms of $d_\Om$ by proceeding as described in Remark \ref{rem:dipendenzecostanti}. 
Finally, as usual, $M$ can be estimated by means of \eqref{gradient-estimate}. 
\end{proof}

If we want to measure the deviation of $u_\nu$ from $R$ in $L^1$-norm, we get a smaller (reduced by one half) stability exponent. The following result improves \cite[Theorem 3.6]{MP2}.

\begin{thm}[Stability with $L^1$-deviation]
\label{thm:Serrin-stability}
Theorem \ref{thm:Improved-Serrin-stability} still holds
with \eqref{general improved stability serrin C} replaced by 
\begin{equation*}
\rho_e-\rho_i\le C\,\nr u_\nu - R \nr_{1,\Ga}^{\tau_N/2} .
\end{equation*}
%
%
%
\end{thm}

\begin{proof}
Instead of applying H\"older's inequality to the right-hand side of \eqref{idwps-h}, we just use the rough bound:
%
%
\begin{equation*}
\int_{\Om} (-u)\, |\na ^2 h|^2\,dx \le
\frac{1}{2}\, \left( M + R \right)\,(M + d_{\Om}) \, \int_\Ga \left| u_\nu - R \right| \, dS_x,
\end{equation*}
since $(u_\nu+R)\,|h_\nu|\le ( M + R)\,(M + d_{\Om})$ on $\Ga$. The conclusion then follows from similar arguments.
\end{proof}

%
%

\begin{rem}
{\rm
If $\Om$ is convex, in view of Remark \ref{rem:removing-distance-convex} we can claim that
the constants $C$ of Theorems \ref{thm:Improved-Serrin-stability} and \ref{thm:Serrin-stability} depend only on $N$, $r_i$, $d_\Om$ (and $\theta$ only in the case $N=3$). 

The dependence of $C$ on $\de_\Ga (z)$ can be removed also in the cases described in Remark \ref{remarkprova:nuova scelte z}.
Regarding the case described in item (i) of that remark, we notice that, since in that situation $z$ is chosen as the center of mass of $\Om$, in the proof of Theorem \ref{thm:Improved-Serrin-stability} we will use item (ii) of Lemma \ref{lem:genericv-trace inequality} instead of item (i), so that $\mu_{2,2,\frac{1}{2}} (\Om, z)$ will be replaced by $\ol{\mu}_{2,2,\frac{1}{2}} (\Om)$.
}
\end{rem}

\bigskip

Since the estimates in Theorems \ref{thm:Improved-Serrin-stability} and \ref{thm:Serrin-stability} do not depend on the particular critical point chosen, as a corollary, we obtain results of closeness to a union of balls:
here, we just illustrate the instance of Theorem \ref{thm:Improved-Serrin-stability}.

\begin{cor}[Closeness to an aggregate of balls]
\label{cor:Serrin-stability-aggregate}
Let $\Ga$, $R$ and $u$ be as in Theorem \ref{thm:Improved-Serrin-stability}.
Then, there exist points $z_1, \dots, z_n$ in $\Om$, $n\ge 1$, and corresponding numbers
%
%
\begin{equation*}
\rho_i^j=\min_{x\in\Ga}|x-z_j| \ \mbox{ and } \ \rho_e^j=\min_{x\in\Ga}|x-z_j|,
\quad j=1, \dots, n,
\end{equation*}
such that
%
%
\begin{equation*}
\bigcup_{j=1}^n B_{\rho_i^j}(z_j)\subset\Om\subset \bigcap_{j=1}^n B_{\rho_e^j}(z_j)
\end{equation*}
and
$$
\max_{1\le j\le n}(\rho_e^j-\rho_i^j)\le C\,\nr u_\nu - R \nr_{2,\Ga}^{\tau_N} .
$$
Here, the exponent $\tau_N$ and the constant $C$ are those of Theorem \ref{thm:Improved-Serrin-stability}.
\par
The number $n$ can be chosen as the number of connected components of the set $\cM$ of all the local minimum points of the solution $u$ of \eqref{serrin1}.
\end{cor}

\begin{proof}
We pick one point $z_j$ from
each connected component of the set of local minimum points of $u$. By applying Theorem \ref{thm:Improved-Serrin-stability} to each $z_j$, the conclusion is then evident.
\end{proof}

\begin{rem}\label{rem:serrin-stimequantitativefissandoparametri}
{\rm
The estimates presented in Theorems \ref{thm:Improved-Serrin-stability}, \ref{thm:Serrin-stability} may be interpreted as stability estimates, once that we fixed some a priori bounds on the relevant parameters: here, we just illustrate the case of Theorem \ref{thm:Improved-Serrin-stability}. 
Given three positive constants $\ol{d}$, $\ul{r}$, and $\ul{\de}$, let ${\mathcal D}={\mathcal D} (\ol{d}, \ul{r}, \ul{\de})$ be the class of bounded domains $\Om\subset\RR^N$ with boundary of class $C^{2}$, such that 
$$
d_{ \Om} \le \ol{d}, \quad r_i(\Om), r_e(\Om)\ge \ul{r}, \quad \de_\Ga (z) \ge \ul{\de}.
$$
Then, for every $\Om\in {\mathcal D} $, we have that
$$
\rho_e-\rho_i\le C\, \nr u_\nu - R \nr_{2,\Ga}^{\tau_N},
$$
where $\tau_N$ is that appearing in \eqref{general improved stability serrin C} and $C$ is a constant depending on $N$, $\ol{d}$, $\ul{r}$, $\ul{\de}$ (and $\theta$ only in the case $N=3$).
\par
If we relax the a priori assumption that $\Om\in {\mathcal D}$ (in particular if we remove the lower bound $\ul{r}$), it may happen that, 
as the deviation $\nr u_\nu - R \nr_{2,\Ga}$ tends to $0$, $\Om$ tends to the ideal configuration of two or more disjoint balls, while $C$ diverges since $\ul{r}$ tends to $0$. The configuration of more balls connected with tiny (but arbitrarily long) tentacles has been quantitatively studied in \cite{BNST}.
}
\end{rem}

\section{Stability for Alexandrov's Soap Bubble Theorem}\label{sec:Alexandrov's SBT}


This section is devoted to the stability issue for the Soap Bubble Theorem.

As already noticed, in light of \eqref{L2-norm-hessian}, \eqref{identity-SBT2} can be rewritten in terms of $h$, as stated in the following.

\begin{lem}
Under the same assumptions of Theorem \ref{thm:identitySBT}, if we set $h=q-u$, then
%
%
it holds that
\begin{multline}\label{identity-SBT-h}
\frac1{N-1}\int_{\Om} |\na ^2 h|^2 dx+
\frac1{R}\,\int_\Ga (u_\nu-R)^2 dS_x = \\
-\int_{\Ga}(H_0-H)\,h_\nu\,u_\nu\,dS_x+
\int_{\Ga}(H_0-H)\, (u_\nu-R)\,q_\nu\, dS_x.
\end{multline}
\end{lem}

Next, we derive the following lemma, that parallels and is a useful consequence of Lemma \ref{lem:genericv-trace inequality}.

\begin{lem}
\label{lem:improving-SBT}
Let $\Om\subset\RR^N$, $N\ge2$, be a bounded domain with boundary $\Ga$ of class $C^2$.
%
%
Denote by $H$ the mean curvature of $\Ga$ and let $H_0$ be the constant defined in \eqref{def-R-H0}.
\par
Then, the following inequality holds:
\begin{equation}
\label{improving-SBT}
\nr  u_\nu-R\nr_{2,\Ga} \le 
R\left\{ d_\Om +\frac{M (M+R)}{r_i}\left(1+\frac{N}{r_i \, \mu_{2,2,\frac{1}{2} }(\Om,z)^2 }\right)\right\} \nr H_0-H\nr_{2,\Ga}.
\end{equation}
\end{lem}

\begin{proof}
Discarding the first summand on the left-hand side of \eqref{identity-SBT-h} and applying H\"older's inequality on its right-hand side gives that
$$
\frac1{R} \nr  u_\nu-R\nr_{2,\Ga}^2 \le \nr H_0-H\nr_{2,\Ga}\left( M \nr h_\nu\nr_{2,\Ga} + d_\Om \,\nr  u_\nu-R\nr_{2,\Ga} \right),
$$ 
since $u_\nu\le M$ and $|q_\nu|\le d_\Om $ on $\Ga$. Thus, inequality \eqref{ineq-feldman}
implies that
\begin{multline*}
\nr  u_\nu-R\nr_{2,\Ga}^2 \le 
\\
R\left\{ d_\Om +\frac{M (M+R)}{r_i}\left(1+\frac{N}{r_i \, \mu_{2,2,\frac{1}{2} }(\Om,z)^2}\right)\right\} \nr H_0-H\nr_{2,\Ga} \nr  u_\nu-R\nr_{2,\Ga},
\end{multline*}
from which \eqref{improving-SBT} follows at once.
\end{proof}

We are now ready to prove our main result for Alexandrov's Soap Bubble Theorem. The result that we present here, improves (for every $N \ge 4$) the exponents of estimates obtained in \cite[Theorem 1.2]{MP2}.

\begin{thm}[Stability for the Soap Bubble Theorem] 
\label{thm:SBT-improved-stability}
Let $N\ge 2$ and let $\Ga$ be a surface of class $C^{2}$, which is the boundary of a bounded domain $\Om\subset\RR^N$. Denote by $H$ the mean curvature of $\Ga$ and let $H_0$ be the constant defined in \eqref{def-R-H0}.
\par
Then,
%
%
for some point $z\in\Om$
there exists a positive constant $C$ such that
\begin{equation}
\label{SBT-improved-stability-C}
\rho_e-\rho_i \le C\,\nr H_0-H\nr_{2,\Ga}^{\tau_N} ,
\end{equation}
with the following specifications:
%
%
%
\begin{enumerate}[(i)]
\item $\tau_N=1$ for $N=2$ or $3$;
\item $\tau_4$ is arbitrarily close to one, in the sense that for any $\theta>0$, there exists a positive constant $C$ such that \eqref{SBT-improved-stability-C} holds with $\tau_4= 1- \theta $;
\item $\tau_N = 2/(N-2) $ for $N\ge 5$.
\end{enumerate}

The constant $C$ depends on $N$, $r_i$, $r_e$, $d_\Om$, $\de_\Ga (z)$, and $\theta$ (only in the case $N=4$).
\end{thm}
\begin{proof}
As before, we choose $z\in\Om$ to be any local minimum point of $u$ in $\Om$.
Discarding the second summand on the left-hand side of \eqref{identity-SBT-h} and applying H\"older's inequality on its right-hand side, as in the previous proof, gives that
\begin{multline*}
\frac1{N-1}\,\int_\Om |\na^2 h|^2 dx\le \\ 
R\left\{ d_\Om +\frac{M (M+R)}{r_i}\left(1+\frac{N}{r_i \, \mu_{2,2,\frac{1}{2} }(\Om,z)^2}\right)\right\} \nr H_0-H\nr_{2,\Ga} \nr  u_\nu-R\nr_{2,\Ga}, \le \\
R^2\left\{ d_\Om +\frac{M (M+R)}{r_i}\left(1+\frac{N}{r_i \, \mu_{2,2,\frac{1}{2} }(\Om,z)^2 }\right)\right\}^2 \nr H_0-H\nr_{2,\Ga}^2,
\end{multline*}
where the second inequality follows from Lemma \ref{lem:improving-SBT}.
\par
The conclusion then follows from
Theorem \ref{thm:SBT-W22-stability}.
\end{proof}

\begin{rem}
{\rm
Estimates similar to those of Theorem \ref{thm:SBT-improved-stability} can also be obtained as a direct corollary of Theorem \ref{thm:Improved-Serrin-stability}, by means of \eqref{improving-SBT}. As it is clear, in this way the exponents $\tau_N$ would be worse than those obtained in Theorem \ref{thm:SBT-improved-stability}.
}
\end{rem}

We now present a stability result with a weaker deviation (at the cost of getting a smaller stability exponent), analogous to Theorem \ref{thm:Serrin-stability}.
To this aim, we notice that from \eqref{H-fundamental} and \eqref{L2-norm-hessian} we can easily deduce the following inequality 
%
%
%
\begin{equation}
\label{fundamental-stability2} 
\frac1{N-1}\int_{\Om} |\na^2 h |^2 \,dx + \int_{\Ga}(H_0-H)^-\, (u_\nu)^2\,dS_x \le 
\int_{\Ga}(H_0-H)^+\, (u_\nu)^2\,dS_x
\end{equation}
(here, we use the positive and negative part functions $(t)^+=\max(t,0)$ and $(t)^-=\max(-t,0)$). That inequality tells us that, if we have an {\it a priori} bound $M$ for
$u_\nu$ on $\Ga$, then its left-hand side is small if the integral
\begin{equation*}
\int_{\Ga}(H_0-H)^+\,dS_x
\end{equation*}
is also small.  In particular, if $H$ is not too much smaller than $H_0$, then it cannot be too much larger than $H_0$ and the Cauchy-Schwarz deficit cannot be too large.  
%
%

It is clear that, 
$$
\int_\Ga (H_0-H)^+\,dS_x,
$$
is a deviation weaker than $\nr H_0-H\nr_{1,\Ga}$.

The result that we present here, improves (for every $N \ge 4$) the estimates obtained in \cite[Theorem 4.1]{MP}.

\begin{thm}[Stability with $L^1$-type deviation]
\label{thm:SBT-stability}
Theorem \ref{thm:SBT-improved-stability} still holds with \eqref{SBT-improved-stability-C} replaced by
\begin{equation*}
\rho_e-\rho_i \le C\,  \left\lbrace \int_\Ga (H_0-H)^+\,dS_x \right\rbrace^{\tau_N /2} .
\end{equation*}
\end{thm}

\begin{proof}
%
%
From \eqref{fundamental-stability2}, we infer that
\begin{equation*}
\nr \na^2 h\nr_{2,\Om}\le M\,\sqrt{N-1}\,  \left\lbrace \int_\Ga (H_0-H)^+\,dS_x \right\rbrace^{1/2}
\end{equation*}
%
%

The conclusion then follows from Theorem \ref{thm:SBT-W22-stability}.
\end{proof}

\begin{rem}\label{rem:ontheconstants-SBT-removingdelta-seconvexocasidiremark}
{\rm
(i) If $\Om$ is convex, the dependence of the constants $C$ in Theorems \ref{thm:SBT-improved-stability} and \ref{thm:SBT-stability} can be reduced to the parameters $N$, $r_i$, $d_\Om$ (and $\theta$ only in the case $N=4$), as described in Remark \ref{rem:removing-distance-convex}. The dependence on the parameter $\de_\Ga (z)$ can be removed also in the cases described in Remark \ref{remarkprova:nuova scelte z}.

(ii) Results of closeness to a union of balls analogous to Corollary \ref{cor:Serrin-stability-aggregate} can be easily derived also for Theorems \ref{thm:SBT-improved-stability} and \ref{thm:SBT-stability}.
}
\end{rem}

%
%
%
%

\begin{rem}\label{rem:stimequantitativefissandoparametri}
{\rm
The estimates presented in Theorems \ref{thm:SBT-improved-stability}, \ref{thm:SBT-stability} may be interpreted as stability estimates, once some a priori information is available: here, we just illustrate the case of Theorem \ref{thm:SBT-improved-stability}. 
Given three positive constants $\ol{d}$, $\ul{r}$, and $\ul{\de}$, let $\cS=\cS(\ol{d}, \ul{r}, \ul{\de})$ be the class of connected surfaces $\Ga\subset\RR^N$ of class $C^{2}$, where $\Ga$ is the boundary of a bounded domain $\Om$, such that 
$$
d_{ \Om} \le \ol{d}, \quad r_i(\Om), r_e(\Om)\ge \ul{r}, \quad \de_\Ga (z) \ge \ul{\de}.
$$
Then, for every $\Ga\in\cS$ we have that
$$
\rho_e-\rho_i\le C\, \nr H_0-H\nr_{2,\Ga}^{\tau_N},
$$
where $\tau_N$ is that appearing in \eqref{SBT-improved-stability-C} and the constant $C$ depends on $N$, $\ol{d}$, $\ul{r}$, $\ul{\de}$ (and $\theta$ only in the case $N=4$).
\par
If we relax the a priori assumption that $\Ga\in\cS$ (in particular if we remove the lower bound
$\ul{r}$), it may happen that, 
as the deviation $\nr H_0 - H\nr_{2,\Ga}$ tends to $0$, $\Om$ tends to the ideal configuration of two or more mutually tangent balls, while $C$ diverges since $\ul{r}$ tends to $0$. Such a configuration can be observed, for example, as limit of sets created by truncating (and then smoothly completing) unduloids with very thin necks. This phenomenon (called bubbling) has been quantitatively studied in \cite{CM} by considering strictly mean convex surfaces and by using the uniform deviation $\nr H_0 - H\nr_{\infty,\Ga}$.
}
\end{rem}

\section{Quantitative bounds for an asymmetry}\label{sec:stabbyasymmetrySBT}

Another consequence of Lemma \ref{lem:improving-SBT} is the following inequality that shows an optimal stability exponent for any $N\ge 2$. The number $\cA(\Om)$, defined as
\begin{equation}
\label{asymmetry}
\cA(\Om)=\inf\left\{\frac{|\Om\De B^x|}{|B^x|}: x \mbox{ center of a ball $B^x$ with radius $R$} \right\} ,
\end{equation}
is some sort of asymmetry similar to the so-called Fraenkel asymmetry (see \cite{Fr}). Here, $\Om\De B^x$ denotes the symmetric difference of $\Om$ and $B^x$, and $R$ is the constant defined in \eqref{def-R-H0}.

\begin{thm}[Stability by asymmetry, \cite{MP2}]
\label{th:asymmetry}
Let $N\ge 2$ and let $\Ga$ be a surface of class $C^{2}$, which is the boundary of a bounded domain $\Om\subset\RR^N$. Denote by $H$ the mean curvature of $\Ga$ and let $H_0$ be the constant defined in \eqref{def-R-H0}.
\par
Then it holds that
\begin{equation}
\label{eq:SBTstabinmisura}
\cA(\Om)\le C \, \nr H_0 - H \nr_{2,\Ga},
\end{equation}
for some positive constant $C$.
\end{thm}

\begin{proof}
We use \cite[inequality (2.14)]{Fe}: we have that 
%
%
\begin{equation*}
\frac{|\Om \De B_R^z|}{|B_R^z|} \le C \, \nr u_\nu - R \nr_{2,\Ga},
\end{equation*}
where $B_R^z$ is a ball of radius $R$ (as defined in \eqref{def-R-H0}) and centered at the point $z$ described in item (ii) of Remark \ref{remarkprova:nuova scelte z}. Hence, we obtain that
$$
\cA(\Om)\le C \, \nr u_\nu - R \nr_{2,\Ga},
$$
by the definition \eqref{asymmetry}.
Thus, thanks to \eqref{improving-SBT}, we obtain \eqref{eq:SBTstabinmisura}.
\par
Here, if we estimate $L_0$ as described in item (iii) of Remark \ref{rem:stime mu HS in item ii}, we can see that $C$ depends on $N, d_\Om, r_i, |\Ga|/|\Om|$.
\end{proof}

\begin{rem}[On the asymmetry $\cA(\Om)$]{\rm

Notice that, for any $x\in \Om$, we have that 

$$
\frac{|\Om \De B_R^x|}{|B_R^x|} \le \frac{|B_{\rho_e}^x \setminus B_{\rho_i}^x|}{|B_R^x|}=
\frac{\rho_e^N-\rho_i^N}{R^N} \le \frac{N \, \rho_e^{N-1}}{R^N} \, (\rho_e - \rho_i),
$$
and $\rho_e\le d_\Om$.
Thus, if $d_\Om/R$ remains bounded and $(\rho_e - \rho_i)/R$ tends to $0$, then the ratio $|\Om \De B_R^x|/|B_R^x|$ does it too.
\par
The converse is not true in general. For example, consider a lollipop made by a ball and a stick with fixed length $L$ and vanishing width; as that width vanishes, the ratio $d_\Om/R$ remains bounded, while $|\Om \De B_R^x|/|B_R^x|$ tends to zero and $\rho_e - \rho_i \ge L >0$.

\bigskip

If we fix $r_i$, $r_e$, and $d_\Om$, we have the following result.
}
\end{rem}

\begin{thm}[\cite{MP2}]
\label{thm:comparing asymmetry}
Let $\Om\subset\RR^N$, $N\ge 2$, be a bounded domain satisfying the uniform interior and exterior sphere conditions with radii $r_i$ and $r_e$, and set $\ul{r}=\min{ \left[ r_i, r_e \right] }$.
\par
Then, we have that
$$
\rho_e - \rho_i \le \max{ \left[ 2 \, d_\Om , \frac{d_\Om^2 }{2 \, \ul{r} } \right] }  \, \cA(\Om)^{\frac{1}{N}} .
$$
\end{thm}
\begin{proof}
Let $x$ be any point in $\Om$.
It is clear that 
\begin{equation}
\label{maxcasiincomparing asymmetry}
\max(\rho_e - R, R- \rho_i) \ge \frac{\rho_e - \rho_i}{2}.
\end{equation}
If that maximum is $\rho_e - R$, at a point $y$ where the ball centered at $x$ with radius $\rho_e$ touches $\Ga$, we consider the interior touching ball $B_{r_i}$ . 
\par
If $2 r_i < (\rho_e - \rho_i)/2$ then $B_{r_i} \subset \Om \setminus B_R^x$ and hence \begin{equation*}
\frac{|\Om \De B_R^x|}{|B_R^x|} \ge \left( \frac{r_i}{R} \right)^N.
\end{equation*}
If, else,  $2 r_i \ge (\rho_e - \rho_i)/2$, $B_{r_i}$ contains a ball of radius $(\rho_e - \rho_i)/4$ still touching $\Ga$ at $y$. Such a ball is contained in $\Om \setminus B_R^x$, and hence
\begin{equation*}
\frac{|\Om \De B_R^x|}{|B_R^x|} \ge \left( \frac{\rho_e - \rho_i}{4 \, R} \right)^N.
\end{equation*}
\par
Thus, we proved that
\begin{equation}\label{eq:asymmetrynuovatogliendoepsilon}
\frac{|\Om \De B_R^x|}{|B_R^x|} \ge \min \left\lbrace \left( \frac{r_i}{R} \right)^N , \left( \frac{\rho_e - \rho_i}{4 \, R} \right)^N \right\rbrace .
\end{equation}
\par
If, else, the maximum in \eqref{maxcasiincomparing asymmetry} is $R- \rho_i$, we proceed similarly, by reasoning on the exterior ball $B_{r_e}$ and $\RR^N\setminus\ol{\Om}$, and we obtain \eqref{eq:asymmetrynuovatogliendoepsilon} with $r_i$ replaced by $r_e$.

Thus, by recalling that $\ul{r}=\min{ \left[ r_i, r_e \right] }$, it always holds that
\begin{equation*}
\frac{|\Om \De B_R^x|}{|B_R^x|} \ge \min \left\lbrace \left( \frac{ \ul{r} }{R} \right)^N , \left( \frac{\rho_e - \rho_i}{4 \, R} \right)^N \right\rbrace ,
\end{equation*}
from which, since $x$ was arbitrarily chosen in $\Om$, we deduce that
$$
\rho_e - \rho_i \le 4 \, R \, \cA(\Om)^{\frac{1}{N}} \quad \mbox{ if } \quad \cA(\Om)\le \left(
\frac{ \ul{r} }{R} \right)^N .
$$
On the other hand, if 
$$
\cA(\Om) > \left( \frac{ \ul{r} }{R} \right)^N
$$
it is trivial to check that
$$
\rho_e - \rho_i \le \frac{d_\Om \, R}{ \ul{r} } \, \cA(\Om)^{\frac{1}{N}}
$$
holds.
Thus, we deduce that
$$
\rho_e - \rho_i \le \max{ \left[ 4 \, R , \frac{d_\Om \, R}{ \ul{r} } \right] }  \, \cA(\Om)^{\frac{1}{N}} .
$$
By proceeding as described at the end of the proof of Theorem \ref{thm:Improved-Serrin-stability}, we can easily obtain the estimate $R \leq d_\Om/2$, from which the conclusion easily follows.
\end{proof}

\begin{rem}
{\rm
Theorem \ref{thm:comparing asymmetry} and \eqref{eq:SBTstabinmisura} give the inequality
$$
\rho_e - \rho_i \le C \, \nr H_0 - H \nr_{2,\Ga}^{1/N}
$$
that, for any $N\ge 2$, is poorer than that obtained in Theorem \ref{thm:SBT-improved-stability}.
}
\end{rem}

\section{Other related stability results}\label{sec:Other stability results mean curvature}

If we suppose that $\Ga$ is mean convex, then we can use Theorem \ref{thm:identityheintze-karcher} to obtain a stability result for Heintze-Karcher inequality.
%
%

The following theorem improves (for every $N\ge 4$) the exponents obtained in \cite[Theorem 4.5]{MP}.
\begin{thm}[Stability for Heintze-Karcher's inequality]
\label{thm:stability-hk}
Let $\Ga$ be a surface of class $C^2$, which is the boundary of a bounded domain $\Om\subset\RR^N$, $N\ge 2$. Denote by $H$ its mean curvature and suppose that $H\ge0$ on $\Ga$.

Then, there exist a point $z\in\Om$ and a positive constant $C$ such that
\begin{equation}
\label{stability-hk-C}
\rho_e-\rho_i\le C\,\left(\int_\Ga\frac{dS_x}{H}-N\,|\Om|\right)^{\tau_N / 2} ,
\end{equation}
with the following specifications:
\begin{enumerate}[(i)]
\item $\tau_N=1$ for $N=2$ or $3$;
\item $\tau_4$ is arbitrarily close to one, in the sense that for any $\theta>0$, there exists a positive constant $C$ such that \eqref{stability-hk-C} holds with $\tau_4= 1- \theta $;
\item $\tau_N = 2/(N-2) $ for $N\ge 5$.
\end{enumerate}

The constant $C$ depends on $N$, $r_i$, $r_e$, $d_\Om$, $\de_\Ga (z)$, and $\theta$ (only in the case $N=4$).
\end{thm}

\begin{proof}
As before, we choose $z\in\Om$ be any local minimum point of $u$ in $\Om$.
Moreover, by \eqref{heintze-karcher-identity} and \eqref{heintze-karcher}, we have that 
$$
\frac1{N-1}\,\int_\Om |\na^2 h|^2\,dx\le \int_\Ga\frac{dS_x}{H}-N\,|\Om|.
$$
Thus, the conclusion follows from Theorem \ref{thm:SBT-W22-stability}.
%
%
%
\end{proof}

Since the deficit of Heintze-Karcher's inequality can be written as
$$
\int_\Ga\frac{dS_x}{H}-N\,|\Om| = \int_\Ga \left( \frac{1}{H} - u_\nu \right) \, dS_x ,
$$
where $u$ is the solution of \eqref{serrin1},
Theorem \ref{thm:stability-hk} also gives a stability estimate for the overdetermined boundary value problem mentioned in item (ii) of Theorem \ref{thm:heintze-karcher}, as stated next. It is clear that the deviation $\int_\Ga \left( 1/H - u_\nu \right) \, dS_x$ is weaker than $\nr 1/H - u_\nu \nr_{1, \Ga}$. The following theorem improves \cite[Theorem 4.8]{MP}.
 
\begin{thm}[Stability for a related overdetermined problem]
\label{thm:OBVP-stability}
%
Let $\Ga$, $\Om$, and $H$ be as in Theorem \ref{thm:stability-hk}.
%
%
Let $u$ be the solution of problem \eqref{serrin1} and $z \in \Om$ be any of its critical points.
\par
There exists a positive constant $C$ such that
\begin{equation}
\label{OBVT-stability-gener-C}
\rho_e-\rho_i\le C\, \left\lbrace \int_\Ga \left( \frac{1}{H} - u_\nu \right) \, dS_x \right\rbrace^{ \tau_N / 2 } ,
\end{equation}
%
%
with the following specifications:
\begin{enumerate}[(i)]
\item $\tau_N=1$ for $N=2$ or $3$;
\item $\tau_4$ is arbitrarily close to one, in the sense that for any $\theta>0$, there exists a positive constant $C$ such that \eqref{OBVT-stability-gener-C} holds with $\tau_4= 1- \theta $;
\item $\tau_N = 2/(N-2) $ for $N\ge 5$.
\end{enumerate}

The constant $C$ depends on $N$, $r_i$, $r_e$, $d_\Om$, $\de_\Ga (z)$, and $\theta$ (only in the case $N=4$).
\end{thm}

%

\begin{rem}
{\rm
It is clear that analogs of items (i) and (ii) of Remark \ref{rem:ontheconstants-SBT-removingdelta-seconvexocasidiremark} hold also for
Theorems \ref{thm:stability-hk} and \ref{thm:OBVP-stability}.
%
%
%
}
\end{rem}

\begin{appendices}
\appendixpage
\noappendicestocpagenum
\addappheadtotoc

\chapter{}\label{appendix:estimates-mu-for-L0John}
\section{Hardy-Poincar\'e inequalities for \texorpdfstring{$L_0$}{L0}-John domains}
%
%
In the present appendix, we show how to prove the estimates given in \eqref{eq:estimatemu-r-p-al-generalizingFeappendix} and \eqref{eq:estimatemu-p-0-generalizingFeappendix}.

As already mentioned in the proof of Lemma \ref{lem:John-two-inequalities}, \cite[Theorem 1.3]{HS} states that if $\Om$ is a $b_0$-John domain, then there exists a constant $c=c(N,\, r, \, p,\, \al, \, \Om)$ such that
\begin{equation*}
\inf_{\la \in \RR} \nr v - \la \nr_{r,\Om} \le c \, \nr \de_\Ga^{\al} \, \na v  \nr_{p, \Om} ,
\end{equation*}
for every $v \in L^1_{loc}(\Om)$ such that $\de_\Ga^\al \, \na v  \in L^p(\Om)$.
Moreover, that proof, which exploits the so-called Whitney decomposition, informs us that
\begin{equation*}
c \le k_{N,\, r, \, p,\, \al} \, b_0^N |\Om|^{\frac{1-\al}{N} +\frac{1}{r} +\frac{1}{p} } .
\end{equation*}

In order to prove \eqref{eq:estimatemu-r-p-al-generalizingFeappendix}, we have to look deeper into that proof.
More technically, a careful inspection of the proof of \cite[Theorem 1.3]{HS} will see that also the following result is proved there (even if not explicitly stated).

\begin{thm}[\cite{HS}]\label{thm:HS}
Let $\Om$ be a $L_0$-John domain in $\RR^N$ with base point $z$, and consider three numbers $r, p, \al$ such that $1 \le p \le r \le \frac{Np}{N-p(1 - \al )}$, $p(1 - \al)<N$, $0 \le \al \le 1$. There exists a constant $k= k_{N, \, r, \, p, \, \al}$ such that, for any function $v \in L^1_{loc} (\Om) $ with $\de_\Ga^\al \na u \in L^p (\Om)$, it holds that
\begin{equation*}
\nr v - v_{Q_0} \nr_{r,\Om} \le k \, L_0^N \, |\Om|^{\frac{1-\al}{N} + \frac{1}{r} - \frac{1}{p}} \nr \de_\Ga^{\al} \, \na v  \nr_{p, \Om},
\end{equation*}
where $Q_0$ is any cube centered at $z$ with $d_{Q_0} \le \mathrm{dist}(Q_0, \Ga) \le 4 \, d_{Q_0}$.
\end{thm}

Now we are ready to prove the following lemma, which clearly gives \eqref{eq:estimatemu-r-p-al-generalizingFeappendix}. We mention that the following lemma generalizes the result obtained in \cite[Lemma 8]{Fe} in the particular case $r=2N/(N-1), \, p=2, \, \al=1/2$.

\begin{lem}\label{lem_generalFe}
Let $\Om$ be a $L_0$-John domain in $\RR^N$ with base point $z$, and consider three numbers $r, p, \al$ such that $1 \le p \le r \le \frac{Np}{N-p(1 - \al )}$, $p(1 - \al)<N$, $0 \le \al \le 1$. Then, for every harmonic function $v$ such that $v(z)=0$ it holds that
\begin{equation*}
\nr v \nr_{r,\Om} \le k \, L_0^N \, |\Om|^{\frac{1-\al}{N} + \frac{1}{r} - \frac{1}{p}} \nr \de_\Ga^{\al} \, \na v  \nr_{p, \Om} ,
\end{equation*}
for a constant $k=k_{N, \, r, \, p, \, \al}$
\end{lem}
\begin{proof}
To simplify formulas, in all the inequalities of the proof we will always use the letter $k$ to denote the constants depending only on $N, \, r, \, p, \, \al$.

If we consider 
$$
B_0 = B_{ \frac{d_{Q_0} }{ 2} } (z),
$$
we have that $Q_0 \subset B_0 \subset \Om$.
Since $v$ is harmonic and $v(z)=0$ we have by the mean value property, $v_{B_0}=0$.
The usual (i.e., unweighted) Poincar\'e inequality in $B_0$ informs us that
\begin{multline*}
\nr v \nr_{r, Q_0} \le \nr v \nr_{r, B_0} \le k \, |Q_0|^{\frac{1}{N} + \frac{1}{r} - \frac{1}{p}} \nr \na v \nr_{p, B_0}
\\
\le k \, |Q_0|^{\frac{1}{N} + \frac{1}{r} - \frac{1}{p} - \frac{\al}{N}} \nr \de_\Ga^\al \na v \nr_{p, B_0}
\le  k \, |Q_0|^{\frac{1 -\al}{N} + \frac{1}{r} - \frac{1}{p}} \nr \de_\Ga^\al \na v \nr_{p, \Om}.
\end{multline*}
By using the above estimate and Theorem \ref{thm:HS} we get that
\begin{equation*}
\nr v \nr_{r, \Om} \le \nr v - v_{Q_0} \nr_{r, \Om} + \left( \frac{|\Om|}{|Q_0|} \right)^\frac{1}{r} \nr v \nr_{r, Q_0} \le k \, L_0^N \, |\Om|^{\frac{1-\al}{N} + \frac{1}{r} - \frac{1}{p}} \nr \de_\Ga^{\al} \, \na v  \nr_{p, \Om}.
\end{equation*}
In the last inequality we used two facts: the first is 
\begin{equation}\label{eq:relationvolumeJohn}
\frac{d_\Om^N}{|Q_0|} \le L_0^N ,
\end{equation}
that can be easily verified by using the definition of $L_0$-John domain with base point $z$ with a point $x \in \Om$ such that $|x-z| \ge d_\Om /2$; the second is that, being $L_0 \ge 1$, we have that
$$L_0^{\frac{N-p(1-\al)}{p}} \le L_0^N,$$
since by the assumptions we know that $p \ge 1$ and $p(1- \al) < N$.
\end{proof}

In an analogous way, we can also prove \eqref{eq:estimatemu-p-0-generalizingFeappendix}.
In fact, an inspection of the proof of \cite[Theorem 8.5]{HS1} will see that the following result is proved there (even if not explicitly stated).

\begin{thm}[\cite{HS1}]\label{thm:HS2}
Let $\Om$ be a $L_0$-John domain in $\RR^N$ with base point $z$,
and let $p \in \left[ 1, \infty \right)$. Then, for every function $v\in W^{1,p} (\Om)$, it holds that
\begin{equation*}
\nr v - v_{Q_0} \nr_{p,\Om} \le k_{N,p} \, L_0^{3N(1 + \frac{N}{p})} \, d_\Om \, \nr \na v  \nr_{p, \Om},
\end{equation*}
where $Q_0$ is any cube centered at $z$ with $d_{Q_0} \le \mathrm{dist}(Q_0, \Ga) \le 4 \, d_{Q_0}$.
\end{thm}

Thus, by proceeding as in Lemma \ref{lem_generalFe}, we can now prove \eqref{eq:estimatemu-p-0-generalizingFeappendix}. We do it in the following lemma.

\begin{lem}
Let $\Om$ be a $L_0$-John domain in $\RR^N$ with base point $z$,
and let $p \in \left[ 1, \infty \right)$. Then, for every harmonic function $v$ such that $v(z)=0$ it holds that
\begin{equation*}
\nr v \nr_{p,\Om} \le k_{N,p} \, L_0^{3N(1 + \frac{N}{p})} \, d_\Om \, \nr \de_\Ga^{\al} \, \na v  \nr_{p, \Om}.
\end{equation*}
\end{lem}
\begin{proof}
If we consider
$$
B_0 = B_{\frac{d_{Q_0}}{2}} (z),
$$
we have that $Q_0 \subset B_0 \subset \Om$.
Since $v$ is harmonic and $v(z)=0$ we have by the mean value property, $v_{B_0}=0$.
The usual (i.e., unweighted) Poincar\'e inequality in $B_0$ informs us that
\begin{equation*}
\nr v \nr_{p, Q_0} \le \nr v \nr_{p, B_0} \le k_{N,p} |Q_0|^{\frac{1}{N}} \nr \na v \nr_{p, B_0} \le k_{N,p} |Q_0|^{\frac{1}{N}} \nr \na v \nr_{p, \Om}.
\end{equation*}
By using the above estimate and Theorem \ref{thm:HS2} we get that
\begin{equation*}
\nr v \nr_{p, \Om} \le \nr v - v_{Q_0} \nr_{p, \Om} + \left( \frac{|\Om|}{|Q_0|} \right)^\frac{1}{p} \nr v \nr_{p, Q_0} \le k_{N,p} \, L_0^{3N(1 + \frac{N}{p})} \, d_\Om \, \nr \na v  \nr_{p, \Om}.
\end{equation*}
In the last inequality we used two facts: the first is
$$
\left( \frac{|\Om|}{|Q_0|} \right)^\frac{1}{p} \le k_{N,p} \, L_0^{3N(1 + \frac{N}{p})},
$$
which follows by using \eqref{eq:relationvolumeJohn} together with the inequality 
$$L_0^{\frac{N}{p}} \le L_0^{3N(1 + \frac{N}{p})} $$
that holds trivially since $L_0 \ge1$ and $N/p \le 3N(1 + N / p)$.
The second is the trivial observation that 
$$
\frac{|Q_0|^\frac{1}{N}}{d_\Om}  \le 1,
$$
since $Q_0 \subset \Om$.
\end{proof}

\end{appendices}

%
%

\part*{Other works}
\addcontentsline{toc}{part}{Other works}

\input{ChapterLNCabreTesi}
\input{ChapterMPLittlewoodTesi}

\end{document}

%% file: ChapterLNCabreTesi.tex
\chapter{Stable solutions to some elliptic problems: minimal cones, the Allen-Cahn equation, and blow-up solutions}\label{chap:LNCIME}
\chaptermark{Stable solutions to some elliptic problems}
\centerline{Xavier Cabr\'e\footnote{Universitat Polit\`ecnica de Catalunya, Departament de Matem\`{a}tiques, Diagonal 647, 
08028 Barcelona, Spain
({\tt xavier.cabre@upc.edu})
}\footnote{ ICREA, Pg. Lluis Companys 23, 08010 Barcelona, Spain}
and Giorgio Poggesi\footnote{
Dipartimento di Matematica ed Informatica ``U.~Dini'',
Universit\` a di Firenze, viale Morgagni 67/A, 50134 Firenze, Italy
({\tt giorgio.poggesi@unifi.it})
}
}

%
%


\vspace{1cm}

\begin{quotation}
{\bf Abstract.}
{\small These notes record the lectures for the CIME Summer Course taught by the first author in Cetraro during the week of June 19-23, 2017.
The notes contain the proofs of several results on the classification of stable solutions to some nonlinear elliptic equations. The results are crucial steps within the regularity theory of minimizers to such problems. We focus our attention on three different equations, emphasizing that the techniques and ideas in the three settings are quite similar.

%
%
The first topic is the stability of minimal cones. We prove the minimality of the Simons cone in high dimensions, and we give almost all details in the proof of J.~Simons on the flatness of stable minimal cones in low dimensions.  

%
%
Its semilinear analogue is a conjecture on the Allen-Cahn equation posed by E.~De Giorgi in 1978. This is our second problem, for which we discuss some results, as well as an open problem in high dimensions on the saddle-shaped solution vanishing on the Simons cone.

%
%
The third problem was raised by H.~Brezis around 1996 and concerns the boundedness of stable solutions to reaction-diffusion equations in bounded domains. We present proofs on their regularity in low dimensions and discuss the main open problem in this topic. 

%
%
Moreover, we briefly comment on related results for harmonic maps, free boundary problems, and nonlocal minimal surfaces.}
\end{quotation}


\bigskip

%
%
%
%
%

The abstract
%
%
above gives
an account of the topics treated in these lecture notes.

\section{Minimal cones}
\label{sec:mincones}
In this section we discuss two classical results on the theory of minimal surfaces: Simons flatness result on stable minimal cones in low dimensions and the Bombieri-De Giorgi-Giusti counterexample in high dimensions.
The main purpose of these lecture notes is to present the main ideas and computations leading to these deep results -- and to related ones in subsequent sections. Therefore, to save time for this purpose, we do not consider the most general classes of sets or functions (defined through weak notions), but instead we assume them to be regular enough.

Throughout the notes, for certain results we will refer to three other expositions: the books of Giusti \cite{G}
and of Colding and Minicozzi \cite{ColMinLN}, and the CIME lecture notes of Cozzi and Figalli \cite{CF}.
The notes \cite{CabCapThree} by the first author and Capella have a similar spirit to the current ones and may complement them.

\begin{definition}[Perimeter]
Let $E \subset \RR^n$ be an open set, regular enough. For a given open ball $B_R$ we define the {\it perimeter of} $E$ {\it in} $B_R$ as 
\begin{equation*}
P(E, B_R) := H_{n-1} (\pa E \cap B_R ),
\end{equation*}
where $H_{n-1}$ denotes the $(n-1)$-dimensional Hausdorff measure (see Figure~\ref{fig:1}).
\end{definition}

The interested reader can learn from \cite{G, CF} a more general notion of perimeter (defined by duality or in a weak sense) and the concept of set of finite perimeter.

\begin{definition}[Minimal set]\label{def:minimal set}
We say that an open set (regular enough) $E \subset \RR^n$ is a {\it minimal set} (or a {\it set of minimal perimeter}) if and only if, for every given open ball $B_R$, it holds that
$$P(E, B_R)\leq P(F, B_R)$$
for every open set $F\subset\RR^n$ (regular enough) such that $E \setminus B_R = F\setminus B_R$.
\end{definition}
In other words, $E$ has least perimeter in $B_R$ among all (regular) sets which agree with $E$ outside $B_R$.

\begin{figure}[htbp]
\centering
\includegraphics[scale=.25]{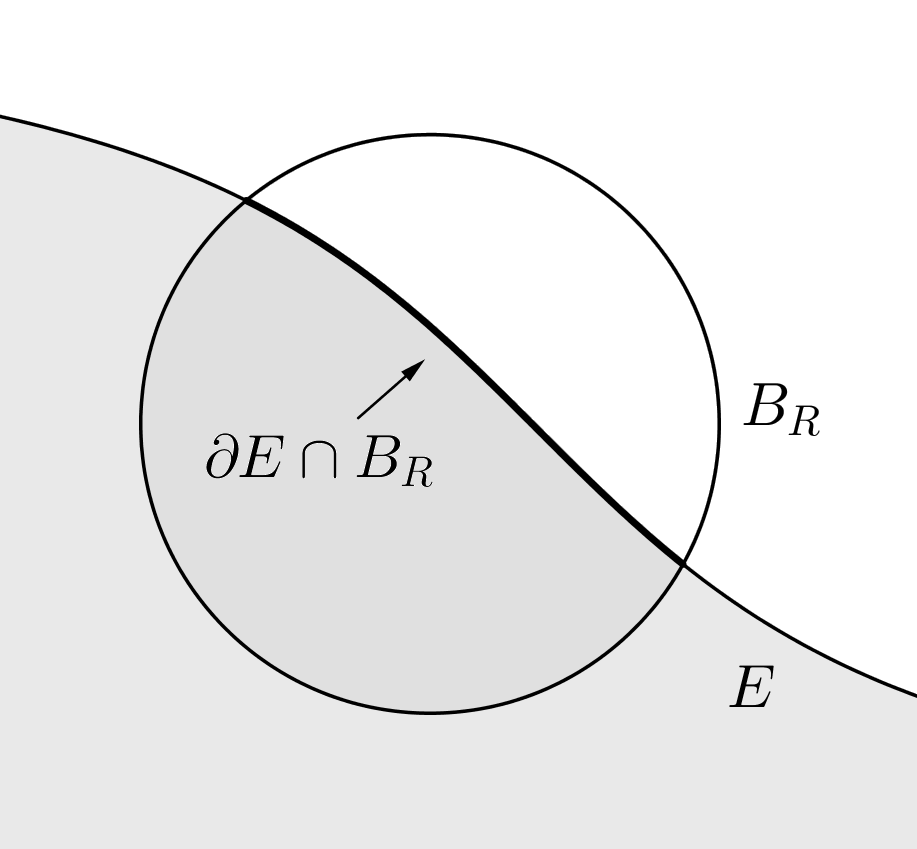}
%
%
\caption{The perimeter of $E$ in $B_R$}
\label{fig:1}       
\end{figure}

To proceed, one considers small perturbations of a given set $E$
and computes the {\it first and second variations of the perimeter functional}. 
To this end,
let $\{\phi_t\}$ be a one-parameter family of maps $\phi_t:\RR^n \to \RR^n$
such that $\phi_0$ is the identity $I$ and all the maps $\phi_t-I$ have compact support (uniformly) contained in $B_R$.

Consider the sets $E_t=\phi_t(E)$. We are interested in the perimeter functional $P(E_t, B_R)$.
One proceeds by choosing $\phi_t=I+t\xi\nu$, which shifts the original set $E$ in the normal direction $\nu$ to its boundary. Here $\nu$ is the outer normal to $E$ and is extended in a neighborhood of $\pa E$ to agree with the gradient of the signed distance function to $\pa E$, as in \cite{G} or in our Subsection \ref{subsec:SimLem} below. On the other hand, $\xi$ is a scalar function with compact support in $B_R$ (see Figure \ref{fig:2}).

\begin{figure}[htbp]
\centering
\includegraphics[scale=.25]{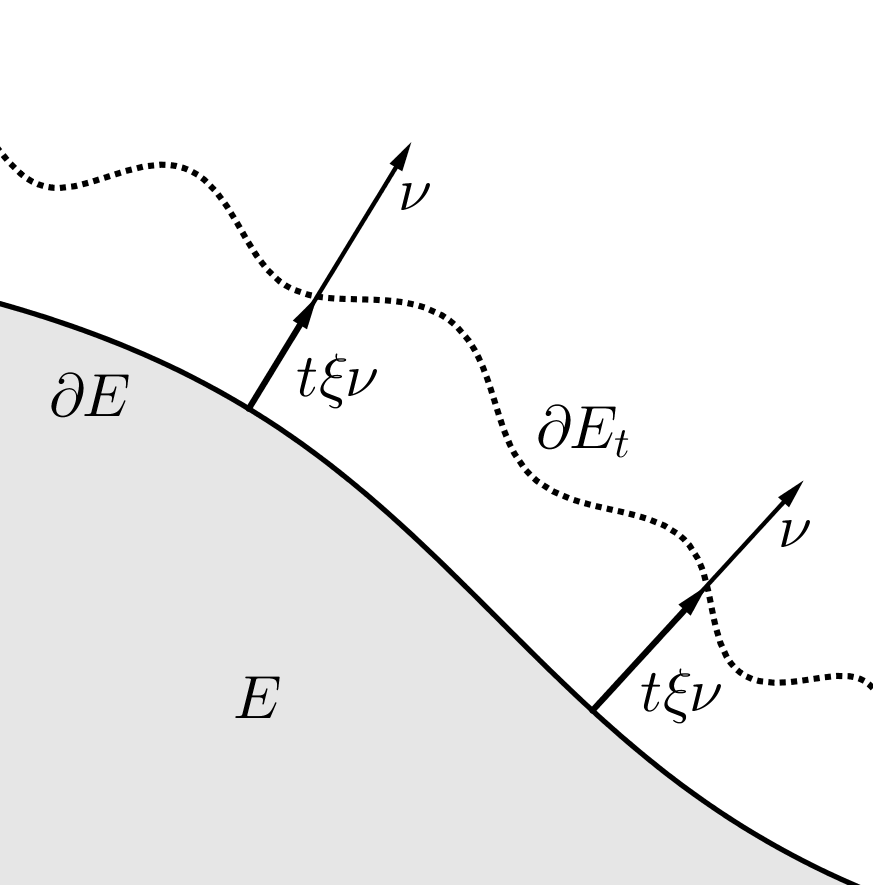}
\caption{A normal deformation $E_t$ of $E$}
\label{fig:2}       
\end{figure}

It can be proved (see chapter 10 of \cite{G}) that the first and second variations of perimeter are given by
\begin{eqnarray}
\left.\frac{d}{dt} P(E_t, B_R) \right|_{t=0} 
&=&\int_{\partial E}{\mathcal H}\xi dH_{n-1},
\label{eq:1-1v}\\
\left.\frac{d^2}{dt^2}  P(E_t, B_R) \right|_{t=0}
&=&\int_{\partial E}\left\{|\delta\xi|^2-(c^2-{\mathcal H}^2)\xi^2
\right\}dH_{n-1},
\label{eq:1-1vBIS}
\end{eqnarray}
where ${\mathcal H}={\mathcal H}(x)$ is the {\it mean curvature} of $\partial E$ at $x$ and $c^2=c^2(x)$ is the sum of the squares 
of the $n-1$ principal
curvatures $k_1, \dots, k_{n-1}$ of $\pa E$ at $x$. More precisely,
$$
\cH(x)= k_1 + \dots + k_{n-1} \quad \mbox{ and } \quad c^2= k_1^2 + \dots + k_{n-1}^2 .
$$
In \eqref{eq:1-1vBIS}, $\delta$ (sometimes denoted by $\na_T$) is the tangential gradient to the surface
$\pa E$, given by
\begin{equation}\label{def:tangentialgradient}
\delta \xi = \na_T \xi = \na \xi - (\na \xi \cdot \nu) \nu
\end{equation}
for any function $\xi$ defined in a neighborhood of $\pa E$. Here $\na$ is the usual Euclidean gradient and $\nu$ is always the normal vector to $\pa E$.
Being $\de$ the tangential gradient, one can check that $\de \xi_{| \pa E}$ depends only on $\xi_{| \pa E}$. It can be therefore computed for functions $\xi: \pa E \to \RR$ defined only on $\pa E$ (and not necessarily in a neighborhood of $\pa E$).


\begin{definition}\label{def:servepertesiLNminimizingminimal}
\begin{enumerate}[(i)]
\item We say that $\pa E$ is a {\it minimal surface} (or a {\it stationary surface}) if the first variation of perimeter vanishes for all balls $B_R$. Equivalently, by \eqref{eq:1-1v}, $\cH = 0$ on $\pa E$.
\item We say that $\pa E$ is a {\it stable minimal surface} if $\cH =0$ and the second variation of perimeter is nonnegative for all balls $B_R$.
\item We say that $\pa E$ is a {\it minimizing minimal surface} if $E$ is a minimal set as in Definition~\ref{def:minimal set}.
\end{enumerate}
\end{definition}
We warn the reader that in some books or articles ``minimal surface'' may mean ``minimizing minimal surface''.

\begin{remark}
\begin{enumerate}[(i)]
\item If $\pa E$ is a minimal surface (i.e., $\cH = 0$), the second variation
of perimeter \eqref{eq:1-1vBIS} becomes
\begin{equation}
\left.\frac{d^2}{dt^2}  P(E_t, B_R) \right|_{t=0} = \int_{\partial E}\left\{|\delta\xi|^2-c^2\xi^2\right\}dH_{n-1} .
\label{eq:minimal-1-2v2}
\end{equation}
\item If $\pa E$ is a minimizing minimal surface, then $\pa E$ is a stable minimal surface. In fact, in this case the function $P(E_t, B_R)$
has a global minimum at $t=0$.
\end{enumerate}
\end{remark}

\subsection{The Simons cone. Minimality}\label{subsec 1.1:MinimalitySimcones}

\begin{definition}[The Simons cone]
The Simons cone $\cSC \subset \RR^{2m}$ is the set
\begin{equation}\label{def:simonscone}
\cSC =\{ x\in\mathbb{R}^{2m}\, : \, x_1^2+\ldots +x_m^2 = x_{m+1}^2
+\ldots +x_{2m}^2 \} .
\end{equation}
In what follows we will also use the following notation:
\begin{equation*}
\cSC =\{ x=(x',x'') \in \RR^m \times \RR^m \, : |x'|^2=|x''|^2 \}.
\end{equation*}
Let us consider the open set
\begin{equation*}
E_S= \left\{x \in \RR^{2m} :  u(x):= |x'|^2 - |x''|^2 < 0 \right\},
\end{equation*}
and notice that $\pa E_S = \cSC$ (see Figure \ref{fig:3}).
\end{definition}

\begin{figure}[htbp]
\centering
\includegraphics[scale=.25]{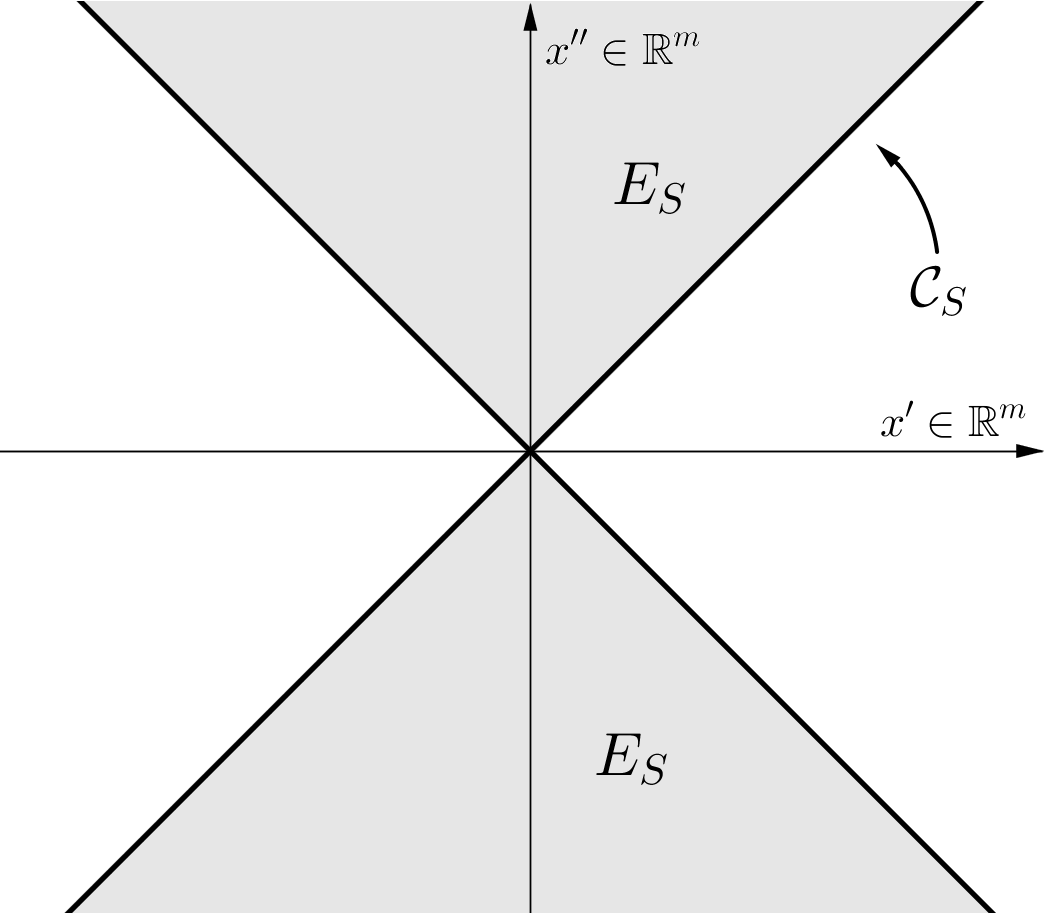}
\caption{The set $E_S$ and the Simons cone $\cSC$}
\label{fig:3}       
\end{figure}

\begin{exercise}
Prove that the Simons cone has zero mean curvature for every integer $m \ge 1$. For this, use the following fact (that you may also try to prove): if
$$
E=\left\{x \in \RR^n :  u(x) < 0 \right\}
$$
for some function $u: \RR^n \to \RR$, then the mean curvature of $\pa E$ is given by
\begin{equation}
\label{mean curvature formula}
{\mathcal H} = \left. \dv \left( \frac{\na u}{| \na u|} \right) \right|_{\pa E}.
\end{equation}
\end{exercise}

\begin{remark}\label{remark:nominimindim2}
It is easy to check that, in $\RR^2$, $\cSC$ is not a minimizing minimal surface. In fact, referring to Figure \ref{fig:4}, the shortest way to go from $P_1$ to $P_2$ is through the straight line. Thus, if we consider as a competitor in $B_R$ the interior of the set
$$F:= \ol{E}_S \cup \ol{T}_1 \cup \ol{T}_2,$$
where $T_1$ is the triangle with vertices $O$, $P_1$, $P_2$, and $T_2$ is the symmetric of $T_1$ with respect to $O$, we have that $F$ has less perimeter in $B_R$ than $E_S$.

\begin{figure}[htbp]
\centering
\includegraphics[scale=.25]{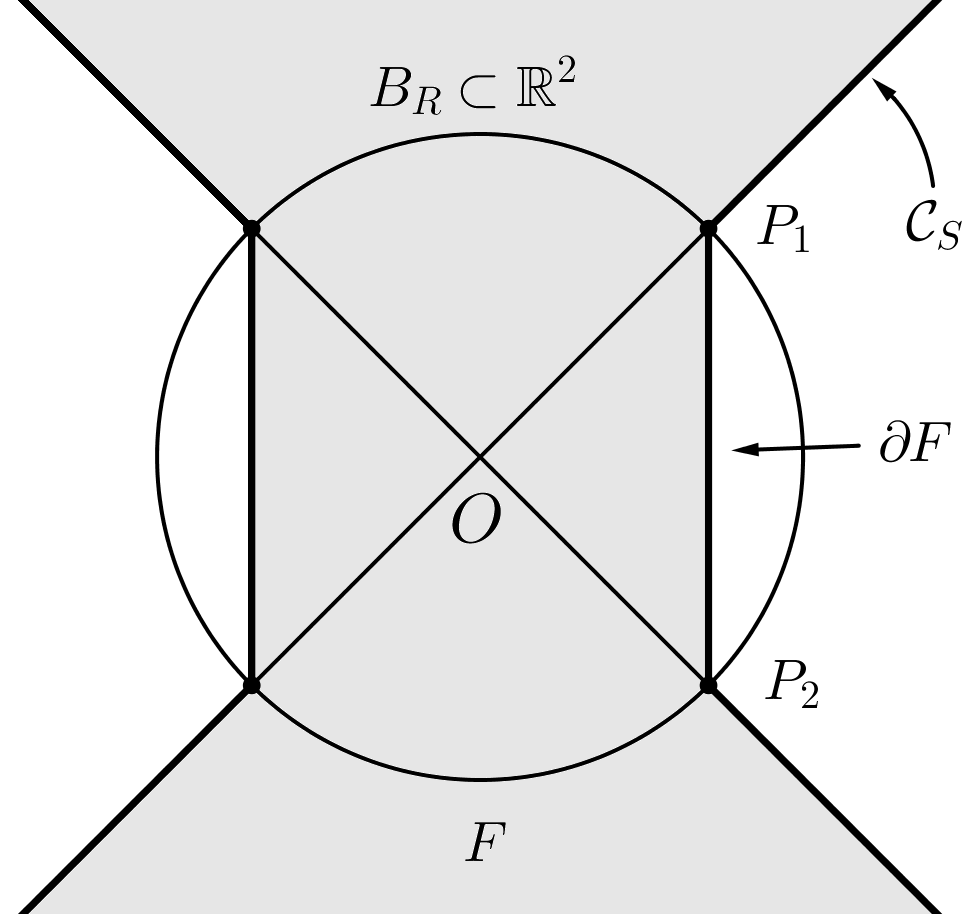}
\caption{The Simons cone $\cSC$ is not a minimizer in $\RR^2$}
\label{fig:4}       
\end{figure}

\end{remark}

In 1969 Bombieri, De Giorgi, and Giusti proved the following result.

\begin{theorem}[Bombieri-De Giorgi-Giusti~\cite{BdGG}]
\label{thm:SimCone} 
If $2m\geq 8$, then $E_S$ is a minimal set in $\mathbb{R}^{2m}$. That is, if  $2m\geq 8$, the Simons cone $\cSC$ is a minimizing minimal surface.
\end{theorem}

The following is a clever proof of Theorem \ref{thm:SimCone} found in 2009 by G.~De Philippis and E.~Paolini (\cite{DePP}). It is based on a {\it calibration argument}.
Let us first define
\begin{equation}\label{def:utilde1}
\tilde{u} = |x'|^4 - |x''|^4  ;
\end{equation}
clearly we have that
\begin{equation*}
E_S= \left\{x \in \RR^{2m} :  \tilde{u}(x) < 0 \right\} \, \mbox{ and } \, \pa E_S = \cSC.
\end{equation*}
Let us also consider the vector field
\begin{equation}\label{def:Xvectorfield}
X= \frac{\na \tilde{u}}{ |\na \tilde{u}| } .
\end{equation}

\begin{exercise}\label{ex:utilde}
Check that if $m \ge 4$, $\dv X$ has the same sign as $\tilde{u}$ in $\RR^{2m}$.
\end{exercise}

\begin{proof}[Proof of Theorem  \ref{thm:SimCone}]
By Exercise \ref{ex:utilde} we know that if $m \ge 4$, $\dv X$ has the same sign as $\tilde{u}$, where $\tilde{u}$ and $X$ are defined in \eqref{def:utilde1} and \eqref{def:Xvectorfield}.
Let $F$ be a competitor for $E_S$ in a ball $B_R$, with $F$ regular enough. We have that
$F \setminus B_R = E_S \setminus B_R$.


Set $\Om := F \setminus E_S $ (see Figure \ref{fig:5}). By using the fact that $\dv X \geq 0$ in $\Om$ and the divergence theorem, we deduce that
\begin{equation}
\label{calideph}
0 \leq \int_{\Om} \dv X \, dx= \int_{\partial E_S \cap \ol{\Om} } X \cdot \nu_\Om \, dH_{n-1} + \int_{\partial F \cap \ol{\Om}} X \cdot \nu_\Om \, dH_{n-1}.
\end{equation}

\begin{figure}[htbp]
\centering
\includegraphics[scale=.25]{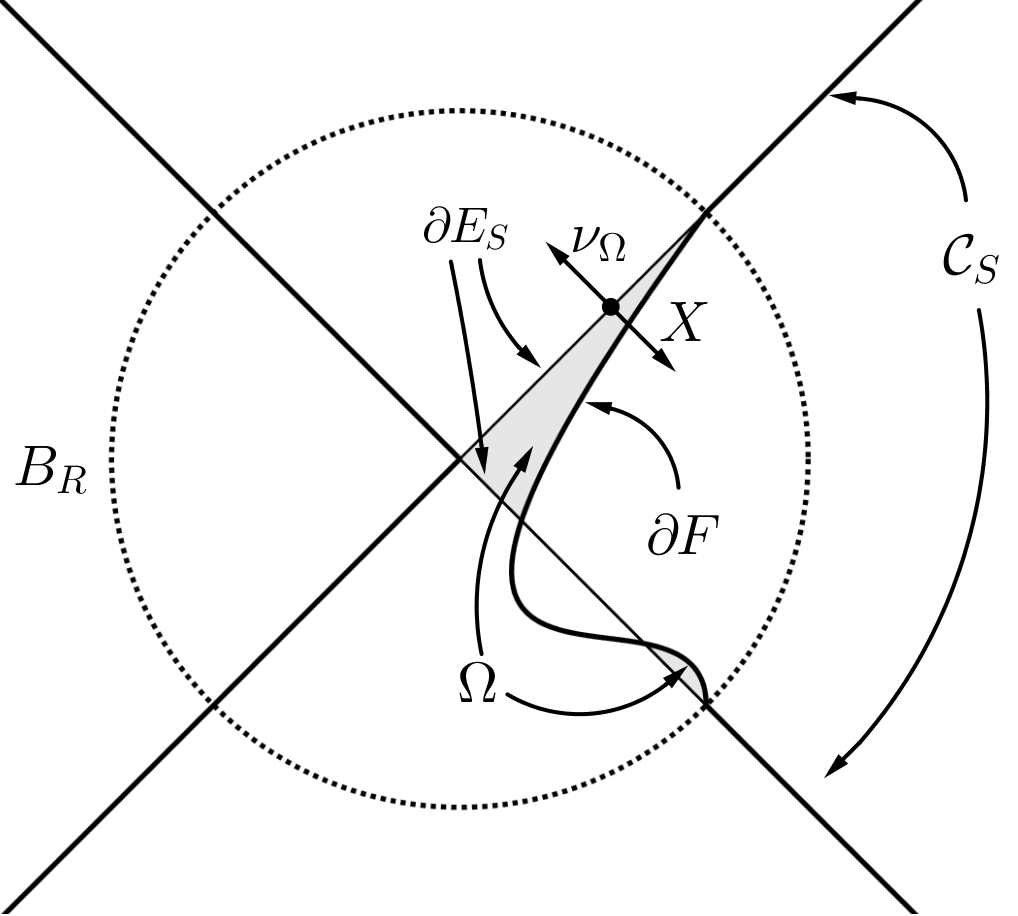}
\caption{A calibration proving that the Simons cone $\cSC$ is minimizing}
\label{fig:5}       
\end{figure}

Since $X= \nu_{E_S}= - \nu_\Om$ on $\partial E_S \cap \ol{\Om}$, and $|X| \leq 1$ (since in fact $|X|= 1$) everywhere (and hence in particular on $\partial F \cap \ol{\Om}$), from \eqref{calideph} we conclude 
\begin{equation}\label{eq:dimcalibration1}
H_{n-1} (\pa E_s \cap \ol{\Om}) \le H_{n-1} (\pa F \cap \ol{\Om}) .
\end{equation}

With the same reasoning it is easy to prove that \eqref{eq:dimcalibration1} holds also for
$\Om := E_S \setminus F$.
Putting both inequalities together, we conclude that $P(E_S,B_R) \leq P(F, B_R)$.

Notice that the proof works for competitors $F$ which are regular enough (since we applied the divergence theorem). However, it can be generalized to very general competitors by using the generalized definition of perimeter, as in \cite[Theorem~1.5]{DePP}.
\end{proof}

Theorem \ref{thm:SimCone} can also be proved with another argument -- but still very much related to the previous one and that also uses the function $\tilde{u} = |x'|^4 - |x''|^4 $. It consists of going to one more dimension $\RR^{2m+1}$ and working with the minimal surface equation for graphs, \eqref{equa:Hgraph} below. This is done in Theorem 16.4 of \cite{G} (see also the proof of Theorem 2.2 in \cite{CabCapThree}).

In the proof above we used a vector field $X$ satisfying the following three properties (with $E= E_S$):
\begin{enumerate}[(i)]
\item $\dv X \geq 0$ in $B_R \setminus E$ and $\dv X \leq 0$ in $E \cap B_R$;
\item $X= \nu_{E}$ on $\partial E \cap B_R$;
\item $|X| \leq 1$ in $B_R$.
\end{enumerate}

\begin{definition}[Calibration]\label{def:calibrationprima}
If $X$ satisfies the three properties above we say that $X$  is a \textit{calibration} for $E$ in $B_R$.
\end{definition}

\begin{exercise}
Use a similar argument to that of our last proof and build a calibration to show that a hyperplane in $\RR^n$ is a minimizing minimal surface.
\end{exercise}

In an appendix, and with the purpose that the reader gets acquainted with another calibration, we present one which solves the isoperimetric problem: balls minimize perimeter among sets of given volume in $\RR^n$.
Note that the first variation (or Euler-Lagrange equation) for this problem is, by Lagrange multipliers,
$\cH = c$, where $c \in \RR$ is a constant.

The following is an alternative proof of Theorem \ref{thm:SimCone}. It uses a {\it foliation argument}, as explained below.
This second proof is probably more transparent (or intuitive) than the previous one and it is used often in minimal surfaces theory, but requires to know the existence of a (regular enough) minimizer (something that was not necessary in the previous proof). This existence result is available and can be proved with tools of the Calculus of Variations (see \cite{CF,G}).

The proof also requires the use of the following important fact. If $\Sigma_1$, $\Sigma_2 \subset B_R$ are two connected hypersurfaces (regular enough), both satisfying $\cH=0$, and such that $\Sigma_1 \cap \Sigma_2 \neq \varnothing$ and $\Sigma_1$ lies on one side of $\Sigma_2$, then $\Sigma_1 \equiv \Sigma_2$ in $B_R$. Lying on one side can be defined as $\Sigma_1 = \pa F_1$, $\Sigma_2 = \pa F_2$, and $F_1 \subset F_2$. The same result holds if $F_1$ satisfies $\cH=0$ and $F_2$ satisfies $\cH \geq 0$.

This result can be proved writing both surfaces as graphs in a neighborhood of a common point $P \in \Sigma_1 \cap \Sigma_2$. The minimal surface equation $\cH=0$ then becomes
\begin{equation}\label{equa:Hgraph}
\dv \left( \frac{\na \fhi_1}{ \sqrt{1+| \na \fhi_1|^2 }  } \right) =0
\end{equation}
for $\fhi_1: \Om \subset \RR^{n-1} \to \RR$ such that $\left( y', \fhi_1 (y') \right) \subset \Om \times \RR $ is a piece of $\Sigma_1$ (after a rotation and translation). Then, assuming that $\fhi_2$ also satisfies \eqref{equa:Hgraph} -- or the appropriate inequality --, one can see that $\fhi_1 - \fhi_2$ is a (super)solution of a second order linear elliptic equation. Since $\fhi_1 - \fhi_2 \geq 0$ (due to the ordering of $\Sigma_1$ and $\Sigma_2$), the strong maximum principle leads to $\fhi_1 - \fhi_2 \equiv 0$ (since $(\fhi_1 -\fhi_2)(0)=0$ at the touching point). See Section 7 of Chapter 1 of \cite{ColMinLN} for more details.

\begin{alternative}[of Theorem \ref{thm:SimCone}]
Note that the hypersurfaces
$$\left\{ x \in \RR^{2m} : \tilde{u} (x) = \la \right\} , $$
with $\la \in \RR$, form a foliation of $\RR^{2m}$, where $\tilde{u}$ is the function defined in \eqref{def:utilde1}.

Let $F$ be a minimizer of the perimeter in $B_R$ among sets that coincide with $E_S$ on $\pa B_R$, and assume that it is regular enough. Since $F$ is a minimizer, in particular $\pa F$ is a solution of the minimal surface equation $\cH=0$.
Since $2m \geq 8$, by \eqref{mean curvature formula} and Exercise \ref{ex:utilde}, the leaves of our foliation $\left\{ x \in \RR^{2m} : \tilde{u} (x) = \la \right\}$ are subsolutions of the same equation for $\la>0$, and supersolutions for $\la<0$ .

If $F \not\equiv E_S$, there will be a first leaf (starting either from $\la= + \infty$ or from $\la = - \infty$) $\left\{ x \in \RR^{2m} : \tilde{u} (x) = \la_* \right\}$, with $\la_* \neq 0$, that touches $\pa F$ at a point in $\ol{B}_R$ that we call $P$ (see Figure \ref{fig:6}).

\begin{figure}[htbp]
\centering
\includegraphics[scale=.25]{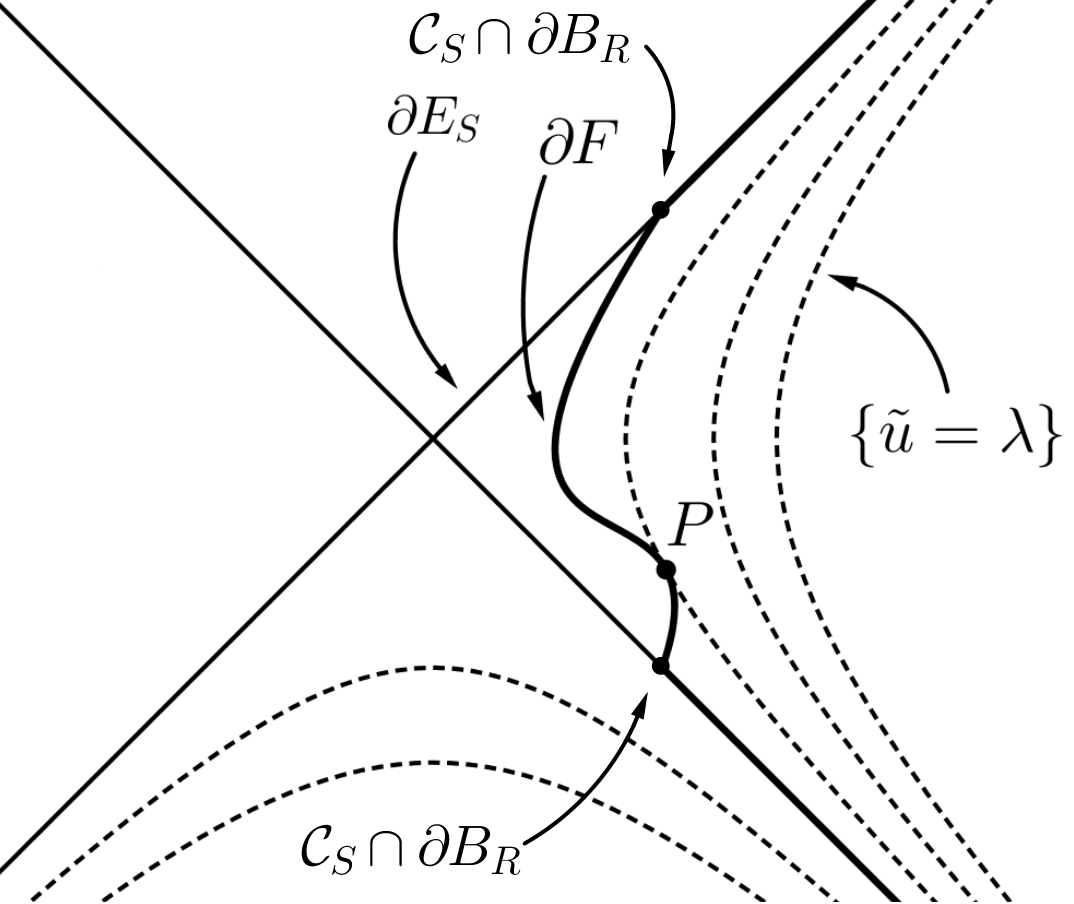}
\caption{The foliation argument to prove that the Simons cone $\cSC$ is minimizing}
\label{fig:6}       
\end{figure}

The point $P$ cannot belong to $\pa B_R$, since it holds that
$$\pa F \cap \pa B_R = \cSC \cap \pa B_R = \left\{ x: \tilde{u} (x) = 0 \right\} \cap \pa B_R  ,$$
and the level sets of $\tilde{u}$ do not intersect each other.
Thus, $P$ must be an interior point of $B_R$. But then we arrive at a contradiction, by the ``strong maximum principle'' argument commented right before this proof, applied with
$\Sigma_1= \pa F$ and $\Sigma_2= \left\{ x \in \RR^{2m} : \tilde{u} (x) = \la_* \right\} $.

As an exercise, write the details to prove the existence of a first leaf touching $\pa F$ at an interior point.

This same foliation argument will be used, in a simpler setting for graphs and the Allen-Cahn equation, in the proof of Theorem \ref{alba} in the next section.
\qed
\end{alternative}

\begin{remark}
The previous foliation argument gives more than the minimality of $\cSC$. It gives {\it uniqueness for the Dirichlet (or Plateau) problem} associated to the minimal surface equation with $\cSC$ as boundary value on $\pa B_R$. 
\end{remark}

\begin{remark}\label{rem:Foliation B-DG-G}
In our alternative proof of Theorem \ref{thm:SimCone} we used a clever foliation made of subsolutions and supersolutions. This sufficed to prove in a simple way Theorem \ref{thm:SimCone}, but required to (luckily) find the auxiliary function $\tilde{u} = |x'|^4 - |x''|^4 $.
Instead, in \cite{BdGG}, Bombieri, De Giorgi, and Giusti considered the foliation made of exact solutions to the minimal surface equation $\cH=0$, when $2 m \ge 8$. To this end, they proceeded as in the following exercise and wrote the minimal surface equation, for surfaces with rotational symmetry in $x'$ and in $x''$, as an ODE in $\RR^2$, finding Equation \eqref{equa:odest} below. They then showed that the solutions of such ODE in the $(s,t)$-plane do not intersect each other (and neither the Simons cone), and thus form a foliation (see Remark \ref{rem:dopohardy} for more information on this).
\end{remark}

\begin{exercise}
Let us set $s= |x'|$ and $t=|x''|$ for $x=(x', x'') \in \RR^m \times \RR^m$.
Check that the following two ODEs are equivalent to the minimal surface equation $\cH=0$ written in the $(s,t)$-variables for surfaces with rotational symmetry in $x'$ and in $x''$.
\begin{enumerate}[(i)]
\item As done in \cite{BdGG}, if we set a parametric representation $ s=s(\tau)$, $t = t(\tau)$, we find
\begin{equation}\label{equa:odest}
s''t'-s't''+(m-1) \left( (s')^2 + (t')^2 \right) \left( \frac{s'}{t} - \frac{t'}{s} \right) = 0 ;
\end{equation}
\item as done in \cite{Da}, if we set $s=e^{z(\te)} \cos (\te)$, $t = e^{z(\te)} \sin (\te)$ we get
$$
z''= \left( 1+(z')^2 \right) \left( (2m -1) - \frac{2(m-1) \cos (2 \te)}{ \sin (2 \te)} \, z' \right) .
$$
\end{enumerate}
The previous ODEs can be found starting from \eqref{mean curvature formula} when $u=u(s,t)$ depends only on $s$ and $t$. Alternatively, they can also be found computing the first variation of the perimeter functional in $\RR^{2m}$ written in the $(s,t)$-variables:
\begin{equation}\label{equa:stenergy}
c \int s^{m-1} t^{m-1} \, dH_1 (s,t),
\end{equation}
for some positive constant $c$, that becomes
%
$$
c \int e^{(2m-1)z(\te)} \, \cos^{m-1} (\te) \, \,  \sin^{m-1} (\te) \, \,  \sqrt{1+ \left( z'(\te) \right)^2 } \, d \te
$$
with the parametrization in point (ii).
\end{exercise}

\begin{remark}
For $n \ge 8$, there exist other minimizing cones, such as some of the {\it Lawson's cones}, defined by
$$
{\mathcal C}_L= \left\{ y= (y', y'') \in \RR^k \times \RR^{n-k} :  |y'|^2= c_{n,k} \, |y''|^2 \right\}
$$
for $k \ge 2$ and $n-k \ge 2$.
For details, see \cite{Da}.
\end{remark}

Notice that if $\pa E$ is a {\it cone} (i.e., $\la \pa E = \pa E$ for every $\la>0$), in the expressions \eqref{eq:1-1v}, \eqref{eq:1-1vBIS}, and \eqref{eq:minimal-1-2v2} we will always consider $\xi$ with compact support outside the origin (thus, not changing the possible singularity of the cone at the origin).

The next theorem was proved by Simons in 1968\footnote{Theorem \ref{thmcone} was proved in 1965 by De Giorgi for $n=3$, in 1966 by Almgren for $n=4$, and finally in 1968 by Simons in any dimension $n \le 7$.} (it is Theorem 10.10 in \cite{G}).
It is a crucial result towards the regularity theory of minimizing minimal surfaces.

\begin{theorem}[Simons~\cite{S}]\label{thmcone}
Let $E \subset \RR^n$ be an open set such that $\pa E$ is a stable minimal cone and 
$\partial E\setminus\{0\}$ is regular. 
Thus, we are assuming $\cH=0$ and
\begin{equation}
\int_{\partial E}\left\{|\delta\xi|^2-c^2\xi^2\right\}dH_{n-1}\geq 0
\label{eq:1-2v2}
\end{equation}
for every $\xi\in C^1(\partial E)$ with 
compact support outside the origin.

If $3 \leq n\leq 7$, then $\partial E$ is a hyperplane.
\end{theorem}

\begin{remark}\label{remarkchemiserve:1.17}
Simons result (Theorem~\ref{thmcone}), together with a blow-up argument and a monotonicity formula (as main tools), lead to the flatness of every minimizing minimal surface in all of $\RR^n$ if $n \le 7$ (see \cite[Theorem 17.3]{G} for a proof). The same tools also give the analyticity of every minimal surface that is minimizing in a given ball of $\RR^n$ if $n \le 7$ (see \cite[Theorem 10.11]{G} for a detailed proof). See also \cite{CF} for a great shorter exposition of these results.
\end{remark}

The dimension $7$ in Theorem \ref{thmcone} is optimal, since by Theorem \ref{thm:SimCone} the Simons cone provides a counterexample in dimension $8$.

The following is a very rough explanation of why the minimizer of the Dirichlet (or Plateau) problem is the Simons cone (and thus passes through the origin) in high dimensions -- in opposition with low dimensions, as in Figure \ref{fig:4}, where the minimizer stays away from the origin. In the perimeter functional written in the $(s,t)$-variables \eqref{equa:stenergy}, the Jacobian $s^{m-1} t^{m-1}$ becomes smaller and smaller near the origin as $m$ gets larger. Thus, lengths near $(s,t)=(0,0)$ become smaller as the dimension $m$ increases.

In order to prove Theorem \ref{thmcone}, we start with some important preliminaries. 
Recalling \eqref{def:tangentialgradient}, for $i=1 \dots ,n$, we define the tangential derivative
$$
\de_i \xi := \pa_i \xi - \nu^i \, \nu^k \xi_k ,
$$
where $\nu= \nu_E = (\nu^1, \dots, \nu^n): \pa E \to \RR^n$ is the exterior normal to $E$ on $\pa E$,
$\pa_i \xi = \pa_{x_i} \xi = \xi_i$ are Euclidean partial derivatives, and we used the standard convention of sum $\sum_{k=1}^{n}$ over repeated indices.
As mentioned right after definition \eqref{def:tangentialgradient}, even if to compute $\pa_i \xi$ requires to extend $\xi$ to a neighborhood of $\pa E$, $\de_i \xi$ is well defined knowing $\xi$ only on
$\pa E$ -- since it is a tangential derivative. Note also that we have $n$ tangential derivatives $\de_1, \dots, \de_n$ and, thus, they are linearly dependent, since $\pa E$ is $(n-1)$-dimensional. However, it is easy to check (as an exercise) that
\begin{equation*}
|\de \xi|^2= \sum_{i=1}^n |\de_i \xi|^2 .
\end{equation*}

We next define {\it the Laplace-Beltrami operator} on $\pa E$ by
\begin{equation}\label{def:Laplace-Beltrami}
\DLB \xi := \sum_{i=1}^n \delta_i\delta_i \xi ,
\end{equation}
acting on functions $\xi: \pa E \to \RR$.
For the reader knowing Riemannian calculus, one can check that
$$
\DLB \xi = {\dv}_T (\na_T \xi) = {\dv}_T (\de \xi) , 
$$
where $\na_T=\de$ is the tangential gradient introduced in \eqref{def:tangentialgradient} and $\dv_T$ denotes the (tangential) divergence on the manifold $\pa E$.

According to \eqref{mean curvature formula}, we have that
$$
\cH = {\dv}_T \nu = \sum_{i=1}^n \de_i \nu^i.
$$ 
We will also use the following {\it formula of integration by parts}:
\begin{equation}\label{eq:Giustitypo}
\int_{\pa E} \de_i \phi \, dH_{n-1} = \int_{\pa E} \cH \phi \nu^i dH_{n-1}                                 
\end{equation}
for every (smooth) hypersurface $\pa E$ and $\phi \in C^1 (\pa E)$ with compact support.
Equation \eqref{eq:Giustitypo} is proved in Giusti's book \cite[Lemma 10.8]{G}. However, there are
%
%
two typos in \cite[Lemma 10.8]{G}: $\cH$ is missed in the identity above, and there is an error of a sign in the proof of \cite[Lemma 10.8]{G}.


Replacing $\phi$ by $\phi \fhi$ in \eqref{eq:Giustitypo}, we deduce that 
\begin{equation}\label{equa:intparts2}
\int_{\pa E} \phi \, \de_i \fhi \, dH_{n-1} = - \int_{\pa E} (\de_i \phi) \fhi \, dH_{n-1}  + \int_{\pa E} \cH \phi \fhi \nu^i dH_{n-1} .
\end{equation}
From this, replacing $\phi$ by $\de_i \phi$ in \eqref{equa:intparts2} and using that $\sum_{i=1}^n \nu^i \de_i \phi = \nu \cdot \de \phi=0$, we also have
\begin{equation}\label{equa:intLB}
\int_{\pa E} \de \phi \cdot \de \fhi \, dH_{n-1} = \sum_{i=1}^n \int_{\pa E} \de_i \phi \, \de_i \fhi \, dH_{n-1} = - \int_{\pa E} (\DLB \phi) \fhi \, dH_{n-1} .
\end{equation}
 
\begin{remark}\label{rem:Jacobi operator}
For a minimal surface $\partial E$, the second variation of perimeter given by \eqref{eq:minimal-1-2v2} can also be rewritten, after \eqref{equa:intLB}, as 
$$
\int_{\partial E} \left\{ -  \DLB \xi -c^2 \xi \right\} \xi \, dH_{n-1}.
$$
The operator $- \DLB -c^2$ appearing in this expression is called {\it the Jacobi operator}. It is the linearization at the minimal surface $\pa E$ of the minimal surface equation $\cH =0$. 
\end{remark}

Towards the proof of Simons theorem, let us now take $\xi=\tilde{c}\eta$ in 
\eqref{eq:1-2v2}, where $\tilde{c}$ and $\eta$ are still arbitrary ($\eta$ with 
compact support outside the origin) and will be chosen later.
We obtain
\begin{eqnarray*}
0 &\le& \int_{\partial E} \left\{ |\delta\xi|^2-c^2 \xi^2 \right\} dH_{n-1} \\
&=& \int_{\partial E} \left\{ \tilde{c}^2 |\delta \eta|^2 + \eta^2 | \delta \tilde{c} |^2
+ \tilde{c} \de \tilde{c} \cdot \de \eta^2 - c^2 \tilde{c}^2 \eta^2 \right\} dH_{n-1} \\
&=&  \int_{\partial E} \left\{  \tilde{c}^2|\delta\eta|^2 - ( \DLB \tilde{c} +c^2 \tilde{c} ) \tilde{c} \eta^2 \right\} dH_{n-1} ,
\end{eqnarray*}
where at the last step we used integration by parts \eqref{equa:intparts2}.
This leads to the inequality
\begin{equation*}
\int_{\partial E}  \left\{ \DLB \tilde{c} +c^2 \tilde{c} \right\} \tilde{c} \eta^2  dH_{n-1} \le  \int_{\partial E} \tilde{c}^2|\delta\eta|^2 dH_{n-1} ,
\end{equation*}
where the term $ \DLB \tilde{c} +c^2 \tilde{c}$ appearing in the first integral is the linearized or Jacobi operator at $\pa E$ acting on $\tilde{c}$.

Now we make the choice $\tilde{c} = c$ and we arrive, as a consequence of stability, to
\begin{equation}\label{eq:Disuguaglianza prima di Teo Simons}
\int_{\partial E} \left\{  \frac{1}{2} \DLB c^2 -  | \delta c|^2 +c^4 \right\} \eta^2  dH_{n-1} \le  \int_{\partial E} c^2|\delta\eta|^2 dH_{n-1} .
\end{equation}

At this point, Simons proof of Theorem \ref{thmcone} uses the following inequality for the Laplace-Beltrami operator $\De_{LB}$ of $c^2$
(recall that $c^2$ is the sum of the squares of the principal curvatures 
of $\partial E$), in the case
when $\partial E$ is a stationary cone.

\begin{lemma}[Simons lemma~\cite{S}]
\label{lmsimons}
Let $E\subset\RR^n$ be an open set such that $\partial E$
is a cone with zero mean curvature and $\partial E\setminus\{0\}$ is regular.
Then, $c^2$ is homogeneous of degree $-2$ and, in $\partial E\setminus\{0\}$, we have
\[
\frac{1}{2} \DLB c^2 - |\delta c|^2 + c^4  \geq \frac{2c^2}{|x|^2}.
\]
\end{lemma}

In Subsection \ref{subsec:SimLem} we will give an outline of the proof of this result.
We now use Lemma \ref{lmsimons} to complete the proof of Theorem \ref{thmcone}.

\begin{proof}[Proof of Theorem \ref{thmcone}]
By using \eqref{eq:Disuguaglianza prima di Teo Simons} together with Lemma~\ref{lmsimons} we obtain
\begin{equation}
0\leq\int_{\partial
E}c^2\left\{|\delta\eta|^2-\frac{2\eta^2}{|x|^2}\right\}dH_{n-1}
\label{eq:1-mcstb}
\end{equation}
for every $\eta\in C^1(\partial E)$ with compact 
support outside the origin.
By approximation, the same holds for $\eta$ Lipschitz instead of $C^1$.


If $r=|x|$, we now choose $\eta$ to be the Lipschitz function
$$
\eta=
\bigg \{
\begin{array}{rl}
r^{- \al} \quad \text{if } r \le 1 \ \\
r^{- \be} \quad \text{if } r \ge 1 . \\
\end{array}
$$
By directly computing
\begin{equation}\label{equaz:provaetadelta}
|\de \eta|^2=
\bigg \{
\begin{array}{rl}
\al^2 r^{- 2 \al -2} \quad \text{if } r \le 1 \ \\
\be^2 r^{- 2 \be -2} \quad \text{if } r \ge 1 , \\
\end{array}
\end{equation}
we realize that if 
\begin{equation}\label{equaz:provaalfabeta}
\al^2<2 \, \mbox{ and } \, \be^2<2 ,
\end{equation}
then in \eqref{eq:1-mcstb} we have $|\de \eta|^2 - 2 \eta^2 / r^2 <0$.
If $\eta$ were an admissible function in \eqref{eq:1-mcstb}, we would then conclude that $c^2 \equiv 0$ on
$\pa E$. This is equivalent to $\pa E$ being an hyperplane.

Now, for $\eta$ to have compact support and hence be admissible, we need to cut-off $\eta$ near $0$ and infinity. As an exercise, one can check that the cut-offs work (i.e., the tails in the integrals tend to zero) if (and only if)
\begin{equation}\label{equazion:cutoffex}
\int_{\partial E}c^2 | \de \eta|^2  dH_{n-1}<\infty,
\end{equation}
or equivalently, since they have the same homogeneity,
$$
\int_{\partial E}c^2\frac{\eta^2}{|x|^2} \, dH_{n-1}<\infty.
$$
By recalling that the Jacobian on $\pa E$ (in spherical coordinates) is $r^{(n-1)-1}$, \eqref{equaz:provaetadelta}, and that, by Lemma~\ref{lmsimons}, $c^2$ is homogeneous of degree $-2$, we deduce that
%
%
\eqref{equazion:cutoffex} is satisfied if $n-6- 2\alpha  >-1$ and $n-6 - 2 \be <-1$. That is, if
\begin{equation}
\alpha < \frac{n -5 }{2} \quad\textrm{ and }\quad  \frac{n-5}{2} < \beta .
\label{eq:1-ineqs}
\end{equation}
%
%
If $3\le n\le 7$ then $(n -5 )^2/4<2$, i.e., $ - \sqrt{2} < (n -5) / 2 < \sqrt{2}$, and thus we can choose $\alpha$ and $\beta$ satisfying
\eqref{eq:1-ineqs} and \eqref{equaz:provaalfabeta}. It then follows that $c^2\equiv 0$, and
hence $\partial E$ is a hyperplane.
\end{proof}

The argument in the previous proof (leading to the dimension $n\le 7$) is very much related to a well known result: Hardy's inequality in $\RR^n$ -- which is presented next.

\subsection{Hardy's inequality}
As already noticed in Remark \ref{rem:Jacobi operator}, for a minimal surface $\partial E$ the second variation of perimeter \eqref{eq:minimal-1-2v2} can also be rewritten, by integrating by parts, as 
$$
\int_{\partial E} \left\{ -  \DLB \xi -c^2 \xi \right\} \xi dH_{n-1} .
$$
This involves the linearized or Jacobi operator $- \DLB -c^2$. If $x= |x| \si= r \si$, with
$\si \in S^{n-1}$, then $c^2= d(\si) / |x|^2$ (if $\pa E$ is a cone and thus $c^2$ is homogeneous of degree $-2$), where $d(\si)$ depends only on the angles $\si$. Thus, we are in the presence of the ``Hardy-type operator''
$$- \DLB - \frac{d(\si)}{|x|^2} ;$$
notice that $\DLB$ and $d(\si) / |x|^2 $ scale in the same way.
Thus, for all admissible functions $\xi$,
$$0 \le \int_{\partial E} \left\{ |\de \xi|^2 - \frac{d(\si)}{|x|^2} \, \xi^2 \right\} dH_{n-1} , \quad \mbox{if $\partial E$ is a stable minimal cone.}$$

Let us analyze the simplest case when $\pa E = \RR^n$ and $d \equiv constant$. Then, the validity or not of the previous inequality is given by Hardy's inequality, stated and proved next.

\begin{proposition}[Hardy's inequality]\label{Prop:Hardy inequality}
If $n \ge 3$ and $\xi \in C_{c}^{1}(\RR^n \setminus \lbrace 0 \rbrace)$, then
\begin{equation}\label{Hardyin}
\frac{(n-2)^{2}}{4}\int _{\RR^n}\frac{\xi ^{2}}{|x|^{2}} \,dx
\leq \int _{\RR^n}\left|\nabla \xi \right|^{2}dx .
\end{equation}
In addition, $(n-2)^{2} / 4$ is the best constant in this inequality and it is not achieved by any $0 \not\equiv \xi \in H^1 (\RR^n)$.

Moreover, if $a> (n-2)^{2} / 4$, then the Dirichlet spectrum of $- \DLB - a / |x|^2$  in the unit ball $B_1$ goes all the way to $- \infty$. That is,
\begin{equation}\label{equa:quotient}
\inf \frac{ \int_{B_1} \lbrace |\na \xi|^2 - a \, \frac{\xi^2}{|x|^2} \rbrace dx }{\int_{B_1} |\xi|^2 dx} = - \infty ,
\end{equation}
where the infimum is taken over $0 \not\equiv \xi \in H_0^1 (B_1)$.
\end{proposition}
\begin{proof}
Using spherical coordinates, for a given $\si \in S^{n-1}$ we can write
\begin{equation}\label{har1}
\int_0^{+\infty} r^{n-1} r^{-2} \xi^2(r \si) \, dr = - \frac{1}{n-2} \int_0^{+\infty} r^{n-2} 2 \xi (r \si) \xi_r (r \si) \, dr .
\end{equation}
Here we integrated by parts, using that $r^{n-3}= \left( r^{n-2} / (n-2) \right)'$.

Now we apply the Cauchy-Schwarz inequality in the right-hand side to obtain
\begin{multline}\label{har2}
- \int_0^{+\infty} r^{n-2} \xi \xi_r \, r^{\frac{n-3}{2} } \, r^{- \frac{n-3}{2} }  \, dr \le
\\
\left( \int_0^{+\infty} r^{n-3} \xi^2 \, dr\right)^{\frac{1}{2}} \, \left(\int_0^{+\infty} r^{n-1} \xi^2_r \, dr \right)^{\frac{1}{2}}.
\end{multline}
Putting together \eqref{har1} and \eqref{har2} we get
$$\int_0^{+\infty} r^{n-3} \xi^2 \, dr \le \frac{2}{n-2} \left( \int_0^{+\infty} r^{n-3} \xi^2 \, dr\right)^{\frac{1}{2}} \, \left(\int_0^{+\infty} r^{n-1} \xi^2_r \, dr \right)^{\frac{1}{2}},$$
that is,
$$\frac{(n-2)^2}{4} \int_0^{+\infty} r^{n-1} \, \frac{\xi^2 }{r^2} \, dr \le \int_0^{+\infty} r^{n-1} \xi^2_r \, dr . $$
By integrating in $\si$ we conclude \eqref{Hardyin}.
An inspection of the equality cases in the previous proof shows that the best constant is not achieved.

Let us now consider $(n-2)^2 / 4 < \al^2 < a$ with $\al \searrow (n-2) / 2$. Take
$$
\xi = r^{- \al} -1  
$$
and cut it off near the origin to be admissible. If we consider the main terms in the quotient \eqref{equa:quotient}, we get
$$
\frac{ \int (\al^2 - a) r^{-2 \al -2} dx }{ \int r^{-2 \al } dx } .
$$
Thus it is clear that, as $\al \searrow (n-2) / 2$, the denominator remains finite independently of the cut-off, while the numerator is as negative as we want after the cut-off. Hence, the quotient tends to $- \infty$.
\end{proof}

\begin{remark}\label{rem:dopohardy}
As we explained in Remark \ref{rem:Foliation B-DG-G}, in \cite{BdGG}, Bombieri, De Giorgi, and Giusti used a foliation made of exact solutions to the minimal surface equation $\cH=0$ when $2 m \ge 8$.
These are the solutions of the ODE \eqref{equa:odest} starting from points $\left( s(0),t(0) \right)= \left( s_0,0 \right) $ in the $s$-axis and with vertical derivative $\left( s'(0), t'(0) \right) = \left( 0,1 \right) $.
They showed that, for $2m \ge 8$, they do not intersect each other, neither intersect the Simons cone $\cSC$. Instead, in dimensions $4$ and $6$ they do not produce a foliation and, in fact, each of them crosses infinitely many times $\cSC$, as showed in Figure \ref{fig:add2}. This reflects the fact that the linearized operator  $- \DLB -c^2$ on $\cSC$ has infinitely many negative eigenvalues, as in the simpler situation of Hardy's inequality in the last statement of Proposition \ref{Prop:Hardy inequality}.
%
\end{remark}

\begin{figure}[htbp]
\centering
\includegraphics[scale=.25]{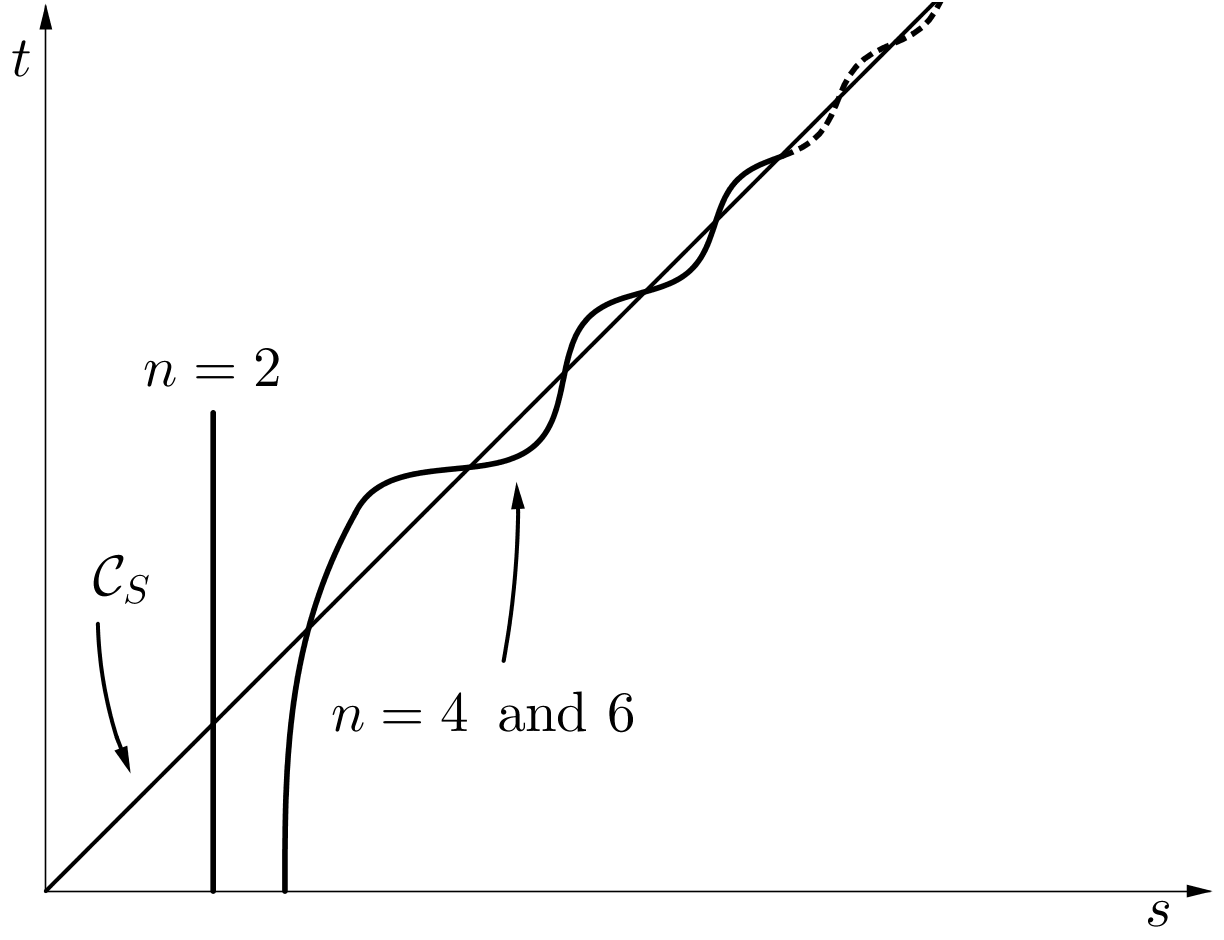}
\caption{Behaviour of the solutions to $\cH = 0$ in dimensions $2$, $4$, and $6$}
\label{fig:add2}       
\end{figure}

\subsection{Proof of the Simons lemma}\label{subsec:SimLem}
As promised, in this section we present the proof of Lemma \ref{lmsimons} with almost all details. 
We follow the proof contained in Giusti's book \cite{G}, where more details can be found (Simons lemma is Lemma 10.9 in \cite{G}). We point out that in the proof of \cite{G} there are the following two typos:
\begin{itemize}
\item as already noticed before, the identity in the statement of \cite[Lemma 10.8]{G} is missing $\cH$ in the second integrand. We wrote the corrected identity in equation \eqref{eq:Giustitypo} of these notes;
\item the label (10.18) is missing in line -8, page 122 of \cite{G}.
\end{itemize}

Alternative proofs of Lemma \ref{lmsimons} using intrinsic Riemaniann tensors can be found in the original paper of Simons \cite{S} from 1968 and also in the book of Colding and Minicozzi \cite{ColMinLN}.

\begin{notations}
We denote by $d(x)$ the signed distance function to $\pa E$, defined by
$$d(x) :=
\bigg \{
\begin{array}{rl}
\operatorname{dist} (x, \pa E) , & \ \ x \in \RR^n \setminus E , \\
- \operatorname{dist} (x, \pa E) , & \ \ x \in E . \\
\end{array}
$$
As we are assuming $E \setminus \{ 0 \}$ to be regular, we have that $d(x)$ is $C^2$ in a neighborhood of $\pa E \setminus \{ 0 \} $.

The normal vector to $\pa E$ is given by 
$$\nu= \na d= \frac{ \na d}{| \na d |};$$
we write
$$\nu= (\nu^1, \dots, \nu^n) = (d_1, \dots, d_n) ,$$
where we adopt the abbreviated notation
$$w_i=w_{x_i} = \pa_i w \, \mbox{ and } \, w_{ij}=w_{x_i x_j} = \pa_{ij} w $$
for partial derivatives in $\RR^n$.
As introduced after Theorem \ref{thmcone}, we will use the tangential derivatives
\begin{equation*}
\de_i := \pa_i - \nu^i \nu^k \pa_k
\end{equation*}
for $i=1, \dots, n$, and thus
$$
\de_i w = w_i - \nu^i \nu^k w_k,
$$
where we adopted the summation convention over repeated indices.
Finally, recall the Laplace-Beltrami operator defined in \eqref{def:Laplace-Beltrami}:
\begin{equation*}
\DLB := \de_i \de_i. 
\end{equation*}

\end{notations}


\begin{remark}
Since
\begin{equation}\label{equazione:nu2}
1= | \nu |^2 = \sum_{k=1}^n d_k^2 ,
\end{equation}
it holds that
$$
d_{jk} d_k =0 \, \mbox{ for } j=1, \dots,n.
$$
Thus, we have
$$
\de_i \nu^j= \de_i d_j = d_{ij} - d_i d_k d_{kj}= d_{ij}=d_{ji},
$$
which leads to
\begin{equation*}
\de_i \nu^j= \de_j \nu^i .
\end{equation*}
\end{remark}

\begin{exercise}
Using $\de_i \nu^j = d_{ij}$, verify that
\begin{equation*}
\cH= \de_i \nu^i ,
\end{equation*}
\begin{equation}
\label{eq:def c indici}
c^2= \de_i \nu^j \de_j \nu^i = \sum_{i,j=1}^n (\de_i \nu^j)^2 .
\end{equation}
\end{exercise}

The identities
\begin{equation*}
\nu^i \de_i = 0 ,
\end{equation*}
\begin{equation}
\label{eq:utileperSimonslemmaproof}
\nu^i \de_j \nu^i = 0 , \, \mbox{ for } j=1,\dots,n ,
\end{equation}
will be used often in the following computations. The first one follows from the definition of $\de_i$, while the second is immediate from \eqref{equazione:nu2}.

The next lemma will be useful in what follows.
\begin{lemma}
The following equations hold for every smooth hypersurface $\pa E$:
\begin{equation}
\label{eq:Giusti 1}
\de_i \de_j = \de_j \de_i + (\nu^i \de_j \nu^k - \nu^j \de_i \nu^k)\de_k ,
\end{equation}
\begin{equation}
\label{eq:Giusti 2}
\DLB \nu^j + c^2 \nu^j = \de_j \cH  \, (=0 \, \mbox{ if $\pa E$ is stationary}) ,
\end{equation}
for all indices $i$ and $j$.
\end{lemma} 
For a proof of this lemma, see \cite[Lemma 10.7]{G}.

Equation \eqref{eq:Giusti 2} is an important one. It says that the normal vector $\nu$ to a minimal surface solves the Jacobi equation $\left( \DLB + c^2 \right) \nu \equiv 0$ on $\pa E$. This reflects the invariance of the equation $\cH = 0$ by translations (to see this, write a perturbation made by a small translation as a normal deformation, as in Figure \ref{fig:2}).

If $\pa E$ is stationary, from \eqref{eq:Giusti 1} and by means of simple calculations, one obtains that
\begin{equation}
\label{eq:Giusti 3}
\DLB \de_k = \de_k \DLB - 2 \nu^k (\de_i \nu^j) \de_i \de_j - 2 (\de_k \nu^j)(\de_j \nu^i) \de_i .
\end{equation}
Equation \eqref{eq:Giusti 3} is the formula with the missed label (10.18) in \cite{G}.


We are ready now to give the

\begin{sketch}[of Lemma \ref{lmsimons}]
By \eqref{eq:def c indici} we can write that
$$
\frac{1}{2} \DLB c^2 = (\de_i \nu^j) \DLB \de_i \nu^j + \sum_{i,j,k} (\de_k \de_i \nu^j)^2.
$$
Then, using \eqref{eq:Giusti 2}, \eqref{eq:Giusti 3}, and the fact $\cH=0$, we have
$$
\frac{1}{2} \DLB c^2 = - (\de_i \nu^j) \de_i (c^2 \nu^j) - 2 (\de_i \nu^j) (\de_k \nu^l) (\de_l \nu^j) (\de_i \nu^k) + \sum_{i,j,k} (\de_k \de_i \nu^j)^2 ,
$$
and by \eqref{eq:Giusti 1}
$$
\frac{1}{2} \DLB c^2  = -c^4 - 2 \nu^i \nu^l (\de_j \de_l \nu^k) (\de_k \de_i \nu^j) + \sum_{i,j,k} (\de_k \de_i \nu^j)^2.
$$
Now, if $x_0 \in \pa E \setminus \{ 0 \}$, we can choose the $x_n$-axis to be the same direction as $\nu(x_0)$. Thus, $\nu(x_0)=(0,\dots,0,1)$ and at $x_0$ we have
$$
\nu^n=1 , \, \de_n=0 ,
$$
$$
\nu^{\al}=0 , \, \de_\al= \pa_\al \, \mbox{ for } \al=1,\dots,n-1 .
$$
Hence, computing from now on always at the point $x_0$, and by using \eqref{eq:Giusti 1} and \eqref{eq:utileperSimonslemmaproof}, we get
\begin{eqnarray*}
\frac{1}{2} \DLB c^2 &=& -c^4 + \sum_{\al,\be,\ga} (\de_\ga \de_\al \nu^\be)^2 + 2 \sum_{\al, \ga} (\de_\ga \de_\al \nu^n)^2 - 2 \sum_{\al, \be} (\de_\al \de_\be \nu^n)^2
\\
&=& -c^4 + \sum_{\al,\be,\ga} (\de_\ga \de_\al \nu^\be)^2 ,
\end{eqnarray*}
where all the greek indices indicate summation from $1$ to $n-1$.

On the other hand, we have 
$$
| \de c |^2 = \frac{1}{c^2} (\de_\al \nu^\be) (\de_\ga \de_\al \nu^\be) (\de_\si \nu^\tau) (\de_\ga \de_\si \nu^\tau) ,
$$
and hence
\begin{equation*}
\frac{1}{2} \DLB c^2 + c^4 - | \de c |^2 = \frac{1}{2 c^2} \sum_{\al, \be, \ga, \si, \tau} \left[ (\de_\si \nu^{\tau}) (\de_\ga \de_\al \nu^\be) - (\de_\al \nu^\be) (\de_\ga \de_\si \nu^\tau) \right]^2 .
\end{equation*}

Now remember that $\pa E$ is a cone with vertex at the origin, and thus
$<x, \nu> = 0$ on $\pa E$.
Since we took $\nu= (0, \dots , 0 , 1)$ at $x_0$, we may choose coordinates in such a way that $x_0$ lies on the $(n-1)$-axis. In particular, $\nu^{n-1} = 0$ at $x_0$ and
$$
0 = \de_i <x, \nu> = <\de_i x, \nu>+<x, \de_i \nu> = <x, \de_i \nu>,
$$
which leads to
$$
\de_i \nu^{n-1} = 0 \, \mbox { at } \, x_0 .
$$

If now the letters $A,B,S,T$ run from $1$ to $n-2$, he have
\begin{equation*}
\begin{split}
\frac{1}{2} \DLB c^2 + c^4 - | \de c |^2 = & \frac{1}{2 c^2} \sum_{A,B,S,T, \ga} \left[ (\de_S \nu^T) (\de_\ga \de_A \nu^B) - (\de_A \nu^B) (\de_\ga \de_S \nu^T) \right]^2 
\\
& + \frac{2}{c^2} \sum_{S,T, \ga, \al} (\de_S \nu^T)^2 (\de_\ga \de_{n-1} \nu^{\al})^2  \ge 2 \sum_{\al, \ga} ( \de_\ga \de_{n-1} \nu^{\al} )^2.
\end{split}
\end{equation*}
From \eqref{eq:Giusti 1}, $\de_i \de_{n-1} = \de_{n-1} \de_i $ and $\de_{n-1} = \pa_{n-1} = \pm \left( x^j /|x| \right) \pa_j$ at $x_0$.
Since $\pa E$ is a cone, $\nu$ is homogeneous of degree 0 and hence $\de_i \nu^\al$ is homogeneous of degree $-1$. Thus, by Euler's theorem on homogeneous functions, we have
$$
\de_{n-1} \de_i \nu^{\al} = \pm  \frac{x^j \pa_j }{|x|}  \de_i \nu^{\al} = \mp \frac{1}{|x|} \de_i \nu^{\al} ,
$$ 
and hence
$$
2 \sum_{i, \al } ( \de_i \de_{n-1} \nu^{\al} )^2 = \frac{2}{ | x |^2} \sum _{i, \al} (\de_i \nu^\al)^2 = \frac{2 c^2}{| x |^2} .
$$
The proof is now completed.
\qed
\end{sketch}

\subsection{Comments on: harmonic maps, free boundary problems, and nonlocal minimal surfaces}
Here we briefly sketch arguments and results similar to the previous ones on minimal surfaces, now for three other elliptic problems.

\subsubsection{Harmonic maps}
Consider the energy
\begin{equation}
\label{eq:harm_enrg}
E(u)=\frac{1}{2}\int_{\Omega}|Du|^2dx
\end{equation}
for $H^1$ maps $u:\Omega\subset\RR^n \to \overline{S^N_+} $ from a domain
$\Omega$ of $\RR^n$ into the closed upper hemisphere
$$\overline{S^N_+}=\{y\in \mathbb{R}^{N+1}: |y|=1, y_{N+1}\ge 0\}.$$

A critical point of $E$ is called a (weakly) {\it harmonic map}.
When a map minimizes $E$ among all maps with values into $\overline{S^N_+}$ and with same boundary
values on $\partial \Omega$, then it is called a 
minimizing harmonic map.

From the energy \eqref{eq:harm_enrg} and the restriction $|u| \equiv 1$, one finds that the equation for harmonic maps is given by
\begin{equation*}
-\Delta u=|Du|^2u \qquad\text{in }\Omega.
\end{equation*}

In 1983, J\"ager and Kaul proved the following theorem, that we state here without proving it (see the original paper \cite{JK} for the proof).

\begin{theorem}[J\"ager-Kaul~\cite{JK}]\label{thm:equator}
The equator map 
\begin{equation*}
u_*:B_1\subset\RR^n \rightarrow \overline{S^n_+},
\quad x\mapsto \left(x/|x|,0\right)
\end{equation*}
is a minimizing harmonic map on the class
$$
{\mathcal C}=\{u\in H^1(B_1\subset\RR^n,S^n) : u=u_* \text{ on }\partial B_1\}
$$
if and only if $n\ge 7$.
\end{theorem}
We just mention that the proof of the ``if" in Theorem \ref{thm:equator} uses a calibration argument.

Later, Giaquinta and Sou\v cek \cite{GiaSouLN}, and independently Schoen and Uhlenbeck \cite{SU2}, proved the following result.

\begin{theorem}[Giaquinta-Sou\v cek \cite{GiaSouLN}; Schoen-Uhlenbeck \cite{SU2}]\label{thm:reg_harm}
\hfill\\
Let $u:B_1\subset\RR^n\to \overline{S^N_+}$ be a minimizing harmonic map,
homogeneous of degree zero, 
into the closed upper hemisphere $\overline{S^N_+}$. If $3\le n \le 6$,
then $u$ is constant.
\end{theorem}

Now we will show an outline of the proof of Theorem~\ref{thm:reg_harm} following~\cite{GiaSouLN}. More details can also be found in Section 3 of \cite{CabCapThree}. This theorem gives an alternative proof of one part of the statement of Theorem~\ref{thm:equator}. Namely, that the equator map $u_*$ 
is not minimizing for $3\le n\le 6$.

\begin{sketch}[of Theorem~\ref{thm:reg_harm}]
After stereographic projection (with respect to the south pole)
$P$ from $S^{N}\subset\RR^{N+1}$ to $\RR^N$, for the new
function $v=P\circ u: B_1\subset\RR^n\to\RR^N$,
the energy \eqref{eq:harm_enrg} (up to a constant factor) 
is given by 
\begin{equation*}
E(v):=\int_{B_1}\frac{|Dv|^2}{(1+|v|^2)^2}dx.
\end{equation*}
In addition, we have $|v|\le 1$ since the image of $u$ is contained
in the closed upper hemisphere.

By testing the function
$$
\xi(x)=v(x)\eta(|x|),
$$ 
where 
$\eta$ is a smooth radial function with compact support in $B_1$, in the equation of the first variation of the energy, that is
\begin{equation*}
\delta E(v) \xi=0 ,
\end{equation*}
one can deduce that either $v$ is constant (and 
then the proof is finished) or
$$
|v|\equiv 1 ,
$$
that we assume from now on.

Since $v$ is a minimizer, we have that the second variation of the energy satisfies
$$\delta^2 E(v) (\xi,\xi) \ge 0 .$$
By choosing here the function 
$$
\xi(x)=v(x)|Dv(x)|\eta(|x|),
$$ 
where 
$\eta$ is a smooth radial function with compact support in $B_1$
(to be chosen later), and setting
$$
c(x):=|Dv(x)|,
$$
one can conclude the proof by similar arguments as in the previous section and by using Lemma \ref{lem:harm}, stated next.
\qed
\end{sketch}

\begin{lemma}\label{lem:harm}
If $v$ is a harmonic map, homogeneous of degree
zero, and with $|v|\equiv 1$, we have
\begin{equation*}
\frac{1}{2}\Delta c^2-|Dc|^2 + c^4 \geq \frac{c^2}{|x|^2}+\frac{c^4}{n-1},
\end{equation*}
where $c:=|Dv|$.
\end{lemma}

This lemma is the analogue result of Lemma~\ref{lmsimons} for
minimal cones. See \cite{GiaSouLN} for a proof of the lemma,
which also follows from Bochner identity (see \cite{SU2}).

\subsubsection{Free boundary problems}
Consider the one-phase free boundary problem:
%
\begin{equation}\label{000:eq:freboundpb}
\left\{
\begin{array}{rcll}
\De u &=& 0 & \quad\mbox{in } E \\
u &=& 0  & \quad\mbox{on } \pa E \\
| \na u | &=& 1  & \quad\mbox{on } \pa E \setminus \lbrace 0 \rbrace ,\\
\end{array}\right.
\end{equation}
where $u$ is homogeneous of degree one and positive in the domain $E \subset \RR^n$ and $\pa E$ is a cone.
We are interested in solutions $u$ that are stable for the Alt-Caffarelli energy functional
\begin{equation*}
E_{B_1}(u) = \int_{B_1} \left\lbrace | \na u |^2 + \mathbbm{1}_{ \left\lbrace u> 0 \right\rbrace  } \right\rbrace  dx
\end{equation*}
with respect to compact domain deformations that do not contain the origin. More precisely, we say that $u$ is {\it stable} if for any smooth vector field $\Psi: \RR^n \rightarrow \RR^n$ with $0 \notin {\rm supp} \Psi \subset B_1$ we have
\begin{equation*}
\left.\frac{d^2}{dt^2} E_{B_1} \left(u \left(x + t \Psi(x) \right) \right) \right|_{t=0} \ge 0.
\end{equation*}
The following result due to Jerison and Savin is contained in \cite{JS}, where a detailed proof can be found.

\begin{theorem}[Jerison-Savin \cite{JS}]\label{JSfreebp}
The only stable, homogeneous of degree one, solutions of \eqref{000:eq:freboundpb} in dimension $n\le 4$ are the one-dimensional solutions $u=(x \cdot \nu )^+$, $\nu \in S^{n-1}$.
\end{theorem}

In dimension $n=3$ this result had been established by Caffarelli, Jerison, and Kenig \cite{CJK}, where they conjectured that it remains true up to dimension $n \le 6$.
On the other hand, in dimension $n=7$, De Silva and Jerison \cite{DJ} provided an example of a nontrivial minimizer.

The proof of Jerison and Savin of Theorem \ref{JSfreebp} is similar to Simons proof of the rigidity of stable minimal cones in low dimensions: they find functions $c$ (now involving the second derivatives of $u$) which satisfy appropriate differential inequalities for the linearized equation. 

Here, the linearized problem is the following:
%
$$
\left\{
\begin{array}{rcll}
\De v &=& 0 & \quad\mbox{in } E \\
v_{\nu} + \cH v &=& 0  & \quad\mbox{on } \pa E \setminus \left\lbrace 0 \right\rbrace . \\
\end{array}\right.
$$
For the function
$$
c^2= \nr D^2 u \nr^2 = \sum_{i,j =1}^n u_{ij}^2 ,
$$
they found the following interior inequality which is similar to the one of the Simons lemma:
$$\frac{1}{2} \De c^2 - |\na c|^2 \geq 2 \, \frac{n-2}{n-1} \, \frac{c^2}{|x|^2}+ \frac{2}{n-1} \, |\na c|^2.$$
In addition, they also need to prove a boundary inequality involving $c_\nu$.
Furthermore, to establish Theorem \ref{JSfreebp} in dimension $n=4$, a more involved function $c$ of the second derivatives of $u$ is needed.

\subsubsection{Nonlocal minimal surfaces}
Nonlocal minimal surfaces, or $\al$-minimal surfaces (where $\al \in (0,1)$), have been introduced in 2010 in the seminal paper of Caffarelli, Roquejoffre, and Savin \cite{CRS}. These surfaces are connected to fractional perimeters and diffusions and, as $\al \nearrow 1$, they converge to classical minimal surfaces.
We refer to the lecture notes \cite{CF} and the survey \cite{DipVal}, where more references can be found.

For $\al$-minimal surfaces and all $\al \in (0,1)$, the analogue of Simons flatness result is only known in dimension $2$ by a result for minimizers of Savin and Valdinoci \cite{SV}.

\section{The Allen-Cahn equation}\label{ALLENCA}
This section concerns the {\it Allen-Cahn equation}
\begin{equation}\label{AlCa}
- \Delta u= u - u^3 \quad \mbox{in $\RR^{n}$}.
\end{equation}
By using equation \eqref{AlCa} and the maximum principle it can be proved that any solution satisfies
$|u| \le 1$. Then, by the strong maximum principle we have that either $|u|<1$  or $u \equiv \pm 1$.
Since $u \equiv \pm 1$ are trivial solutions, from now on we consider $u:\RR^{n}\to (-1,1)$.

We introduce the class of {\it 1-d solutions}: 
$$
u(x)= u_* (x \cdot e) \quad \mbox{ for a vector } \, e \in \RR^n, |e|=1,
$$ 
where
$$
u_* (y)= \tanh\left(\frac{y}{\sqrt{2}} \right).
$$
%
%
%
The solution $u_*$ is sometimes referred to as the {\it layer solution} to \eqref{AlCa};
see Figure \ref{fig:8}.
The fact that $u$ depends only on one variable can be rephrased also by saying that all the level sets
$\{u=s\}$ of $u$ are
hyperplanes.

\begin{figure}[htbp]
\centering
\includegraphics[scale=.25]{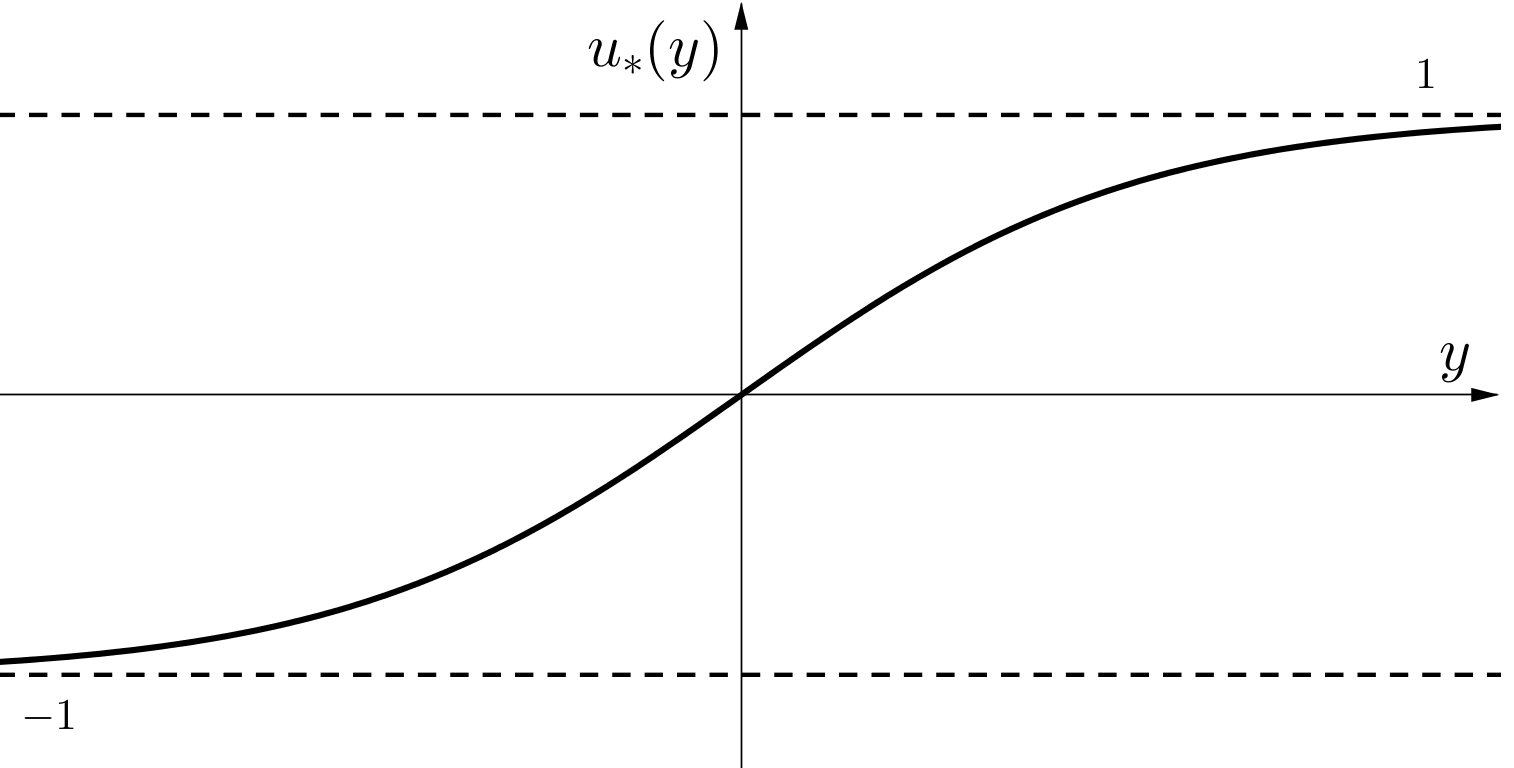}
\caption{The increasing, or layer, solution to the Allen-Cahn equation}
\label{fig:8}       
\end{figure}

\begin{exercise}
Check that the $1$-d functions introduced above are solutions of the Allen-Cahn equation.
\end{exercise}

\begin{remark}\label{rem:1dthenconditions}
Let us take $e= e_n = (0, \dots, 0,1)$ and consider the 1-d solution
$u(x)=u_* (x_n) = \tanh ( x_n / \sqrt{2} )$. It is clear that the following two relations hold:
\begin{equation}\label{monoton}
u_{x_n} >0 \quad \mbox{in }\RR^{n} ,
\end{equation}

\begin{equation}\label{limiting}
\lim_{x_n\rightarrow\pm\infty}u(x', x_n)=\pm 1
\quad \text{ for all }x'\in \RR^{n-1}.
 \end{equation}

\end{remark}

The energy functional associated to equation \eqref{AlCa} is
\begin{equation*}
E_\Om (u):=\int_\Omega \left\{ \frac{1}{2} |\nabla u|^2
+ G(u)\right\} dx ,
\end{equation*}
where $G$ is the double-well potential in Figure \ref{fig:7}:
\begin{equation*}
G(u)= \frac{1}{4} \left( 1- u^2 \right)^2. 
\end{equation*}

\begin{definition}[Minimizer]\label{definit:MinimizAC}
A function  $u:\RR^{n}\to (-1,1)$ is said to be a {\it minimizer} of \eqref{AlCa}
when
$$E_{B_R} (u) \leq E_{B_R} (v)$$
for every open ball $B_R$ and functions $v: \ol{B_R}\to \RR$ such that $v\equiv u$ on 
$\partial B_R$.
\end{definition}

\begin{figure}[htbp]
\centering
\includegraphics[scale=.25]{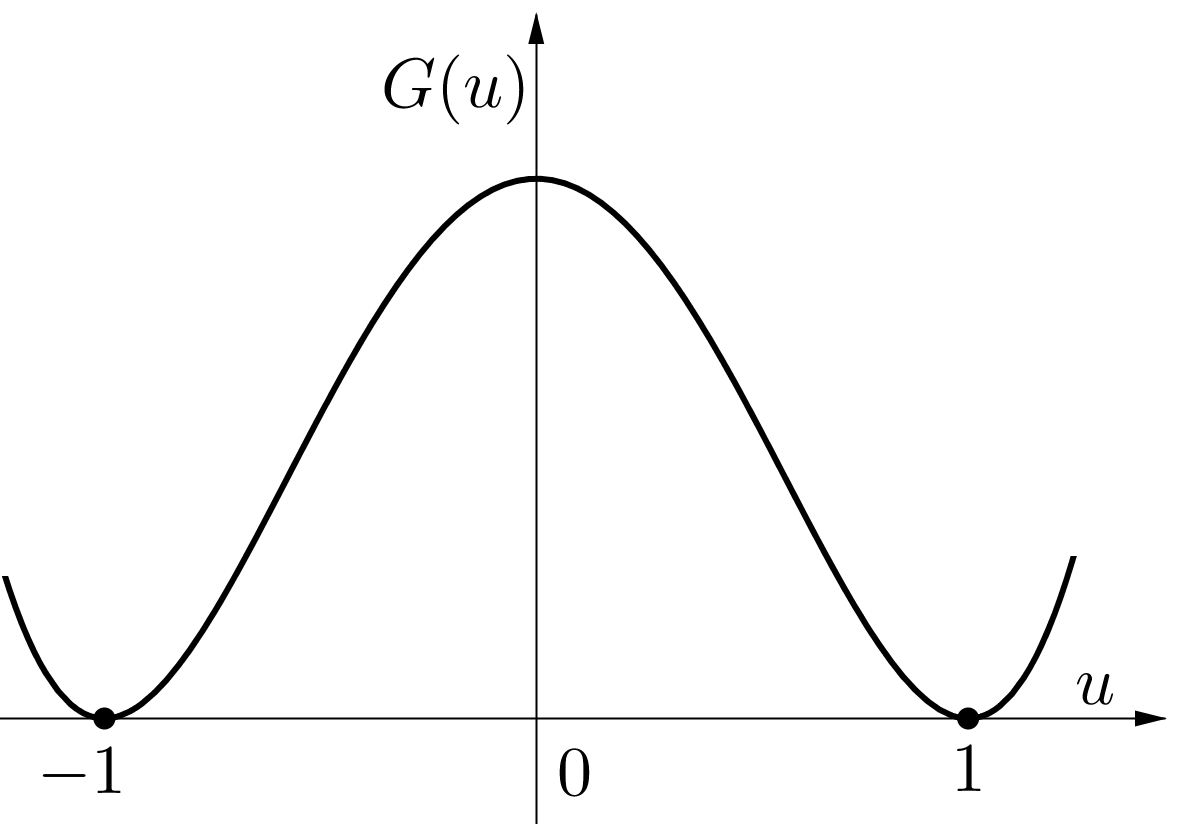}
\caption{The double-well potential in the Allen-Cahn energy}
\label{fig:7}       
\end{figure}

\smallskip
\noindent
{\bf Connection with the theory of minimal surfaces.}
The Allen-Cahn equation has its origin in the theory of phase transitions and it is used as a model for some nonlinear reaction-diffusion processes.
To better understand this, let $\Om\subset \RR^n$ be a bounded domain, and consider the Allen-Cahn equation with parameter $\ve > 0$,
\begin{equation}\label{eq:Allen-Cahnwithparameter}
- \ve^2 \De u= u - u^3 \, \mbox{ in } \, \Om ,
\end{equation}
with associated energy functional given by
\begin{equation}\label{eq:energyACwithparameter}
E_{\ve}(u)= \int_{\Om} \left\lbrace \frac{\ve}{2} \, | \na u|^2 + \frac{1}{\ve} \, G(u) \right\rbrace dx .
\end{equation}
Assume now that there are two populations (or chemical states) A and B and that $u$ is a density measuring the percentage of the two populations at every point: if $u(x) = 1$ (respectively, $u(x)= -1$) at a point $x$, we have only population A at $x$ (respectively, population B); $u(x)=0$ means that at $x$ we have $50\%$ of population A and $50\%$ of population $B$.

By \eqref{eq:energyACwithparameter}, it is clear that in order to minimize $E_{\ve}$ as $\ve$ tends to $0$, $G(u)$ must be very small. From Figure \ref{fig:7} we see that this happens when $u$ is close to $\pm 1$. These heuristics are indeed formally confirmed by a celebrated theorem of Modica and Mortola. It states that, if $u_\ve$ is a family of minimizers of $E_\ve$,
then, up to a subsequence, $u_\ve$ converges in $L^1_{\mathrm{loc}} (\Om)$, as $\ve$ tends to $0$,
to
$$
u_0 = \mathbbm{1}_{ \Om_+  } - \mathbbm{1}_{ \Om_-  } 
$$
for some disjoint sets $\Om_{\pm}$ having as common boundary a surface $\Ga$. In addition, $\Ga$ is a minimizing minimal surface.
Therefore, the result of Modica-Mortola establishes that the two populations tend to their total separation, and in such a (clever) way that the interface surface $\Ga$ of separation has least possible area.

Finally, notice that the $1$-d solution of \eqref{eq:Allen-Cahnwithparameter},
$$
x \mapsto u_* (\frac{x \cdot e}{\ve}) ,
$$
makes a very fast transition from $-1$ to $1$ in a scale of order $\ve$. Accordingly, in Figure \ref{fig:add1}, $u_\ve$ will make this type of fast transition across the limiting minimizing minimal surface $\Ga$.   
The interested reader can see \cite{AAC} for more details.

\smallskip

\begin{figure}[htbp]
\centering
\includegraphics[scale=.25]{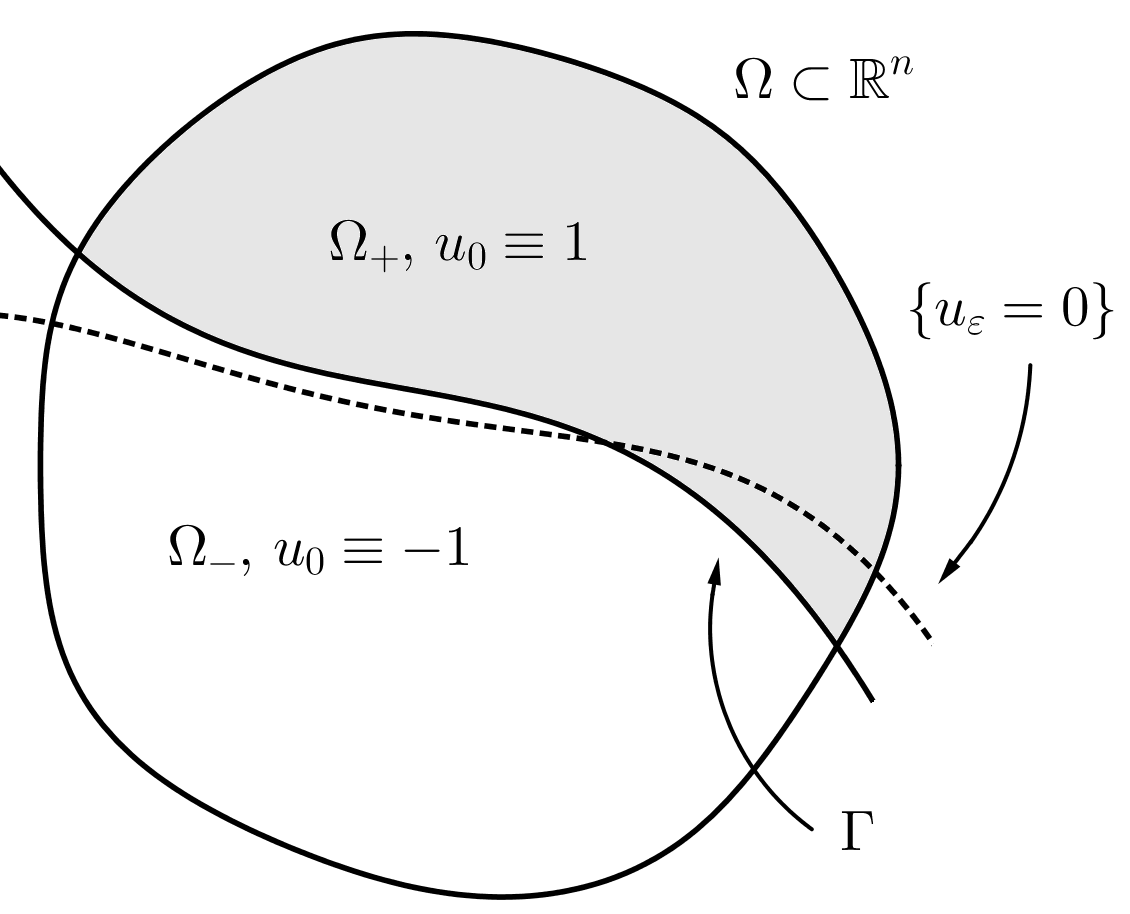}
\caption{The zero level set of $u_\ve$, the limiting function $u_0$, and the minimal surface $\Ga$}
\label{fig:add1}       
\end{figure}

\subsection{Minimality of monotone solutions with limits \texorpdfstring{$\pm 1$}{+-1}}

The following fundamental result shows that monotone solutions with limits $\pm 1$ are minimizers (as in Definition \ref{definit:MinimizAC}).

\begin{theorem}[Alberti-Ambrosio-Cabr\'e \cite{AAC}]\label{alba}
Suppose that $u$ is a solution of \eqref{AlCa} satisfying the monotonicity hypothesis \eqref{monoton} and the condition \eqref{limiting} on limits.
Then, $u$ is a minimizer of \eqref{AlCa} in $\RR^n$.
\end{theorem}

See \cite{AAC} for the original proof of the Theorem~\ref{alba}. It uses a calibration built from a foliation and avoids the use of the strong maximum principle, but it is slightly involved. Instead, the simple proof that we give here was suggested to the first author (after one of his lectures on \cite{AAC}) by  L.~Caffarelli.
It uses a simple foliation argument together with the strong maximum principle, as in the alternative proof of Theorem \ref{thm:SimCone} given in Subsection \ref{subsec 1.1:MinimalitySimcones}.

\begin{proof}[Proof of Theorem \ref{alba}]
Denoting $x= (x' ,x_n) \in \RR^{n-1} \times \RR$, let us consider the functions
$$
u^t (x):= u(x', x_n +t), \, \mbox{ for } \, t\in\RR.
$$
By the monotonicity assumption \eqref{monoton} we have that
\begin{equation}\label{eq:monfoliation}
u^t < u^{t'} \, \mbox{ in } \RR^n , \, \mbox{ if } t<t'.
\end{equation}
Thus, by \eqref{limiting} we have that the graphs of $u^t= u^t (x)$, $t \in \RR$,
form a foliation filling all of $\RR^n \times (-1,1)$.
Moreover, we have that for every $t \in \RR$, $u^t$ are solutions of $-\De u^t =u^t - (u^t)^3$ in $\RR^n$.

By simple arguments of the Calculus of Variations, given a ball $B_R$ it can be proved that there exists
a minimizer $v: \ol{B}_R \to \RR$ of $E_{B_R}$ such that $v=u$ on $\pa B_R$. In particular, $v$ satisfies
%
%
%
\begin{equation*}
\left\{
\begin{array}{rcll}
- \De v &=& v - v^3 & \quad\mbox{in } B_R \\
|v| &<& 1  & \quad\mbox{in } \ol{B}_R \\
v &=& u  & \quad\mbox{on } \pa B_R .\\
\end{array}\right.
\end{equation*}

By \eqref{limiting}, we have that the graph of $u^t$ in the compact set $\ol{B}_R$ is above the graph of $v$ for $t$ large enough, and it is below the graph of $v$ for $t$ negative enough
(see Figure~\ref{fig:9}).
If $v \not\equiv u$, assume that $v<u$ at some point in $B_R$ (the situation $v>u$ somewhere in $B_R$ is done similarly). It follows that, starting from $t= - \infty$, there will exist a first $t_* <0$
such that $u^{t_*}$ touches $v$ at a point $P \in \ol{B}_R$. This means that
$u^{t_*} \le v$ in $\ol{B}_R$ and $u^{t_*} (P) = v(P)$.

\begin{figure}[htbp]
\centering
\includegraphics[scale=.25]{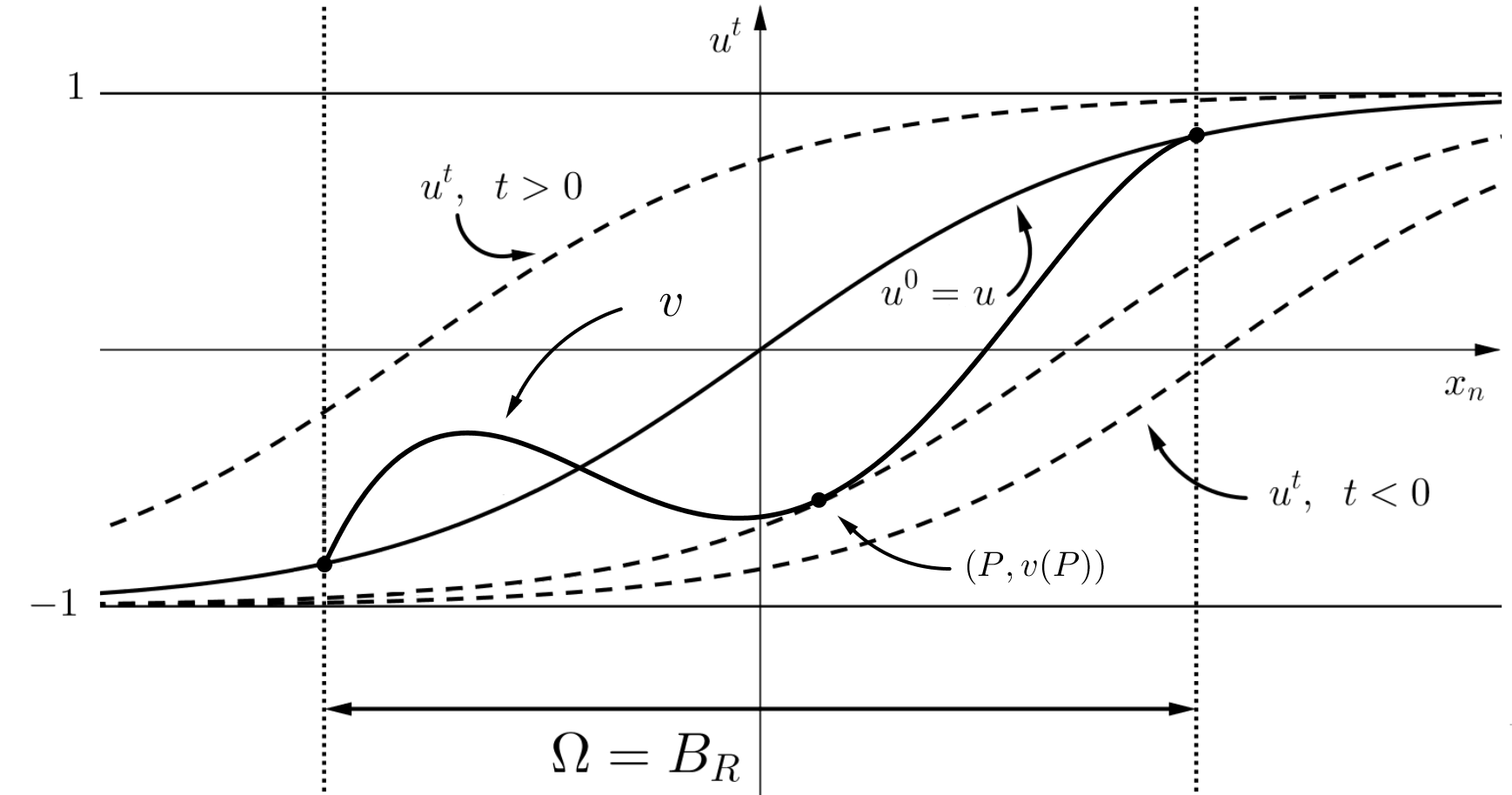}
\caption{The foliation $\lbrace u^t \rbrace$ and the minimizer $v$}
\label{fig:9}       
\end{figure}

By \eqref{eq:monfoliation}, $t_* <0$, and the fact that $v=u = u^0$ on $\pa B_R$, the point $P$ cannot belong to
$\pa B_R$. Thus, $P$ will be an interior point of $B_R$.

But then we have that $u^{t_*}$ and $v$ are two solutions of the same semilinear equation (the Allen-Cahn equation), the graph of $u^{t_*}$ stays below that of $v$, and they touch each other at the interior point $\left( P, v(P) \right) $. This is a contradiction with the strong maximum principle.

Here we leave as an exercise (stated next) to verify that the difference of two solutions of $- \De u = f(u)$ satisfies a linear elliptic equation to which we can apply the strong maximum principle.
This leads to $u^{t_*} \equiv v$, which contradicts $u^{t_*} < v = u^0$ on $\pa B_R$.
%
\end{proof}

\begin{exercise}
Prove that the difference $w:= v_1 - v_2$ of two solutions of a semilinear equation $- \De v= f(v)$, where $f$ is a Lipschitz function, satisfies a linear equation of the form $\De w + c(x) \, w = 0$, for some function $c \in L^{\infty}$. Verify that, as a consequence, this leads to $u^{t_*} \equiv v$ in the previous proof.
\end{exercise}

By recalling Remark \ref{rem:1dthenconditions}, we immediately get the following corollary.
\begin{corollary}
The 1-d solution $u(x)=u_* (x \cdot e)$ is a minimizer of \eqref{AlCa} in $\RR^n$, for every unit vector $e \in \RR^n$. 
\end{corollary}


As a corollary of Theorem \ref{alba}, we easily deduce the following important energy estimates.

\begin{corollary}[Energy upper bounds; Ambrosio-Cabr\'e \cite{ambcab}]\label{tupper}
Let $u$ be a solution of \eqref{AlCa} satisfying
\eqref{monoton} and \eqref{limiting} (or more generally, let $u$ be a minimizer in $\RR^n$).

Then, for all $R \ge 1$ we have
\begin{equation}\label{upperbou}
E_{B_R} (u) \leq C R^{n-1}
\end{equation}
for some constant $C$ independent of $R$.
In particular, since $G \ge 0$, we have that
$$
\int_{B_R} |\nabla u|^2 \,dx
\leq C R^{n-1} 
$$
for all $R \ge 1$.
\end{corollary}

\begin{remark}
The proof of Corollary \ref{tupper} is trivial for $1$-d solutions. Indeed, it is easy to check that $\int_{- \infty}^{+ \infty} \left\lbrace \frac{1}{2} (u_* ')^2 + \frac{1}{4}(1- u_*^2)^2 \right\rbrace dy < \infty$ and, as a consequence, by applying Fubini's theorem on a cube larger than $B_R$, that \eqref{upperbou} holds.
This argument also shows that the exponent $n-1$ in \eqref{upperbou} is optimal (since it cannot be improved for $1$-d solutions).
\end{remark}

The estimates in Corollary \ref{tupper} are fundamental in the proofs of a conjecture of De Giorgi that we treat in the next subsection.

The estimate \eqref{upperbou} was first proved by Ambrosio and the first author in \cite{ambcab}. Later on, in \cite{AAC} Alberti, Ambrosio, and the first author discovered that monotone solutions with limits are minimizers (Theorem \ref{alba} above). This allowed to simplify the original proof of the energy estimates found in \cite{ambcab}, as follows.

\begin{proof}[Proof of Corollary \ref{tupper}]
Since $u$ is a minimizer by Theorem \ref{alba} (or by hypothesis), we can
perform a simple energy comparison argument. Indeed, let
$\phi_R\in C^\infty(\RR^{n})$ satisfy $0\le\phi_R\le 1$ in $\RR^{n}$,
$\phi_R\equiv 1$ in $B_{R-1}$,
$\phi_R\equiv 0$ in $\RR^{n}\setminus B_R$, and 
$\Vert\nabla\phi_R\Vert_\infty\leq 2$. Consider  
$$
v_R:=(1-\phi_R) u+\phi_R.
$$

Since $v_R \equiv u$ on $\pa B_R$, we can compare
the energy of $u$ in $B_R$ with that of $v_R$. 
We obtain
\begin{multline*}
\int_{B_R}\left\{\frac{1}{2} |\nabla u|^2+ G(u) \right\}dx \leq 
\int_{B_R}\left\{\frac{1}{2}|\nabla v_R|^2+ G(v_R) \right\}dx 
\\ =\int_{B_R\setminus B_{R-1}}
\left\{\frac{1}{2}|\nabla v_R|^2+ G(v_R) \right\}dx \leq 
C|B_R\setminus B_{R-1}| \le  CR^{n-1}
\end{multline*}
for every $R \ge 1$, with $C$ independent of $R$. In the second inequality of the chain above we used that $\frac{1}{2}|\nabla v_R|^2+ G(v_R) \le C$ in $B_R\setminus B_{R-1}$ for some constant $C$ independent
of~$R$. This is a consequence of the following exercise.
\end{proof}

\begin{exercise}
Prove that if $u$ is a solution of a semilinear equation $- \De u = f(u)$ in $\RR^n$ and $|u| \le 1$ in $\RR^n$, where $f$ is a continuous nonlinearity, then $|u| + | \na u| \le C$ in $\RR^n$ for some constant $C$ depending only on $n$ and $f$. See \cite{ambcab}, if necessary, for a proof.
\end{exercise}

\subsection{A conjecture of De Giorgi}

In 1978, E. De Giorgi \cite{DG} stated the following conjecture:

\smallskip\noindent
{\em 
{\rm {\bf Conjecture} (DG).}  
Let $u:\RR^{n}\to (-1,1)$ be a solution of the Allen-Cahn equation \eqref{AlCa}
satisfying the monotonicity condition \eqref{monoton}.
Then, $u$ is a $1$-d solution (or equivalently, all level sets $\{u=s\}$ of $u$ are
hyperplanes), at least if $n\leq 8$.
}

\smallskip

This conjecture was proved in 1997 for $n=2$ by Ghoussoub and Gui \cite{GG}, and in 2000 for $n=3$ by
Ambrosio and Cabr\'e \cite{ambcab}.
Next we state a deep result of Savin \cite{Sa} under the only assumption of minimality. This is the semilinear analogue of Simons Theorem \ref{thmcone} and Remark \ref{remarkchemiserve:1.17} on minimal surfaces. As we will see, Savin's result leads to a proof of Conjecture (DG) for $n \le 8$ if the additional condition \eqref{limiting} on limits is assumed.

\begin{theorem}[Savin \cite{Sa}]\label{teosavin}
Assume that $n\leq 7$ and that $u$ is a minimizer of \eqref{AlCa} in~$\RR^n$. 
Then, $u$ is a $1$-d solution.
\end{theorem}

The hypothesis $n\leq 7$ on its statement is sharp.
Indeed, in 2017 Liu, Wang, and Wei \cite{LWW} have shown the existence of a minimizer in $\RR^8$ whose level sets are not hyperplanes. Its zero level set is asymptotic at infinity to the Simons cone. However, a canonical solution described in Subsection \ref{subsec:saddlesha} (and whose zero level set is exactly the Simons cone) is still not known to be a minimizer in $\RR^8$.

Note that Theorem \ref{teosavin} makes no assumptions on the monotonicity or
the limits at infinity of the solution. To prove Conjecture (DG) using Savin's result
(Theorem~ \ref{teosavin}), one needs to make the further assumption \eqref{limiting} on the limits only to guarantee, by Theorem~\ref{alba}, that the solution is actually a minimizer.
Then, Theorem~\ref{teosavin} (and the gain of one more dimension, $n=8$, thanks to the monotonicity of the solution)
leads to the proof of Conjecture (DG) for monotone solutions with limits $\pm 1$.

However, for $4\leq n\leq 8$ the conjecture in its original statement (i.e., without the limits $\pm 1$ as hypothesis) is still open.
To our knowledge no clear evidence is known about its validity
or not.

The proof of Theorem \ref{teosavin} uses an improvement of flatness result for the Allen-Cahn equation developed by Savin, as well as Theorem \ref{thmcone} on the non-existence of stable minimal cones in dimension $n \le 7$.

Instead, the proofs of Conjecture (DG) in dimensions $2$ and $3$ are much simpler. They use the energy estimates of Corollary \ref{tupper} and a Liouville-type theorem developed in \cite{ambcab} (see also \cite{AAC}). As explained next, the idea of the proof originates in the paper \cite{BCN} by Berestycki, Caffarelli, and Nirenberg.

\smallskip
\noindent
{\bf Motivation for the proof of Conjecture (DG) for $n \le 3$.}
In \cite{BCN} the authors made the following heuristic observation. From the equation
$- \De u = f(u)$ and the monotonicity assumption \eqref{monoton}, by differentiating we find that 
\begin{equation}
\label{eq:remforseno}
u_{x_n}>0 \, \mbox{ and } \, L u_{x_n} := \left( - \De - f'(u) \right) u_{x_n}=0 \, \mbox{ in } \RR^n .
\end{equation} 
If we were in a bounded domain $\Om\subset \RR^n$ instead of $\RR^n$ (and we forgot about boundary conditions), from \eqref{eq:remforseno}, we would deduce that $u_{x_n}$ is the first eigenfunction of $L$ and that its first eigenvalue is $0$. As a consequence, such eigenvalue is simple. But then, since we also have that
$$
L u_{x_i} = \left( - \De - f'(u) \right) u_{x_i}=0 \quad \mbox{for } \, i=1, \dots, n-1,
$$
the simplicity of the eigenvalue would lead to 
\begin{equation}\label{eq:remforseno2}
u_{x_i}= c_i u_{x_n} \quad \mbox{for } \, i=1, \dots, n-1 ,
\end{equation}
where $c_i$ are constants. Now, we would conclude that $u$ is a 1-d solution, by the following exercise.
\begin{exercise}
Check that \eqref{eq:remforseno2}, with $c_i$ being constants, is equivalent to the fact that $u$
is a 1-d solution.
\end{exercise}
To make this argument work in the whole $\RR^n$, one needs a Liouville-type theorem. For $n=2$ it was proved in \cite{BCN} and \cite{GG}. Later, \cite{ambcab} used it to prove Conjecture (DG) in $\RR^3$ after proving the crucial energy estimate \eqref{upperbou}. The Liouville theorem requires the right hand side of \eqref{upperbou} to be bounded by $C R^2= C R^{3-1}$.

\smallskip

In 2011, del Pino, Kowalczyk, and Wei \cite{dP} established 
that Conjecture~(DG) does not hold for $n\geq 9$ 
-- as suggested in De Giorgi's original statement.

\begin{theorem}[del Pino-Kowalczyk-Wei \cite{dP}]\label{thm:delP K W}
If $n \ge 9$, there exists a solution of \eqref{AlCa}, satisfying \eqref{monoton} and \eqref{limiting}, and which is not a $1$-d solution.
\end{theorem}

The proof in \cite{dP} uses crucially the minimal graph in $\RR^9$ built by Bombieri, De Giorgi, and Giusti in \cite{BdGG}. This is a minimal surface in $\RR^9$ given by the graph of a function $\phi: \RR^8 \to \RR$ which is antisymmetric with respect to the Simons cone.
The solution of Theorem \ref{thm:delP K W} is built in such a way that its zero level set stays at finite distance from the Bombieri-De Giorgi-Giusti graph.

We consider next a similar object to the previous minimal graph, but in the context of the Allen-Cahn equation: a solution $u: (\RR^8 =) \RR^{2m} \to \RR$ which is antisymmetric with respect to the Simons cone.

\subsection{The saddle-shaped solution vanishing on the Simons cone}\label{subsec:saddlesha}

As in Section \ref{sec:mincones}, let $m \ge 1$ and denote by $\cSC$
the Simons cone \eqref{def:simonscone}.
For $x=(x_1,\dots, x_{2m})\in\RR^{2m}$, $s$ and $t$ denote the two
radial variables
\begin{equation}\label{coor}
s =   \sqrt{x_1^2+...+x_m^2} \quad \mbox{ and } \quad t  =  \sqrt{x_{m+1}^2+...+x_{2m}^2}.
\end{equation}
The Simons cone is given by 
$$
\cSC=\{s=t\}=\partial E, \quad\text{ where }
E =\{s>t\}.
$$ 


\begin{definition}[Saddle-shaped solution]\label{def:saddlesolution} 
We  say  that $u:\RR^{2m}\rightarrow\RR$ is a {\it saddle-shaped
solution}  (or simply a saddle solution) of the Allen-Cahn equation 
\begin{equation}\label{eq2m}
-\Delta u= u- u^3 \quad {\rm in }\;\RR^{2m}
\end{equation}
whenever $u$ is a
solution of
\eqref{eq2m} and, with $s$ and $t$ defined by \eqref{coor},
\renewcommand{\labelenumi}{$($\alph{enumi}$)$}
\begin{enumerate}
\item $u$ depends only on the variables $s$ and $t$. We write
$u=u(s,t)$; 
\item $u>0$ in $E:=\{s>t\}$; 
\item  $u(s,t)=-u(t,s)$ in $\RR^{2m}$.
\end{enumerate}
\end{definition}

\begin{figure}[htbp]
\centering
\includegraphics[scale=.25]{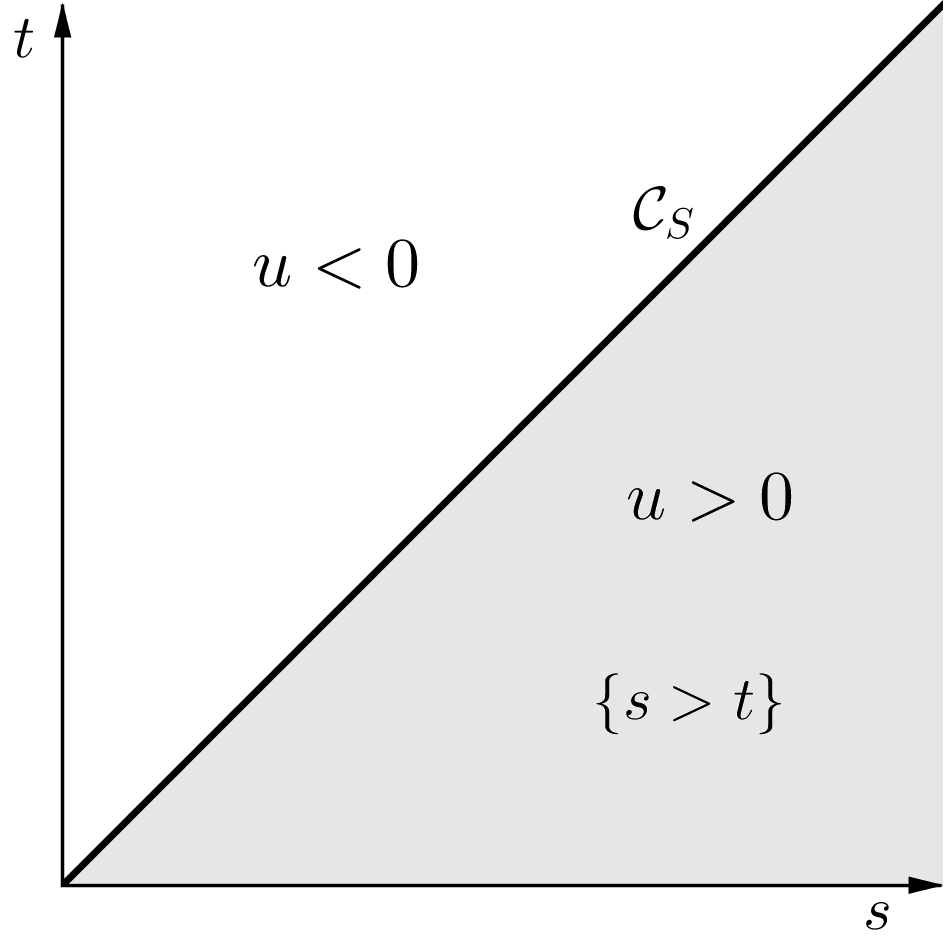}
\caption{The saddle-shaped solution $u$ and the Simons cone $\cSC$}
\label{fig:10}       
\end{figure}

\begin{remark}
Notice that if $u$ is a saddle-shaped solution, then we have $u=0$ on $\cSC$ (see Figure~\ref{fig:10}).
\end{remark}

While the existence of a saddle-shaped solution is easily established, its uniqueness is more delicate. This was accomplished in 2012 by the first author in \cite{Ca}.

\begin{theorem}[Cabr\'e \cite{Ca}] \label{unique}
For every even dimension  $2m\geq 2$,
there exists a unique saddle-shaped solution $u$ of \eqref{eq2m}.
\end{theorem}

Due to the minimality of the Simons cone when $2m \ge 8$ (and also because of the minimizer from \cite{LWW} referred to after Theorem \ref{teosavin}), the saddle-shaped solution is expected to be a minimizer when $2m \ge 8$:

\begin{open problem}
Is the saddle-shaped solution a minimizer of \eqref{eq2m} in $\RR^8$, or at least in higher even dimensions?
\end{open problem}

Nothing is known on this open problem except for the following result.
It establishes stability (a weaker property than minimality) for $2m \ge 14$. Below, we sketch its proof.

\begin{theorem}[Cabr\'e \cite{Ca}]\label{stab} 
If $2m\geq 14$, the saddle-shaped solution $u$ of \eqref{eq2m} is stable in $\RR^{2m}$, in the sense of the following definition.
\end{theorem}

\begin{definition}[Stability]
We say that a
solution $u$ of $- \De u = f(u)$ in $\RR^n$ is {\it stable} if the second variation of the energy with respect to compactly
supported perturbations $\xi$ is nonnegative. That is, if
\begin{equation*}
\int_{{\RR}^{n}}\left\{ |\nabla\xi|^2-f'(u)\xi^2\right\}dx\geq 0 \quad \mbox{ for all } \;\xi\in C^1_c(\RR^{n}) .
\end{equation*}
\end{definition}

In the rest of this section, we will take $n=2m$ and $f$ to be the Allen-Cahn nonlinearity, i.e.,
$f (u) = u- u^3 .$

\begin{sketch}[of Theorem \ref{stab}]
Notice that
\begin{equation}\label{eqst}
u_{ss}+u_{tt}+(m-1){\Big (}\frac{u_s}{s}+\frac{u_t}{t}{\Big
)}+ f(u) =0 ,
\end{equation}
for $s>0$ and $t>0$,
is equation \eqref{eq2m} expressed in the 
$(s,t)$ variables.
Let us introduce the function
\begin{equation}\label{0001:eq:deffi}
\varphi := t^{-b}u_s -s^{-b}u_t .
\end{equation}
Differentiating \eqref{eqst} with respect to $s$ (and to $t$), one finds equations satisfied by $u_s$ (and by $u_t$) -- and which involve a zero order term with coefficient $f'(u)$. These equations, together with some more delicate monotonicity properties of the saddle-shaped solution established in \cite{Ca}, can be used to prove the following fact.

For $2m\geq 14$, one can choose $b>0$ in \eqref{0001:eq:deffi} (see \cite{Ca} for more details) such that $\varphi$ is a positive supersolution of the linearized problem, i.e.:
\begin{equation}\label{posivar}
\varphi >0 \quad \text{in } \{st> 0\},
\end{equation}
\begin{equation}\label{superpf}
\{\Delta + f'(u)\} \varphi \leq 0 \quad \text{in } \RR^{2m}\setminus\{st=0\}=\{st>0\} .
\end{equation}

Next, using \eqref{posivar} and \eqref{superpf}, we can verify the stability
condition of $u$ for any $C^1$ test function
$\xi = \xi(x)$ with compact support in $\{st>0\}$. Indeed, multiply \eqref{superpf}
by $\xi^2/\varphi$ and integrate by parts to get
\begin{eqnarray*}
\int_{\{st>0\}} f'(u)\,\xi^2 \, dx  
& = &\int_{\{st>0\}} f'(u) \varphi\, \frac{\xi^2}{\varphi} \, dx \\
& \leq & \int_{\{st>0\}} -\Delta \varphi \, \frac{\xi^2}{\varphi}  \, dx \\
& = &\int_{\{st>0\}} 
\nabla \varphi \, \nabla\xi\, \frac{2\xi}{\varphi}  \, dx 
-\int_{\{st>0\}} \frac{|\nabla \varphi|^2}{\varphi^2}\, \xi^2 \, dx .
\end{eqnarray*}
Now, using the Cauchy-Schwarz inequality, we are led to
\begin{equation*}
\int_{\{st>0\}} f'(u)\,\xi^2 \, dx \leq \int_{\{st>0\}} |\nabla\xi|^2 \, dx .
\end{equation*}

Finally, by a cut-off argument we can prove that this same inequality holds also for every function $\xi \in C^1_c(\RR^{2m})$.
\qed
\end{sketch}

\begin{remark}
Alternatively to the variational proof seen above, another way to establish stability from the existence of a positive supersolution to the linearized problem is by using the maximum principle (see \cite{BNV} for more details).
\end{remark}

\section{Blow-up problems} 

In this final section, we consider positive solutions of the semilinear problem

\begin{equation}\label{pb}
\left\{
\begin{array}{rcll}
-\De u &=& f(u) & \quad\mbox{in } \Om \\
u &>& 0  & \quad\mbox{in } \Om \\
u &=& 0  & \quad\mbox{on } \pa\Om ,\\
\end{array}\right.
\end{equation}
where $\Om \subset \RR^n$ is a smooth bounded domain, $n\geq 1$, and $f: \RR^+ \to \RR$ is $C^1$.

The associated energy functional is
\begin{equation}\label{eq:energiaultimasec}
E_\Om (u):=\int_\Omega \left\{ \frac{1}{2} |\nabla u|^2 - F(u) \right\} dx ,
\end{equation}
where $F$ is such that $F' = f$.

\subsection{Stable and extremal solutions. A singular stable solution for \texorpdfstring{$n \ge 10$}{n>=10}}

We define next the class of stable solutions to \eqref{pb}. It includes any local minimizer, i.e., any minimizer of \eqref{eq:energiaultimasec} under small perturbations vanishing on $\pa \Om$.

\begin{definition}[Stability]
A solution $u$ of \eqref{pb} is said to be \emph{stable} if 
the second variation of the energy with respect to $C^1$ perturbations $\xi$ vanishing on $\pa \Om$ is nonnegative. That is, if
\begin{equation} \label{LN:stability}
\int_{\Omega} f'(u)\xi^2\,dx
\leq
\int_{\Omega} |\nabla\xi|^2\,dx \quad\textrm{for all } \xi \in C^1 (\ol{\Om} ) \textrm{ with } \xi_{|\pa \Om} \equiv 0.
\end{equation}
\end{definition}


There are many nonlinearities for which \eqref{pb} admits a (positive) stable solution. 
Indeed, replace $f(u)$ by $\lambda f(u)$
in \eqref{pb}, with $\lambda\geq 0$:
\begin{equation}\label{pbla}
\left\{
\begin{array}{rcll}
-\Delta u &=& \la f(u) & \quad\mbox{in $\Omega$}\\
u &=& 0  & \quad\mbox{on $\pa\Omega$.}\\
\end{array}\right.
\end{equation}
Assume that $f$ is positive, nondecreasing, and superlinear at $+\infty$, that is,
\begin{equation} \label{hypf}
f(0) > 0, \quad f'\geq 0 \quad\textrm{and}\quad\lim_{t\to +\infty}\frac{f(t)}t=+\infty.
\end{equation}
Note that also in this case we look for positive solutions (when $\la>0$), since $f>0$.
We point out that, for $\la>0$, $u\equiv 0$ is not a solution.

\begin{proposition}\label{prop:extremal}
Assuming \eqref{hypf}, there exists an extremal parameter $\lambda^*\in (0,+\infty)$ such that if $0\le\lambda<\lambda^*$ then
\eqref{pbla} admits a minimal stable classical solution $u_\lambda$. Here ``minimal'' means the smallest among all the solutions, while ``classical'' means of class $C^2$. Being classical is a consequence of $u_\la \in L^{\infty} ( \Om)$ if $\la < \la^*$.

On the other hand, if $\lambda>\lambda^*$ then \eqref{pbla} has no classical solution.

The family of classical solutions  $\left\{ u_{\lambda }:0\le \lambda <\lambda ^{*}\right\}$ 
is increasing in $\lambda$, and its limit as $\lambda\uparrow\lambda^{*}$ is a weak solution $u^*=u_{\lambda^*}$ of \eqref{pbla} for $\la=\la^*$. 
\end{proposition}

\begin{definition}[Extremal solution]
The function $u^*$ given by Proposition \ref{prop:extremal} is called \emph{the extremal solution} of \eqref{pbla}.
\end{definition}

For a proof of Proposition \ref{prop:extremal} see the book \cite{Dupaigne} by L. Dupaigne.
The definition of weak solution (the sense in which $u^*$ is a solution) requires $u^* \in L^1 (\Om)$, $f(u^* ) \mathrm{dist}(\cdot, \pa \Om) \in L^1 (\Om)$, and the equation to be satisfied in the distributional sense after multiplying it by test functions vanishing on $\pa \Om$ and integrating by parts twice (see \cite{Dupaigne}). Other useful references regarding extremal and stable solutions are \cite{Brezis}, \cite{Cextremal}, and \cite{CabBound}.

Since 1996, Brezis has raised several questions regarding stable and extremal solutions; see for instance \cite{Brezis}. They have led to interesting works, some of them described next. One of his questions is the following.

\medskip
\noindent
\textbf{Question} (Brezis). Depending on the dimension $n$ or on the domain $\Om$, is the extremal solution $u^*$ of \eqref{pbla} bounded (and therefore classical) or is it unbounded?
More generally, one may ask the same question for the larger class of stable solutions to \eqref{pb}.

\medskip

\begin{figure}[htbp]
\centering
\includegraphics[scale=.25]{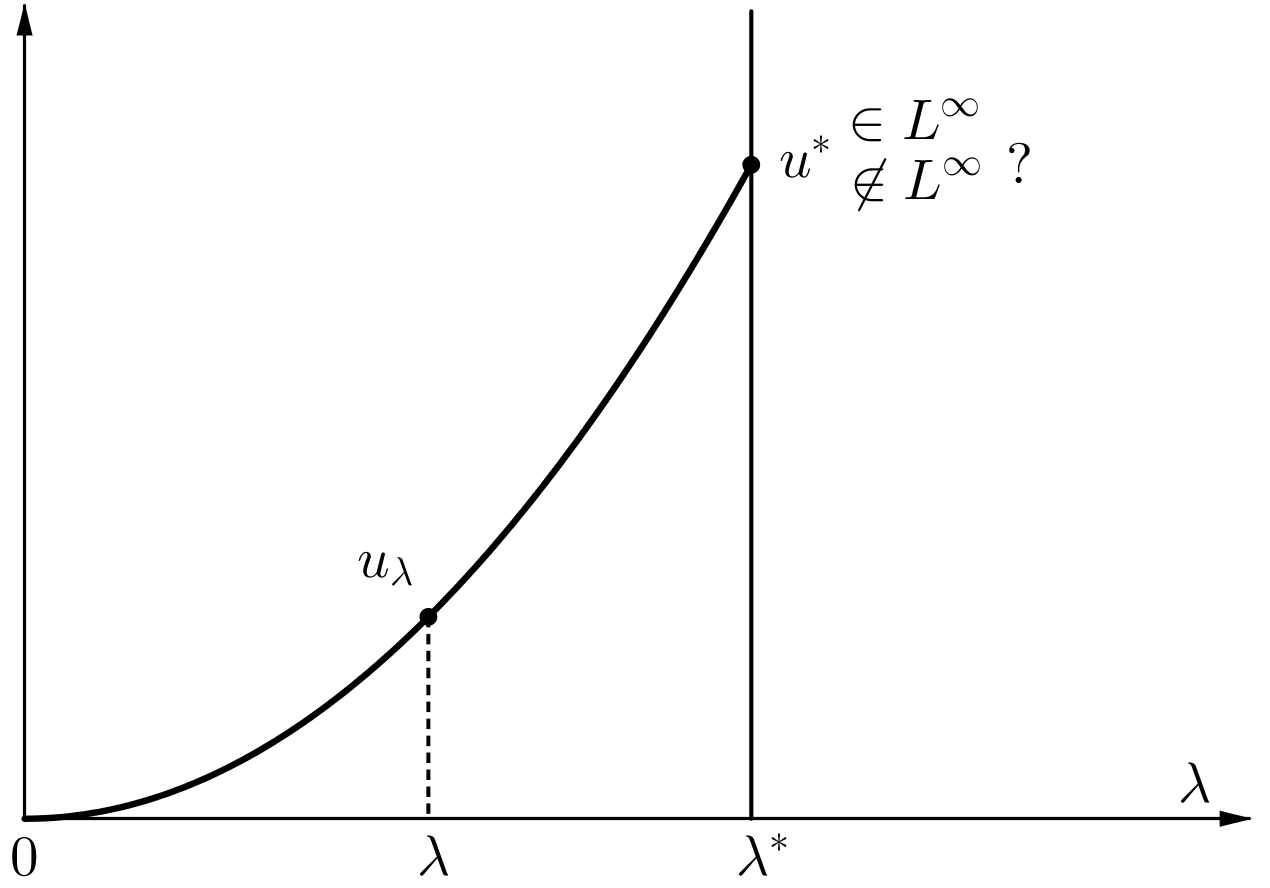}
\caption{The family of stable solutions $u_\la$ and the extremal solution $u^*$}
\label{fig:11}       
\end{figure}

The following is an explicit example of stable unbounded (or singular) solution.

It is easy to check that, for $n\geq 3$, the function $\tilde u=-2\log\vert{x}\vert$ is a solution of \eqref{pb} in $\Om =B_1$, 
the unit ball, for $f(u)= 2(n-2) e^u$. 
%
Let us now consider the linearized operator at $\tilde u$, which is given by  
$$
-\Delta -2(n-2)e^{\tilde u} =
-\Delta  -\frac{2(n-2)}{\vert{x}\vert^2}.
$$
If $n\ge 10$, then its first Dirichlet eigenvalue in $B_1$ is nonnegative.
This is a consequence of {\it Hardy's inequality} \eqref{Hardyin}:
$$
\frac{(n-2)^2}{4}\int_{B_1}\frac{\xi^2}{\vert{x}\vert^2} \, dx \;\; \leq\;\;
\int_{B_1} \vert\nabla \xi\vert^2 dx
\qquad\quad
\mbox{ for every } \xi\in H^1_0(B_1),
$$
and the fact that $2(n-2)\leq (n-2)^2/4$ if $n\geq 10$.
Thus we proved the following result.
\begin{proposition}\label{prop:dimensopt}
For $n\geq 10$, $\tilde u = -2\log\vert{x}\vert$ is an $H^1_0(B_1)$ stable weak solution of $-\De u = 2 (n-2) e^u $ in $B_1$, $u>0$ in $B_1$, $u=0$ on $\pa B_1$.
\end{proposition}

Thus, in dimensions $n\geq 10$ there exist unbounded $H^1_0$ stable weak solutions of \eqref{pb}, 
even in the unit ball and for the exponential nonlinearity.
It is believed that $n \ge 10$ could be the optimal dimension for this fact, as we describe next.

\subsection{Regularity of stable solutions. The Allard and Michael-Simon Sobolev inequality}

The following results give $L^{\infty}$ bounds for stable solutions. To avoid technicalities we state the bounds for the extremal solution but, more generally, they also apply to every stable weak solution of \eqref{pb} which is the pointwise limit of a sequence of bounded stable solutions to similar equations (see \cite{Dupaigne}). 


\begin{theorem}[Crandall-Rabinowitz \cite{CR}]
\label{thm:CranRab}
Let $u^*$ be the extremal solution of \eqref{pbla} with $f(u)=e^u$ or $f(u)=(1+u)^p$, $p>1$. If $n\le 9$, then $u^* \in L^\infty (\Om)$.
\end{theorem}

\begin{sketch}[in the case $f(u)= e^u$]
Use the equation in \eqref{pbla} for the classical solutions $u=u_\la$ ($\la < \la^*$), together with the stability condition \eqref{LN:stability} for the test function $\xi=e^{\al u} -1$ (for a positive exponent $\al$ to be chosen later).
More precisely, start from \eqref{LN:stability} -- with $f'$ replaced by $\la f'$ -- and to proceed with $\int_\Om \al^2 e^{2 \al u} |\na u|^2$, write $ \al^2 e^{2 \al u} | \na u|^2 = ( \al / 2) \na \left( e^{2 \al u} -1 \right) \na u$, and integrate by parts to use \eqref{pbla}.
For every $\al < 2$, verify that this leads, after letting $\la \uparrow \la^*$,
to $e^{u^*} \in L^{2 \al +1} (\Om)$. As a consequence, by Calder\'{o}n-Zygmund theory and Sobolev embeddings, $u^* \in W^{2, 2 \al +1} (\Om) \subset L^\infty (\Om)$ if $2 ( 2 \al +1) > n$. This requires that $n \le 9$.
\qed
\end{sketch}

Notice that the nonlinearities $f(u)=e^u$ or $f(u)=(1+u)^p$ with $p>1$ satisfy~\eqref{hypf}.

In the radial case $\Om=B_1$ we have the following result.

\begin{theorem}[Cabr\'e-Capella \cite{CabrCappLN}]
\label{corolrad}
Let $u^*$ be the extremal solution of \eqref{pbla}. Assume that $f$ satisfies \eqref{hypf} and that $\Om=B_1$.
If $1\leq n\leq 9$, then $u^* \in L^\infty (B_1)$. 
\end{theorem}

As mentioned before, this theorem also holds for every $H_0^1 (B_1)$ stable weak solution of \eqref{pb}, for any $f \in C^1$.
Thus, in view of Proposition \ref{prop:dimensopt}, the dimension $n \le 9$ is optimal in this result.

We turn now to the nonradial case and we present the currently known results. First, in 2000 Nedev solved the case $n \le 3$.

\begin{theorem}[Nedev \cite{ND}]\label{thm3pre}
Let $f$ be convex and satisfy \eqref{hypf}, and $\Omega\subset\RR^{n}$ be a smooth bounded domain.
If $n\leq 3$, then $u^* \in L^\infty(\Omega)$.
\end{theorem}

In 2010, Nedev's result was improved to dimension four:

\begin{theorem}[Cabr\'e \cite{Cabre}; Villegas \cite{V}]\label{thm3}
Let $f$ satisfy \eqref{hypf}, $\Omega\subset\RR^{n}$ be a smooth bounded domain, and $1 \le n\leq 4$.
If $n\in \{3,4\}$ assume either that $f$ is a convex nonlinearity or that $\Omega$ is a convex domain. 
Then, $u^* \in L^\infty(\Omega)$.
\end{theorem}

For $3\leq n \leq 4$, \cite{Cabre} requires $\Omega$ to be convex, while $f$ needs not be
convex. Some years later, S. Villegas~\cite{V} succeeded to use both \cite{Cabre} and \cite{ND} when $n=4$
to remove the requirement that $\Omega$ is convex by further assuming that $f$ is convex.

\begin{open problem}
For every $\Om$ and for every $f$ satisfying \eqref{hypf}, is the extremal solution $u^*$ -- or, in general, $H_0^1$ stable weak solutions of \eqref{pb} -- always bounded in dimensions $5,6,7,8,9$?
\end{open problem}

We recall that the answer to this question is affirmative when $\Om = B_1$, by Theorem \ref{corolrad}. We next sketch the proof of this radial result, as well as the regularity theorem in the nonradial case up to $n \le 4$.
In the case $n=4$, we will need the following remarkable result.

\begin{theorem}[Allard; Michael and Simon]
\label{Sobolev}
Let  $M\subset  \RR^{m+1}$ be an immersed smooth $m$-dimensional 
compact hypersurface without boundary.

Then, for every $p\in  [1, m)$,
there exists a constant $C = C(m,p)$ depending only on the dimension $m$ 
and the exponent $p$ such that, for every $C^\infty$ function $v : M  \to \RR$, 
\begin{equation}
\label{MSsob}
\left( \int_M |v|^{p^*} \,dV \right)^{1/p^*} \leq C(m,p)
\left( \int_ M (|\nabla v|^p +  | \cH v|^p) \,dV \right)^{1/p},
\end{equation}
where $\cH$ is the mean curvature of $M$ and $p^* = mp/(m - p)$.
\end{theorem}

This theorem dates from 1972 and has its origin in an important result of Miranda from 1967. It stated that \eqref{MSsob} holds with $\cH=0$ if $M$ is a minimal surface in $\RR^{m+1}$.
See the book \cite{Dupaigne} for a proof of Theorem \ref{Sobolev}.

\begin{remark}
Note that this Sobolev inequality contains a term involving the mean curvature of $M$ on its right-hand side. This fact makes, in a remarkable way, that the constant $C(m,p)$ in the inequality does not depend on the geometry of the manifold~$M$.
\end{remark}

\begin{sketch}[of Theorems \ref{corolrad}, \ref{thm3pre}, and \ref{thm3}] 
For Theorem \ref{thm3pre} the test function to be used is $\xi=h(u)$, for some $h$ depending on $f$ (as in the proof of Theorem~\ref{thm:CranRab}).

Instead, for Theorems \ref{corolrad} and \ref{thm3}, the proofs start by writing 
the stability condition \eqref{LN:stability} for the test 
function $\xi= \tilde{c} \eta$, where $\eta|_{\partial\Omega}\equiv0$.
This was motivated by the analogous computation that we have presented for minimal surfaces right after Remark \ref{rem:Jacobi operator}.
Integrating by parts, one easily deduces that
\begin{equation}\label{stabc}
\int _{\Omega} \left( \Delta \tilde{c} +f'(u) \tilde{c} \right) \tilde{c} \eta ^{2}\,dx \le
\int _{\Omega} \tilde{c}^{2}\left|\nabla \eta\right|^{2}\,dx.
\end{equation}
Next, a key point is to choose a function $\tilde{c}$ satisfying an appropriate equation for
the linearized operator $\Delta + f'(u)$. In the radial case (Theorem \ref{corolrad}) the choice of $\tilde{c}$ and the final choice of $\xi$ are
\begin{equation*}
\tilde{c} =u_{r} \quad\text{ and }\quad\xi=u_r r(r^{-\alpha}-(1/2)^{-\alpha})_{+} ,
\end{equation*}
where $r=|x|$, $\alpha>0$, and $\xi$ is later truncated near the origin to make it Lipschitz.
The proof in the radial case is quite simple after computing the equation satisfied by $u_r$.


For the estimate up to dimension 4 in the nonradial case (Theorem~\ref{thm3}), \cite{Cabre} takes
\begin{equation}\label{choice4}
\tilde{c} =\left|\nabla u\right| \quad\text{ and }\quad\xi=\left|\nabla u\right|\varphi(u) ,
\end{equation}
where, in dimension $n=4$, $\varphi$ is chosen depending on the solution $u$ itself.

We make the choice \eqref{choice4} and, in particular, we take
$\tilde{c} =\left|\nabla u\right|$ in \eqref{stabc}.
It is easy to check that, in the set 
$\left\{\left|\nabla u\right|>0\right\}$, we have
\begin{equation}\label{dipassaggio}
\left(\Delta + f'(u)\right)|\nabla u|=\frac{1}{|\nabla u|}
\left(\sum_{i,j}u_{ij}^2-\sum_i\left(\sum_{j}u_{ij}\frac{u_j}
{|\nabla u|}\right)^2\right).
\end{equation}
Taking an orthonormal basis in which the last vector is the normal $\nabla u/|\nabla u|$ to the level set of $u$ (through a given point $x \in \Om$), and the other vectors are the principal directions of the level set at $x$, one easily sees that \eqref{dipassaggio} can be written as
\begin{equation}\label{eq:grad}
\left( \Delta+f'(u)\right)\left|\nabla u\right|=
\frac{1}{\left|\nabla u\right|} \left(\left|\nabla_T\left|\nabla u\right|\right|^2+
\left|A\right|^2\left|\nabla u\right|^2\right)\quad \text{in}\ 
\Omega\cap\left\{\left|\nabla u\right|>0\right\},
\end{equation}
where $\left|A\right|^2=\left|A\left(x\right)\right|^2$ is the squared norm of 
the second fundamental form of the level set of $u$ passing through a 
given point $x\in\Omega\cap\left\{\left|\nabla u\right|>0\right\}$, i.e., the 
sum of the squares of the principal curvatures of the level set. In the notation of the first section on minimal surfaces, $|A|^2 = c^2$.
On the other hand, as in that section $\nabla_T = \de$ denotes the tangential gradient to the level set. 
Thus, \eqref{eq:grad} involves geometric information of the level sets 
of $u$.

Therefore, using the stability condition \eqref{stabc},
we conclude that
\begin{equation}\label{semi1}
\int_{\left\{\left|\nabla u\right|>0\right\}} 
\left( |\nabla_T |\nabla u||^2 +|A|^2|\nabla u|^2\right)\eta^2\,dx
\leq \int_\Omega |\nabla u|^2 |\nabla \eta|^2 \,dx .
\end{equation}

Let us define
$$
T:=\max_{\ol{\Omega} }u=\left\|u\right\|_{L^\infty (\Omega)}
\quad\text{ and }\quad \Gamma_s:=\left\{x\in\Omega:u(x)=s\right\}
$$
for $s\in(0,T)$.  

We now use \eqref{semi1} with  
$\eta=\varphi(u)$, 
where $\varphi$ is a Lipschitz function in $\left[0,T\right]$ 
with $\varphi(0)=0$.
The right hand side of \eqref{semi1} becomes
\begin{eqnarray*}
\int_{\Omega}\left|\nabla u\right|^2\left|\nabla\eta\right|^2dx&=&
\int_{\Omega}\left|\nabla u\right|^4\varphi'(u)^2dx\\
&=&\int_{0}^{T}\left(\int_{\Gamma_s}\left|\nabla u\right|^3\,dV_s
\right)\varphi'(s)^2\,ds,
\end{eqnarray*}
by the {\it coarea formula}. Thus, \eqref{semi1} can be written as
\begin{eqnarray*}
&&\hspace{-1cm} \int_{0}^{T}\left(\int_{\Gamma_s}\left|\nabla u\right|^3\,dV_s
\right)\varphi'(s)^2\,ds\\
&&\geq\int_{\left\{\left|\nabla 
u\right|>0\right\}}\left(\left|\nabla_T\left|
\nabla u\right|\right|^2+\left|A\right|^2\left|\nabla 
u\right|^2\right)\varphi(u)^2dx\\
&&=\int_{0}^{T}\left(\int_{\Gamma_s\cap\left\{\left|\nabla 
u\right|>0\right\}}\frac{1}{\left|\nabla u\right|}
\left(\left|\nabla_T\left|
\nabla u\right|\right|^2+\left|A\right|^2\left|\nabla 
u\right|^2\right)\,dV_s\right)\varphi(s)^2\,ds\\
&&=\int_{0}^{T}\left(\int_{\Gamma_s\cap\left\{\left|\nabla 
u\right|>0\right\}}
\left( 4\left|\nabla_T\left|
\nabla u\right|^{1/2}\right|^2+\left(\left|A\right|\left|\nabla 
u\right|^{1/2}\right)^{2}\right) \,dV_s\right)\varphi(s)^2\,ds .
\end{eqnarray*}
We conclude that
\begin{equation}
\int_0^T h_1(s) \varphi(s)^2 \,ds
\leq
\int_0^T h_2(s)  \varphi'(s)^2 \,ds,
\label{semi3}
\end{equation}
for all Lipschitz functions 
$\varphi:\left[0,T\right]\rightarrow\mathbb{R}$ with 
$\varphi(0)=0$, where 
\begin{equation*}
h_1(s):=\int_{\Gamma_s} 
\left( 4|\nabla_T |\nabla u|^{1/2}|^2 +\left( |A||\nabla u|^{1/2} \right)^2\right) \,dV_s\, ,\
h_2(s):=\int_{\Gamma_s} |\nabla u|^3\,dV_s 
\end{equation*}
for every regular value $s$ of $u$.
We recall that, by Sard's theorem, almost every $s\in(0,T)$ is a regular value of $u$.

Inequality \eqref{semi3}, with $h_1$ and $h_2$ as defined above, leads to a bound for $T$
(that is, to an $L^{\infty}$ estimate and hence to Theorem \ref{thm3}) after choosing an 
appropriate test function $\varphi$ in \eqref{semi3}. 
In dimensions 2 and 3 we can choose a simple function $\varphi$ in \eqref{semi3} and use well known geometric inequalities 
about the curvature of manifolds (note that $h_1$ involves the curvature 
of the level sets of $u$). Instead, in dimension 4 we need to use the geometric Sobolev inequality of Theorem \ref{Sobolev} on each level set of $u$.
Note that $\cH^2 \le (n-1) |A|^2$.
This gives the following lower bound for $h_1 (s)$:
$$
c(n) \left( \int_{\Ga_s} | \na u |^{\frac{n-1}{n-3} }  \right)^{\frac{n-3}{n-1}} \le h_1 (s).
$$
Comparing this with $h_2 (s)$, which appears in the right hand side of \eqref{semi3}, we only know how to derive an $L^{\infty}$-estimate for $u$ (i.e., a bound on $T= \max u$) when the exponent $(n-1) / (n-3)$ in the above inequality is larger than or equal to the exponent $3$ in $h_2 (s)$.
This requires $n \le 4$.
See \cite{Cabre} for details on how the proof is finished.
\qed
\end{sketch}


\section{Appendix: a calibration giving the optimal isoperimetric inequality}
%
%

Our first proof of Theorem \ref{thm:SimCone} used a calibration. To understand better the concept and use of ``calibrations'', we present here another one. It leads to a proof of the isoperimetric problem.

The isoperimetric problems asks which sets in $\RR^n$ minimize perimeter for a given volume. Making the first variation of perimeter (as in Section \ref{sec:mincones}), but now with a volume constraint, one discovers that a minimizer $\Om$ should satisfy $\cH=c$ (with $c$ a constant), at least in a weak sense, where $\cH$ is the mean curvature of $\pa \Om$. Obviously, balls satisfy this equation -- they have constant mean curvature.
The isoperimetric inequality states that the unique minimizers are, indeed, balls. In other words, we have:
\begin{theorem}[The isoperimetric inequality]\label{thm:servepertesiLNisoperimetric}
We have
\begin{equation}
\label{eq:isoperimetric}
\frac{|\pa \Om|}{| \Om|^{\frac{n-1}{n}}} \ge \frac{|\pa B_1|}{| B_1|^{\frac{n-1}{n}}} 
\end{equation}
for every bounded smooth domain $\Om \subset \RR^n$.
In addition, if equality holds in \eqref{eq:isoperimetric}, then $\Om$ must be a ball.
\end{theorem}

In 1996 the first author found the following proof of the isoperimetric problem. It uses a calibration (for more details see \cite{Caisoperimetric}).

\begin{sketch}[of the isoperimetric inequality]
The initial idea was to characterize the perimeter $|\pa \Om|$ as in \eqref{calideph}-\eqref{eq:dimcalibration1}, that is, as
$$
| \pa \Om | = \sup_{ \nr X \nr_{L^\infty} \le 1} \int_{ \pa \Om} X \cdot \nu \, dH_{n-1} .
$$
Taking $X$ to be a gradient, we have that
$$
| \pa \Om | = \int_{ \pa \Om} \na u \cdot \nu \, dH_{n-1} = \int_{\pa \Om} u_\nu \, dH_{n-1} ,
$$
for every function $u$ such that $u_\nu=1$ on $\pa \Om$.
Let us take $u$ to be the solution of
\begin{equation}\label{eqappisop}
\left\{
\begin{array}{rcll}
\De u &=& c & \quad\mbox{in } \Om \\
u_\nu &=& 1  & \quad\mbox{on } \pa\Om ,\\
\end{array}\right.
\end{equation}
where $c$ is a constant that, by the divergence theorem, is given by
$$
c= \frac{|\pa \Om|}{| \Om| }.
$$
It is known that there exists a unique solution $u$ to \eqref{eqappisop} (up to an additive constant).

Now let us see that $X=\na u$ (where $X$ was the notation that we used in the proof of Theorem \ref{thm:SimCone}) can play the role of a calibration.
In fact, in analogy with Definition~\ref{def:calibrationprima} we have:

\begin{enumerate}[]
\item(i-bis) \,$\dv \na u = \frac{|\pa \Om|}{| \Om| } \, \mbox{ in } \, \Om $;
\item(ii-bis)  \,$\na u \cdot \nu = 1 \, \mbox{ on } \, \pa \Om$;
\item(iii-bis)  \,$B_1 (0) \subset \na u (\Ga_u)$, where
$$\Ga_u = \left\{ x \in \Om : u(y) \ge u(x) + \na u(x) \cdot (y-x)  \, \mbox{ for every } \, y \in \ol{\Om}  \right\} $$
is the {\it lower contact set of $u$}, that is, the set of the points of $\Om$ at which the tangent plane to $u$ stays below $u$ in $\Om$. 
\end{enumerate}
The relations (i-bis) and (ii-bis) follow immediately from \eqref{eqappisop}. 
In the following exercise, we ask to establish (iii-bis) and finish the proof of \eqref{eq:isoperimetric}.

We point out that this proof also gives that $\Om$ must be a ball if equality holds in \eqref{eq:isoperimetric}.
\qed
\end{sketch}

\begin{exercise}
Establish (iii-bis) above. For this, use a foliation-contact argument (as in the alternative proof of Theorem \ref{thm:SimCone} and in the proof of Theorem \ref{alba}), foliating now $\RR^n \times \RR$ by parallel hyperplanes.

Next, finish the proof of \eqref{eq:isoperimetric}. For this, consider the measures of the two sets in (iii-bis), compute $| \na u (\Ga_u)|$ using the {\it area formula}, and control $\det D^2 u$ using the geometric-arithmetic means inequality.

\end{exercise}

%
%
%
%

%
\begin{quote}
{\bf Acknowledgments.}
{\small The authors wish to thank Lorenzo Cavallina for producing the figures of this work.

The first author is member of the Barcelona Graduate School of Mathematics and is supported by MINECO grants MTM2014-52402-C3-1-P and MTM2017-84214-C2-1-P. He is also part of the Catalan research group 2017 SGR 1392.

The second author was partially supported by PhD funds of the Universit\`{a} di Firenze and he is a member of the Gruppo Nazionale Analisi Matematica Probabilit\'{a} e Applicazioni (GNAMPA) of the Istituto Nazionale di Alta Matematica (INdAM). This work was partially written while the second author was visiting the Departament de Matem\`{a}tiques of the Universitat Polit\`ecnica de Catalunya, that he wishes to thank for hospitality and support.}
\end{quote}

%% file: ChapterMPLittlewoodTesi.tex
\chapter{Littlewood's fourth principle}
\label{chap:Littlewoodfourthprinciple}
\centerline{Rolando Magnanini\footnote{Dipartimento di Matematica e Informatica ``U.~Dini'',
Universit\` a di Firenze, viale Morgagni 67/A, 50134 Firenze, Italy
({\tt magnanin@math.unifi.it}).}
and Giorgio Poggesi\footnote{
Dipartimento di Matematica ed Informatica ``U.~Dini'',
Universit\` a di Firenze, viale Morgagni 67/A, 50134 Firenze, Italy
({\tt giorgio.poggesi@unifi.it}).
}
}

\vspace{1cm}

\begin{quotation}
{\bf Abstract.}
{\small In Real Analysis, Littlewood's three principles are known as heuristics that help teach the essentials of measure theory and reveal the analogies between the concepts of topological space and continuous function on one side and those of measurable space and measurable function on the other one. They are based on important and rigorous statements, such as Lusin's and Egoroff-Severini's theorems, and have ingenious and elegant proofs. We shall comment on those theorems and show how their proofs 
can possibly be made simpler by introducing a \textit{fourth principle}. These alternative
proofs make even more manifest those analogies and show that Egoroff-Severini's theorem can be considered as the natural generalization of the classical Dini's monotone convergence theorem. 
}
\end{quotation}

\section{Introduction.}

John Edenson Littlewood (9 June 1885 - 6 September 1977) was a British mathematician.
In 1944, he wrote an influential textbook, \textit{Lectures on the Theory of Functions} (\cite{Littlewood}), in which he proposed three principles as guides for working in real analysis;  these are heuristics to help teach the essentials of measure theory, as Littlewood himself wrote in \cite{Littlewood}:
\begin{quotation}
The extent of knowledge [of real analysis] required is nothing like so great as is sometimes supposed. There are three principles, roughly expressible in the following terms:
every (measurable) set is nearly a finite sum of intervals;
every function (of class $L^{\lambda}$) is nearly continuous;
every convergent sequence is nearly uniformly convergent.
Most of the results of the present section are fairly intuitive applications of these ideas, and the student armed with them should be equal to most occasions when real variable theory is called for. If one of the principles would be the obvious means to settle a problem if it were ``quite'' true, it is natural to ask if the ``nearly'' is near enough, and for a problem that is actually soluble it generally is.
\end{quotation}

To benefit our further discussion, we shall express Littlewood's principles 
and their rigorous statements in forms
that are slightly different from those originally stated.

The first principle descends directly from the very definition of Lebesgue measurability of a set.

\begin{principle1}
Every measurable set is nearly closed.
\end{principle1}

The second principle relates the measurability of a function to the more familiar property of continuity.

\begin{principle2}
Every measurable function is nearly continuous.
\end{principle2}

The third principle connects the pointwise convergence of a sequence of functions to
the standard concept of uniform convergence.

\begin{principle3}
Every sequence of measurable functions that converges pointwise almost everywhere is nearly 
uniformly convergent.
\end{principle3}

These principles are based on important theorems that give a rigorous meaning to the term ``nearly''.
We shall recall these in the next section along with their ingenious proofs that give a taste of the 
standard arguments used in Real Analysis. 
\par
In Section 3, we will discuss a {\it fourth principle} that associates the concept of finiteness
of a function to that of its boundedness.

\begin{principle4}
Every measurable function that is finite almost everywhere is nearly 
bounded.
\end{principle4}

In the mathematical literature (see \cite{DiB}, \cite{CD}, \cite{Littlewood}, \cite{Ma}, \cite{Royden}, \cite{Ru}, \cite{Ta}), the proof of the second principle is based on the third; it can be easily seen that the fourth principle can be derived from the second.
\par
However, we shall see that the fourth principle can also be proved independently;
this fact makes possible a proof of the second principle \textit{without} appealing for the third,
that itself can be derived from the second, by a totally new proof based on \textit{Dini's monotone convergence theorem}.
\par
 As in \cite{Littlewood}, to make our discussion as simple as possible, we shall consider the Lebesgue measure $m$ for the real line $\mathbb {R}$. 

\section{The three principles}
We recall the definitions of  {\it inner} and {\it outer measure} of
a set $E\subseteq\RR$: they are 
\begin{eqnarray*}
&&m_i(E)=\sup\{ |K|: K \mbox{ is compact and } K\subseteq E\},\\
&&m_e(E)=\inf\{|A| : A\mbox{ is open and } A\supseteq E\},
\end{eqnarray*}
where the number $|K|$ is the infimum of the total lengths of all the finite unions of open intervals that contain $K$; accordingly, $|A|$ is the supremum of the total lengths of all the finite unions of closed intervals contained in $A$.
It always holds that $m_i(E)\le m_e(E)$. The set $E$ is (Lebesgue) measurable if and only if 
$m_i(E)=m_e(E)$;
when this is the case, the {\it measure} of $E$ is $m(E)=m_i(E)=m_e(E)$; thus $m(E)\in[0,\infty]$ 
and it can be proved that $m$ is a measure on the $\sigma$-algebra of Lebesgue measurable subsets of $\RR$.
\par
By the properties of the supremum, it is easily seen that, for any pair of subsets $E$ and $F$ of $\RR$, $m_e(E\cup F)\le m_e(E)+m_e(F)$ and $m_e(E)\le m_e(F)$ if $E\subseteq F$.
\vskip.1cm
The first principle is a condition for the measurability of subsets of $\RR$. 

\begin{theorem}[First Principle]
\label{th:first}
Let $E\subset\RR$ be a set of finite outer measure. 
\par
Then, $E$ is measurable if and only if for every $\eps > 0$ there exist two sets $K$ and $F$, with 
$K$ closed (compact), $K \cup F = E$ and $m_e(F) < \eps$.
\end{theorem}

This is what is meant by \textit{nearly closed}.

\begin{proof}
If $E$ is measurable, for any $\eps>0$ we can find a compact set $K\subseteq E$ and an open
set $A\supseteq E$ such that
$$
m(K)>m(E)-\eps/2 \ \mbox{ and } \ m(A)<m(E)+\eps/2.
$$
The set $A\setminus K$ is open and contains $E\setminus K$. Thus, by setting $F=E\setminus K$, we have $E=K\cup F$ and
$$
m_e(F)\le m(A)-m(K)<\eps.
$$
\par
Viceversa, for every $\eps>0$ we have: 
$$
m_e(E)=m_e(K\cup F)\le m_e(K)+m_e(F)<m(K)+\eps\le m_i(E)+\eps.
$$
Since $\eps$ is arbitrary, then $m_e(E)\le m_i(E)$.
\end{proof}

The second and third principles concern measurable functions from (measurable) subsets of $\RR$
to the {\it extended real line} $\ovr{\RR}=\RR\cup\{+\infty\}\cup\{-\infty\}$,
that is functions are allowed to have values $+\infty$ and $-\infty$. 
\par
Let $f:E\to\ovr{\RR}$ be a function defined on a measurable subset $E$ of $\RR$. 
We say that $f$ is \textit{measurable} if the \textit{level sets} defined by
$$
L(f,t)=\{ x\in E: f(x)>t\}
$$
are measurable subsets of $\RR$ for every $t\in\RR$. It is easy to verify that if we replace $L(f,t)$ with $L^*(f,t)=\{ x\in E: f(x) \geq t\}$ we have an equivalent definition.  
\par
Since the countable union and intersection of measurable sets are measurable, it is not hard to show that
the pointwise infimum and supremum of a sequence of measurable functions $f_n:E\to\ovr{\RR}$ 
are measurable functions as well as the function defined for any $x\in E$ by
$$
\limsup_{n\to\infty} f_n(x)=\inf_{k\ge 1}\sup_{n\ge k} f_n(x).
$$
\par
Since the countable union of sets of measure zero has measure zero and the difference between $E$ and any set of measure zero is measurable, the same definitions and conclusions hold even
if the functions $f$ and $f_n$ are defined \textit{almost everywhere} (denoted for short by a.e.), that is
if the subsets of $E$ in which they are not defined have measure zero. In the same spirit,
we say that a function or a sequence of functions satisfies a given property \textit{a.e.} in $E$, if that property holds with the exception of a subset of measure zero.
\par
As already mentioned, the third principle is needed to prove the second and is
known as \textit{Egoroff's theorem} or \textit{Egoroff-Severini's theorem}.\footnote{Dmitri Egoroff, a Russian physicist and geometer and Carlo Severini, an Italian mathematician, published independent proofs of this theorem in 1910 and 1911
(see \cite{Eg} and \cite{Severini}); Severini's assumptions are more restrictive. Severini's result is not very well-known, since it is hidden in a paper on orthogonal polynomials, published in Italian.}

\begin{theorem}[Third Principle; Egoroff-Severini]
\label{th:Egoroff}
Let $E\subset \RR$ be a measurable set with finite measure and
let $f:E\to\ovr{\RR}$ be measurable and finite a.e. in $E$.

The sequence of measurable functions $f_n:E\to\ovr{\RR}$  converges a.e. to $f$ in $E$ for $n\to\infty$ if and only if, for every $\eps> 0$, there exists a closed set $K \subseteq E$ such that $m(E \setminus K) < \eps$ and $f_n$ converges uniformly to $f$ on $K$.
\end{theorem}

This is what we mean by \textit{nearly uniformly convergent}.

\begin{proof}
If $f_n \to f$ a.e. in $E$ as $n\to\infty$, the subset of $E$ in which $f_n \to f$ pointwise
has the same measure as $E$; hence, without loss of generality, we can assume that $f_n(x)$
converges to $f(x)$ for every $x\in E$.

Consider the functions defined by
\begin{equation}
\label{defgn}
g_n(x)=\sup_{k\ge n} |f_k(x)-f(x)|, \ \ x\in E
\end{equation}
and the sets
\begin{equation}
\label{defEnm}
E_{n,m} =\left\{x \in E : g_n(x) <\frac1{m} \right\}
\ \mbox{ for } \ n, m\in\NN.
\end{equation}
Observe that, if $x\in E$, then $g_n(x)\to 0$ as $n\to\infty$ and hence for any $m\in\NN$
$$
E=\bigcup_{n=1}^\infty E_{n,m}.
$$
As $E_{n,m}$ is increasing with $n$, the monotone convergence theorem implies that $m(E_{n,m})$ converges to $m(E)$ for $n\to\infty$  and for any $m \in\NN$.
Thus, for every $\eps>0$ and $m \in\NN$, there exists an index $\nu=\nu(\eps, m)$ such that
$
m(E \setminus E_{\nu,m}) < \eps/ {2^{m+1}}.
$

\par
The measure of the set $F = \bigcup\limits_{m = 1}^{\infty} (E \setminus E_{\nu,m})$ is arbitrary small, in fact
$$
m(F) \leq \sum_{m=1}^{\infty} m(E \setminus E_{\nu,m}) < \eps/ 2.
$$ 
Also, since $E\setminus F$ is measurable, by Thorem \ref{th:first}  there exists a compact set $K \subseteq E \setminus F$ such that $m(E \setminus F) - m(K) < \eps/ 2$, and hence 
$$
m(E \setminus K) = m(E\setminus F)+m(F) - m(K) < \eps.
$$
\par
Since $K \subseteq E \setminus F = \bigcap\limits_{m=1}^{\infty} E_{\nu(\eps, m), m}$ we have that 
$$
|f_n(x) - f(x)| < \frac{1}{m} \ \mbox{ for any } \ x\in K \ \mbox{ and } \ n \ge \nu(\eps, m),
$$ 
by the definitions of $E_{\nu,m}$ and $g_{n}$; this means that $f_n$ converges uniformly to $f$ on $K$ as $n\to\infty$.
\vskip.1cm
Viceversa, if for every $\eps> 0$ there is a closed set $K \subseteq E$ with
$m(E \setminus K) < \eps$ and $f_n \rightarrow f$ uniformly on $K$, then by
choosing $\eps=1/m$ we can say that there is a closed set $K_m \subseteq E$ such that $f_n \to f$ uniformly on $K_m$ and $m(E \setminus K_m) < 1 / m$. 
\par
Therefore, $f_n(x) \rightarrow f(x)$ for any $x$ in the set $F = \bigcup\limits_{m=1}^\infty K_m$ and 
$$
m(E \setminus F) = m\Bigl(\bigcap_{m=1}^\infty (E \setminus K_m)\Bigr) \leq m(E \setminus K_m) <\frac1{m} \ \mbox{ for any } \ m\in\NN,
$$ 
which implies that $m(E \setminus F) = 0$. Thus, $f_n\to f$ a.e. in $E$ as $n\to\infty$. 

\end{proof}

The second principle corresponds to \textit{Lusin's theorem} (see \cite{Lu}),\footnote{N. N. Lusin or Luzin was a student of Egoroff. For biographical notes on Egoroff and Lusin see \cite{GK}.} that we state
here in a form similar to Theorems \ref{th:first} and \ref{th:Egoroff}.

\begin{theorem}[Second Principle; Lusin]
\label{th:Lusin}
Let $E\subset \RR$ be a measurable set with finite measure and
let $f:E\to\ovr{\RR}$ be finite a.e. in $E$.
\par
Then, $f$ is measurable in $E$ if and only if, for every $\eps>0$, there exists a closed set $K \subseteq E$ such that $m(E \setminus K) < \eps$ and the restriction of $f$ to $K$ is continuous.
\end{theorem}

This is what we mean by \textit{nearly continuos}.
\vskip.1cm
The proof of Lusin's theorem is done by approximation by simple functions. A \textit{simple function}
is a measurable function that has a finite number of real values. 
If $c_1, \ldots, c_n$ are the \textit{distinct} values of a simple function $s$,
then $s$ can be conveniently represented as
$$
s = \sum\limits_{j = 1}^{n} c_{j}\mathcal{X}_{E_j},
$$
where $\mathcal{X}_{E_j}$ is the characteristic function of the set $E_j = \left\{ x \in E \textrm{ : } s(x) = c_j \right\}$. Notice that the $E_j$'s form a
covering of $E$ of pairwise disjoint measurable sets.
\par
Simple functions play 
a crucial role in Real Analysis; this is mainly due to the following result of which we shall
omit the proof.

\begin{theorem}[Approximation by Simple Functions, \cite{Royden}, \cite{Ru}]
\label{th:approximation}
Let $E\subseteq\RR$ be a measurable set and let $f: E\to [0,+\infty]$  be a measurable function.
\par
Then, there exists an increasing sequence of non-negative simple functions $s_n$ that 
converges pointwise to $f$ in $E$ for $n\to\infty$.
\par
Moreover, if $f$ is bounded, then $s_n$ converges to $f$ uniformly in $E$.
\end{theorem}

We can now give the proof of Lusin's theorem.

\begin{proof}
Any measurable function $f$ can be decomposed as $f=f^+-f^-$, where $f^+=\max(f,0)$ and 
$f^-=\max(-f,0)$ are measurable and non-negative functions. Thus, we can always suppose that $f$
is non-negative and hence, by Theorem \ref{th:approximation}, it 
can be approximated pointwise by a sequence of simple functions.
\par
We first prove that a simple function $s$ is nearly continuos. 
Since the sets $E_j$ defining $s$ are measurable, if we fix $\eps> 0$ we can find closed subsets $K_j$ of $E_j$ such that $m(E_j \setminus K_j) < \eps/n$ for $j=1,\dots, n$.
The union $K$ of the sets $K_j$ is also a closed set and, since the $E_j$'s cover $E$, we have that $m(E \setminus K) < \eps$. Since the closed sets $K_j$ are pairwise disjoint (as the $E_j$'s are pairwise disjoint) and $s$ is constant on $K_j$ for all $j=1,\dots,n$, we conclude that $s$ is continuous in $K$.
\par
Now, if $f$ is measurable and non-negative, 
let $s_n$ be a sequence of simple functions that converges pointwise to $f$ and fix an $\eps>0$. 
\par
As the $s_n$'s are nearly continuous, for any natural number $n$, there exists a closed set $K_n \subseteq E$ such that $m(E \setminus K_n) < \eps/{2^{n+1}}$ and $s_n$ is continuous in $K_n$. By Theorem \ref{th:Egoroff}, there exists a closed set $K_0 \subseteq E$ such that $m(E \setminus K_0) < \eps/2$ and $s_n$ converges uniformly to $f$ in $K_0$ as $n\to\infty$. Thus, in the set
$$
K = \bigcap\limits_{n=0}^\infty K_n
$$ 
the functions $s_n$ are all continuous and converge uniformly to $f$. Therefore $f$ is continuous in $K$
and
$$
m(E \setminus K) = m\Bigl(\bigcup\limits_{n=0}^\infty (E \setminus K_n)\Bigr) \leq \sum_{n=0}^{\infty} m(E \setminus K_n) < \eps.
$$
\par
Viceversa, if $f$ is nearly continuous, fix an $\eps>0$ and let $K$ be a closed subset of $E$ such that $m(E \setminus K) < \eps$ and $f$ is continuous in $K$. For any $t \in \RR$, we have:
 $$
L^*(f,t)= \left\{x \in K : f(x) \geq t\right\} \cup \left\{x \in E \setminus K : f(x) \geq t\right\}.
$$ 
The former set in this decomposition is closed, as the restriction of $f$ to $K$ is continuous, 
while the latter is clearly a subset of $E \setminus K$ and hence its outer measure must be less than $\eps$. By Theorem \ref{th:first}, $L^*(f,t)$ is measurable (for any $t\in\RR$), which means that $f$
is measurable. 
\end{proof}

\section{The fourth principle}

We shall now present alternative proofs of Theorems \ref{th:Egoroff} and \ref{th:Lusin}.  
They are based on a fourth principle, that corresponds to the following theorem.

\begin{theorem}[Fourth Principle]
\label{th:fourth}
Let $E\subset \RR$ be a measurable set with finite measure and let $f:E\to\ovr{\RR}$ be a measurable function.
\par
Then, $f$ is finite a.e. in $E$ if and only if, for every $\eps > 0$, there exists a closed set $K \subseteq E$ such that $m(E \setminus K) < \eps$ and $f$ is bounded on $K$.
\end{theorem}

This is what we mean by \textit{nearly bounded}.
\begin{proof}
If $f$ is finite a.e., we have that 
$$
m(\left\{x \in E : |f(x)| = \infty\right\})=0.
$$
As $f$ is measurable, $|f|$ is also measurable and so are the sets 
$$
L(|f|,n) = \left\{x \in E :|f(x)| > n\right\}, \ n\in\NN.
$$ 
Observe that the sequence of sets $L(|f|,n)$ is decreasing and
$$
\bigcap\limits_{n=1}^\infty L(|f|,n) = \left\{x \in E :|f(x)| = \infty\right\}.
$$ 
As $m(L(|f|,1)) \leq m(E) < \infty$, we can
apply the (downward) monotone convergence theorem and infer that
$$
\lim_{n \to \infty} m(L(|f|,n)) = m(\left\{x \in E :|f(x)| = \infty\right\}) = 0.
$$ 
\par
Thus, if we fix $\eps>0$,  there is an $n_{\eps} \in \mathbb N$ such that $m(L(|f|,n_{\eps})) < \frac{\eps}{2}$. Also, we can find a closed subset $K$ of the measurable set
$E \setminus L(|f|,n_{\eps})$ such that $m(E \setminus L(|f|,n_{\eps}))-m(K)< \frac{\eps}{2}$. Finally, since $K \subseteq E \setminus L(|f|,n_{\eps})$, $|f|$ is obviously bounded by $n_\eps$ on $K$
and
$$
m(E\setminus K)=m(E \setminus L(|f|,n_{\eps}))+m(L(|f|,n_{\eps})\setminus K)<\eps.
$$
\par
Viceversa, if $f$ is nearly bounded, then for any $n \in\NN$ there exists a closed set $K_n \subseteq E$ such that $m(E \setminus K_n) < 1/n$ and $f$ is bounded (and hence finite) in $K_n$.
Thus, $\left\{x \in E :|f(x)| =\infty\right\} \subseteq E \setminus K_n$ for any $n \in\NN$, and hence $$
m(\left\{x \in E:|f(x)| = \infty\right\}) \leq \lim_{n\to\infty} m(E \setminus K_n) =0,
$$ 
that is $f$ is finite a.e..
\end{proof}

\begin{remark}
Notice that this theorem can also be derived from Theorem \ref{th:Lusin}. 
In fact,
without loss of generality, the closed set $K$ provided by Theorem \ref{th:Lusin} can be 
taken to be compact and hence, $f$ is surely bounded on $K$, being continuous on a compact set.
\end{remark}

More importantly for our aims, Theorem \ref{th:fourth} enables us to prove Theorem \ref{th:Lusin} \textit{without} using Theorem \ref{th:Egoroff}.

\begin{proof}[Alternative proof of Theorem \ref{th:Lusin}]
The proof runs similarly to that presented in Section 2. If $f$ is measurable, without loss of generality, we can assume that $f$ is non-negative and hence $f$ can be approximated pointwise by a sequence of simple functions $s_n$, which we know are nearly continuous. Thus, for any $\eps>0$, we can still construct the sequence of closed subsets $K_n$ of $E$ such that $m(E \setminus K_n) < \eps/{2^{n+1}}$ and $s_n$ is continuous in $K_n$. 
\par
Now, as $f$ is finite a.e., Theorem \ref{th:fourth} implies that it is nearly bounded, that is we can find a closed subset $K_0$ of $E$ in which $f$ is bounded and $m(E \setminus K_0) < \eps/ 2$. 
We apply the second part of the Theorem \ref{th:approximation} and infer that $s_n$ converges uniformly to $f$ in $K_0$. As seen before, we conclude that $f$ is continuous in the intersection $K$ of all the $K_n$'s,
because in $K$ it is the uniform limit of  the sequence of continuous functions $s_n$. As before
$m(E\setminus K)<\eps$.
\par
The reverse implication remains unchanged.
\end{proof}

In order to give our alternative proof of Theorem \ref{th:Egoroff}, we need to recall
a classical result for sequences of continuous functions.

\begin{theorem}[Dini]
\label{th:Dini}
Let $K$ be a compact subset of $\RR$ and let be given a sequence of continuous functions $f_n:K\to \RR$ that converges pointwise and monotonically in $K$ to a function $f:K\to\RR$.
\par
If $f$ is also continuous, then $f_n$ converges uniformly to $f$.
\end{theorem}

\begin{proof}
We shall prove the theorem when $f_n$ is monotonically increasing. 

For each $n \in \NN$, set $h_n = f - f_n$; as $n\to\infty$ the continuos functions $h_n$ decrease pointwise to $0$ on $K$.

Fix $\eps>0$. The sets $A_n = \left\{x \in K:h_n(x) < \eps\right\}$ are relatively open in $K$, since the $h_n$'s are continuous; also,  $A_n \subseteq A_{n+1}$ for every $n\in\NN$, since the ${h_n}$'s decrease; finally,
the $A_n$'s cover $K$, since the $h_n$ converge pointwise to $0$. 
\par
By compactness, $K$ is then covered by a finite number $m$ of the $A_n$'s, which means that $A_m = K$ for some $m\in\NN$. This implies that  $|f(x) - f_n(x)| < \eps$ for all $n \geq m$ and $x \in K$, as desired.
\end{proof}

\begin{remark}
\label{semicont}
The conclusion of Theorem \ref{th:Dini} still holds true if we assume that the sequence of $f_n$'s is increasing (respectively decreasing) and $f$ and all the $f_n$'s are lower (respectively upper) semicontinuous. (We say that $f$ is lower semicontinuous if the level sets
$L_+(f,t)$ are open for every $t \in \RR$; $f$ is upper semicontinuous if $-f$ is lower semicontinuous.)
\end{remark}

Now, Theorem \ref{th:Egoroff} can be proved by appealing to Theorems \ref{th:Lusin} and \ref{th:Dini}.

\begin{proof}[Alternative proof of Theorem \ref{th:Egoroff}]
As in the classical proof of this theorem, we can always assume that $f_n(x) \to f(x)$ for every $x\in E$.
\par
Consider the functions and sets defined in \eqref{defgn} and \eqref{defEnm}, respectively.
We shall first show that there exists an $\nu\in\NN$ such that $g_n$ is nearly bounded for every
$n\ge \nu$. In fact, as already observed, since $g_n\to 0$ pointwise in $E$ as $n\to\infty$, we have that  
$$
E=\bigcup_{n=1}^\infty E_{n,1},
$$
and the $E_{n,1}$'s increase with $n$. Hence, if we fix $\eps>0$, there is a $\nu\in\NN$ such that $m(E\setminus E_\nu)<\eps/2$. Since
$E_\nu$ is measurable, by Theorem \ref{th:first} we can find a closed subset $K$ of $E_\nu$ such that
$m(E_\nu\setminus K)<\eps/2$.
\par
Therefore, $m(E\setminus K)<\eps$ and for every $n\ge \nu$
$$
0\le g_n(x)\le g_\nu(x)<1, \ \mbox{ for any } \ x\in K.
$$
\par
Now, being $g_n$ nearly bounded  in $E$ for every $n\ge \nu$, the alternative proof of Theorem \ref{th:Lusin} implies that $g_n$ is nearly continuous in $E$, that is for every $n\ge \nu$ there exists a closed
subset $K_n$ of $E$ such that $m(E\setminus K_n)<\eps/2^{n-\nu+1}$ and $g_n$ is continuous on $K_n$. 
The set 
$$
K=\bigcap_{n=\nu}^\infty K_n
$$
is closed, $m(E\setminus K)<\eps$ and on $K$ the functions $g_n$ are continuos for any $n\ge \nu$ 
and monotonically descrease to $0$ as $n\to\infty$.
\par
By Theorem \ref{th:Dini}, the $g_n$'s converge to $0$ uniformly on $K$.
This means that the $f_n$'s converge to $f$ uniformly on $K$ as $n\to\infty$. 
\par
The reverse implication remains unchanged.
\end{proof}

\begin{remark}
Egoroff's theorem can be considered, in a sense, as the \textit{natural} substitute of Dini's theorem, 
in case the monotonicity assumption is removed.
In fact, notice that the sequence of the $g_n$'s defined in \eqref{defgn} is decreasing; however, the $g_n$'s are in general no longer upper semicontinuous (they are only lower semicontinuous) and Dini's theorem (even in the form described in Remark \ref{semicont}) cannot be applied. In spite of that, the $g_n$'s \textit{remain} measurable if the $f_n$'s are so. 
\end{remark}

Of course, all the proofs presented in Sections 2 and 3 still work if we replace the real line $\RR$
by an Euclidean space of any dimension.
\vskip.1cm
Theorems \ref{th:Egoroff}, \ref{th:Lusin} and \ref{th:fourth} can also be extended to general measure spaces not necessarily endowed with a topology: the intersted reader can refer to \cite{CD}, \cite{Ox}, \cite{Royden} and \cite{Ta}.